\documentclass{amsart}

\usepackage{enumerate, amsmath, amsfonts, amssymb, amsthm, wasysym, graphics, graphicx, xcolor, url, a4wide, pdflscape, multido, xargs, colortbl, multicol, multirow, calc, shuffle, ifthen, hvfloat, xstring, stmaryrd, blkarray}
\usepackage[normalem]{ulem}
\usepackage{hyperref, hypcap}
\hypersetup{colorlinks=true, citecolor=blue, linkcolor=blue}
\usepackage[noabbrev,capitalise]{cleveref}

\usepackage{tikz}\usetikzlibrary{trees, snakes, shapes, arrows, matrix, calc, graphs}
\tikzset{blueNode/.style={blue, edge from parent/.style={blue, draw}}}
\tikzset{redNode/.style={red, edge from parent/.style={red, draw}}}
\tikzset{greenNode/.style={green, edge from parent/.style={green, draw}}}
\tikzset{cyanNode/.style={cyan, edge from parent/.style={cyan, draw}}}
\tikzset{blueEdge/.style={edge from parent/.style={blue, draw}}}
\tikzset{redEdge/.style={edge from parent/.style={red, draw}}}
\tikzset{violetEdge/.style={edge from parent/.style={violet, draw}}}
\graphicspath{{figures/}{figures/exmTwist/}}

\usepackage{etoolbox}
\usetikzlibrary{cd}

\title{Signaletic operads}
\thanks{This research was partially supported by the Labex DigiCosme (project ANR-11-LABEX-0045-DIGICOSME) operated by the ANR as part of the program ``Investissement d'Avenir'' Idex Paris-Saclay (ANR-11-IDEX-0003-02).}

\author[F.~Hivert]{Florent Hivert}
\address[F.~Hivert]{Laboratoire de Recherche en Informatique (LRI) -- Universit\'e Paris-Sud, CNRS, Universit\'e Paris-Saclay -- B\^atiment 650, Universit\'e Paris Sud, 91405 ORSAY CEDEX}
\email{florent.hivert@lri.fr}
\urladdr{http://www.lri.fr/~hivert/}

\author[V.~Pilaud]{Vincent Pilaud} 
\address[V.~Pilaud]{Universitat de Barcelona}
\email{vincent.pilaud@ub.edu}
\urladdr{https://www.ub.edu/comb/vincentpilaud/}
\thanks{VP was partially supported by the French ANR grants SC3A~(15\,CE40\,0004\,01), CAPPS~(17\,CE40\,0018), and CHARMS (ANR-19-CE40-0017), by the French--Austrian project PAGCAP (ANR-21-CE48-0020 \& FWF I 5788),  by the Spanish grant PID2022-137283NB-C21 of MCIN/AEI/10.13039/501100011033 / FEDER, UE, and by Departament de Recerca i Universitats de la Generalitat de Catalunya (2021 SGR 00697).}

\newtheorem{theorem}{Theorem}[section]
\newtheorem{corollary}[theorem]{Corollary}
\newtheorem{proposition}[theorem]{Proposition}
\newtheorem{lemma}[theorem]{Lemma}
\newtheorem{definition}[theorem]{Definition}
\newtheorem{conjecture}[theorem]{Conjecture}
\Crefname{conjecture}{Conjecture}{Conjectures}

\theoremstyle{definition}
\newtheorem{example}[theorem]{Example}
\newtheorem{remark}[theorem]{Remark}

\newcommand{\N}{\mathbb{N}} 
\newcommand{\K}{\mathbb{K}} 
\newcommand{\C}{\mathbb{C}} 
\newcommand{\fS}{\mathfrak{S}} 

\newcommand{\set}[2]{\left\{ #1 \;\middle|\; #2 \right\}} 
\newcommand{\bigset}[2]{\big\{ #1 \mid #2 \big\}} 
\newcommand{\ssm}{\smallsetminus} 
\newcommand{\dotprod}[2]{\langle #1 | #2 \rangle} 
\newcommand{\one}{{1\!\!1}} 
\newcommand{\eqdef}{\mbox{\,\raisebox{0.2ex}{\scriptsize\ensuremath{\mathrm:}}\ensuremath{=}\,}} 
\newcommand{\defeq}{\mbox{~\ensuremath{=}\raisebox{0.2ex}{\scriptsize\ensuremath{\mathrm:}} }} 
\newcommand{\mat}[1]{\mathsf{#1}} 
\DeclareMathOperator{\supp}{supp} 
\DeclareMathOperator{\sign}{sign} 

\newcommand{\product}{\cdot} 
\newcommand{\shiftedShuffle}{\,\bar\shuffle\,} 
\newcommand{\alphabet}{\mathcal{A}} 
\newcommand{\Perm}{\mathsf{Perm}} 
\newcommand{\Shuffle}{\mathsf{Shuffle}} 
\newcommand{\FQSym}{\mathsf{FQSym}} 
\newcommand{\WQSym}{\mathsf{WQSym}} 
\DeclareMathOperator{\linearExtensions}{\mathcal{L}} 
\DeclareMathOperator{\LinExt}{LinExt} 
\DeclareMathOperator{\LexMin}{LexMin} 
\DeclareMathOperator{\Std}{Std} 
\DeclareMathOperator{\Des}{Des} 

\newcommandx{\Operad}[1][1=O]{\mathcal{#1}} 
\newcommand{\Operations}{\mathfrak{B}} 
\newcommandx{\operation}[1][1=b]{\mathfrak{#1}} 
\newcommand{\Free}[1][\Operations]{\mathbf{Free}(#1)} 
\newcommand{\Syntax}[1][\Operations]{\mathbf{Trees}(#1)} 
\newcommandx{\tree}[1][1=t]{\mathfrak{#1}} 
\newcommand{\Relations}{\mathfrak{R}} 
\newcommand{\algebra}{\mathsf{Alg}} 
\newcommand{\freeAlgebra}{\mathsf{FAlg}} 
\newcommand{\Rewr}{\overset{*}{\to}}
\newcommand{\Hilbert}{\mathcal{H}} 
\newcommand{\koszul}{{\smash{\b{!}}}}
\newcommand{\whiteManin}{\,\square\,} 
\newcommand{\blackManin}{\,\blacksquare\,} 
\newcommand{\Res}{\mathsf{Res}} 
\newcommand{\Ins}{\mathsf{Ins}} 
\newcommand{\morphism}{\mathsf{m}} 
\newcommand{\id}{\mathsf{id}} 

\newcommand{\As}{\mathsf{As}} 
\newcommand{\Dend}{\mathsf{Dend}} 
\newcommand{\Dias}{\mathsf{Diass}} 
\newcommand{\Dup}{\mathsf{Dup}} 
\newcommand{\Dupdual}{\mathsf{Dup}^\koszul} 
\newcommand{\Quad}{\mathsf{Quad}} 
\newcommand{\Twist}{\mathsf{Twist}} 
\newcommand{\Zinb}{\mathsf{Zinb}} 

\newcommand{\EulPol}[1]{\mathrm{Eul}_{#1}} 
\newcommand{\EulNum}[2]{\genfrac{\langle}{\rangle}{0pt}{}{#1}{#2}} 

\newcommand{\op}[1]{
	\foreach \letters in {#1} {
		\StrLeft{\letters}{1}[\firstletter]
		\StrRight{\letters}{1}[\lastletter]
		\IfEqCase{\lastletter}{%
			{a}{\red}%
			{b}{\blue}%
			{c}{\green}%
			{f}{\bf}%
		}%
		\hspace{-.1mm}
		\IfEqCase{\firstletter}{%
			{l}{\prec}%
			{m}{\prec\hspace*{-.27cm}\succ}%
			{r}{\succ}%
			{o}{\odot}%
			{*}{\color{gray}\prec}%
			{n}{\color{white}\prec}%
		} \hspace{-.1mm}
	}
}

\newcommand{\compoR}[2]{
	\begin{tikzpicture}[baseline=-.5cm, level 1/.style={sibling distance = .8cm, level distance = .6cm}, level 2/.style={sibling distance = .8cm, level distance = .5cm}]
		\node [rectangle, draw] {$\op{#1}$}
			child {node {}}
			child {node [rectangle, draw] {$\op{#2}$}
				child {node {}}
				child {node {}}
			}
		;
	\end{tikzpicture}
} 
\newcommand{\compoL}[2]{
	\begin{tikzpicture}[baseline=-.5cm, level 1/.style={sibling distance = .8cm, level distance = .6cm}, level 2/.style={sibling distance = .8cm, level distance = .5cm}]
		\node [rectangle, draw] {$\op{#2}$}
			child {node [rectangle, draw] {$\op{#1}$}
				child {node {}}
				child {node {}}
			}
			child {node {}}
		;
	\end{tikzpicture}
} 

\newcommand{\ind}[1]{\mathsf{#1}} 
\newcommand{\series}{{\ddagger}} 
\newcommand{\destVect}[2]{\llparenthesis\, #1\, \rrparenthesis_{#2}} 
\newcommand{\kcuts}{\mathrm{rcuts}_k} 
\newcommand{\bcuts}{\mathrm{bcuts}} 

\newcommandx{\signaletic}[1][1=k]{\mathsf{Sig}_{#1}} 
\newcommandx{\messySignaletic}{\mathsf{MSig}_{1}} 
\newcommandx{\messySignaleticParallel}[1][1=k]{\smash{\mathsf{MSig}^\parallel_{#1}}} 
\newcommandx{\messySignaleticSeries}[1][1=k]{\smash{\mathsf{MSig}^\series_{#1}}} 
\newcommandx{\messySignaleticParallelSeries}[1][1=\b{k}]{\smash{\mathsf{MSig}^{\parallel\series}_{#1}}} 
\newcommandx{\messySignaleticSeriesParallel}[1][1=\b{k}]{\smash{\mathsf{MSig}^{\series\parallel}_{#1}}} 
\newcommandx{\tidySignaletic}[1][1=c]{\mathsf{TSig}_{#1}} 
\newcommandx{\tidySignaleticParallel}[1][1=k]{\smash{\mathsf{TSig}^\parallel_{#1}}} 
\newcommandx{\tidySignaleticSeries}[1][1=k]{\smash{\mathsf{TSig}^\series_{#1}}} 
\newcommandx{\tidySignaleticParallelSeries}[1][1=\b{k}]{\smash{\mathsf{TSig}^{\parallel\series}_{#1}}} 
\newcommandx{\tidySignaleticSeriesParallel}[1][1=\b{k}]{\smash{\mathsf{TSig}^{\series\parallel}_{#1}}} 
\newcommandx{\constraint}[1][1 = c]{\mathbf{#1}} 

\newcommandx{\citelangis}[1][1=k]{\mathsf{Cit}_{#1}} 
\newcommandx{\messyCitelangisParallel}[1][1=k]{\smash{\mathsf{MCit}^\parallel_{#1}}} 
\newcommandx{\messyCitelangisSeries}[1][1=k]{\smash{\mathsf{MCit}^\series_{#1}}} 
\newcommandx{\messyCitelangisParallelSeries}[1][1=\b{k}]{\smash{\mathsf{MCit}^{\parallel\series}_{#1}}} 
\newcommandx{\messyCitelangisSeriesParallel}[1][1=\b{k}]{\smash{\mathsf{MCit}^{\series\parallel}_{#1}}} 
\newcommandx{\tidyCitelangisParallel}[1][1=k]{\smash{\mathsf{TCit}^\parallel_{#1}}} 
\newcommandx{\tidyCitelangisSeries}[1][1=k]{\smash{\mathsf{TCit}^\series_{#1}}} 
\newcommandx{\tidyCitelangisParallelSeries}[1][1=\b{k}]{\smash{\mathsf{TCit}^{\parallel\series}_{#1}}} 
\newcommandx{\tidyCitelangisSeriesParallel}[1][1=\b{k}]{\smash{\mathsf{TCit}^{\series\parallel}_{#1}}} 
\newcommand{\cR}{\mathcal{R}} 

\newcommandx{\Zinbiel}[1][1=k]{\mathsf{Zin}_{#1}} 
\newcommandx{\messyZinbielParallel}[1][1=k]{\smash{\mathsf{MZin}^\parallel_{#1}}} 
\newcommandx{\messyZinbielSeries}[1][1=k]{\smash{\mathsf{MZin}^\series_{#1}}} 
\newcommandx{\tidyZinbielParallel}[1][1=k]{\smash{\mathsf{TZin}^\parallel_{#1}}} 
\newcommandx{\tidyZinbielSeries}[1][1=k]{\smash{\mathsf{TZin}^\series_{#1}}} 

\newcommandx{\posParallel}[1][1=k]{\smash{\mathsf{Pos}^\parallel_{#1}}} 
\newcommandx{\posSeries}[1][1=k]{\smash{\mathsf{Pos}^\series_{#1}}} 
\newcommandx{\boundedPosParallel}[1][1=k]{\smash{\mathsf{BP}^\parallel_{#1}}} 
\newcommandx{\rootedPosSeries}[1][1=k]{\smash{\mathsf{RP}^\series_{#1}}} 

\newcommand{\messyEvalPermParallel}{\smash{\mathrm{MEval}^\parallel}} 
\newcommand{\messyEvalPermSeries}{\smash{\mathrm{MEval}^\series}} 
\newcommand{\tidyEvalPermParallel}{\smash{\mathrm{TEval}^\parallel}} 
\newcommand{\tidyEvalPermSeries}{\smash{\mathrm{TEval}^\series}} 
\newcommand{\EvalPosetParallel}{\smash{\mathrm{PEval}^\parallel}} 
\newcommand{\EvalPosetSeries}{\smash{\mathrm{PEval}^\series}} 
\DeclareMathOperator{\Root}{Root} 

\makeatletter
\newcommand{\oset}[2]{%
  {\mathop{#2}\limits^{\vbox to .1\ex@{\kern-\tw@\ex@
   \hbox{\scalebox{0.4}{$#1$}}\vss}}}}
\makeatother

\newcommand{\specialComposition}[3]{\mathbin{\smash{\oset{\mathrm{#1}#2}{\circ}\hspace{-.02cm}_{#3}}}} 
\newcommand{\posetParallelCirc}[1]{\specialComposition{P}{\!\parallel}{#1}} 
\newcommand{\posetSeriesCirc}[1]{\specialComposition{P}{\,\series}{#1}} 
\newcommand{\messyParallelCirc}[1]{\specialComposition{M}{\!\parallel}{#1}} 
\newcommand{\messySeriesCirc}[1]{\specialComposition{M}{\,\series}{#1}} 
\newcommand{\tidyParallelCirc}[1]{\specialComposition{T}{\!\parallel}{#1}} 
\newcommand{\tidySeriesCirc}[1]{\specialComposition{T}{\,\series}{#1}} 

\newcommand{\ie}{\textit{i.e.}~} 
\newcommand{\eg}{\textit{e.g.}~} 
\newcommand{\aka}{\textit{aka.}~} 
\newcommand{\viceversa}{\textit{vice versa}} 
\newcommand{\qandq}{\quad\text{and}\quad} 
\newcommand{\qqandqq}{\qquad\text{and}\qquad} 
\newcommand{\parabf}[1]{\vspace{.3cm}\noindent\textbf{#1}.} 
\newcommand{\paraul}[1]{\vspace{.3cm}\noindent\uline{#1}.} 
\newcommand{\warning}{\vspace{.3cm}\noindent\fbox{\textbf{Warning!}}~} 
\newcommand{\defn}[1]{\emph{\textsf{\color{blue} #1}}} 
\definecolor{green}{RGB}{57,181,74} 
\newcommand{\blue}{\color{blue}} 
\newcommand{\green}{\color{green}} 
\newcommand{\red}{\color{red}} 
\newcommand{\violet}{\color{violet}} 
\newcommand{\cyan}{\color{cyan}} 
\renewcommand{\b}[1]{\mathbf{#1}} 
\newcommand{\transpose}[1]{{}^\mathsf{t}#1} 

\makeatletter
\def\l@section{\@tocline{1}{7pt}{0pc}{}{}}
\makeatother
\let\oldtocpart=\tocpart
\renewcommand{\tocpart}[2]{\bf\large\oldtocpart{#1}{#2}}
\let\oldtocsection=\tocsection
\renewcommand{\tocsection}[2]{\bf\oldtocsection{#1}{#2}}
\let\oldtocsubsubsection=\tocsubsubsection
\renewcommand{\tocsubsubsection}[2]{\quad\oldtocsubsubsection{#1}{#2}}

\begin{document}

\begin{abstract}
We introduce $k$-signaletic operads and their Koszul duals, generalizing the dendriform, diassociative and duplicial operads (which correspond to the~$k=1$ case).
We show that the Koszul duals of the $k$-signaletic operads act on multipermutations and that the resulting algebras are free, thus providing combinatorial models for these operads.
Finally, motivated by these actions on multipermutations, we introduce similar operations on multiposets which yield yet another relevant operad obtained as Manin powers of the $L$-operad.
\end{abstract}

\vspace*{-.7cm}

\maketitle

\vspace{-.7cm}

\tableofcontents

\clearpage


\section{Introduction}
\label{sec:introduction}

\parabf{Diassociative and dendriform algebras and operads}
In his study of Leibniz algebras, a kind of non commutative Lie algebras, J.-L.~Loday introduced two more families of algebras named diassociative and dendriform~\cite{Loday-dialgebras}.
They are vector spaces endowed with two binary operations~$\op{l}$ and~$\op{r}$ satisfying certain axioms regarding the compositions of two operations (see \cref{subsubsec:dendriformDiassociativeOperads} for precise definitions).
In particular, both operations are associative for diassociative algebras, and the sum of the two operations is associative for dendriform algebras.

The close connections between these two families of algebras are better understood via the theory of (non-symmetric) operads.
An operad is an algebraic structure abstracting a type of algebras.
Their theory allows, via a suitable notion of morphism, to understand relations between algebras of different types.
More precisely operads provide a formalization of operations transforming several inputs into one output, of the ways to compose these operations, and of the relations between different possible compositions (see \cref{subsec:operads}).
It turns out that the diassociative operad~$\Dias$ and the dendriform operad~$\Dend$ are both binary quadratic Koszul operads and that they are Koszul duals to each other (see \cref{subsec:Koszul}).
An enumerative consequence of this duality is that the Hilbert series~$\Hilbert_\Dias$ and~$\Hilbert_\Dend$ of these two operads are Lagrange inverse to each other:
\begin{gather*}
\Hilbert_\Dias \big( - \Hilbert_\Dend(-t) \big) = t,\\
\text{where}\qquad
\Hilbert_\Dias(t) = \sum_{p \ge 1} p \, t^p = \frac{t}{1-t^2}
\qqandqq
\Hilbert_\Dend(t) = \sum_{p \ge 1} C_p \, t^p = \frac{1-\sqrt{1-4t}}{2t}-1
\end{gather*}
where~$C_p \eqdef \frac{1}{p+1} \binom{2p}{p}$ denotes the $p$-th Catalan number.

Combinatorial models for the free algebras on the diassociative and dendriform operads were described in~\cite{Loday-dialgebras, LodayRonco}.
The free diassociative algebra is described in~\cite{Loday-dialgebras} in terms of middle entries in monomials.
The free dendriform algebra was described by J.-L.~Loday and M.~Ronco in~\cite{LodayRonco} as an algebra on planar binary trees.
Other important dendriform algebras include the shuffle algebra and the algebra of permutations of C.~Malvenuto and C.~Reutenauer~\cite{MalvenutoReutenauer} (also known as the algebra of free quasi-symmetric functions~$\FQSym$), where the classical shuffle product splits into two partial shuffles according to the provenance of the first letter:
\begin{gather*}
{\shuffle} = {\op{l}} + {\op{r}},
\qquad\text{where}\qquad
xX \op{l} yY = x (X \shuffle yY)
\qqandqq
xX \op{r} yY = y (xX \shuffle Y).
\end{gather*}

\parabf{Further splitting of associative products}
\enlargethispage{-.6cm}
Several generalizations of dendriform algebras were recently introduced and studied, including tridendriform algebras~\cite{LodayRonco-tridendriform}, quadri-algebras~\cite{AguiarLoday, Foissy}, ennea-algebras~\cite{Leroux-ennea}, $m$-dendriform algebras~\cite{Novelli}, polydendriform algebras~\cite{Giraudo-pluriassociativeI, Giraudo-pluriassociativeII}, and $k$-twistiform algebras~\cite{Pilaud-brickAlgebra}.
These structures rely on various decompositions of an associative product into several operations, often motivated by the combinatorics of the shuffle product.
For instance, there are two natural ways to split further the shuffle product into four partial shuffles, according to the provenance of:
\begin{enumerate}[(i)]
\item the first and last letters, defining a quadri-algebra structure of~\cite{AguiarLoday, Foissy}, or
\item the first two letters, defining a $2$-twistiform algebra structure of~\cite{Pilaud-brickAlgebra}.
\end{enumerate}
While the corresponding $\Quad$ and $\Twist$ operads are not isomorphic, it turns out that they have the same Hilbert series, which is the Lagrange inverse of~$\sum_{p \ge 1} p^2 \, t^p$.
In fact, this equi-enumeration extends to two infinite families of operads generalizing the $\Dend$, $\Quad$ and $\Twist$ operads:
\begin{enumerate}[(i)]
\item the $k$-th Manin power of the dendriform operad ($\Quad$ is the Manin square of $\Dend$),
\item the $k$-twistiform operads defined in~\cite{Pilaud-brickAlgebra}.
\end{enumerate}
Their Hilbert series coincide and is the Lagrange inverse of~$\sum_{p \ge 1} p^k \, t^p$.
The objective of this paper is to explore further the twistiform operads in order to find a good algebraic interpretation of this equi-enumeration.

\parabf{$k$-signaletic operads}
Our approach via the Koszul duals of all these operads offers an enlightening perspective.
The first step is to come back to J.-L.~Loday's original interpretation of the diassociative operad.
Consider a syntax tree~$\tree$ on~$\{\op{l}, \op{r}\}$, where~$\op{l}$ and~$\op{r}$ are binary operations, and imagine that a car starts at the root of~$\tree$ and follows the signs in the nodes of~$\tree$ until it reaches a leaf.
The diassociative operad is then the quotient of the free operad on~$\{\op{l}, \op{r}\}$ by the equivalence relation identifying syntax trees of the same arity in which the car reaches the same destination (see \cref{subsec:signaleticDiassDualDup}).

This interpretation then naturally generalizes to operads with~$2^k$ binary operations~$\{\op{l}, \op{r}\}^k$.
Namely, for a syntax tree on~$\{\op{l}, \op{r}\}^k$, we imagine that $k$ cars start at the root of~$\tree$ and follow the signs in the nodes of~$\tree$ according to the rules described below until they reach the leaves.
We define the \defn{$k$-signaletic operad} as the quotient of the free operad on~$\{\op{l}, \op{r}\}^k$ by the equivalence relation identifying syntax trees with the same arity and in which the $k$ cars reach the same destinations (see \cref{sec:signaleticOperads}).
As the equivalence classes are determined by the destinations of the $k$ cars, and all destination vectors are possible, the Hilbert series of the $k$-signaletic operad is obviously given by~$\sum_{p \ge 1} p^k \, t^k$.

Of course, the destinations of the cars crucially depend on the \defn{circulation rule} telling each car which signal it should consider in each node.
There are two natural possible circulation rules:
\begin{enumerate}[(i)]
\item \defn{Parallel}: All $k$ cars start together, and the $i$-th car always follows the indication given by the $i$-th letter of the traffic signal at each node.
\item \defn{Series}: The $k$ cars start one after the other, and each car always follows the indication given by the leftmost remaining letter of the traffic signal at each node and erases it.
\end{enumerate}
This results in two non-isomorphic but obviously equi-enumerated versions of the $k$-signaletic operad.
The parallel version is the $k$-th Manin power of the diassociative operad, while the series version is a somewhat twisted $k$-th Manin power of the diassociative operad.

Besides the two possible circulation rules, we can also impose two \defn{ordering rules}:
\begin{enumerate}[(i)]
\item \defn{Tidy}: All signals not used by a car must point to the left,
\item \defn{Messy}: No further constraints on the remaining signals.
\end{enumerate}
Again, the resulting operads are non-isomorphic but equi-enumerated.
For instance, when~$k = 1$, the tidy signaletic operad is an analogue of the dual duplicial operad, while the messy signaletic operad is the diassociative operad.
We therefore obtain four versions of the $k$-signaletic operad: $\tidySignaleticParallel$, $\messySignaleticParallel$, $\tidySignaleticSeries$, $\messySignaleticSeries$ (parallel or series, tidy or messy).

It turns out that these $k$-signaletic operads are all quadratic and Koszul operads (see \cref{subsec:HilbertSeriesKoszulitySignaletic}).
Remarkably, our proof is independent of which version we consider.
Namely, we orient the $k$-signaletic quadratic relations according to the $k$-signaletic Tamari lattice, an order generalizing the classical Tamari lattice to syntax trees on~$\{\op{l}, \op{r}\}^k$.
We then show that the resulting $k$-signaletic rewriting system converges to certain right $k$-signaletic combs.
The Koszulity then follows from a classical result of the operad theory.

\parabf{$k$-citelangis operads}
At that stage, it is natural to consider the Koszul duals of the $k$-signaletic operads, which we call \defn{$k$-citelangis operads} (see \cref{sec:citelangisOperads}).
Again, $k$-citelangis operads arise in four flavours: $\tidyCitelangisParallel$, $\messyCitelangisParallel$, $\tidyCitelangisSeries$, $\messyCitelangisSeries$ (parallel or series, tidy or messy).
From their presentation, one recognizes relevant operads:
\begin{enumerate}[(i)]
\item the messy parallel $k$-citelangis operad is the $k$-th Manin power of the dendriform operad,
\item the tidy parallel $k$-citelangis operad is the $k$-th Manin power of a twisted duplicial~operad,
\item the messy series $k$-citelangis operad is a twisted $k$-th Manin power of the dendriform operad, and coincides with the $k$-twistiform operad of~\cite{Pilaud-brickAlgebra},
\item the tidy series $k$-citelangis operad is a twisted $k$-th Manin power of a twisted duplicial operad.
\end{enumerate}
While the $k$-citelangis operads are not isomorphic, they are all equi-enumerated: their Hilbert series is the Lagrange inverse of~$\sum_{p \ge 1} p^k \, t^p$ (see \cref{subsubsec:HilberSeriesCitelangis}).
This gives in particular the algebraic interpretation of the equi-enumeration between the $\Quad$ and~$\Twist$ operads observed earlier.

By Koszul duality, it is not difficult to orient the quadratic relations in these operads into a converging $k$-citelangis rewriting system.
The normal forms of this system are syntax trees where each node imposes some constraints only to its right child.
Therefore, the Hilbert series of the $k$-citelangis operads can be interpreted as a fixed point of the substitution into the generating function of the right combs that are in normal form.
Moreover, the latter has a natural expression in terms of the power of the transition matrix~$M_k$ of the automaton where each operation of~$\{\op{l}, \op{r}\}^k$ points towards its possible children.
While simple to express, the resulting matrices~$M_k$ seem very interesting (see \cref{subsubsec:normalFormsCitelangis}): for instance, we conjecture that their minimal polynomials are multiples of the Eulerian polynomial~$\EulPol{k}(t) \eqdef \sum_{j=0}^{k-1} \EulNum{k}{j} \, t^j$ (where the Eulerian number~$\EulNum{k}{j}$ is the number of permutations of~$\fS_k$ with precisely~$j$ descents).
These polynomials are ubiquitous in this paper since
\[
\sum_{p \ge 1} p^k \, t^p = \frac{t \cdot \EulPol{k}(t)}{(1-t)^{k+1}} \, .
\]

\parabf{Actions on multipermutations and free $k$-citelangis algebras}
Once the presentations of (the four versions of) $k$-citelangis operads are established, we look for natural actions of these operads and combinatorial models for their free algebras (see \cref{sec:actions}).

As already mentioned, this has been done for the $\Dend$ and $\Quad$ operads.
Namely, the algebra of permutations~$\FQSym$ has a dendriform algebra structure where the shuffle product decomposes into two operations depending on the provenance of the first letter of the result~\cite{Loday-dialgebras}, and a quadri-algebra structure where the shuffle product decomposes into four operations depending on the provenance of the first and last letters of the result~\cite{AguiarLoday}.
Moreover, it was shown that the resulting algebras are free~\cite{Loday-dialgebras, Foissy, Vong}, so that the subalgebras generated by the permutation~$1$ for~$\Dend$ and the permutation~$12$ for $\Quad$ provide combinatorial models for the corresponding operads.
Unfortunately, the parallel world is quite limited: these actions cannot be extended to messy parallel $k$-citelangis structures for~$k \ge 2$ because a \mbox{permutation has only two ends!}

In contrast, $\FQSym$ has a natural messy series $k$-citelangis algebra structure for any~$k \ge 1$, where the shuffle product decomposes into $2^k$ operations depending on the provenance of the first $k$ letters of the result~\cite{Pilaud-brickAlgebra}.
Nevertheless, the question of the freeness of the resulting algebras was left widely open in~\cite{Pilaud-brickAlgebra}, and we were lacking combinatorial models for the $k$-twistiform operads.
This paper settles these two problems.

To start with, we need to revisit the combinatorics of the parallel $2$-citelangis operads (see \cref{subsec:actionParallelCitalangisOperads}).
We first observe that (a version of) the tidy parallel $2$-citelangis operad also naturally acts on permutations: the four operations act by concatenation, except that they select the first and last letters depending on their first and last signs.
The resulting tidy parallel $2$-citelangis algebra is free on the permutations that admit no bounded cut, \ie a value~$\gamma$ such that except the first and last letters, all letters smaller than~$\gamma$ arrive before all letters larger than~$\gamma$ in the permutation.
We then observe a triangularity relation between the tidy and messy parallel $2$-citelangis operations on permutations: namely, the result of a tidy operation is given by the lexicographically minimal permutation among the result of the corresponding messy operation.
This triangularity transports the freeness of the tidy parallel $2$-citelangis algebra to that of the messy parallel $2$-citelangis algebra on permutations, thus providing an alternative proof of the results of~\cite{Foissy, Vong}.

This approach to the parallel $2$-citelangis operads provides a solid prototype to tackle the series $k$-citelangis operads (see \cref{subsec:actionSeriesCitalangisOperads}).
We indeed define a similar action of the tidy series $k$-citelangis operad on $k$-permutations (\ie permutations of a multiset where each value is repeated $k$ times).
The resulting tidy series $k$-citelangis algebra is free on the $k$-permutations that admit no $k$-rooted cut, \ie a value~$\gamma$ such that except the first $k$ letters, all letters smaller than~$\gamma$ arrive before all letters larger than~$\gamma$ in the permutation.
By the same triangularity argument, we obtain that the messy series $k$-citelangis algebra on permutations defined in~\cite{Pilaud-brickAlgebra} is free.

\parabf{Operations on posets and poset operads}
As the messy $k$-citelangis operations on permutations are defined as partial shuffle products, their results are sums over all linear extensions of certain $k$-posets (\ie partial orders on a multiset where each value is repeated $k$ times).
The $k$-posets are bounded for the parallel action, and $k$-rooted for the series action (meaning that they have a chain of $k$ successive minimal elements).
The messy $k$-citelangis operations can be described directly on these $k$-posets, but the resulting operations do not satisfy all messy $k$-citelangis relations.
It motivates the introduction of operads~$\boundedPosParallel[2]$ on bounded $2$-posets and~$\rootedPosSeries$ on $k$-rooted $k$-posets defined by these operations (see \cref{subsubsec:parallelPosetOperad,subsubsec:seriesPosetOperad}).
The resulting operads are quadratic and Koszul, with a unique quadratic relation.
The Hilbert series of their Koszul dual is thus~$t + 2^k t^2 + t^3$, which by Lagrange inversion enables us to compute their own Hilbert series.

\parabf{$k$-Zinbiel operads}
Finally, motivated by the actions of the series $k$-citelangis operads on \mbox{$k$-rooted} cuttable $k$-permutations, we define the series \defn{$k$-Zinbiel operad}~$\messyZinbielSeries$ on all $k$-rooted $k$-permutations (see \cref{subsubsec:seriesZinbiel}).
This operad completes our operad zoo: the linear extensions map provides a natural morphism from the $k$-rooted poset operad~$\rootedPosSeries$ to the $k$-Zinbiel operad~$\messyZinbielSeries$, and the messy series $k$-citelangis operad~$\messyCitelangisSeries$ is a suboperad of the $k$-Zinbiel operad~$\messyZinbielSeries$.
We also study a similar operad in the parallel situation when~$k = 2$.

\parabf{Open questions}
Our work motivates many interesting research directions and open questions, some of which are briefly discussed in \cref{sec:openQuestions}.
A short summary:
\begin{itemize}
\item Are there general notions of twisted Manin products of operads that would enlighten series signaletic and citelangis operads in a similar way the classical Manin products explain the parallel signaletic and citelangis operads?
\item What are the properties of the $k$-Zinbiel operads as symmetric operads? Are they Koszul? What are their duals?
\item What is the natural generalization of tridendriform algebras in parallel and series? Is there a natural interpretation in terms of traffic rules for their Koszul dual operads?
\item Is there a relevant generalization of the duplicial operad in the series situation? While the parallel situation allows quite some freedom to extend both the duplicial and twisted duplicial operads, the series situation seems more limited, but might reveal an interesting operad.
\item Is it possible to interpolate between the tidy and messy versions of our operads with \mbox{$q$-analogues}, in the same way the $q$-deformation of the shuffle product interpolates between the classical shuffle and the concatenation products?
\end{itemize}

\parabf{Diagrammatic summary}
Finally, we believe that the following diagram gives a good summary of the combinatorial maps between all the series operads (a similar diagram holds for the parallel operads when~$k = 2$).
In this diagram, plain arrows are operad morphisms, while dashed arrows are just bijections of normal forms.
Different operads and combinatorial objects in this diagram should become clear along the text.

\vspace{.5cm}
\begin{center}
\begin{tikzcd}[column sep=2cm]
	  \blue \rootedPosSeries
	& \blue \messyZinbielSeries
	& \blue \tidyZinbielSeries \\[-.8cm]
	  \begin{minipage}{2.8cm} \begin{center} $k$-rooted \\[-.1cm] posets \end{center} \end{minipage} \arrow[r, "\LinExt",  two heads]
	& \begin{minipage}{2.8cm} \begin{center} $k$-rooted \\[-.1cm] $k$-permutations \end{center} \end{minipage} \arrow[r, dashrightarrow, "\LexMin", hook, two heads]
	& \begin{minipage}{2.8cm} \begin{center} $k$-rooted \\[-.1cm] $k$-permutations \end{center} \end{minipage} \\[.3cm]
	  \phantom{\big(} \blue \posSeries \phantom{\big)} \arrow[u, hook]
	& \phantom{\big(} \blue \messyCitelangisSeries \phantom{\big)} \arrow[u, hook]
	& \phantom{\big(} \blue \tidyCitelangisSeries \phantom{\big)} \arrow[u, hook] \\[-.8cm]
	  \begin{minipage}{2.8cm} \begin{center} $k$-rooted cuttable \\[-.1cm] posets \end{center} \end{minipage} \arrow[r, "\LinExt",  two heads]
	& \begin{minipage}{2.8cm} \begin{center} $k$-rooted cuttable \\[-.1cm] $k$-permutations \end{center} \end{minipage} \arrow[r, dashrightarrow, "\LexMin", hook, two heads]
	& \begin{minipage}{2.8cm} \begin{center} $k$-rooted cuttable \\[-.1cm] $k$-permutations \end{center} \end{minipage} \\[-.3cm]
	  \phantom{1} & \textsc{\green\;\,koszul \; duality} \arrow[l, no head, squiggly, green] \arrow[r, no head, squiggly, green] & \phantom{1} \\[-.3cm]
	  \phantom{\big(} \blue (\posSeries)^\koszul \phantom{\big)} \arrow[uu, leftrightarrow, squiggly, green]
	& \phantom{\big(} \blue \messySignaleticSeries \phantom{\big)} \arrow[uu, leftrightarrow, squiggly, green]
	& \phantom{\big(} \blue \tidySignaleticSeries \phantom{\big)} \arrow[uu, leftrightarrow, squiggly, green] \\[-.8cm]
	  \begin{minipage}{2.8cm} \begin{center} $L$-destination \\[-.1cm] vectors \end{center} \end{minipage} \arrow[r, hookleftarrow]
	& \begin{minipage}{2.8cm} \begin{center} destination \\[-.1cm] vectors \end{center} \end{minipage} \arrow[r, "=", dashed, leftrightarrow]
	& \begin{minipage}{2.8cm} \begin{center} destination \\[-.1cm] vectors \end{center} \end{minipage}
\end{tikzcd}
\end{center}

\parabf{Acknowledgements}
We are deeply grateful to anonymous referees for impressively numerous comments and extremely constructive suggestions on the presentation of this paper.

\newpage


\section{Preliminaries}
\label{sec:preliminaries}

This section recalls some algebraic and combinatorial preliminaries and fixes notations.
To start with, we denote by~$\N$ the set of non-negative integers, including~$0$.
For an integer~$n$, we denote~$[n] \eqdef \{1,2,\dots,n\}$, with the convention that~$[0] = \varnothing$.
Let~$\K$ be a field of characteristic~$0$.


\subsection{Operads}
\label{subsec:operads}

\enlargethispage{.5cm}
An operad is an algebraic structure abstracting a type of algebras.
We give here the formal definitions needed in this paper.
We refer to~\cite{LodayVallette, Markl} for classical references on operads, and to~\cite{Chapoton-operads, Giraudo-nonsymmetricOperadsCombinatorics, Giraudo-these} for more combinatorial approches.
Note that we work with non-symmetric operads.

\begin{definition}
A (unital, non-symmetric) \defn{operad} (in the category of vector spaces) is a graded $\K$-vector space~$\Operad = \bigoplus_{p \ge 1} \Operad(p)$ endowed with a distinguished element~$\one \in \Operad(1)$ (called the \defn{unit} of~$\Operad$) and linear maps (called \defn{partial compositions} of~$\Operad$)
\[
\circ_i : \Operad(p) \otimes \Operad(q) \to \Operad(p+q-1) \qquad\text{for } p,q \ge 1 \text{ and } i \in [p]
\]
such that the following three relations hold:
\[
\begin{array}{@{}l@{\;\;}c@{\;\;}l}
\text{(\defn{unitality})} & \one \circ_1 \operation[p] = \operation[p] = \operation[p] \circ_i \one & \text{for all } \operation[p] \in \Operad(p), \; i \in [p], \\
\text{(\defn{series axiom})} & (\operation[p] \circ_i \operation[q]) \circ_{i+j-1} \operation[r] = \operation[p] \circ_i (\operation[q] \circ_j \operation[r]) & \text{for all } \operation[p] \in \Operad(p), \; \operation[q] \in \Operad(q), \; \operation[r] \in \Operad(r), \; i \in [p], \; j \in [q], \\
\text{(\defn{parallel axiom})} & (\operation[p] \circ_i \operation[q]) \circ_{j+q-1} \operation[r] = (\operation[p] \circ_j \operation[r]) \circ_i \operation[q] & \text{for all } \operation[p] \in \Operad(p), \; \operation[q] \in \Operad(q), \; \operation[r] \in \Operad(r), \; i < j \in [p].
\end{array}
\]
The elements of~$\Operad(p)$ are called \defn{operations} of arity~$p$.
\end{definition}

Alternatively, one can define the operad~$\Operad$ by its \defn{compositions maps}
\[
\circ : \Operad(p) \otimes \big( \Operad(q_1) \otimes \dots \otimes \Operad(q_p) \big) \to \Operad(q_1+\dots+q_p) \qquad \text{for } p,q_1,\dots,q_p \ge 1
\]
linearly defined for~$\operation[p] \in \Operad(p)$ and~$\operation[q]_1 \in \Operad(q_1), \dots, \operation[q]_p \in \Operad(q_p)$ by
\[
\operation[p] \circ (\operation[q]_1 \otimes \dots \otimes \operation[q]_p) = (((\operation[p] \circ_p \operation[q]_p) \circ_{p-1} \operation[q]_{p-1}) \dots) \circ_1 \operation[q]_1.
\]
(Note that partial compositions can be retrieved from the global composition maps using the unit).

\begin{definition}
An \defn{operad morphism} is a unital graded linear map~$\phi: \Operad \to \Operad'$ between two operads which commutes with partial compositions:
\[
\phi(\operation[p] \circ_i \operation[q]) = \phi(\operation[p]) \circ_i \phi(\operation[q]) \text{ for all } \operation[p] \in \Operad(p), \operation[q] \in \Operad(q) \text{ and } i \in [p].
\]
\end{definition}

\begin{definition}
An \defn{operad ideal} of~$\Operad$ is a graded subspace~$\Operad[I] = \bigoplus_{p \ge 1} \Operad[I](p)$ of~$\Operad$ such that~$\operation[p] \circ_i \operation[q] \in \Operad[I]$ and $\operation[q] \circ_j \operation[p] \in \Operad[I]$ for any~$\operation[p] \in \Operad(p)$, $\operation[q] \in \Operad[I](q)$, $i \in [p]$ and~$j \in [q]$.
\end{definition}

One can then define as usual the quotient~$\Operad / \Operad[I]$ which is itself an operad. We denote by~$\langle \Relations \rangle$ the operad ideal generated by a graded subspace~$\Relations$ of~$\Operad$, \ie the smallest operad ideal of~$\Operad$ containing~$\Relations$.

\begin{definition}
The \defn{Hilbert series} of an operad~$\Operad$ is the formal power series~$\Hilbert_{\Operad}(t)$ encoding the dimensions of~$\Operad(p)$ for~$p \ge 1$:
\[
\Hilbert_{\Operad}(t) \eqdef \sum_{p \ge 1} \dim \Operad(p) \, t^p.
\]
\end{definition}

\begin{definition}
\label{def:algebraOverOperad}
An \defn{algebra over an operad~$\Operad$} is a $\K$-vector space~$\algebra$ endowed with an action ${\cdot : \Operad(p) \otimes \algebra^{\otimes p} \to \algebra}$, for each~$p \ge 1$,  satisfying the relations imposed by the compositions in~$\Operad$, meaning that for all~$\operation[p] \in \Operad(p)$, $\operation[q] \in \Operad(q)$, $a_1 \otimes \dots \otimes a_{p+q-1} \in \algebra^{\otimes p+q-1}$ and~$i \in [p]$,
\[
(\operation[p] \circ_i \operation[q]) \cdot (a_1 \otimes \dots \otimes a_{p+q-1}) = \operation[p] \cdot (a_1 \otimes \dots \otimes a_{i-1} \otimes \operation[q] \cdot (a_i \otimes \dots \otimes a_{i+q-1}) \otimes a_{i+q} \otimes \dots \otimes a_{p+q-1}).
\]
\end{definition}

In other words, the elements of the operad~$\Operad$ are linear operations on~$\algebra$, and the partial compositions on~$\Operad$ correspond to partial compositions in~$\algebra$.
In particular, there is a notion of free algebra over~$\Operad$.

\begin{definition}
The \defn{free algebra over an operad~$\Operad$ generated by a set~$S$} is the $\K$-vector space
\(
\freeAlgebra \eqdef \bigoplus_{p \ge 1} \Operad(p) \otimes \K S^{\otimes p}
\)
endowed with the action ${\cdot : \Operad(p) \otimes \freeAlgebra^{\otimes p} \to \freeAlgebra}$ defined for all ${\operation[p] \in \Operad(p)}$, $\operation[q]_1 \in \Operad(q_1), \dots, \operation[q]_p \in \Operad(q_p)$, and~$v_1 \in \K S^{\otimes q_i}, \dots, v_p \in \K S^{\otimes q_p}$ by
\[
\operation[p] \cdot \big( (\operation[q]_1 \otimes v_1) \otimes \dots \otimes (\operation[q]_p \otimes v_p)\big) \eqdef \big( \operation[p] \circ (\operation[q]_1 \otimes \dots \otimes \operation[q]_p) \big) \otimes (v_1 \otimes \dots \otimes v_p).
\]
\end{definition}

In particular, the free algebra over~$\Operad$ generated by a singleton is isomorphic, as a (graded) vector space, to the operad~$\Operad$ itself.


\subsection{Syntax trees, free operads, presentations, and rewriting systems}
\label{subsec:syntaxTrees}

Operads are about ways to compose operations.
To manipulate these compositions and the relations between them, we need the fundamental tool of syntax trees.

\begin{definition}
A \defn{syntax tree} over a graded set~$\Operations \eqdef \bigsqcup_{p \ge 1} \Operations(p)$ is a rooted planar tree~$\tree$ where each internal node of arity~$p$ is labeled by an element of~$\Operations(p)$. 
We denote by~$\Syntax(p)$ the set of all syntax trees over~$\Operations$ of arity~$p$ and~$\Syntax \eqdef \bigsqcup_{p\ge 1} \Syntax(p)$.
\end{definition}

Throughout, the \defn{arity} of a node is its number of children, and the \defn{arity} of a tree is its number of leaves. 
Note that the tree of arity~$1$, with no internal node and only one leaf, is a valid syntax~tree.

\begin{definition}
The \defn{free unital operad} over a graded set~$\Operations \eqdef \bigsqcup_{p \ge 1} \Operations(p)$ is the operad
\[
\Free \eqdef \bigoplus_{p \ge 1} \Free(p),
\]
where:
\begin{itemize}
\item $\Free(p)$ is the $\K$-vector space generated by syntax trees on~$\Operations$ of arity~$p$,
\item the partial composition~$\circ_i$ is the linear map defined on two syntax trees~$\operation[s]$ and~$\operation[t]$ by grafting the root of~$\operation[t]$ on the $i$-th leaf of~$\operation[s]$, and extended by linearity to~$\Free$, and
\item the unit is the syntax tree of arity~$1$, with no internal node and only one leaf.
\end{itemize}
\end{definition}

As its name suggests, the free operad is a free object in the category of non-symmetric operads so that any operad can be obtained by a quotient of a free operad.

\begin{definition}
A \defn{presentation} of an operad~$\Operad$ is a pair~$(\Operations, \Relations)$ where~$\Operations \eqdef \bigsqcup_{p \ge 1} \Operations(p)$ is a graded set and~$\Relations$ is a subspace of the free operad~$\Free$ such that~$\Operad$ is isomorphic to the quotient~$\Free/\langle \Relations \rangle$. The elements of~$\Operations$ are called \defn{generators} while the elements of~$\Relations$ are called \defn{relations} of the presentation.
\end{definition}

\begin{definition}
An operad is called \defn{binary} if it admits a presentation~$(\Operations, \Relations)$ in which the generators~$\Operations$ all have arity~$2$. A binary operad is \defn{quadratic} if it admits a presentation~$(\Operations, \Relations)$ in which the relations~$\Relations$ only involve trees with two nodes (that is operad elements of arity~$3$).
\end{definition}

In this paper, we only work with binary quadratic operads.
To compute with presentations, it is often useful to use rewriting.

\begin{definition}
A \defn{quadratic rewriting rule} is a pair $(\tree[s], \operation[q])$ formed by a syntax tree $\tree[s]$ of $\Syntax(3)$ and a linear combination $\operation[q]$ in $\Free(3)$. A \defn{quadratic rewriting system} over $\Operations$ is a collection of quadratic rewriting rules. We say that a tree~$\tree\in\Syntax$ \defn{rewrites} into an element~$\operation[p]\in\Free$ by this system and we write $\tree\to\operation[p]$ if there exists an integer $i$, a tree $\tree[u]$ of arity at least $i$, a rule~$(\tree[s], \operation[q])$, and a triple of trees $(\tree[l]_1, \tree[l]_2, \tree[l]_3)$ such that
\[
\tree = \tree[u] \circ_i \big(\tree[s] \circ (\tree[l]_1, \tree[l]_2, \tree[l]_3)\big)
\qqandqq
\operation[p] = \tree[u] \circ_i \big(\operation[q] \circ (\tree[l]_1, \tree[l]_2, \tree[l]_3)\big).
\]
We say that $\tree[s]$ is the \defn{pattern} that was rewritten to $\operation[q]$. Extending by linearity and abusing notation we still denote by~$\operation[p] \to \operation[q]$ this binary relation on $\Free$. We moreover denote by $\Rewr$ its reflexive and transitive closure.
\end{definition}

\begin{example}
We show here an example of an application of a rewriting rule. We fix $\Operations\eqdef\{\op{l}, \op{r}\}$ and we suppose that the system contains the following rewriting rule:
\[
\left(\compoR{l}{l}\, ,\ \compoL{l}{l} \!\!+\!\! \compoR{l}{r}\right)
\]

\pagebreak
Then the following is a possible rewriting:
\bigskip
\[
	\begin{tikzpicture}[baseline=-1.2cm, level/.style={sibling distance = 25mm/#1, level distance = 1.2cm/sqrt(#1)}]
		\node [rectangle, draw] {$\op{r}$}
			child {node [rectangle, draw, fill=red!50] {$\op{l}$}
				child {node [rectangle, draw] {$\op{r}$}
					child {node {\phantom{1}}}
					child {node {\phantom{1}}}
				}
				child {node [rectangle, draw, fill=red!50] {$\op{l}$}
					child {node [rectangle, draw] {$\op{l}$}
						child {node {\phantom{1}}}
						child {node {\phantom{1}}}
					}
					child {node [rectangle, draw] {$\op{r}$}
						child {node {\phantom{1}}}
						child {node [rectangle, draw] {$\op{r}$}
							child {node {\phantom{1}}}
							child {node {\phantom{1}}}
						}
					}
				}
			}
			child {node [rectangle, draw] {$\op{l}$}
				child {node [rectangle, draw] {$\op{r}$}
					child {node {\phantom{1}}}
					child {node {\phantom{1}}}
				}
				child {node {\phantom{1}}}
			}
		;
	\end{tikzpicture}
	\longrightarrow
	\begin{tikzpicture}[baseline=-1.2cm, level/.style={sibling distance = 25mm/#1, level distance = 1.2cm/sqrt(#1)}]
		\node [rectangle, draw] {$\op{r}$}
			child {node [rectangle, draw, fill=red!50] {$\op{l}$}
				child {node [rectangle, draw, fill=red!50] {$\op{l}$}
					child {node [rectangle, draw] {$\op{r}$}
						child {node {\phantom{1}}}
						child {node {\phantom{1}}}
					}
					child {node [rectangle, draw] {$\op{l}$}
						child {node {\phantom{1}}}
						child {node {\phantom{1}}}
					}
				}
				child {node [rectangle, draw] {$\op{r}$}
					child {node {\phantom{1}}}
					child {node [rectangle, draw] {$\op{r}$}
						child {node {\phantom{1}}}
						child {node {\phantom{1}}}
					}
				}
			}
			child {node [rectangle, draw] {$\op{l}$}
				child {node [rectangle, draw] {$\op{r}$}
					child {node {\phantom{1}}}
					child {node {\phantom{1}}}
				}
				child {node {\phantom{1}}}
			}
		;
	\end{tikzpicture}
	+
	\begin{tikzpicture}[baseline=-1.2cm, level/.style={sibling distance = 25mm/#1, level distance = 1.2cm/sqrt(#1)}]
		\node [rectangle, draw] {$\op{r}$}
			child {node [rectangle, draw, fill=red!50] {$\op{l}$}
				child {node [rectangle, draw] {$\op{r}$}
					child {node {\phantom{1}}}
					child {node {\phantom{1}}}
				}
				child {node [rectangle, draw, fill=red!50] {$\op{r}$}
					child {node [rectangle, draw] {$\op{l}$}
						child {node {\phantom{1}}}
						child {node {\phantom{1}}}
					}
					child {node [rectangle, draw] {$\op{r}$}
						child {node {\phantom{1}}}
						child {node [rectangle, draw] {$\op{r}$}
							child {node {\phantom{1}}}
							child {node {\phantom{1}}}
						}
					}
				}
			}
			child {node [rectangle, draw] {$\op{l}$}
				child {node [rectangle, draw] {$\op{r}$}
					child {node {\phantom{1}}}
					child {node {\phantom{1}}}
				}
				child {node {\phantom{1}}}
			}
		;
	\end{tikzpicture}
\]
\end{example}

\begin{definition}
A \defn{normal form} of a rewriting system is a tree that is not rewritable by the rewriting system, which means that it does not contain any pattern~$\tree[s]$ of a rewriting rule~$(\tree[s], \operation[q])$.
\end{definition}

\begin{definition}
A rewriting system is
\begin{itemize}
\item \defn{terminating} if there are no infinite rewriting sequences, so that any tree can be finitely rewritten into a linear combination of normal forms,
\item \defn{confluent} if for any~$\tree \in \Syntax$ and~$\operation[p], \operation[q] \in \Free$ such that $\tree \to \operation[p]$ and $\tree \to \operation[q]$, there exists $\operation[r] \in \Free$ such that $\operation[p] \Rewr \operation[r]$ and $\operation[q] \Rewr \operation[r]$, so that any tree rewrites to at most one combination of normal forms,
\item \defn{convergent} when it is both terminating and confluent, so that any tree rewrites as a unique linear combination of normal forms.
\end{itemize}
\end{definition}

Any convergent rewriting system defines a presentation where the relation space~$\Relations$ is linearly spanned by $\tree[s] - \operation[q]$ for all rules~$(\tree[s], \operation[q])$ in the system. In this case, the quotient $\Free/\langle \Relations \rangle$ can be identified with the linear span of the normal forms of the rewriting system. Conversely, one can orient a presentation.

\begin{definition}
An \defn{orientation} of a binary quadratic presentation~$(\Operations, \Relations)$ is a quadratic rewriting system such that~$\Relations$ is linearly spanned by~$\tree[s] - \operation[q]$ for all rules~$(\tree[s], \operation[q])$ of the system.
\end{definition}


\subsection{Koszulity and Koszul duality}
\label{subsec:Koszul}

A Koszul operad is defined in terms of acyclicity of a certain complex naturally associated to it~\cite{LodayVallette}. 
In this paper we will only use the technique introduced by~\cite{Hoffbeck} that amounts to exhibiting a so-called Poincar\'e\,--\,Birkhoff\,--\,Witt basis. Moreover, we will only use this tools in the specific context of set operads (\ie where the right hand side of each rewriting rule consists of a single tree). This allows the following reformulation of~\cite{DotsenkoKhoroshkin} (see also~\cite{Giraudo-pluriassociative1}), which we take as definition.

\begin{definition}
A set operad~$\Operad$ is \defn{Koszul} if it admits a quadratic presentation whose relations can be oriented into a convergent rewriting system. Moreover, the set of normal forms of the rewriting system is called a \defn{Poincar\'e\,--\,Birkhoff\,--\,Witt basis} of~$\Operad$.
\end{definition}

Next comes the notion of the Koszul dual operad~$\Operad^\koszul$ of a quadratic operad~$\Operad$, which has particularly nice features when~$\Operad$ is Koszul. Consider the free operad~$\Free$ on~$\Operations$. The homogeneous component of degree~$3$ of~$\Free$ can be endowed with a bilinear form~$\dotprod{\cdot}{\cdot}$ defined~by
\[
\dotprod{\tree}{\tree'} = \delta_{\tree = \tree'} \cdot \sign(\tree),
\]
where~$\delta$ denotes the Kronecker delta and~$\sign(\operation[a] \circ_1 \operation[b]) = 1$ while~$\sign(\operation[a] \circ_2 \operation[b]) = -1$.

\begin{definition}
\label{def:Koszul}
The \defn{Koszul dual} of a quadratic operad~$\Operad$ presented by~$(\Operations, \Relations)$ is the quadratic operad~$\Operad^\koszul$ presented by~$(\Operations, \Relations^\koszul)$, whose relations~$\Relations^\koszul$ are given by the orthogonal complement for~$\dotprod{\cdot}{\cdot}$ of the relations~$\Relations$ in~$\Free$.
\end{definition}

We will often apply \cref{def:Koszul} in the following specific situation.

\begin{lemma}
\label{lem:koszulDual}
Consider an equivalence relation~$\equiv$ on $\Syntax(3)$.
Then the spaces linearly spanned by
\begin{align*}
& \bigset{\tree - \tree'}{\tree \equiv \tree'} \cup \bigset{\tree}{\tree \text{ alone in its $\equiv$-congruence class}} \\
\text{and} \qquad &
\bigset{\sum\nolimits_{\tree \in C} \sign(\tree) \cdot \tree}{C \text{ non-trivial $\equiv$-equivalence class}}
\end{align*}
repectively, are orthogonal complements, hence define two Koszul dual binary quadratic operads.
\end{lemma}

\begin{proof}
Let~$R_C \eqdef \sum\nolimits_{\tree \in C} \sign(\tree) \cdot \tree$ for some non-trivial $\equiv$-equivalence class~$C$.
If~$\tree$ is alone in its $\equiv$-congruence class, then it is disjoint from~$C$, hence~$\dotprod{\tree}{\tree[c]} = 0$ for all~$\tree[c] \in C$, so that~$\dotprod{\tree}{R_C} = 0$.
If~$\tree \equiv \tree'$, then we have two situations
\begin{itemize}
\item if~$\tree, \tree' \notin C$, then~$\dotprod{\tree}{\tree[c]} = \dotprod{\tree'}{\tree[c]} = 0$ for all~$\tree[c] \in C$, so that~$\dotprod{\tree-\tree'}{R_C} = 0$,
\item if~$\tree, \tree' \in C \cap (\Operations \circ_1 \Operations)$, then $\dotprod{\tree-\tree'}{R_C} = \dotprod{\tree}{R_C} - \dotprod{\tree'}{R_C} = \sign(\tree) \cdot \dotprod{\tree}{\tree} - \sign(\tree') \cdot \dotprod{\tree'}{\tree'} = \sign(\tree)^2 - \sign(\tree')^2 = 1 - 1 = 0$.
\qedhere
\end{itemize}
\end{proof}

The dual of a Kozul operad is always Koszul.
Moreover, the Hilbert series of Koszul dual Koszul operads are related by Lagrange inversion as stated in the following theorem.

\begin{theorem}[{\cite[Thm.~7.6.13]{LodayVallette}, \cite[Eq.~2.3.3]{Giraudo-nonsymmetricOperadsCombinatorics}}]
\label{thm:HilbertKoszulLagrange}
The Hilbert series of two Koszul dual Koszul operads~$\Operad$ and~$\Operad^\koszul$ satisfy
\[
\Hilbert_{\Operad}(-\Hilbert_{\Operad^\koszul}(-t)) = t.
\]
\end{theorem}


\subsection{Six particular operads}
\label{subsec:sixParticularOperads}

This paper deals with generalizations of the following three specific pairs of Koszul dual operads.


\subsubsection{Dendriform and diassociative operads}
\label{subsubsec:dendriformDiassociativeOperads}

\begin{definition}
The \defn{dendriform operad}~$\Dend$ is the quadratic operad over ${\Operations \eqdef \{\op{l}, \op{r}\}}$ defined by the three linear relations:
\medskip
\[
\compoR{l}{l} \!\!+\!\! \compoR{l}{r} \!\!=\!\! \compoL{l}{l}
\qquad
\compoR{r}{l} \!\!=\!\! \compoL{r}{l}
\qquad
\compoR{r}{r} \!\!=\!\! \compoL{l}{r} \!\!+\!\! \compoL{r}{r}.
\]
\end{definition}

\begin{definition}
\label{def:dias}
The \defn{diassociative operad}~$\Dias$ is the quadratic operad over~$\Operations \eqdef \{\op{l}, \op{r}\}$ defined by the five relations:
\medskip
\[
\compoR{l}{l} \!\!=\!\! \compoR{l}{r} \!\!=\!\! \compoL{l}{l}
\qquad
\compoR{r}{l} \!\!=\!\! \compoL{r}{l}
\qquad
\compoR{r}{r} \!\!=\!\! \compoL{l}{r} \!\!=\!\! \compoL{r}{r}.
\]
\end{definition}

\begin{proposition}[\cite{Loday-dialgebras,LodayVallette}]
\label{prop:dendDiass}
\begin{enumerate}[(i)]
\item The dendriform and diassociative operads are Koszul.
\item The dendriform and diassociative operads are Koszul dual operads.
\item The diassociative operad is isomorphic to the operad~$(\Operad[P], \one, (\circ_i))$ with $p$-th homogeneous component $\Operad[P](p) \eqdef \K \set{\destVect{r}{p}}{r \in [p]}$, unit~$\one = \destVect{1}{1}$ and where the composition~$\circ_i$ is defined~by
\[
\destVect{r}{p} \circ_i \destVect{s}{q} =
\begin{cases}
\destVect{r}{p+q-1} & \text{if } r < i, \\
\destVect{r + s - 1}{p+q-1} & \text{if } r = i, \\
\destVect{r + q - 1}{p+q-1} & \text{if } r > i.
\end{cases}
\]
\item The Hilbert series of the dendriform and diassociative operads are given by
\[
\Hilbert_{\Dend}(t) = \sum_{p \ge 1} C_p \, t^p = \frac{1-\sqrt{1-4t}}{2t}-1
\qqandqq
\Hilbert_{\Dias}(t) = \sum_{p \ge 1} p \, t^p = \frac{t}{(1-t)^2},
\]
where~$C_p \eqdef \frac{1}{p+1} \binom{2p}{p}$ denotes the $p$-th Catalan number.
\end{enumerate}
\end{proposition}

\begin{remark}
The shuffle algebra (see \cref{def:shuffle}) can be endowed with the structure of a dendriform algebra~\cite{Loday-dialgebras} defined for two words~$xX$ and~$yY$ by
\[
xX \op{l} yY = x (X \shuffle yY)
\qqandqq
xX \op{r} yY = y (xX \shuffle Y).
\]
Similarly, C.~Malvenuto and C.~Reutenauer's Hopf algebra~$\FQSym$ on permutations can be endowed with the structure of a dendriform algebra by splitting the shifted shuffle product (see \cref{def:shiftedShuffle}).
\end{remark}


\subsubsection{Duplicial and dual duplicial operads}
\label{subsubsec:duplicialOperads}

\begin{definition}
The \defn{duplicial operad}~$\Dup$ is the quadratic operad over ${\Operations \eqdef \{\op{l}, \op{r}\}}$ defined by the three relations:
\medskip
\[
\compoR{l}{l} \!\!=\!\! \compoL{l}{l}
\qquad
\compoR{r}{l} \!\!=\!\! \compoL{r}{l}
\qquad
\compoR{r}{r} \!\!=\!\! \compoL{r}{r}.
\]
\end{definition}

\begin{definition}
The \defn{dual duplicial operad}~$\Dupdual$ is the quadratic operad over~$\Operations \eqdef \{\op{l}, \op{r}\}$ defined by the five linear relations:
\medskip
\[
\compoR{l}{l} \!\!=\!\! \compoL{l}{l}
\quad
\compoR{l}{r} \!\!= 0
\quad
\compoR{r}{l} \!\!=\!\! \compoL{r}{l}
\quad
0 =\!\! \compoL{l}{r}
\quad
\compoR{r}{r} \!\!=\!\! \compoL{r}{r}.
\]
\end{definition}

\begin{proposition}[\cite{Loday-generalizedBialgebras,LodayVallette}]
\label{prop:dup}
\begin{enumerate}[(i)]
\item The duplicial and dual duplicial operads are Koszul.
\item The duplicial and dual duplicial operads are Koszul dual operads.
\item The duplicial operad is isomorphic to the operad~$(\Operad[P], \one, (\circ_i))$ with $p$-th homogeneous component $\Operad[P](p) \eqdef \K \set{\destVect{r}{p}}{r \in [p]}$, unit~$\one = \destVect{1}{1}$ and where the composition~$\circ_i$ is defined~by
\[
\destVect{r}{p} \circ_i \destVect{s}{q} =
\begin{cases}
0 & \text{if } r < i \text{ and } s \ne 1, \\
\destVect{r}{p+q-1} & \text{if } r < i \text{ and } s = 1, \\
\destVect{r + s - 1}{p+q-1} & \text{if } r = i, \\
0 & \text{if } r > i \text{ and } s \ne q, \\
\destVect{r + q - 1}{p+q-1} & \text{if } r > i \text{ and } s = q.
\end{cases}
\]
\item The Hilbert series of the duplicial and dual duplicial operads are given by
\[
\Hilbert_{\Dup}(t) = \sum_{p \ge 1} C_p \, t^p = \frac{1-\sqrt{1-4t}}{2t}-1
\qqandqq
\Hilbert_{\Dupdual}(t) = \sum_{p \ge 1} p \, t^p = \frac{t}{(1-t)^2},
\]
where~$C_p \eqdef \frac{1}{p+1} \binom{2p}{p}$ denotes the $p$-th Catalan number.
\end{enumerate}
\end{proposition}

\begin{remark}
\label{rem:associativeProductsDend}
Any operation~$\operation \eqdef \mu \, {\op{l}} \,  + \, \nu \, {\op{r}}$ with~$\mu, \nu \in \K$ in the operad~$\Dupdual$ is associative.
In other words, $\op{l}$ and~$\op{r}$ define a compatible associative structure.
\end{remark}

\subsubsection{Leibniz and Zinbiel operads}
\label{subsubsec:LeibnizZinbieOperads}

Leibniz algebras are natural generalizations of Lie algebras with non antisymmetric brackets, and Zinbiel algebras play a salient role in the analysis of divided power algebras~\cite{Loday-generalizedBialgebras, LodayVallette}.
The Zinbiel operad is usually defined as the symmetric Koszul dual of the Leibniz symmetric operad.
As a symmetric operad, it is generated by~$\op{l}$ with the unique relation $(x \op{l} y) \op{l} z = x \op{l} (y \op{l} z + z \op{l} y)$.
In this paper, we use the non-symmetric Zinbiel operad~$\Zinb$, that has the same components~$\Zinb(n)$ and the same composition rules as the symmetric one, but the action of the symmetric group on~$\Zinb(n)$ is trivial for all~$n \ge 2$.
Alternatively, following~\cite{ChapotonHivertAlMould}, we can directly define the non-symmetric Zinbiel operad~$\Zinb$ combinatorially as follows.

\pagebreak
\begin{definition}
The (non-symmetric) \defn{Zinbiel operad}~$\Zinb$ is the operad whose arity~$n$ component~$\Zinb(n)$ has for basis the set of permutations~$\Perm(n)$ and where the composition is given as follows.
Let~$\sigma$ and~$\tau$ be two permutations of degree~$m$ and~$n$ respectively, and let~$i \in [m]$.
Write~$\sigma = \lambda \, i \, \mu$ and~$\tau = f \, \theta$ where $f$ is the first letter of~$\tau$. Then the composition~is~given~by
\[
\sigma \circ_i \tau = \lambda[i,n] \, f[i-1] \, \big( \mu[i,n] \shuffle \theta[i-1] \big)\,,
\]
where we use the one-line notation for permutations, and the usual shuffle product $\shuffle$ of \cref{def:shuffle} and the operad shifting notation of \cref{def:shiftMultipermutation} and \cref{def:shiftMultipermutationPlace}. In this combinatorial realization, the generator $\op{l}$ of the presentation is given by~${\op{l}} \eqdef 12$.
\end{definition}

Finally, it was observed \eg in~\cite[Example~4.2]{ChapotonHivertAlMould}, that the dendriform operad is the suboperad generated by ${\op{l}} \eqdef 12$ and ${\op{r}} \eqdef 21$ of the non-symmetric Zinbiel operad.


\subsection{Manin products}
\label{subsec:ManinProducts}

The original Manin product was defined for quadratic algebras and generalized by B.~Vallette~\cite{Vallette} to any category endowed with two coherent monoidal products which include, as a particular case, non-symmetric operads. As an application, he showed that
\[
\Quad = \Dend \whiteManin \Dend = \Dend \blackManin \Dend.
\]
We will generalize this statement to some of our operads in \cref{subsec:parallelSignaleticOperads,subsec:parallelCitelangisOperads}.
As the exposition in~\cite{Vallette} is too general for our purposes, we follow the former version of~\cite{Foissy} still available on \href{https://arxiv.org/abs/1411.6501v1}{\texttt{arxiv:1411.6501v1}}.

Consider two operads~$\Operad \eqdef \Free/\langle \Relations \rangle$ and~$\Operad' \eqdef \Free[\Operations']/\langle \Relations' \rangle$. Then there is a natural inclusion~$\Psi : \Free[\Operations \times \Operations'] \to \Free[\Operations] \otimes \Free[\Operations']$, which sends a syntax tree on~$\Operations \times \Operations'$ to a pair formed by a syntax tree on~$\Operations$ and a syntax tree on~$\Operations'$.

\begin{definition}
The \defn{white Manin product}~$\Operad \whiteManin \Operad'$ is the operad with generators~$\Operations \times \Operations'$ and relations~$\Relations'' \subseteq \Free[\Operations \times \Operations']$ defined by
\[
\Relations'' \eqdef \Psi^{-1} \big( \Relations \otimes \Free[\Operations'] + \Free[\Operations] \otimes \Relations' \big).
\]
The \defn{black Manin product} is defined by Koszul duality as~$\Operad \blackManin \Operad' \eqdef \big( \Operad^\koszul \whiteManin \Operad'^\koszul \big)^\koszul$.
\end{definition}

For quadratic operads the white product can be computed without relying on presentation by the following characterization, which is also closer to the original definition of V.~Gizburg and M.~Kapranov in the case of symmetric operads~\cite{GinzburgKapranov}.

\begin{proposition}[{[\href{https://arxiv.org/abs/1411.6501v1}{\texttt{arxiv:1411.6501v1}}, Proposition~2]}]
\label{prop:whiteManin}
For two quadratic operads~$\Operad, \Operad'$, the white Manin product~$\Operad \whiteManin \Operad'$ is the suboperad of the tensor operad~$\Operad \otimes \Operad'$ generated by the homogeneous component of degree~$2$.
\end{proposition}


\subsection{Eulerian numbers and polynomials}
\label{subsec:EulerianNumbers}

We now focus on Eulerian numbers, which are underlying all the present study. Indeed, from a purely combinatorial point of view, the main result of the present paper is to give a combinatorial interpretation of the Lagrange inverse of the generating series $\sum_p p^k t^p$ which involve Eulerian numbers. These numbers were introduced by L.~Euler in the context of differential calculus.

\begin{definition}
For~$j < k$, the \defn{Eulerian number}~$\EulNum{k}{j}$ is the number of permutations~$\sigma$ of~$\fS_k$ with precisely~$j$ descents (\ie positions~$i \in [k]$ such that~$\sigma_i > \sigma_{i+1}$). The $k$-th \defn{Eulerian polynomial} is the polynomial
\[
\EulPol{k}(t) \eqdef \sum\limits_{j = 0}^{k-1} \EulNum{k}{j} \, t^{j} = \sum_{\sigma \in \fS_k} t^{|\Des(\sigma)|} \, .
\]
\end{definition}

\begin{example}
\cref{table:EulerianNumbers} gathers the first Eulerian numbers.
The first Eulerian polynomials~are:
\[
\EulPol{1}(t) = 1, \quad \EulPol{2}(t) = t + 1, \quad \EulPol{3}(t) = t^2 + 4t + 1, \quad \EulPol{4}(t) = t^3 + 11t^2 + 11t + 1, \quad \dots
\]
\end{example}

We will use the following classical identity on Eulerian numbers.

\begin{lemma}[\cite{Worpitzky}]
\label{lem:Worpitzky}
For any~$p,k \in \N$, we have $\displaystyle p^k = \sum_{j=0}^{k-1} \EulNum{k}{j} \binom{j+p}{k}$.
\end{lemma}

\begin{proof}
It is shown by induction using the identity $\EulNum{k}{j} = (j+1) \EulNum{k-1}{j} + (k-j) \EulNum{k-1}{j-1}$, see \eg \cite{Stanley} (the insertion of~$k$ in a permutation of size~$k-1$ creates a new descent if and only if it is performed at an ascent position). Another approach is to observe that the standardization provides a bijection between the words on~$[p]$ with $k$ letters and the pairs~$(\sigma, \pi)$ where~$\sigma$ is a permutations of~$\fS_k$ and $\pi$ a composition of~$p$ (with possible empty parts) refining the descent composition of~$\sigma$.
\end{proof}

\cref{lem:Worpitzky} translates to an identity of generating functions, and to a differential recursion.

\begin{proposition}
For any~$k \in \N$, we have
\(
\displaystyle
\sum_{p \ge 1} p^k t^p = \frac{t \cdot \EulPol{k}(t)}{(1-t)^{k+1}} \, .
\)
\end{proposition}

\begin{proposition}
For any~$k \in \N$, we have
\(
\displaystyle
\EulPol{k}(t) = t (1-t) \EulPol{k-1}'(t) + \big( (k-1) t + 1 \big) \EulPol{k-1}(t).
\)
\end{proposition}

\begin{table}
	\[
	\begin{array}{l|rrrrrrrrrr}
		\raisebox{-.1cm}{$k$} \backslash \, \raisebox{.1cm}{$j$}
		   & 0 &    1 &     2 &      3 &       4 &       5 &      6 &     7 &    8 & 9 \\[.1cm]
		\hline
		1  & 1 &      &       &        &         &         &        &       &      &   \\
		2  & 1 &    1 &       &        &         &         &        &       &      &   \\
		3  & 1 &    4 &     1 &        &         &         &        &       &      &   \\
		4  & 1 &   11 &    11 &      1 &         &         &        &       &      &   \\
		5  & 1 &   26 &    66 &     26 &       1 &         &        &       &      &   \\
		6  & 1 &   57 &   302 &    302 &      57 &       1 &        &       &      &   \\
		7  & 1 &  120 &  1191 &   2416 &    1191 &     120 &      1 &       &      &   \\
		8  & 1 &  247 &  4293 &  15619 &   15619 &    4293 &    247 &     1 &      &   \\
		9  & 1 &  502 & 14608 &  88234 &  156190 &   88234 &  14608 &   502 &    1 &   \\
		10 & 1 & 1013 & 47840 & 455192 & 1310354 & 1310354 & 455192 & 47840 & 1013 & 1 \\
	\end{array}
	\]
	\caption{The Eulerian numbers~$\displaystyle\EulNum{k}{j}$ for~$0 \le j < k \le 10$.}
	\label{table:EulerianNumbers}
\end{table}


\subsection{Tamari lattice}
\label{subsec:TamariLattice}

The other classical combinatorial ingredient of the present paper is a lattice on the set of binary trees of a given arity, defined by D.~Tamari~\cite{Tamari}. We recall some basic~facts.

\begin{definition}
\label{def:TamariLattice}
A \defn{right rotation} in a binary tree is a substitution of the form
\[
	\begin{tikzpicture}[baseline=0.4cm, level 1/.style={sibling distance = .8cm, level distance = .6cm}, level 2/.style={sibling distance = .8cm, level distance = .5cm}]
		\node [circle, draw] {} [grow' = up]
			child {node [circle, draw] {}
				child {node {$A$}}
				child {node {$B$}}
			}
			child {node {$C$}}
		;
	\end{tikzpicture}
	\quad\longrightarrow\quad
	\begin{tikzpicture}[baseline=0.4cm, level 1/.style={sibling distance = .8cm, level distance = .6cm}, level 2/.style={sibling distance = .8cm, level distance = .5cm}]
		\node [circle, draw] {} [grow' = up]
			child {node {$A$}}
			child {node [circle, draw] {}
				child {node {$B$}}
				child {node {$C$}}
			}
		;
	\end{tikzpicture}
\]
where~$A$, $B$ and~$C$ are three binary subtrees and the substitution is performed at any node of the tree (not necessarily at the root).
The \defn{Tamari lattice} is the partial order on binary trees with~$n$ internal nodes whose cover relations are given by right rotations.
\end{definition}

The Tamari lattice is known to be a lattice.
Its Hasse diagram (graph of cover relations) is the skeleton (graph of vertices and edges, see \cite{Ziegler-polytopes}) of an $(n-1)$-dimensional polytope, called the associahedron, see \eg~\cite{Loday}.
We have represented in \cref{fig:TamariLatticeBinaryTrees} the Tamari lattice for~$n = 4$ and in \cref{fig:associahedra} J.-L.~Loday's $2$- and $3$-dimensional associahedra.

\begin{figure}
	\centerline{
    \begin{tikzpicture}[scale=2, inner sep=2pt]
    	\node (a) at (4.2,1) {
    		\begin{tikzpicture}[xscale=.3, yscale=.4, inner sep=1pt]
    			\node (1) at (1,4) {$1$};
    			\node (2) at (2,3) {$2$};
    			\node (3) at (3,2) {$3$};
    			\node (4) at (4,1) {$4$};
    			\draw[-] (4) -- (3) -- (2) -- (1);
    			\draw[-] (1) -- +(-.7,.7);
    			\draw[-] (1) -- +(.7,.7);
    			\draw[-] (2) -- +(.7,.7);
    			\draw[-] (3) -- +(.7,.7);
    			\draw[-] (4) -- +(.7,.7);
    		\end{tikzpicture}	
    	};
    	\node (b) at (3,2) {
    		\begin{tikzpicture}[xscale=.3, yscale=.4, inner sep=1pt]
    			\node (1) at (1,3) {$1$};
    			\node (2) at (2,4) {$2$};
    			\node (3) at (3,2) {$3$};
    			\node (4) at (4,1) {$4$};
    			\draw[-] (4) -- (3) -- (1) -- (2);
    			\draw[-] (1) -- +(-.7,.7);
    			\draw[-] (2) -- +(-.7,.7);
    			\draw[-] (2) -- +(.7,.7);
    			\draw[-] (3) -- +(.7,.7);
    			\draw[-] (4) -- +(.7,.7);
    		\end{tikzpicture}	
    	};
    	\node (c) at (2,3) {
    		\begin{tikzpicture}[xscale=.3, yscale=.4, inner sep=1pt]
    			\node (1) at (1,2) {$1$};
    			\node (2) at (2,4) {$2$};
    			\node (3) at (3,3) {$3$};
    			\node (4) at (4,1) {$4$};
    			\draw[-] (4) -- (1) -- (3) -- (2);
    			\draw[-] (1) -- +(-.7,.7);
    			\draw[-] (2) -- +(-.7,.7);
    			\draw[-] (2) -- +(.7,.7);
    			\draw[-] (3) -- +(.7,.7);
    			\draw[-] (4) -- +(.7,.7);
    		\end{tikzpicture}	
    	};
    	\node (d) at (1,4) {
    		\begin{tikzpicture}[xscale=.3, yscale=.4, inner sep=1pt]
    			\node (1) at (1,2) {$1$};
    			\node (2) at (2,3) {$2$};
    			\node (3) at (3,4) {$3$};
    			\node (4) at (4,1) {$4$};
    			\draw[-] (4) -- (1) -- (2) -- (3);
    			\draw[-] (1) -- +(-.7,.7);
    			\draw[-] (2) -- +(-.7,.7);
    			\draw[-] (3) -- +(-.7,.7);
    			\draw[-] (3) -- +(.7,.7);
    			\draw[-] (4) -- +(.7,.7);
    		\end{tikzpicture}	
    	};
    	\node (e) at (2,5) {
    		\begin{tikzpicture}[xscale=.3, yscale=.4, inner sep=1pt]
    			\node (1) at (1,1) {$1$};
    			\node (2) at (2,3) {$2$};
    			\node (3) at (3,4) {$3$};
    			\node (4) at (4,2) {$4$};
    			\draw[-] (1) -- (4) -- (2) -- (3);
    			\draw[-] (1) -- +(-.7,.7);
    			\draw[-] (2) -- +(-.7,.7);
    			\draw[-] (2) -- +(.7,.7);
    			\draw[-] (3) -- +(-.7,.7);
    			\draw[-] (4) -- +(.7,.7);
    		\end{tikzpicture}	
    	};
    	\node (f) at (3,6) {
    		\begin{tikzpicture}[xscale=.3, yscale=.4, inner sep=1pt]
    			\node (1) at (1,1) {$1$};
    			\node (2) at (2,2) {$2$};
    			\node (3) at (3,4) {$3$};
    			\node (4) at (4,3) {$4$};
    			\draw[-] (1) -- (2) -- (4) -- (3);
    			\draw[-] (1) -- +(-.7,.7);
    			\draw[-] (2) -- +(-.7,.7);
    			\draw[-] (3) -- +(-.7,.7);
    			\draw[-] (3) -- +(.7,.7);
    			\draw[-] (4) -- +(.7,.7);
    		\end{tikzpicture}	
    	};
    	\node (g) at (4.2,7) {
    		\begin{tikzpicture}[xscale=.3, yscale=.4, inner sep=1pt]
    			\node (1) at (1,1) {$1$};
    			\node (2) at (2,2) {$2$};
    			\node (3) at (3,3) {$3$};
    			\node (4) at (4,4) {$4$};
    			\draw[-] (1) -- (2) -- (3) -- (4);
    			\draw[-] (1) -- +(-.7,.7);
    			\draw[-] (2) -- +(-.7,.7);
    			\draw[-] (3) -- +(-.7,.7);
    			\draw[-] (4) -- +(-.7,.7);
    			\draw[-] (4) -- +(.7,.7);
    		\end{tikzpicture}	
    	};
    	\node (h) at (3.5,3.1) {
    		\begin{tikzpicture}[xscale=.3, yscale=.4, inner sep=1pt]
    			\node (1) at (1,3) {$1$};
    			\node (2) at (2,2) {$2$};
    			\node (3) at (3,3) {$3$};
    			\node (4) at (4,1) {$4$};
    			\draw[-] (4) -- (2) -- (1); \draw[-] (2) -- (3);
    			\draw[-] (1) -- +(-.7,.7);
    			\draw[-] (1) -- +(.7,.7);
    			\draw[-] (3) -- +(-.7,.7);
    			\draw[-] (3) -- +(.7,.7);
    			\draw[-] (4) -- +(.7,.7);
    		\end{tikzpicture}	
    	};
    	\node (i) at (3,4) {
    		\begin{tikzpicture}[xscale=.3, yscale=.4, inner sep=1pt]
    			\node (1) at (1,1) {$1$};
    			\node (2) at (2,4) {$2$};
    			\node (3) at (3,3) {$3$};
    			\node (4) at (4,2) {$4$};
    			\draw[-] (1) -- (4) -- (3) -- (2);
    			\draw[-] (1) -- +(-.7,.7);
    			\draw[-] (2) -- +(-.7,.7);
    			\draw[-] (2) -- +(.7,.7);
    			\draw[-] (3) -- +(.7,.7);
    			\draw[-] (4) -- +(.7,.7);
    		\end{tikzpicture}	
    	};
    	\node (j) at (4.2,4) {
    		\begin{tikzpicture}[xscale=.3, yscale=.4, inner sep=1pt]
    			\node (1) at (1,2) {$1$};
    			\node (2) at (2,1) {$2$};
    			\node (3) at (3,3) {$3$};
    			\node (4) at (4,2) {$4$};
    			\draw[-] (2) -- (1); \draw[-] (2) -- (4) -- (3);
    			\draw[-] (1) -- +(-.7,.7);
    			\draw[-] (1) -- +(.7,.7);
    			\draw[-] (3) -- +(-.7,.7);
    			\draw[-] (3) -- +(.7,.7);
    			\draw[-] (4) -- +(.7,.7);
    		\end{tikzpicture}	
    	};
    	\node (k) at (4.2,5.5) {
    		\begin{tikzpicture}[xscale=.3, yscale=.4, inner sep=1pt]
    			\node (1) at (1,1) {$1$};
    			\node (2) at (2,3) {$2$};
    			\node (3) at (3,2) {$3$};
    			\node (4) at (4,3) {$4$};
    			\draw[-] (1) -- (3) -- (2); \draw[-] (3) -- (4);
    			\draw[-] (1) -- +(-.7,.7);
    			\draw[-] (2) -- +(-.7,.7);
    			\draw[-] (2) -- +(.7,.7);
    			\draw[-] (4) -- +(-.7,.7);
    			\draw[-] (4) -- +(.7,.7);
    		\end{tikzpicture}	
    	};
    	\node (l) at (5.4,4) {
    		\begin{tikzpicture}[xscale=.3, yscale=.4, inner sep=1pt]
    			\node (1) at (1,2) {$1$};
    			\node (2) at (2,3) {$2$};
    			\node (3) at (3,1) {$3$};
    			\node (4) at (4,2) {$4$};
    			\draw[-] (3) -- (1) -- (2); \draw[-] (3) -- (4);
    			\draw[-] (1) -- +(-.7,.7);
    			\draw[-] (2) -- +(-.7,.7);
    			\draw[-] (2) -- +(.7,.7);
    			\draw[-] (4) -- +(-.7,.7);
    			\draw[-] (4) -- +(.7,.7);
    		\end{tikzpicture}	
    	};
    	\node (m) at (6.6,3) {
    		\begin{tikzpicture}[xscale=.3, yscale=.4, inner sep=1pt]
    			\node (1) at (1,3) {$1$};
    			\node (2) at (2,2) {$2$};
    			\node (3) at (3,1) {$3$};
    			\node (4) at (4,2) {$4$};
    			\draw[-] (3) -- (2) -- (1); \draw[-] (3) -- (4);
    			\draw[-] (1) -- +(-.7,.7);
    			\draw[-] (1) -- +(.7,.7);
    			\draw[-] (2) -- +(.7,.7);
    			\draw[-] (4) -- +(-.7,.7);
    			\draw[-] (4) -- +(.7,.7);
    		\end{tikzpicture}	
    	};
    	\node (n) at (6,6) {
    		\begin{tikzpicture}[xscale=.3, yscale=.4, inner sep=1pt]
    			\node (1) at (1,2) {$1$};
    			\node (2) at (2,1) {$2$};
    			\node (3) at (3,2) {$3$};
    			\node (4) at (4,3) {$4$};
    			\draw[-] (2) -- (1); \draw[-] (2) -- (3) -- (4);
    			\draw[-] (1) -- +(-.7,.7);
    			\draw[-] (1) -- +(.7,.7);
    			\draw[-] (3) -- +(-.7,.7);
    			\draw[-] (4) -- +(-.7,.7);
    			\draw[-] (4) -- +(.7,.7);
    		\end{tikzpicture}	
    	};
    	\draw[-, ultra thick] (a) -- (b);
    	\draw[-, ultra thick] (a) -- (h);
    	\draw[-, ultra thick] (a) -- (m);
    	\draw[-, ultra thick] (b) -- (c);
    	\draw[-, ultra thick] (b) -- (l);
    	\draw[-, ultra thick] (c) -- (d);
    	\draw[-, ultra thick] (c) -- (i);
    	\draw[-, ultra thick] (d) -- (e);
    	\draw[-, ultra thick] (e) -- (f);
    	\draw[-, ultra thick] (f) -- (g);
    	\draw[-, ultra thick] (h) -- (d);
    	\draw[-, ultra thick] (h) -- (j);
    	\draw[-, ultra thick] (i) -- (e);
    	\draw[-, ultra thick] (i) -- (k);
    	\draw[-, ultra thick] (j) -- (f);
    	\draw[-, ultra thick] (j) -- (n);
    	\draw[-, ultra thick] (k) -- (g);
    	\draw[-, ultra thick] (l) -- (k);
    	\draw[-, ultra thick] (m) -- (l);
    	\draw[-, ultra thick] (m) -- (n);
    	\draw[-, ultra thick] (n) -- (g);
    \end{tikzpicture}
    }
	\caption{The Tamari lattice on binary trees. See \cref{def:TamariLattice}.}
	\label{fig:TamariLatticeBinaryTrees}
\end{figure}
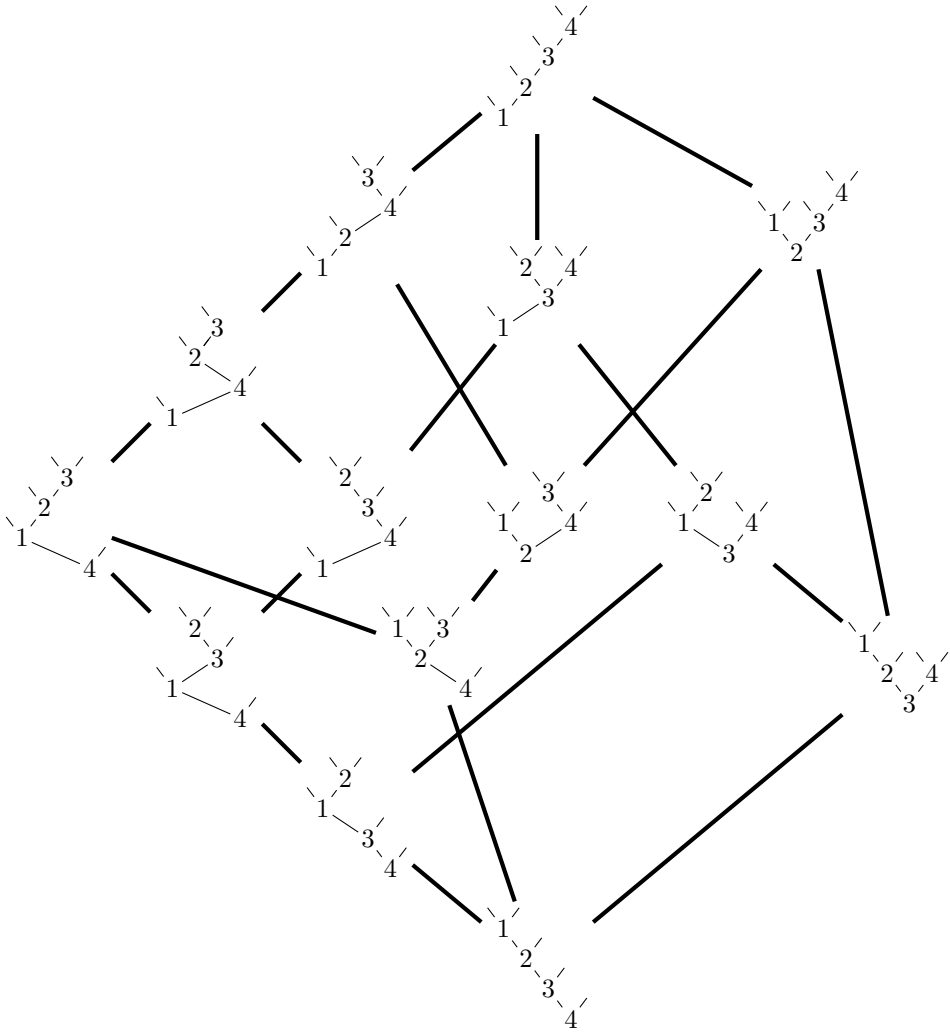

\begin{figure}
    \centerline{\raisebox{.8cm}{\includegraphics[scale=.6]{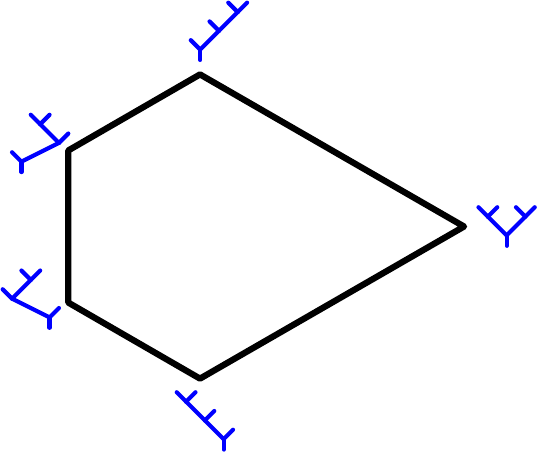}} \quad \includegraphics[scale=.6]{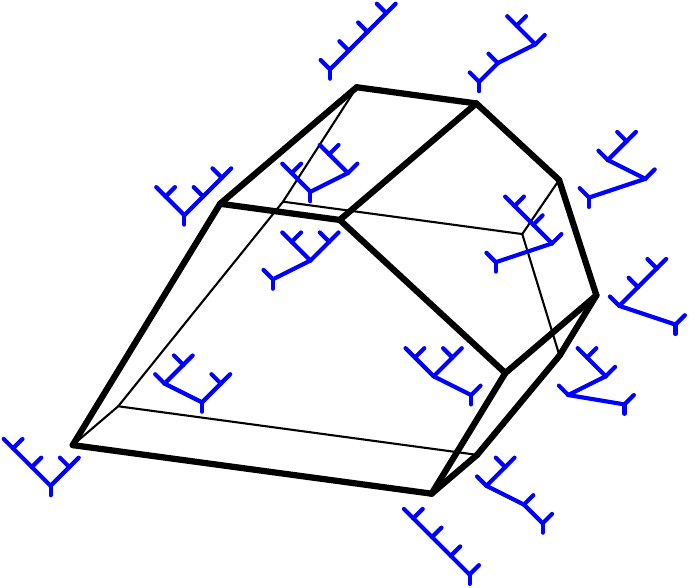}}
	\caption{J.-L.~Loday's associahedra~\cite{Loday} of dimension $2$ (left) and $3$ (right).}
	\label{fig:associahedra}
\end{figure}


\subsection{Multipermutations and multiposets}
\label{subsec:multiobjects}

We finally introduce some combinatorial objects and tools needed to define actions of our operads in \cref{sec:actions}.


\subsubsection{Multisets}
\label{subsubsec:multisets}

We call \defn{multiset} a set
\[
M \eqdef \{1_1, 1_2, \dots, 1_{a_1}, 2_1, 2_2, \dots, 2_{a_2}, \dots, m_1, m_2, \dots, m_{a_m}\}
\]
with several copies of each element that we distinguish with indices when necessary.
For brevity, we use the monomial notation~$M \eqdef 1^{\{a_1\}} 2^{\{a_2\}} \dots m^{\{a_m\}}$ where the exponents are surrounded by curly brackets to avoid any ambiguity.
We denote by~$|M| \eqdef \sum_{i \in [m]} a_i$ the cardinality of~$M$, by~$\max(M) \eqdef m$ its largest element, and by~$\supp(M) \eqdef \set{i \in [m]}{a_i \ne 0}$ the support of~$M$.
It will be convenient to have the following restriction, shift and standardization operations.
These operations are classically used to define the product and coproduct of Hopf algebras and the compositions of operads.

\begin{definition}
\label{def:restrictionMultiset}
For a multiset~$M \eqdef 1^{\{a_1\}} \dots m^{\{a_m\}}$ and a subset~$L \eqdef \{\ell_1 < \dots < \ell_p\}$ of~$[m]$  (possibly empty), we define the \defn{restricted} multiset~$M^{|L} \eqdef 1^{\{a_{\ell_1}\}} \dots p^{\{a_{\ell_p}\}}$.
For~$w$ in~$M^{|L}$ we let~$\overline{w} \eqdef \ell_w$ denote the corresponding element in~$M$.
\end{definition}

\begin{definition}
\label{def:shiftMultiset}
For a multiset~$M \eqdef 1^{\{a_1\}} \dots m^{\{a_m\}}$ and an integer~$j$, we define the \defn{shifted} multiset $M[j] \eqdef (1+j)^{\{a_1\}} \dots (m+j)^{\{a_m\}}$.
For~$w$ in~$M[j]$, we let~$\overline{w} \eqdef w-j$ denote the corresponding element~of~$M$.
\end{definition}

\begin{definition}
\label{def:shiftMultisetPlace}
For a multiset~$M \eqdef 1^{\{a_1\}} \dots m^{\{a_m\}}$ and two integers~$i \le m$ and~$j$, we define the multiset $M[i,j] \eqdef 1^{\{a_1\}} \dots (i-1)^{\{a_{i-1}\}} (i+1+j-1)^{\{a_{i+1}\}} \dots (m+j-1)^{\{a_m\}}$.
In other words, we erase all values~$i$ and replace values~$v > i$ by~$v+j-1$.
For~$w$ in~$M[i,j]$, we let~$\overline{w}$ denote the corresponding element in~$M$, that is
\[
\overline{w} \eqdef
\begin{cases}
w & \text{if } w \le i-1, \\
w-j+1 & \text{if } w \ge i+j.
\end{cases}
\]
\end{definition}

Note that the notation $\overline{x}$ of the three previous definitions is ambiguous on purpose.
Its intended meaning should be always clear from the context.

\begin{definition}
\label{def:standardizationMultiset}
The \defn{standardization} of a multiset~$M \eqdef 1^{\{a_1\}} 2^{\{a_2\}} \dots m^{\{a_m\}}$ is the bijection~$\Std$ from~$M$ to~$[|M|]$ defined by
\[
\Std(i_u) \eqdef u + \sum_{j < i} a_j.
\]
\end{definition}

For example, consider the multiset~$M = 1^{\{3\}}2^{\{1\}}4^{\{2\}} = \{1_1, 1_2, 1_3, 2_1, 4_1, 4_2\}$.
Then we have~$M^{|\{1,4\}} = 1^{\{3\}}2^{\{2\}}$, $M[3] = 4^{\{3\}}5^{\{1\}}7^{\{2\}}$, $M[2,4] = 1^{\{3\}}7^{\{2\}}$, $\Std(2_1) = 4$ and~$\Std(4_2) = 6$.

\begin{remark}
\label{rem:relationOperatorsMultisets}
Note that there are natural relations between these operations. Among others:
\[
(M^{|L})^{|X} = M^{|\set{\ell_x}{x \in X}},
\quad
(M^{|L})[j] = (M[j])^{|\set{\ell+j}{\ell \in L}}
\qandq
(M[i,j])[k] = (M[k])[i+k, j].
\]
\end{remark}


\subsubsection{Multipermutations}
\label{subsubsec:multipermutations}

We consider a permutation of a finite set~$X$ as a word whose letters are the elements of~$X$.

\begin{definition}
\label{def:multipermutation}
A \defn{multipermutation} of a multiset~$M \eqdef 1^{\{a_1\}} 2^{\{a_2\}} \dots m^{\{a_m\}}$ is a permutation of~$M$ where the copies of the same integer appear in natural order, meaning that~$i_u$ is to the left of~$i_v$ for all $i \in [m]$ and~$0 < u < v \le a_i$.
Since the different copies appear in natural order, there is no need to distinguish between them when writing a multipermutation.

For two integers~$m$ and~$k>0$, a \defn{$k$-permutation} of degree~$m$ is a multipermutation of the multiset~$1^{\{k\}} 2^{\{k\}} \dots m^{\{k\}}$.
We denote by~$\Perm_k$ the set of $k$-permutations, by~$\Perm_k(m)$ those of degree $m$, and by~${\FQSym_k = \bigoplus_{m > 0} \K\Perm_k(m)}$ the graded $\K$-vector space whose $m$-th homogeneous component has basis~$\Perm_k(m)$.
\end{definition}

The cardinality of~$\Perm_k(m)$ is a multinomial coefficient:
\[
|\Perm_k(m)| = \binom{mk}{k^{\{m\}}} = \frac{(mk)!}{(k!)^m} \, .
\]
For example, the word~$132123$ is one of the $90$ $2$-permutations of degree~$3$, the word~$321312132$ is one of the~$1680$ $3$-permutations of degree~$3$, and the word~$31421324$ is one of the~$2520$ $2$-permutations of degree~$4$.
We now define the restriction, shift and standardization operations on multipermutations, similar to \cref{def:restrictionMultiset,def:shiftMultiset,def:shiftMultisetPlace,def:standardizationMultiset}.

\begin{definition}
\label{def:restrictionMultipermutation}
For a multipermutation~$\mu$ of~$M$ and a subset~$L \eqdef \{\ell_1 < \dots < \ell_p\}$ of~$[\max(M)]$ (possibly empty), the \defn{restricted} multipermutation~$\mu^{|L}$ of~$M^{|L}$ is obtained from~$\mu$ by keeping only the values of~$\mu$ that belong to~$L$ while preserving their original order, and replacing~$\ell_w$ by~$w$.
\end{definition}

\begin{definition}
\label{def:shiftMultipermutation}
For a multipermutation~$\mu$ of~$M$ and an integer~$j$, the \defn{shifted} multipermutation~$\mu[j]$ on~$M[j]$ is defined by~$\mu[j]_p = \mu_p+j$ for any~$p \in [|M|]$.
\end{definition}

\begin{definition}
\label{def:shiftMultipermutationPlace}
For a multipermutation~$\mu$ of~$M$ and two integers~$i \le \max(M)$ and~$j$, the multipermutation $\mu[i,j]$ of~$M[i,j]$ is obtained from~$\mu$ by erasing all values~$i$ and replacing all values~$v > i$ by~$v + j -1$.
\end{definition}

\begin{definition}
\label{def:standardizationMultipermutation}
The \defn{standardization} of a multipermutation~$\mu$ of~$M$ is the permutation~$\Std(\mu)$ of~$[|M|]$ defined by~$\Std(\mu)_p = \Std(\mu_p)$ for any~$p \in [|M|]$.
\end{definition}

For example, consider the multipermutation~$\mu = 31421324$.
Then we have~$\mu^{|\{2,3,4\}|} = 231213$, $\mu[3] = 64754657$, $\mu[2,4] = 617167$, and $\Std(\mu) = 51732648$.

\begin{remark}
\label{rem:relationOperatorsMultipermutations}
As in \cref{rem:relationOperatorsMultiposets}, there are natural relations between these operations, including:
\[
(\mu^{|L})^{|X} = \mu^{|\set{\ell_x}{x \in X}},
\quad
(\mu^{|L})[j] = (\mu[j])^{|\set{\ell+j}{\ell \in L}}
\qandq
(\mu[i,j])[k] = (\mu[k])[i+k, j].
\]
\end{remark}

The standardization clearly defines a bijection from the set of $k$-permutations of degree $m$ to the set of $k$-monotone permutations of~$[mk]$, \ie permutations~$\sigma$ of~$[mk]$ such that the letters ${ki, ki+1, \dots, k(i+1)-1}$ appear in this particular order form left to right in~$\sigma$ for all~${i \in [m]}$. Equivalently, $\sigma$ appears in the $m$-iterated shifted shuffle product (see \cref{def:shiftedShuffle}) of the identity permutation of degree~$k$.
Although we could work with monotone permutations, we believe that most statements are more natural on multipermutations.


\subsubsection{Concatenation and shifted concatenation}
\label{subsubsec:concatenation}

We quickly recall the concatenation product on words and the shifted concatenation product on permutations.
We fix a finite alphabet~$\alphabet$ and denote by~$\alphabet^n$ the set of words on~$\alphabet$ of length~$n$.
Note that~$\alphabet^0 = \{\varepsilon\}$ consists of the empty word of length~$0$.

\begin{definition}
\label{def:concatenation}
The \defn{free associative algebra} over~$\alphabet$ is the graded $\K$-vector space~$\alphabet^* \eqdef \bigoplus_{n \in \N} \K \alphabet^n$ endowed with the \defn{concatenation}~${X \cdot Y \eqdef XY}$.
\end{definition}

We denote by~$\alphabet^{\ge k} \eqdef \bigoplus_{n \ge k} \K \alphabet^n$.
It is stable by the concatenation product.
Note that for~${\alphabet = [m]}$ where~$m > 0$, the concatenation product does not stabilizes the subspace of $k$-permutations. For this, we need a slight modification of this product obtained by shifting. 

\begin{definition}
\label{def:shiftedConcatenation}
The \defn{shifted concatenation}~$\mu \,\bar\cdot\, \nu$ of two multipermutations~$\mu$ of~$M$ and~$\nu$ of~$N$ is the permutation of~$M \cup N[m]$ defined by~$\mu \,\bar\cdot\, \nu \eqdef \mu \cdot \nu[m]$, where~$m \eqdef \max(M)$.
\end{definition}

For example ${\blue123213}\ \bar\cdot\ {\red2112} = {\blue123213}{\red5445}$.
Note that the shifted concatenation is associative and graded on $\FQSym_k$.
Indeed the shifted concatenation of two multipermutations of degrees $m$ and $n$ is of degree $m+n$.


\subsubsection{Shuffle and shifted shuffle}
\label{subsubsec:shuffle}

We now quickly recall the shuffle product on words and the shifted shuffle product on permutations.

\begin{definition}
\label{def:shuffle}
The \defn{shuffle algebra} over~$\alphabet$ is the graded $\K$-vector space $\Shuffle(\alphabet) \eqdef \bigoplus_{n \in \N} \K \alphabet^n$ endowed with the \defn{shuffle product} defined inductively by~$X \shuffle \varepsilon \eqdef X$, $\varepsilon \shuffle Y \eqdef Y$ and \linebreak
\(
xX \shuffle yY \eqdef x(X \shuffle yY) + y(xX \shuffle Y)
\)
for any two words~$X,Y$ on~$\alphabet$.
\end{definition}

We denote by~$\Shuffle^{\ge k} \eqdef \bigoplus_{n \ge k} \K \alphabet^n$.
It is stable by the shuffle product.
Note that for~${\alphabet = [m]}$ where~$m > 0$, the shuffle product does not stabilizes the subspace of $k$-permutations. For this, we need a slight modification of this product obtained by shifting. 

\begin{definition}
\label{def:shiftedShuffle}
The \defn{shifted shuffle product}~$\mu \shiftedShuffle \nu$ of two multipermutations~$\mu$ of~$M$ and~$\nu$ of~$N$ is defined as~$\mu \shiftedShuffle \nu \eqdef \mu \shuffle \nu[m]$, where~$m \eqdef \max(M)$.
By definition, it is a linear combination of multipermutations of~$M \cup N[m]$.
\end{definition}

\begin{lemma}
The standardization~$\Std$ and the shifted shuffle product~$\shiftedShuffle$ commute: for any two multipermutations~$\mu$ and~$\nu$, we have $\Std(\mu \shiftedShuffle \nu) = \Std(\mu) \shiftedShuffle \Std(\nu)$.
\end{lemma}

The shifted shuffle product endows the vector space~$\FQSym_k$ with a graded algebra structure that generalizes the algebra on permutations of C.~Malvenuto and C.~Reutenauer~\cite{MalvenutoReutenauer}. By standardization (see~\cref{def:standardizationMultipermutation}), it is also a subalgebra of the classical algebra~$\FQSym$ on permutations. For example
\[
	{\blue123213} \shiftedShuffle {\red2112}
	= {\blue123213}{\red5445}
	+ {\blue12321}{\red5}{\blue3}{\red445}
	+ \cdots (\text{$210$~terms}) \cdots
	+ {\red544}{\blue1}{\red5}{\blue123213}
	+ {\red5445}{\blue123213}.
\]


\subsubsection{Multiposets}
\label{subsubsec:multiposets}

Recall that a \defn{poset} is a set endowed with a partial order (a reflexive, antisymmetric and transitive binary relation).
We depict posets by their Hasse diagram which is the graph of their cover relations and we always orient them from bottom to top.
The following definition is illustrated on \cref{fig:multiposets}.

\begin{definition}
\label{def:multiposet}
A \defn{multiposet} on a multiset~$M \eqdef 1^{\{a_1\}} 2^{\{a_2\}} \dots m^{\{a_m\}}$ is a partial order~$\le_M$ on the set~$\{1_1, 1_2, \dots, 1_{a_1}, 2_1, 2_2, \dots, 2_{a_2}, \dots, m_1, m_2, \dots, m_{a_m}\}$ where~$i_u \le_M i_v$ for all~$i \in [\max(M)]$ and~$0 < u \le v \le a_i$.
Since the different copies are naturally ordered there is no need to distinguish between them when drawing the Hasse diagram of a multiposet.
For two integers~$m$ and~$k>0$, a \defn{$k$-poset} of degree $m$ is a multiposet on the multiset~$1^{\{k\}} 2^{\{k\}} \dots m^{\{k\}}$.
\end{definition}

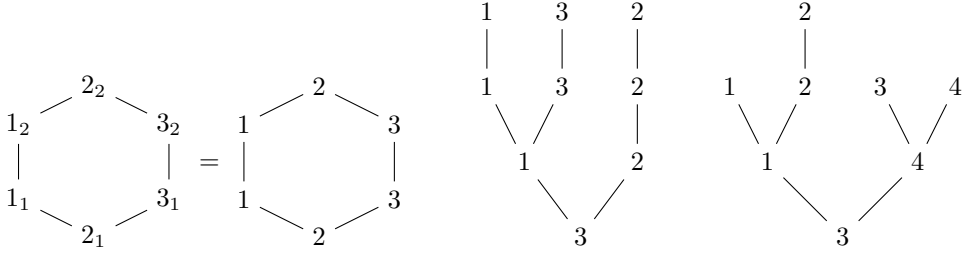
\begin{figure}[t]
	\centerline{
	\begin{tikzpicture}[baseline=.9cm]
		\node (1a) at (0,.5) {$1_1$};
		\node (1b) at (0,1.5) {$1_2$};
		\node (2a) at (1,0) {$2_1$};
		\node (2b) at (1,2) {$2_2$};
		\node (3a) at (2,.5) {$3_1$};
		\node (3b) at (2,1.5) {$3_2$};
		\draw (1a) -- (1b);
		\draw (1b) -- (2b);
		\draw (2a) -- (1a);
		\draw (2a) -- (3a);
		\draw (3a) -- (3b);
		\draw (3b) -- (2b);
	\end{tikzpicture}
	=
	\begin{tikzpicture}[baseline=.9cm]
		\node (1a) at (0,.5) {$1$};
		\node (1b) at (0,1.5) {$1$};
		\node (2a) at (1,0) {$2$};
		\node (2b) at (1,2) {$2$};
		\node (3a) at (2,.5) {$3$};
		\node (3b) at (2,1.5) {$3$};
		\draw (1a) -- (1b);
		\draw (1b) -- (2b);
		\draw (2a) -- (1a);
		\draw (2a) -- (3a);
		\draw (3a) -- (3b);
		\draw (3b) -- (2b);
	\end{tikzpicture}
	\qquad
	\begin{tikzpicture}[baseline=.9cm, level/.style={level distance = 1cm}, level 1/.style={sibling distance = 1.5cm}, level 2/.style={sibling distance = 1cm}]
		\node {$3$} [grow' = up]
			child {node {$1$}
				child {node {$1$}
					child {node {$1$}}
				}
				child {node {$3$}
					child {node {$3$}}
				}
			}
			child {node {$2$}
				child {node {$2$}
					child {node {$2$}}
				}
			}
		;
	\end{tikzpicture}
	\qquad
	\begin{tikzpicture}[baseline=.9cm, level/.style={level distance = 1cm}, level 1/.style={sibling distance = 2cm}, level 2/.style={sibling distance = 1cm}]
		\node {$3$} [grow' = up]
			child {node {$1$}
				child {node {$1$}}
				child {node {$2$}
					child {node {$2$}}
				}
			}
			child {node {$4$}
				child {node {$3$}}
				child {node {$4$}}
			}
		;
	\end{tikzpicture}
	}
	\caption{A $2$-poset of degree~$3$ (left), a $3$-poset of degree~$3$ (middle), and a $2$-poset of degree~$4$ (right). See \cref{def:multiposet}.}
	\label{fig:multiposets}
\end{figure}

We now define the restriction, shift and standardization operations on multiposets, similar to \cref{def:restrictionMultiset,def:shiftMultiset,def:shiftMultisetPlace,def:standardizationMultiset,def:restrictionMultipermutation,def:shiftMultipermutation,def:shiftMultipermutationPlace,def:standardizationMultipermutation}.
These definitions are illustrated on \cref{fig:shiftStandardizationMultiposets}.

\begin{figure}
	\centerline{\begin{tabular}{ccccc}
	\begin{tikzpicture}[level/.style={level distance = 1cm}, level 1/.style={sibling distance = 2cm}, level 2/.style={sibling distance = 1cm}]
		\node {$3$} [grow' = up]
			child {node {$1$}
				child {node {$1$}}
				child {node {$2$}
					child {node {$2$}}
				}
			}
			child {node {$4$}
				child {node {$3$}}
				child {node {$4$}}
			}
		;
	\end{tikzpicture}
	&
	\begin{tikzpicture}[level/.style={level distance = 1cm}, level 1/.style={sibling distance = 2cm}, level 2/.style={sibling distance = 1cm}]
		\node {$2$} [grow' = up]
			child {node {$1$}
				child {node {$1$}}
			}
			child {node {$3$}
				child {node {$2$}}
				child {node {$3$}}
			}
		;
	\end{tikzpicture}
	&
	\begin{tikzpicture}[level/.style={level distance = 1cm}, level 1/.style={sibling distance = 2cm}, level 2/.style={sibling distance = 1cm}]
		\node {$6$} [grow' = up]
			child {node {$4$}
				child {node {$4$}}
				child {node {$5$}
					child {node {$5$}}
				}
			}
			child {node {$7$}
				child {node {$6$}}
				child {node {$7$}}
			}
		;
	\end{tikzpicture}
	&
	\begin{tikzpicture}[level/.style={sibling distance = 1cm, level distance = 1cm}]
		\node {$1$} [grow' = up]
			child {node {$1$}}
			child {node {$2$}
				child {node {$2$}}
			}
		;
		\node[xshift=1.2cm] {$7$} [grow' = up]
			child {node {$7$}}
		;
	\end{tikzpicture}
	&
	\begin{tikzpicture}[level/.style={level distance = 1cm}, level 1/.style={sibling distance = 2cm}, level 2/.style={sibling distance = 1cm}]
		\node {$5$} [grow' = up]
			child {node {$1$}
				child {node {$2$}}
				child {node {$3$}
					child {node {$4$}}
				}
			}
			child {node {$7$}
				child {node {$6$}}
				child {node {$8$}}
			}
		;
	\end{tikzpicture}
	\\
	$\le_M$ & $\le_{M^{|\{2,3,4\}}}$ & $\le_{M[3]}$ & $\le_{M[3,4]}$ & $\Std(\le_M)$
	\end{tabular}
	}
	\caption{Shifts and standardization of multiposets. See \cref{def:restrictionPoset,def:shiftPoset,def:shiftPosetPlace,def:standardizationMultiposet}.}
	\label{fig:shiftStandardizationMultiposets}
\end{figure}
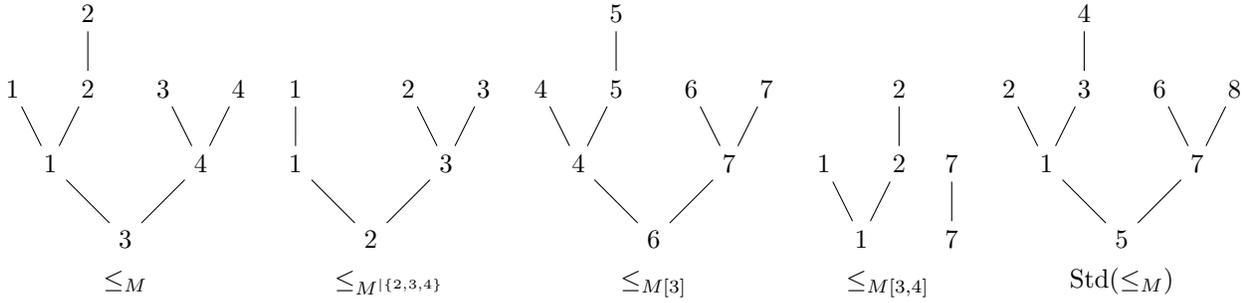

\begin{definition}
\label{def:restrictionPoset}
For a multiposet~$\le_M$ on~$M$ and a subset~$L \eqdef \{\ell_1, \dots, \ell_p\}$ of~$[\max(M)]$ (possibly empty), we define the \defn{restricted} multiposet~$\le_{M^{|L}}$ on~$M^{|L}$ by~${i_u \le_{M^{|L}} j_v \iff (\ell_i)_u \le_M (\ell_j)_v}$.
\end{definition}

\begin{definition}
\label{def:shiftPoset}
For a multiposet~$\le_M$ on~$M$ and an integer~$j$, we define the \defn{shifted} multiposet~$\le_{M[j]}$ on~$M[j]$ by~${x_u \le_{M[j]} y_v \iff x_u - j \le_M y_v - j}$.
\end{definition}

\begin{definition}
\label{def:shiftPosetPlace}
For a multiposet~$\le_M$ on~$M$ and two integers~$i \le \max(M)$ and~$j$, we define the multiposet~$\le_{M[i,j]}$ on~$M[i,j]$ by~${x \le_{M[i,j]} y \iff \overline{x} \le_M \overline{y}}$ where~$x \mapsto \overline{x}$ is the injection from~$M[i,j]$ to~$M$ described in \cref{def:shiftMultisetPlace}.
\end{definition}

\begin{definition}
\label{def:standardizationMultiposet}
The \defn{standardization} of a multiposet~$\le_M$ on~$M$ is the poset~$\le_{\Std(M)}$ of~$\Std(M)$ defined by~$\Std(x) \le_{\Std(M)} \Std(y) \iff x \le_M y$ for all~$x,y \in M$.
\end{definition}

\begin{remark}
\label{rem:relationOperatorsMultiposets}
As in \cref{rem:relationOperatorsMultiposets}, there are natural relations between these operations, including:
\[
{\le_{(M^{|L})^{|X}}} = {\le_{M^{|\set{\ell_x}{x \in X}}}},
\quad
{\le_{(M^{|L})[j]}} = {\le_{(M[j])^{|\set{\ell+j}{\ell \in L}}}}
\qandq
{\le_{(M[i,j])[k]}} = {\le_{(M[k])[i+k, j]}}.
\]
\end{remark}

We conclude with two relevant operations on multiposets.

\begin{definition}
\label{def:disjointUnionOrderedSum}
Consider two multiposets~$\le_M$ and~$\le_N$ on two disjoint multisets~$M$ and~$N$ (meaning~$\supp(M) \cap \supp(N) = \varnothing$), and let~$P \eqdef M \sqcup N$.
We define 
\begin{itemize}
\item the \defn{disjoint union}~${\le_M} \sqcup {\le_N}$ to be the multiposet~$\le_\sqcup$ on~$P$ where~$x \le_\sqcup y$ if and only if either $x \in M$, $y \in M$ and~$x \le_M y$, or~$x \in N$, $y \in N$ and~$x \le_N y$.
\item the \defn{ordered sum}~${\le_M} + {\le_N}$ to be the multiposet~$\le_+$ on~$P$ where~$x \le_+ y$ if and only if either $x \le_\sqcup y$, or~$x \in M$ and $y \in N$.
\end{itemize}
\end{definition}


\subsubsection{Tree multiposets}
\label{subsubsec:treeMultiposets}

We are now interested in multiposets with special shapes, namely trees and forests. They can be defined recursively or directly in terms of the partial order.

\begin{definition}
\label{def:treeMultiposet}
A multiposet $\le_M$ is a \defn{forest} if $x \le_M z$ and $y \le_M z$ implies either $x \le_M y$ or~$y \le_M x$ for all $x, y, z \in M$. It is a \defn{tree} if it is a forest with a unique minimal element, called the~\defn{root}.
\end{definition}

For example, the leftmost multiposet of \cref{fig:multiposets} is not a tree while the middle and rightmost multiposets of \cref{fig:multiposets} are.
Clearly a multiposet is a forest (resp.~a tree) if its Hasse diagram is a forest (resp.~a tree) in the graph-theoretical sense. We therefore freely use the vocabulary of trees (such as children, subtree, \dots) on tree multiposets.

\begin{definition}
\label{def:intervalLabelled}
A tree multiposet is \defn{interval labelled} if for any subtrees~$t$ and~$t'$ with the same parent, all the labels that appear in~$t$ are strictly smaller (in the sense of the standard order on~$[\max(M)]$) than all the labels that appear in~$t'$ or \viceversa. In other words, the labels that appear in distinct descendant subtrees of a given node belong to disjoint intervals.
\end{definition}

For instance, the middle tree multiposet of \cref{fig:multiposets} is not interval labelled (since the left descendant tree of the root contains~$1$ and~$3$ while the right descendant tree of the root contains~$2$) while the right tree multiposet of \cref{fig:multiposets} is interval labelled.


\subsubsection{Linear extensions}
\label{subsubsec:linearExtensions}

A multiposet~$\le_M$ on~$M$ is \defn{linear} if it is a total order, meaning that any two elements of~$M$ are comparable for~$\le_M$. Note that the multipermutations of~$M$ are precisely the linear multiposets of~$M$.

\begin{definition}
A \defn{linear extension} of a multiposet~$\le_M$ is a multipermutation~$\sigma$ that extends~$\le_M$, meaning that~$x \le_M y$ implies that~$x$ appears to the left of~$y$ in~$\sigma$ for all~$x,y \in M$.
We denote by $\linearExtensions(\le_M)$ the set of linear extensions of the multiposet~$\le_M$.
\end{definition}

For example, the multipermutations~$231132$, $321312132$ and~$31421324$ are linear extensions of the multiposets of \cref{fig:multiposets}.
The following statement is illustrated by the examples of the previous sections.

\begin{lemma}
The linear extensions operation~$\linearExtensions$ commutes with the restriction, the shift and the standardization: for any multiposet~$\le_M$, any subset~$L \subseteq [\max(M)]$ and any integers~$i \le \max(M)$ and~$j$, we have
\[
\begin{array}{ccc}
\linearExtensions(\le_{M^{|L}}) = \bigset{\sigma^{|L}}{\sigma \in \linearExtensions(\le_M)}
&&
\linearExtensions(\le_{M[j]}) = \bigset{\sigma[j]}{\sigma \in \linearExtensions(\le_M)}
\\[.2cm]
\linearExtensions(\le_{M[i,j]}) = \bigset{\sigma[i,j]}{\sigma \in \linearExtensions(\le_M)}
& \,\qandq\, &
\linearExtensions(\Std(\le_M)) = \bigset{\Std(\mu)}{\mu \in \linearExtensions(\le_M)}.
\end{array}
\]
\end{lemma}

The last point of this lemma enables us to transport classical properties of linear extensions of posets to linear extensions of multiposets. For instance, the hook length formula enables us to count the number of linear extensions of a tree multiposet.

\begin{lemma}[\cite{BjornerWachs}]
The number of linear extensions of a tree multiposet~$\le_M$ is given by
\[
\frac{|M|!}{\prod_{x \in M} |M_x|}
\]
where $M_x$ denotes the subtree rooted at~$x \in M$.
\end{lemma}

For instance, the number of linear extensions of the middle and rightmost tree multiposets of \cref{fig:multiposets} are~$9! / (9 \cdot 5 \cdot 3 \cdot 2^3) = 336$ and $8! / (8 \cdot 4 \cdot 3 \cdot 2) = 210$ respectively.

\begin{lemma}
\label{lem:disjointUnionOrderedSumLinearExtensions}
For any two multiposets~$\le_M$ and~$\le_N$ on two disjoints multisets~$M$ and~$N$ (meaning~$\supp(M) \cap \supp(N) = \varnothing$),
\begin{gather*}
\linearExtensions({\le_M} \sqcup {\le_N}) = \linearExtensions(\le_M) \shuffle \linearExtensions(\le_N) = \bigcup\nolimits_{\substack{\sigma \in \linearExtensions(\le_M) \\ \tau \in \linearExtensions(\le_N)}} \sigma \shuffle \tau, \\
\linearExtensions({\le_M} + {\le_N}) = \linearExtensions(\le_M) \cdot \linearExtensions(\le_N) = \set{\sigma \cdot \tau}{\sigma \in \linearExtensions(\le_M), \; \tau \in \linearExtensions(\le_N)},
\end{gather*}
where~$\shuffle$ denotes the shuffle product of \cref{def:shuffle} and~$\cdot$ denotes the concatenation product of \cref{def:concatenation}.
\end{lemma}

\begin{proof}
A permutation is a linear extensions of~${\le_M} \sqcup {\le_N}$ if and only if its restriction to~$M$ is a linear extension of~${\le_M}$ and its restriction to~$N$ is a linear extension of~${\le_N}$, hence if and only if it belongs to~$\linearExtensions(\le_M) \shuffle \linearExtensions(\le_N)$.
A permutation is a linear extensions of~${\le_M} + {\le_N}$ if and only if its first $|M|$ letters form a linear extension of~${\le_M}$ and its last $|N|$ letters form a linear extension of~${\le_N}$, hence if and only if it belongs to~$\linearExtensions(\le_M) \cdot \linearExtensions(\le_N)$.
\end{proof}


\clearpage
\section{Signaletic operads}
\label{sec:signaleticOperads}

This paper essentially relies on an interpretation of the diassociative and dual duplicial operads in terms of traffic signals in an arborescent road. This interpretation then admits four natural extensions that we call messy/tidy series/parallel signaletic operads.


\subsection{Signaletic interpretation of the diassociative and dual duplicial operads}
\label{subsec:signaleticDiassDualDup}

\setlength{\fboxsep}{1.5pt}
We consider that a syntax tree~$\tree$ on the operations~$\{\op{l}, \op{r}\}$ is an arborescent road with a designated entry point where each branching node is occupied by a traffic signal~$\boxed{\op{l}}$ or~$\boxed{\op{r}}$. A car arrives at the root of~$\tree$ and drives through the tree~$\tree$ following at each branching node the direction indicated by the traffic signal. The car ends at a certain leaf of~$\tree$ that we call the \defn{destination} of~$\tree$. This procedure is illustrated on three syntax trees in \cref{fig:twotrees}. 

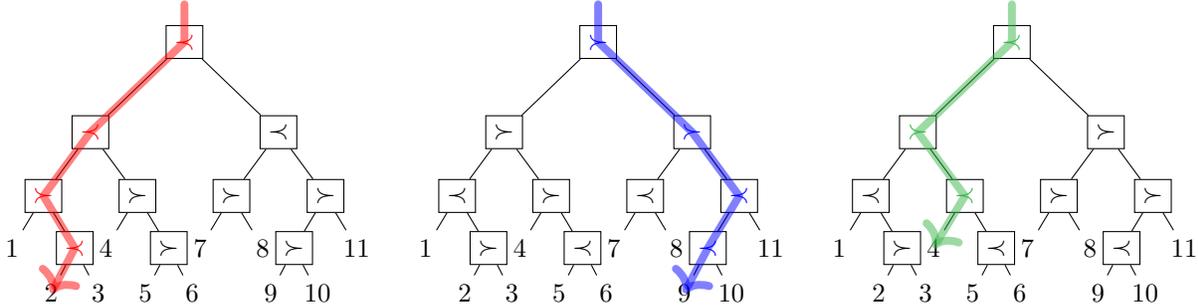
\begin{figure}[h]
	\centerline{
	\begin{tikzpicture}[level/.style={sibling distance=25mm/#1, level distance = 1.2cm/sqrt(#1)}]
		\node [rectangle, draw] (root) {$\op{la}$}
			child {node [rectangle, draw] (l) {$\op{la}$}
				child {node [rectangle, draw] (ll) {$\op{ra}$}
					child {node (lll) {1}}
					child {node [rectangle, draw] (llr) {$\op{la}$}
						child {node (llrl) {2}}
						child {node (llrr) {3}}
					}
				}
				child {node [rectangle, draw] (lr) {$\op{r}$}
					child {node (lrl) {4}}
					child {node [rectangle, draw] (lrr) {$\op{r}$}
						child {node (lrrl) {5}}
						child {node (lrrr) {6}}
					}
				}
			}
			child {node [rectangle, draw] (r) {$\op{l}$}
				child {node [rectangle, draw] (rl) {$\op{r}$}
					child {node (rll) {7}}
					child {node (rlr) {8}}
				}
				child {node [rectangle, draw] (rr) {$\op{r}$}
					child {node [rectangle, draw] (rrl) {$\op{r}$}
						child {node (rrll) {9}}
						child {node (rrlr) {10}}
					}
					child {node (rrr) {11}}
				}
			}
		;
		\draw[join=round, line cap=round, line width = 3pt,   red, opacity=.5, ->] ($(root) + (0,.5)$) -- (root.center) -- (l.center) -- (ll.center) -- (llr.center) -- (llrl.center);
		\node [rectangle, draw, xshift=5.5cm] (root) {$\op{rb}$}
			child {node [rectangle, draw] (l) {$\op{r}$}
				child {node [rectangle, draw] (ll) {$\op{l}$}
					child {node (lll) {1}}
					child {node [rectangle, draw] (llr) {$\op{r}$}
						child {node (llrl) {2}}
						child {node (llrr) {3}}
					}
				}
				child {node [rectangle, draw] (lr) {$\op{r}$}
					child {node (lrl) {4}}
					child {node [rectangle, draw] (lrr) {$\op{l}$}
						child {node (lrrl) {5}}
						child {node (lrrr) {6}}
					}
				}
			}
			child {node [rectangle, draw] (r) {$\op{rb}$}
				child {node [rectangle, draw] (rl) {$\op{l}$}
					child {node (rll) {7}}
					child {node (rlr) {8}}
				}
				child {node [rectangle, draw] (rr) {$\op{lb}$}
					child {node [rectangle, draw] (rrl) {$\op{lb}$}
						child {node (rrll) {9}}
						child {node (rrlr) {10}}
					}
					child {node (rrr) {11}}
				}
			}
		;
		\draw[join=round, line cap=round, line width = 3pt,  blue, opacity=.5, ->] ($(root) + (0,.5)$) -- (root.center) -- (r.center) -- (rr.center) -- (rrl.center) -- (rrll.center);
		\node [rectangle, draw, xshift=11cm] (root) {$\op{lc}$}
			child {node [rectangle, draw] (l) {$\op{rc}$}
				child {node [rectangle, draw] (ll) {$\op{l}$}
					child {node (lll) {1}}
					child {node [rectangle, draw] (llr) {$\op{r}$}
						child {node (llrl) {2}}
						child {node (llrr) {3}}
					}
				}
				child {node [rectangle, draw] (lr) {$\op{lc}$}
					child {node (lrl) {4}}
					child {node [rectangle, draw] (lrr) {$\op{l}$}
						child {node (lrrl) {5}}
						child {node (lrrr) {6}}
					}
				}
			}
			child {node [rectangle, draw] (r) {$\op{r}$}
				child {node [rectangle, draw] (rl) {$\op{r}$}
					child {node (rll) {7}}
					child {node (rlr) {8}}
				}
				child {node [rectangle, draw] (rr) {$\op{r}$}
					child {node [rectangle, draw] (rrl) {$\op{l}$}
						child {node (rrll) {9}}
						child {node (rrlr) {10}}
					}
					child {node (rrr) {11}}
				}
			}
		;
		\draw[join=round, line cap=round, line width = 3pt, green, opacity=.5, ->] ($(root) + (0,.5)$) -- (root.center) -- (l.center) -- (lr.center) -- (lrl.center);
	\end{tikzpicture}}
	\caption{Interpretation of the diassociative operad in terms of syntax tree traversal.}
	\label{fig:twotrees}
\end{figure}

\noindent
Observe that the diassociative relations are compatible with the destination:

\medskip
\centerline{$
\begin{tikzpicture}[baseline=-.5cm, level/.style={sibling distance = .8cm, level distance = .8cm}]
	\node [rectangle, draw] (root) {$\op{la}$}
		child {node (l) {1}}
		child {node [rectangle, draw] (r) {$\op{l}$}
			child {node (rl) {2}}
			child {node (rr) {3}}
		}
	;
	\draw[join=round, line cap=round, line width = 3pt, red, opacity=.5, ->] ($(root) + (0,.5)$) -- (root.center) -- (l.center);
\end{tikzpicture}
\!\!=\!\!
\begin{tikzpicture}[baseline=-.5cm, level/.style={sibling distance = .8cm, level distance = .8cm}]
	\node [rectangle, draw] (root) {$\op{la}$}
		child {node (l) {1}}
		child {node [rectangle, draw] (r) {$\op{r}$}
			child {node (rl) {2}}
			child {node (rr) {3}}
		}
	;
	\draw[join=round, line cap=round, line width = 3pt, red, opacity=.5, ->] ($(root) + (0,.5)$) -- (root.center) -- (l.center);
\end{tikzpicture}
\!\!=\!\!
\begin{tikzpicture}[baseline=-.5cm, level/.style={sibling distance = .8cm, level distance = .8cm}]
	\node [rectangle, draw] (root) {$\op{la}$}
		child {node [rectangle, draw] (l) {$\op{la}$}
			child {node (ll) {1}}
			child {node (lr) {2}}
		}
		child {node (r) {3}}
	;
	\draw[join=round, line cap=round, line width = 3pt, red, opacity=.5, ->] ($(root) + (0,.5)$) -- (root.center) -- (l.center) -- (ll.center);
\end{tikzpicture}
\qquad
\begin{tikzpicture}[baseline=-.5cm, level/.style={sibling distance = .8cm, level distance =.8cm}]
	\node [rectangle, draw] (root) {$\op{ra}$}
		child {node (l) {1}}
		child {node [rectangle, draw] (r) {$\op{la}$}
			child {node (rl) {2}}
			child {node (rr) {3}}
		}
	;
	\draw[join=round, line cap=round, line width = 3pt, red, opacity=.5, ->] ($(root) + (0,.5)$) -- (root.center) -- (r.center) -- (rl.center);
\end{tikzpicture}
\!\!=\!\!
\begin{tikzpicture}[baseline=-.5cm, level/.style={sibling distance = .8cm, level distance = .8cm}]
	\node [rectangle, draw] (root) {$\op{la}$}
		child {node [rectangle, draw] (l) {$\op{ra}$}
			child {node (ll) {1}}
			child {node (lr) {2}}
		}
		child {node (r) {3}}
	;
	\draw[join=round, line cap=round, line width = 3pt, red, opacity=.5, ->] ($(root) + (0,.5)$) -- (root.center) -- (l.center) -- (lr.center);
\end{tikzpicture}
\qquad
\begin{tikzpicture}[baseline=-.5cm, level/.style={sibling distance = .8cm, level distance = .8cm}]
	\node [rectangle, draw] (root) {$\op{ra}$}
		child {node (l) {1}}
		child {node [rectangle, draw] (r) {$\op{ra}$}
			child {node (rl) {2}}
			child {node (rr) {3}}
		}
	;
	\draw[join=round, line cap=round, line width = 3pt, red, opacity=.5, ->] ($(root) + (0,.5)$) -- (root.center) -- (r.center) -- (rr.center);
\end{tikzpicture}
\!\!=\!\!
\begin{tikzpicture}[baseline=-.5cm, level/.style={sibling distance = .8cm, level distance = .8cm}]
	\node [rectangle, draw] (root) {$\op{ra}$}
		child {node [rectangle, draw] (l) {$\op{l}$}
			child {node (ll) {1}}
			child {node (lr) {2}}
		}
		child {node (r) {3}}
	;
	\draw[join=round, line cap=round, line width = 3pt, red, opacity=.5, ->] ($(root) + (0,.5)$) -- (root.center) -- (r.center);
\end{tikzpicture}
\!\!=\!\!
\begin{tikzpicture}[baseline=-.5cm, level/.style={sibling distance = .8cm, level distance = .8cm}]
	\node [rectangle, draw] (root) {$\op{ra}$}
		child {node [rectangle, draw] (l) {$\op{r}$}
			child {node (ll) {1}}
			child {node (lr) {2}}
		}
		child {node (r) {3}}
	;
	\draw[join=round, line cap=round, line width = 3pt, red, opacity=.5, ->] ($(root) + (0,.5)$) -- (root.center) -- (r.center);
\end{tikzpicture}$}
\medskip

\noindent
This shows that two equivalent syntax trees have the same destination. The reverse statement can be shown using normal forms as will be generalized in \cref{thm:signaleticOperadsQuadratic}.

\begin{proposition}[\cite{Loday-dialgebras}]
Two syntax trees on~$\{\op{l}, \op{r}\}$ with the same arity represent the same operation in the diassociative operad if and only if they have the same destination.
\end{proposition}

Consider now the dual duplicial operad~$\Dupdual$.
We observe that its relations are also compatible with the destination. The difference is that a syntax tree, regarded as an element of~$\Dupdual$, vanishes if some signal outside the route of the car does not point towards this route:

\medskip
\centerline{$
\begin{tikzpicture}[baseline=-.5cm, level/.style={sibling distance = .8cm, level distance = .8cm}]
	\node [rectangle, draw] (root) {$\op{la}$}
		child {node (l) {1}}
		child {node [rectangle, draw] (r) {$\op{l}$}
			child {node (rl) {2}}
			child {node (rr) {3}}
		}
	;
	\draw[join=round, line cap=round, line width = 3pt, red, opacity=.5, ->] ($(root) + (0,.5)$) -- (root.center) -- (l.center);
\end{tikzpicture}
\!\!=\!\!
\begin{tikzpicture}[baseline=-.5cm, level/.style={sibling distance = .8cm, level distance = .8cm}]
	\node [rectangle, draw] (root) {$\op{la}$}
		child {node [rectangle, draw] (l) {$\op{la}$}
			child {node (ll) {1}}
			child {node (lr) {2}}
		}
		child {node (r) {3}}
	;
	\draw[join=round, line cap=round, line width = 3pt, red, opacity=.5, ->] ($(root) + (0,.5)$) -- (root.center) -- (l.center) -- (ll.center);
\end{tikzpicture}
\quad\;
\begin{tikzpicture}[baseline=-.5cm, level/.style={sibling distance = .8cm, level distance = .8cm}]
	\node [rectangle, draw] (root) {$\op{la}$}
		child {node (l) {1}}
		child {node [rectangle, draw] (r) {$\op{r}$}
			child {node (rl) {2}}
			child {node (rr) {3}}
		}
	;
	\draw[join=round, line cap=round, line width = 3pt, red, opacity=.5, ->] ($(root) + (0,.5)$) -- (root.center) -- (l.center);
\end{tikzpicture}
\!\!= 0
\quad\;
\begin{tikzpicture}[baseline=-.5cm, level/.style={sibling distance = .8cm, level distance =.8cm}]
	\node [rectangle, draw] (root) {$\op{ra}$}
		child {node (l) {1}}
		child {node [rectangle, draw] (r) {$\op{la}$}
			child {node (rl) {2}}
			child {node (rr) {3}}
		}
	;
	\draw[join=round, line cap=round, line width = 3pt, red, opacity=.5, ->] ($(root) + (0,.5)$) -- (root.center) -- (r.center) -- (rl.center);
\end{tikzpicture}
\!\!=\!\!
\begin{tikzpicture}[baseline=-.5cm, level/.style={sibling distance = .8cm, level distance = .8cm}]
	\node [rectangle, draw] (root) {$\op{la}$}
		child {node [rectangle, draw] (l) {$\op{ra}$}
			child {node (ll) {1}}
			child {node (lr) {2}}
		}
		child {node (r) {3}}
	;
	\draw[join=round, line cap=round, line width = 3pt, red, opacity=.5, ->] ($(root) + (0,.5)$) -- (root.center) -- (l.center) -- (lr.center);
\end{tikzpicture}
\quad\;
0 =\!\!
\begin{tikzpicture}[baseline=-.5cm, level/.style={sibling distance = .8cm, level distance = .8cm}]
	\node [rectangle, draw] (root) {$\op{ra}$}
		child {node [rectangle, draw] (l) {$\op{l}$}
			child {node (ll) {1}}
			child {node (lr) {2}}
		}
		child {node (r) {3}}
	;
	\draw[join=round, line cap=round, line width = 3pt, red, opacity=.5, ->] ($(root) + (0,.5)$) -- (root.center) -- (r.center);
\end{tikzpicture}
\quad\;
\begin{tikzpicture}[baseline=-.5cm, level/.style={sibling distance = .8cm, level distance = .8cm}]
	\node [rectangle, draw] (root) {$\op{ra}$}
		child {node (l) {1}}
		child {node [rectangle, draw] (r) {$\op{ra}$}
			child {node (rl) {2}}
			child {node (rr) {3}}
		}
	;
	\draw[join=round, line cap=round, line width = 3pt, red, opacity=.5, ->] ($(root) + (0,.5)$) -- (root.center) -- (r.center) -- (rr.center);
\end{tikzpicture}
\!\!=\!\!
\begin{tikzpicture}[baseline=-.5cm, level/.style={sibling distance = .8cm, level distance = .8cm}]
	\node [rectangle, draw] (root) {$\op{ra}$}
		child {node [rectangle, draw] (l) {$\op{r}$}
			child {node (ll) {1}}
			child {node (lr) {2}}
		}
		child {node (r) {3}}
	;
	\draw[join=round, line cap=round, line width = 3pt, red, opacity=.5, ->] ($(root) + (0,.5)$) -- (root.center) -- (r.center);
\end{tikzpicture}$}
\medskip

\begin{proposition}
A syntax tree~$\tree$ on~$\{\op{l}, \op{r}\}$ vanishes in the dual duplicial operad except if all signals outside the route of the car point toward this route.
Two non-vanishing syntax trees on~${\{\op{l}, \op{r}\}}$ represent the same operation in the dual duplicial operad if and only if they have the same destination.
\end{proposition}

For our purposes, we need to consider a slight modification of the dual duplicial operad.
We still consider a syntax tree on~$\{\op{l}, \op{r}\}$ and we let a car traverse the tree, but we now impose that all signals outside the route of the car point to the left.
We will see as a special case of \cref{subsubsec:tidyParallelSignaleticOperads} that this defines a quadratic operad~$\Dupdual_{\op{l}}$ over~$\{\op{l}, \op{r}\}$ defined by the five linear relations:

\medskip
\centerline{$
\begin{tikzpicture}[baseline=-.5cm, level/.style={sibling distance = .8cm, level distance = .8cm}]
	\node [rectangle, draw] (root) {$\op{la}$}
		child {node (l) {1}}
		child {node [rectangle, draw] (r) {$\op{l}$}
			child {node (rl) {2}}
			child {node (rr) {3}}
		}
	;
	\draw[join=round, line cap=round, line width = 3pt, red, opacity=.5, ->] ($(root) + (0,.5)$) -- (root.center) -- (l.center);
\end{tikzpicture}
\!\!=\!\!
\begin{tikzpicture}[baseline=-.5cm, level/.style={sibling distance = .8cm, level distance = .8cm}]
	\node [rectangle, draw] (root) {$\op{la}$}
		child {node [rectangle, draw] (l) {$\op{la}$}
			child {node (ll) {1}}
			child {node (lr) {2}}
		}
		child {node (r) {3}}
	;
	\draw[join=round, line cap=round, line width = 3pt, red, opacity=.5, ->] ($(root) + (0,.5)$) -- (root.center) -- (l.center) -- (ll.center);
\end{tikzpicture}
\quad\;
\begin{tikzpicture}[baseline=-.5cm, level/.style={sibling distance = .8cm, level distance = .8cm}]
	\node [rectangle, draw] (root) {$\op{la}$}
		child {node (l) {1}}
		child {node [rectangle, draw] (r) {$\op{r}$}
			child {node (rl) {2}}
			child {node (rr) {3}}
		}
	;
	\draw[join=round, line cap=round, line width = 3pt, red, opacity=.5, ->] ($(root) + (0,.5)$) -- (root.center) -- (l.center);
\end{tikzpicture}
\!\!= 0
\quad\;
\begin{tikzpicture}[baseline=-.5cm, level/.style={sibling distance = .8cm, level distance =.8cm}]
	\node [rectangle, draw] (root) {$\op{ra}$}
		child {node (l) {1}}
		child {node [rectangle, draw] (r) {$\op{la}$}
			child {node (rl) {2}}
			child {node (rr) {3}}
		}
	;
	\draw[join=round, line cap=round, line width = 3pt, red, opacity=.5, ->] ($(root) + (0,.5)$) -- (root.center) -- (r.center) -- (rl.center);
\end{tikzpicture}
\!\!=\!\!
\begin{tikzpicture}[baseline=-.5cm, level/.style={sibling distance = .8cm, level distance = .8cm}]
	\node [rectangle, draw] (root) {$\op{la}$}
		child {node [rectangle, draw] (l) {$\op{ra}$}
			child {node (ll) {1}}
			child {node (lr) {2}}
		}
		child {node (r) {3}}
	;
	\draw[join=round, line cap=round, line width = 3pt, red, opacity=.5, ->] ($(root) + (0,.5)$) -- (root.center) -- (l.center) -- (lr.center);
\end{tikzpicture}
\quad\;
0 =\!\!
\begin{tikzpicture}[baseline=-.5cm, level/.style={sibling distance = .8cm, level distance = .8cm}]
	\node [rectangle, draw] (root) {$\op{ra}$}
		child {node [rectangle, draw] (l) {$\op{r}$}
			child {node (ll) {1}}
			child {node (lr) {2}}
		}
		child {node (r) {3}}
	;
	\draw[join=round, line cap=round, line width = 3pt, red, opacity=.5, ->] ($(root) + (0,.5)$) -- (root.center) -- (r.center);
\end{tikzpicture}
\quad\;
\begin{tikzpicture}[baseline=-.5cm, level/.style={sibling distance = .8cm, level distance = .8cm}]
	\node [rectangle, draw] (root) {$\op{ra}$}
		child {node (l) {1}}
		child {node [rectangle, draw] (r) {$\op{ra}$}
			child {node (rl) {2}}
			child {node (rr) {3}}
		}
	;
	\draw[join=round, line cap=round, line width = 3pt, red, opacity=.5, ->] ($(root) + (0,.5)$) -- (root.center) -- (r.center) -- (rr.center);
\end{tikzpicture}
\!\!=\!\!
\begin{tikzpicture}[baseline=-.5cm, level/.style={sibling distance = .8cm, level distance = .8cm}]
	\node [rectangle, draw] (root) {$\op{ra}$}
		child {node [rectangle, draw] (l) {$\op{l}$}
			child {node (ll) {1}}
			child {node (lr) {2}}
		}
		child {node (r) {3}}
	;
	\draw[join=round, line cap=round, line width = 3pt, red, opacity=.5, ->] ($(root) + (0,.5)$) -- (root.center) -- (r.center);
\end{tikzpicture}$}

\medskip
\noindent
The opposite rule, forcing all signals not located on the route to point to the right, also yields an operad~$\Dupdual_{\op{r}}$.
We let~$\Dupdual_{\odot} \eqdef \Dupdual$ (resp.~$\Dupdual_\star \eqdef \Dias$) denote the classical dual duplicial operad (resp.~diassociative operad), where the signals not located on the route point towards the route (resp.~have no constraints).
In contrast, note that the rule forcing all signals not located on the route to point away from the route does not define an operad.
We will see as a special case of \cref{subsec:parallelSignaleticOperads} that the four operads~$\Dupdual_\star$, $\Dupdual_{\odot}$, $\Dupdual_{\op{l}}$ and~$\Dupdual_{\op{r}}$ are quadratic~and~Koszul.

In \cref{subsec:parallelSignaleticOperads,subsec:seriesSignaleticOperads}, we define and study two natural generalizations of this signaletic interpretation of the diassociative and dual duplicial operads. Namely, we fix an integer~$k \ge 1$, we consider a syntax tree~$\tree$ on the operations~$\{\op{l}, \op{r}\}^k$, and we assume that $k$ cars arrive at the root of~$\tree$. These $k$ cars will drive through the tree~$\tree$ following the directions indicated by the traffic signal at each branching node in two different ways:
\begin{itemize}
\item \defn{Parallel}: The $k$ cars all start together at the root of~$\tree$, and the $i$-th car always follows the indication given by the $i$-th letter of the traffic signal at each branching node. This corresponds to the white Manin product of the diassociative and duplicial operads.
\item \defn{Series}: The $k$ cars start one after the other at the root of~$\tree$, and each car always follows the indication given by the leftmost remaining letter of the traffic signal at each branching node and erases it.
\end{itemize}
These two circulation rules are illustrated in \cref{fig:traversingTreeParallel,fig:traversingTreeSeries}.
Moreover, we can impose the following additional constraints:
\begin{itemize}
\item \defn{Tidy}: All signals not used by a car must point to the left (as in~$\Dupdual_{\op{l}}$).
\item \defn{Messy}: No further constraints on the remaining signals (as in~$\Dias$).
\end{itemize}
We will see that the two circulation rules (parallel and series) and the two ordering rules (messy or tidy) define four families of operads.
Although not isomorphic, these four families of operads have identical Hilbert series and very similar behaviors.
We will actually observe that many proof techniques can be applied to both families irrespective of the series or parallel circulation rule, and of the tidy or messy constraints. 

In fact, these two circulation rules can also be mixed as quickly discussed in \cref{subsec:seriesParallelOperads}.
We have preferred to first discuss them separately to simplify the presentation.


\subsection{Parallel signaletic operads}
\label{subsec:parallelSignaleticOperads}

We start with the parallel signaletic operads, and we treat separately the messy and tidy situations.
Although these operads are just white Manin powers of the diassociative and duplicial operads, we use an elementary combinatorial presentation that will allow to define similarly series signaletic operads with a slight modification of the circulation rule.


\subsubsection{Messy parallel signaletic operads}
\label{subsubsec:messyParallelSignaleticOperads}

Fix an integer~$k \ge 0$ and consider a syntax tree~$\tree$ on the operations~$\Operations_k \eqdef \{\op{l}, \op{r}\}^k$ with~$n$ leaves that we label~$1, \dots, n$ from left to right. Assume that $k$ cars~$c_1, \dots, c_k$ arrive simultaneously at the root of~$\tree$ and that the $j$-th car~$c_j$ always follows the indication given by the $j$-th letter of the traffic signal at each branching node. We call \defn{parallel routes} the paths followed by the $k$ cars in the syntax tree~$\tree$. Finally, each car~$c_j$ ends at a certain leaf labeled~$\ell_j$ and we call \defn{parallel destination vector} of~$\tree$ the vector~$\destVect{\ell_1, \dots, \ell_k}{n}$. See \cref{fig:traversingTreeParallel}.

\begin{figure}[!h]
	\centerline{
	\begin{tikzpicture}[level/.style={sibling distance=6cm/#1, level distance = 1.2cm/sqrt(#1)}]
		\node [rectangle, draw] (root) {$\op{la,rb,lc}$}
			child {node [rectangle, draw] (l) {$\op{la,r,rc}$}
				child {node [rectangle, draw] (ll) {$\op{ra,l,l}$}
					child {node (lll) {1}}
					child {node [rectangle, draw] (llr) {$\op{la,r,r}$}
						child {node (llrl) {2}}
						child {node (llrr) {3}}
					}
				}
				child {node [rectangle, draw] (lr) {$\op{r,r,lc}$}
					child {node (lrl) {4}}
					child {node [rectangle, draw] (lrr) {$\op{r,l,l}$}
						child {node (lrrl) {5}}
						child {node (lrrr) {6}}
					}
				}
			}
			child {node [rectangle, draw] (r) {$\op{l,rb,r}$}
				child {node [rectangle, draw] (rl) {$\op{r,l,r}$}
					child {node (rll) {7}}
					child {node (rlr) {8}}
				}
				child {node [rectangle, draw] (rr) {$\op{r,lb,r}$}
					child {node [rectangle, draw] (rrl) {$\op{r,lb,l}$}
						child {node (rrll) {9}}
						child {node (rrlr) {10}}
					}
					child {node (rrr) {11}}
				}
			}
		;
		\draw[join=round, line cap=round, line width = 3pt,   red, opacity=.5, ->] ($(root) + (-.25,.5)$) -- ($(root) + (-.25,0)$) -- ($(l) + (-.25,0)$) -- ($(ll) + (-.25,0)$) -- ($(llr) + (-.25,0)$) -- ($(llrl) + (-.25,0)$);
		\draw[join=round, line cap=round, line width = 3pt,  blue, opacity=.5, ->] ($(root) + (  0,.5)$) -- ($(root) + (  0,0)$) -- ($(r) + (  0,0)$) -- ($(rr) + (  0,0)$) -- ($(rrl) + (  0,0)$) -- ($(rrll) + (  0,0)$);
		\draw[join=round, line cap=round, line width = 3pt, green, opacity=.5, ->] ($(root) + ( .25,.5)$) -- ($(root) + ( .25,0)$) -- ($(l) + ( .25,0)$) -- ($(lr) + ( .25,0)$) -- ($(lrl) + ( .25,0)$);
	\end{tikzpicture}
	}
	\caption{Traversing the syntax tree in parallel. The parallel routes are marked, and the parallel destination vector is~$\destVect{2,9,4}{11}$.}
	\label{fig:traversingTreeParallel}
\end{figure}
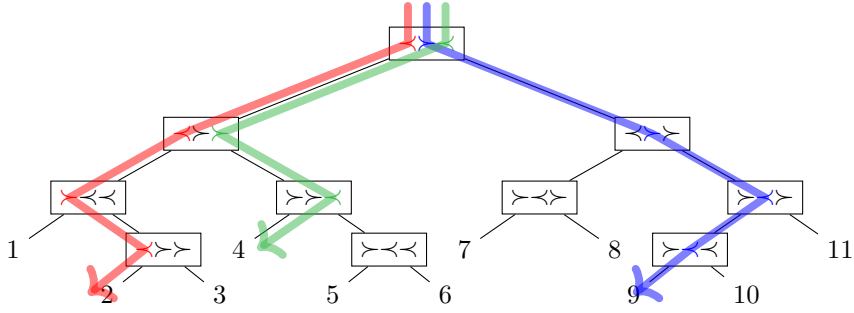

\pagebreak
We say that two syntax trees~$\tree, \tree'$ on~$\Operations_k$ with the same arity are \defn{messy parallel $k$-signaletic equivalent} and we write~$\tree \approxeq^{\parallel} \tree'$ if they have the same parallel destination vector.

\begin{proposition}
The messy parallel $k$-signaletic equivalence is compatible with grafting of syntax trees: $\tree \circ_i \tree[s] \approxeq^\parallel \tree' \circ_i \tree[s]'$ for any syntax trees~$\tree \approxeq^\parallel \tree'$ of arity~$p$ and $\tree[s] \approxeq^\parallel \tree[s]'$ of arity~$q$, and any~$i \in [p]$.
\end{proposition}

\begin{proof}
Consider the $j$-th car~$c_j$ and denote by~$\ind{p}_j$ its parallel destination in~$\tree$ and by~$\ind{q}_j$ its parallel destination in~$\tree[s]$. Then its parallel destination~$\ind{r}_j$ in~$\tree \circ_i \tree[s]$ is given by
\begin{equation}
\label{eq:messySignaleticParallelDestinations}
\ind{r}_j = 
\begin{cases}
\ind{p}_j & \text{if } \ind{p}_j < i, \\
\ind{p}_j + \ind{q}_j - 1 & \text{if } \ind{p}_j = i, \\
\ind{p}_j + q - 1 & \text{if } \ind{p}_j > i.
\end{cases}
\qedhere
\end{equation}
\end{proof}

Here are some examples of \cref{eq:messySignaleticParallelDestinations}: for $\ind{p} \eqdef \destVect{2, 1, 4, 2, 2}{5}$ and $\ind{q} \eqdef \destVect{3, 1, 6, 1, 2}{6}$, we have $\ind{p} \circ_1 \ind{q} = \destVect{7, 1, 9, 7, 7}{10}$, $\ind{p} \circ_2 \ind{q} = \destVect{4, 1, 9, 2, 3}{10}$, $\ind{p} \circ_3 \ind{q} = \ind{p} \circ_4 \ind{q} = \destVect{2, 1, 9, 2, 2}{10}$, and $\ind{p} \circ_5 \ind{q} = \destVect{2, 1, 4, 2, 2}{10}$.

\begin{definition}
The \defn{messy parallel $k$-signaletic operad}~$\messySignaleticParallel$ is the quotient of the free operad on~${\Operations_k \eqdef \{\op{l}, \op{r}\}^k}$ by the messy parallel $k$-signaletic equivalence.
\end{definition}

We can immediately observe the connection with the Manin products of \cref{subsec:ManinProducts}.

\begin{proposition}
\label{prop:messyParallelSignaleticManinProduct}
For any two integers $k$ and $l$, we have
\[
\messySignaleticParallel = \Dias^{\whiteManin k} = \Dias^{\blackManin k}
\qqandqq
\messySignaleticParallel[k+l] = \messySignaleticParallel[k]\whiteManin\messySignaleticParallel[l].
\]
\end{proposition}

\begin{proof}
A syntax tree~$\tree$ on~$\Operations_k \times \Operations_l$ can be viewed as a pair~$(\tree_1, \tree_2)$ of syntax trees on~$\Operations_k$ and~$\Operations_l$ respectively  with the same shape.
Moreover, two syntax trees~$\tree, \tree'$ on~$\Operations_k \times \Operations_l$ are messy parallel $(k+l)$-signaletic equivalent if for their corresponding pairs~$(\tree_1,\tree_2)$ and~$(\tree_1', \tree_2')$ of syntax trees on~$\Operations_k$ and~$\Operations_l$, the trees~$\tree_1$ and~$\tree_1'$ are messy parallel $k$-signaletic equivalent and the trees~$\tree_2$ and~$\tree_2'$ are messy parallel $l$-signaletic equivalent respectively.
Hence, we get~${\messySignaleticParallel[k+l] = \messySignaleticParallel[k]\whiteManin\messySignaleticParallel[l]}$.
As~$\messySignaleticParallel[1] = \Dias$, this immediately shows that~${\messySignaleticParallel = \Dias^{\whiteManin k}}$.
The other equality follow from the fact that~${\Dend \whiteManin \Dend = \Dend \blackManin \Dend}$, which was shown in~\cite{Vallette}.
\end{proof}

\cref{def:dias,prop:dendDiass,prop:whiteManin,prop:messyParallelSignaleticManinProduct} imply that the messy parallel $k$-signaletic operad~$\messySignaleticParallel$ is quadratic and Koszul.
We will see an alternative proof in \cref{subsec:HilbertSeriesKoszulitySignaletic}, with an argument uniform for all $k$-signaletic operads (it applies to both messy and tidy and to both parallel and series).
We now extract from the definitions the quadratic relations which provide a presentation of~$\messySignaleticParallel$.

\begin{definition}
\label{def:messySignaleticRelationsParallel}
We call \defn{messy parallel $k$-signaletic relations} all the quadratic relations of~$\messySignaleticParallel$, that is all the relations of the form $\tree_1 = \tree_2$ where $\tree_1$ and $\tree_2$ are two syntax trees with two nodes (\ie of arity~$3$) sharing the same parallel destination vector.
\end{definition}

\begin{remark}
\label{rem:messySignaleticRelationsParallel}
These relations can be explicitly described for any~$k$ as follows.
For any destination vector~${\ind{p} \in [3]^k}$ of arity~$3$, the messy parallel $k$-signaletic operad~$\messySignaleticParallel$ satisfies the quadratic relation~$\tree_1 = \tree_2$ for any two syntax trees~$\tree_1$ and~$\tree_2$ of one of the following forms:
\begin{align*}
	&
	\begin{tikzpicture}[baseline=-.5cm, level/.style={sibling distance = .8cm, level distance = .8cm}]
		\node [rectangle, draw, minimum height=.6cm] {$\operation[a]_\ind{p}$}
			child {node {$1$}}
			child {node [rectangle, draw, minimum height=.6cm] {$\operation[b]_\ind{p}$}
				child {node {$2$}}
				child {node {$3$}}
			}
		;
	\end{tikzpicture}
	\quad\text{where}\quad
	(\operation[a]_\ind{p})_i = \begin{cases} \op{l} & \text{if } \ind{p}_i = 1 \\ \op{r} & \text{if } \ind{p}_i = 2 \\ \op{r} & \text{if } \ind{p}_i = 3 \end{cases}
	\qandq
	(\operation[b]_\ind{p})_i = \begin{cases} \op{l} & \text{if } \ind{p}_i = 2 \\ \op{r} & \text{if } \ind{p}_i = 3 \end{cases}
	\quad\text{for all } i \in [k],
\\
	\text{or}\quad &
	\begin{tikzpicture}[baseline=-.5cm, level/.style={sibling distance = .8cm, level distance = .8cm}]
		\node [rectangle, draw, minimum height=.6cm] {$\operation[c]_\ind{p}$}
			child {node [rectangle, draw, minimum height=.6cm] {$\operation[d]_\ind{p}$}
				child {node {$1$}}
				child {node {$2$}}
			}
			child {node {$3$}}
		;
	\end{tikzpicture}
	\quad\text{where}\quad
	(\operation[c]_\ind{p})_i = \begin{cases} \op{l} & \text{if } \ind{p}_i = 1 \\ \op{l} & \text{if } \ind{p}_i = 2 \\ \op{r} & \text{if } \ind{p}_i = 3 \end{cases}
	\qandq
	(\operation[d]_\ind{p})_i = \begin{cases} \op{l} & \text{if } \ind{p}_i = 1 \\ \op{r} & \text{if } \ind{p}_i = 2 \end{cases}
	\quad\text{for all } i \in [k].
\end{align*}
Note that there are $2^{1+2k}$ syntax trees of arity~$3$ on~$\Operations_k$ but only $3^k$ parallel destination vectors (\ie messy parallel $k$-signaletic equivalence classes).
Therefore, there are $2^{1+2k}-3^k$ independent messy parallel $k$-signaletic relations among syntax trees of arity~$3$ on~$\Operations_k$.
\end{remark}

\begin{example}[Messy parallel $0$-, $1$- and~$2$-signaletic relations]
\label{exm:messySignaleticRelationsParallel}
The messy parallel $0$-signaletic relation is the associative relation:
\allowdisplaybreaks
\begin{gather}
\compoR{n}{n} = \compoL{n}{n}. \tag{$\messySignaleticParallel[]\,.$}\label{eq:messySignaleticParallel0}
\end{gather}
In other words, the messy parallel $0$-signaletic operad~$\messySignaleticParallel[0]$ is just the associative operad~$\As$.

The messy parallel $1$-signaletic relations are the~$5$ diassociative relations:
\allowdisplaybreaks
\begin{gather}
\compoR{l}{l} = \compoR{l}{r} = \compoL{l}{l}, \tag{$\messySignaleticParallel[]\,1$}\label{eq:messySignaleticParallel1} \\[-.1cm]
\compoR{r}{l} = \compoL{r}{l}, \tag{$\messySignaleticParallel[]\,2$}\label{eq:messySignaleticParallel2} \\[-.1cm]
\compoR{r}{r} = \compoL{l}{r} = \compoL{r}{r}. \tag{$\messySignaleticParallel[]\,3$}\label{eq:messySignaleticParallel3}
\end{gather}
In other words, the messy parallel $1$-signaletic operad~$\messySignaleticParallel[1]$ is just the diassociative operad~$\Dias$.

The messy parallel $2$-signaletic relations are the following~$23$ relations:
\begin{gather}
\compoR{l,l}{l,l} = \compoR{l,l}{l,r} = \compoR{l,l}{r,l} = \compoR{l,l}{r,r} = \compoL{l,l}{l,l}, \tag{$\messySignaleticParallel[]\,11$}\label{eq:messySignaleticParallel11} \\[-.1cm]
\compoR{l,r}{l,l} = \compoR{l,r}{r,l} = \compoL{l,r}{l,l}, \tag{$\messySignaleticParallel[]\,12$}\label{eq:messySignaleticParallel12} \\[-.1cm]
\compoR{l,r}{l,r} = \compoR{l,r}{r,r} = \compoL{l,l}{l,r} = \compoL{l,r}{l,r}, \tag{$\messySignaleticParallel[]\,13$}\label{eq:messySignaleticParallel13} \\[-.1cm]
\compoR{r,l}{l,l} = \compoR{r,l}{l,r} = \compoL{r,l}{l,l}, \tag{$\messySignaleticParallel[]\,21$}\label{eq:messySignaleticParallel21} \\[-.1cm]
\compoR{r,r}{l,l} = \compoL{r,r}{l,l}, \tag{$\messySignaleticParallel[]\,22$}\label{eq:messySignaleticParallel22} \\[-.1cm]
\compoR{r,r}{l,r} = \compoL{r,l}{l,r} = \compoL{r,r}{l,r}, \tag{$\messySignaleticParallel[]\,23$}\label{eq:messySignaleticParallel23} \\[-.1cm]
\compoR{r,l}{r,l} = \compoR{r,l}{r,r} = \compoL{l,l}{r,l} = \compoL{r,l}{r,l}, \tag{$\messySignaleticParallel[]\,31$}\label{eq:messySignaleticParallel31} \\[-.1cm]
\compoR{r,r}{r,l} = \compoL{l,r}{r,l} = \compoL{r,r}{r,l}, \tag{$\messySignaleticParallel\,32$}\label{eq:messySignaleticParallel32} \\[-.1cm]
\compoR{r,r}{r,r} = \compoL{l,l}{r,r} = \compoL{l,r}{r,r} = \compoL{r,l}{r,r} = \compoL{r,r}{r,r}. \tag{$\messySignaleticParallel[]\,33$}\label{eq:messySignaleticParallel33}
\end{gather}
Note that these relations were already considered as the dual of the $\Quad$ operad, see for instance~\cite[Proposition~3]{Foissy}.
The translation between our notations and that of~\cite[Proposition~3]{Foissy} is the following:
\[
{\op{l,l}} = \; \nwarrow,
\qquad
{\op{l,r}} = \; \nearrow,
\qquad
{\op{r,l}} = \; \swarrow,
\qqandqq
{\op{r,r}} = \; \searrow.
\]
\end{example}

\begin{remark}
\label{rem:restrictionMessySignaleticParallel}
For any subset $I$ of $[k]$ of cardinality~$\ell \le k$, there is a surjective restriction operad morphism ${\Res_I : \messySignaleticParallel[k] \to \messySignaleticParallel[\ell]}$ sending a destination vector $\ind{p} = \destVect{\ind{p}_1, \dots, \ind{p}_k}{n}$ to the destination subvector ${\Res_I(\ind{p})\eqdef\destVect{\ind{p}_i \mid i \in I}{n}}$. It is equivalently defined on operations by~$\Res_I(\operation)\eqdef(\operation_i)_{i\in I}$.
This is actually a general fact for Manin products of non-symmetric set operads (see \cref{subsec:ManinProducts} for more details). Indeed, each (non-symmetric) set operad~$\Operad$ has a morphism~$\morphism$ to the non-symmetric associative operad~$\As$ sending any operation to the unique operation of~$\As$ with the same arity. Noting that~$\As$ is the neutral element for the Manin product, we observe that~$\Res_I$ is nothing but the following Manin product of morphisms:
\[
\Res_I = \morphism_1 \whiteManin \morphism_2 \whiteManin \dots \whiteManin \morphism_k
\qquad\text{where}\qquad
\morphism_i=\begin{cases}\id&\text{if $i\in I$}, \\ \morphism &\text{otherwise.}\end{cases}
\]
\end{remark}

\begin{remark}
\label{rem:extensionMessySignaleticParallel}
Conversely, since both operations~${\op{l}}$ and~${\op{r}}$ are associative in~$\messySignaleticParallel[1]$, for any~${i \in [k+1]}$ and~$\operation[a] \in \{\op{l}, \op{r}\}$, the map $\Ins^{\operation[a]}_i : \messySignaleticParallel[k] \to \messySignaleticParallel[k+1]$ inserting~$\operation[a]$ at position~$i$ in every nodes is an injective operad morphism.
On destination vectors of degree~$n$ it amounts to respectively insert~$1$ or~$n$ at position~$i$.
Similarly to the previous \cref{rem:restrictionMessySignaleticParallel}, this is inserting $\As \to \messySignaleticParallel[1]$ in a Manin product of identity morphisms.
\end{remark}


\subsubsection{Tidy parallel signaletic operads}
\label{subsubsec:tidyParallelSignaleticOperads}

We now impose additional constraints on the trafic signals not contained in the parallel routes.
We say that a syntax tree~$\tree$ on~$\Operations_k$ is \defn{tidy parallel} if all the traffic signals not contained in its parallel routes \textbf{point to the left}.
Otherwise, we say that~$\tree$ is \defn{messy parallel}.
We define the \defn{tidy parallel destination vector} of a syntax tree~$\tree$ to be its parallel destination vector if~$\tree$ is tidy parallel, and to be $0$ if~$\tree$ is messy parallel.
Finally, we say that two syntax trees~$\tree, \tree'$ on~$\Operations_k$ are \defn{tidy parallel $k$-signaletic equivalent} and we write~$\tree \equiv^\parallel \tree'$ if they have the same tidy parallel destination vector.
In particular, all messy parallel syntax trees are equivalent to the zero tree~$0$.

\begin{proposition}
The tidy parallel $k$-signaletic equivalence is compatible with grafting of syntax trees: $\tree \circ_i \tree[s] \equiv^\parallel \tree' \circ_i \tree[s]'$ for any syntax trees~$\tree \equiv^\parallel \tree'$ of arity~$p$ and $\tree[s] \equiv^\parallel \tree[s]'$ of arity~$q$, and any~$i \in [p]$.
\end{proposition}

\begin{proof}
Let~$\ind{p}$ and $\ind{q}$ denote the tidy parallel destination vectors of~$\tree[t]$ and~$\tree[s]$ respectively.
Then the parallel series destination vector~$\ind{r}$ of~$\tree[t] \circ_i \tree[s]$ is $0$ unless $\ind{q}_\ell = 1$ for all $j \in [k]$ such that $\ind{p}_j \ne i$, and if so, is equal to the messy parallel destination vector, that is for all~$j \in [k]$,
\begin{equation}
\label{eq:tidySignaleticParallelDestinations}
\ind{r}_j = 
\begin{cases}
\ind{p}_j & \text{if } \ind{p}_j < i, \\
\ind{p}_j + \ind{q}_j - 1 & \text{if } \ind{p}_j = i, \\
\ind{p}_j + q - 1 & \text{if } \ind{p}_j > i.\\
\end{cases}
\qedhere
\end{equation}
\end{proof}

Here are some examples of \cref{eq:tidySignaleticParallelDestinations}: for $\ind{p} \eqdef \destVect{2, 1, 4, 2, 2}{5}$ and $\ind{q} \eqdef \destVect{3, 1, 1, 4, 1}{4}$, we have $\ind{p} \circ_1 \ind{q} = \ind{p} \circ_3 \ind{q} = \ind{p} \circ_4 \ind{q} = \ind{p} \circ_5 \ind{q} = 0$ and $\ind{p} \circ_2 \ind{q} = \destVect{4, 1, 7, 5, 2}{8}$.

\begin{definition}
\label{def:tidySignaleticParallel}
The \defn{tidy parallel $k$-signaletic operad}~$\tidySignaleticParallel$ is the quotient of the free operad on~${\Operations_k \eqdef \{\op{l}, \op{r}\}^k}$ by the tidy parallel $k$-signaletic equivalence.
\end{definition}

Similarly to \cref{prop:messyParallelSignaleticManinProduct},  we can immediately observe the connection with the white Manin product of \cref{subsec:ManinProducts}.

\begin{proposition}
\label{prop:tidyParallelSignaleticManinProduct}
For any two integers $k$ and $l$, we have
\[
\tidySignaleticParallel = {\Dupdual_{\op{l}}}^{\whiteManin k}
\qqandqq
\tidySignaleticParallel[k+l] = \tidySignaleticParallel[k]\whiteManin\tidySignaleticParallel[l].
\]
\end{proposition}

This implies that the tidy parallel $k$-signaletic operad~$\tidySignaleticParallel$ is quadratic and Koszul.
As already mentioned, we will see a uniform argument in \cref{subsec:HilbertSeriesKoszulitySignaletic}.
We now extract from the definitions the quadratic relations which provide a presentation of~$\tidySignaleticParallel$.

\begin{definition}
\label{def:tidySignaleticRelationsParallel}
We call \defn{tidy parallel $k$-signaletic relations} all the quadratic relations of~$\tidySignaleticParallel$, that is all the relations of the form $\tree = 0$ where $\tree$ is a messy parallel syntax tree with two nodes (\ie of arity~$3$), and all the relations of the form $\tree_1 = \tree_2$ where $\tree_1$ and $\tree_2$ are two tidy syntax trees with two nodes (\ie of arity~$3$) sharing the same parallel destination vector.
\end{definition}

\begin{remark}
\label{rem:tidySignaleticRelationsParallel}
These relations can be explicitly described for any~$k$ as follows.
First, the tidy parallel $k$-signaletic operad~$\tidySignaleticParallel$ satisfies the quadratic relations
\[
	\begin{tikzpicture}[baseline=-.5cm, level/.style={sibling distance = .8cm, level distance = .8cm}]
		\node [rectangle, draw, minimum height=.6cm] {$\operation[a]$}
			child {node {$1$}}
			child {node [rectangle, draw, minimum height=.6cm] {$\operation[b]$}
				child {node {$2$}}
				child {node {$3$}}
			}
		;
	\end{tikzpicture}
	\!\!= 0
	\qqandqq
	\begin{tikzpicture}[baseline=-.5cm, level/.style={sibling distance = .8cm, level distance = .8cm}]
		\node [rectangle, draw, minimum height=.6cm] {$\operation[c]$}
			child {node [rectangle, draw, minimum height=.6cm] {$\operation[d]$}
				child {node {$1$}}
				child {node {$2$}}
			}
			child {node {$3$}}
		;
	\end{tikzpicture}
	\!\!= 0
\]
for any~$\operation[a], \operation[b], \operation[c], \operation[d] \in \Operations_k$ such that~$\operation[a]_i = {\op{l}}$ and~$\operation[b]_i = {\op{r}}$ for some~$i \in [k]$, and~$\operation[c]_j = {\op{r}}$ and~$\operation[d]_j = {\op{r}}$ for some~$j \in [k]$ .
Second, for any destination vector~${\ind{p} \in [3]^k}$ of arity~$3$, the tidy parallel $k$-signaletic operad~$\tidySignaleticParallel$ satisfies the quadratic relation
\[
	\begin{tikzpicture}[baseline=-.5cm, level/.style={sibling distance = .8cm, level distance = .8cm}]
		\node [rectangle, draw, minimum height=.6cm] {$\operation[a]_\ind{p}$}
			child {node {$1$}}
			child {node [rectangle, draw, minimum height=.6cm] {$\operation[b]_\ind{p}$}
				child {node {$2$}}
				child {node {$3$}}
			}
		;
	\end{tikzpicture}
	=
	\begin{tikzpicture}[baseline=-.5cm, level/.style={sibling distance = .8cm, level distance = .8cm}]
		\node [rectangle, draw, minimum height=.6cm] {$\operation[c]_\ind{p}$}
			child {node [rectangle, draw, minimum height=.6cm] {$\operation[d]_\ind{p}$}
				child {node {$1$}}
				child {node {$2$}}
			}
			child {node {$3$}}
		;
	\end{tikzpicture}
\]
where~$\operation[a]_\ind{p}, \operation[b]_\ind{p}, \operation[c]_\ind{p}, \operation[d]_\ind{p} \in \Operations_k$ are defined by
\[
	(\operation[a]_\ind{p})_i \eqdef \begin{cases} \op{l} & \!\!\text{if } \ind{p}_i = 1 \\ \op{r} & \!\!\text{if } \ind{p}_i = 2 \\ \op{r} & \!\!\text{if } \ind{p}_i = 3 \end{cases}
	\quad
	(\operation[b]_\ind{p})_i \eqdef \begin{cases} \op{l} & \!\!\text{if } \ind{p}_i = 1 \\ \op{l} & \!\!\text{if } \ind{p}_i = 2 \\ \op{r} & \!\!\text{if } \ind{p}_i = 3 \end{cases}
	\quad
	(\operation[c]_\ind{p})_i \eqdef \begin{cases} \op{l} & \!\!\text{if } \ind{p}_i = 1 \\ \op{l} & \!\!\text{if } \ind{p}_i = 2 \\ \op{r} & \!\!\text{if } \ind{p}_i = 3 \end{cases}
	\quad
	(\operation[d]_\ind{p})_i \eqdef \begin{cases} \op{l} & \!\!\text{if } \ind{p}_i = 1 \\ \op{r} & \!\!\text{if } \ind{p}_i = 2 \\ \op{l} & \!\!\text{if } \ind{p}_i = 3 \end{cases}
\]
for all~$i \in [k]$.
Again, there are $2^{1+2k}-3^k$ independent tidy parallel $k$-signaletic relations among syntax trees of arity~$3$ on~$\Operations_k$.
\end{remark}

\begin{example}[Tidy parallel $0$-, $1$- and~$2$-signaletic relations]
The tidy parallel $0$-signaletic relation is the associative relation:
\allowdisplaybreaks
\begin{gather}
\compoR{n}{n} = \compoL{n}{n}. \tag{$\tidySignaleticParallel[]\,.$}\label{eq:tidySignaleticParallel0}
\end{gather}
In other words, the tidy parallel $0$-signaletic operad~$\tidySignaleticParallel[0]$ is just the associative operad~$\As$.

The tidy parallel $1$-signaletic relations are the~$5$ twisted dual duplicial relations:
\allowdisplaybreaks
\begin{gather}
\compoR{l}{l} = \compoL{l}{l} \qandq \compoR{l}{r} = 0, \tag{$\tidySignaleticParallel[]\,1$}\label{eq:tidySignaleticParallel1} \\[-.1cm]
\compoR{r}{l} = \compoL{r}{l}, \tag{$\tidySignaleticParallel[]\,2$}\label{eq:tidySignaleticParallel2} \\[-.1cm]
\compoR{r}{r} = \compoL{l}{r} \qandq \compoL{r}{r} = 0. \tag{$\tidySignaleticParallel[]\,3$}\label{eq:tidySignaleticParallel3}
\end{gather}
In other words, the tidy parallel $1$-signaletic operad~$\tidySignaleticParallel[1]$ is just the twisted dual duplicial operad~$\Dupdual_{\op{l}}$.

The tidy parallel $2$-signaletic relations are the following~$23$ relations:
\begin{gather}
\compoR{l,l}{l,l} = \compoL{l,l}{l,l} \qandq \compoR{l,l}{l,r} = \compoR{l,l}{r,l} = \compoR{l,l}{r,r} = 0, \tag{$\tidySignaleticParallel[]\,11$}\label{eq:tidySignaleticParallel11} \\[-.1cm]
\compoR{l,r}{l,l} = \compoL{l,r}{l,l} \qandq \compoR{l,r}{r,l} = 0, \tag{$\tidySignaleticParallel[]\,12$}\label{eq:tidySignaleticParallel12} \\[-.1cm]
\compoR{l,r}{l,r} = \compoL{l,l}{l,r} \qandq \compoR{l,r}{r,r} = \compoL{l,r}{l,r} = 0, \tag{$\tidySignaleticParallel[]\,13$}\label{eq:tidySignaleticParallel13} \\[-.1cm]
\compoR{r,l}{l,l} = \compoL{r,l}{l,l} \qandq \compoR{r,l}{l,r} = 0, \tag{$\tidySignaleticParallel[]\,21$}\label{eq:tidySignaleticParallel21} \\[-.1cm]
\compoR{r,r}{l,l} = \compoL{r,r}{l,l}, \tag{$\tidySignaleticParallel[]\,22$}\label{eq:tidySignaleticParallel22} \\[-.1cm]
\compoR{r,r}{l,r} = \compoL{r,l}{l,r} \qandq \compoL{r,r}{l,r} = 0, \tag{$\tidySignaleticParallel[]\,23$}\label{eq:tidySignaleticParallel23} \\[-.1cm]
\compoR{r,l}{r,l} = \compoL{l,l}{r,l} \qandq \compoR{r,l}{r,r} = \compoL{r,l}{r,l} = 0, \tag{$\tidySignaleticParallel[]\,31$}\label{eq:tidySignaleticParallel31} \\[-.1cm]
\compoR{r,r}{r,l} = \compoL{l,r}{r,l} \qandq \compoL{r,r}{r,l} = 0, \tag{$\tidySignaleticParallel[]\,32$}\label{eq:tidySignaleticParallel32} \\[-.1cm]
\compoR{r,r}{r,r} = \compoL{l,l}{r,r} \qandq \compoL{l,r}{r,r} = \compoL{r,l}{r,r} = \compoL{r,r}{r,r} = 0. \tag{$\tidySignaleticParallel[]\,33$}\label{eq:tidySignaleticParallel33}
\end{gather}
\end{example}

\begin{remark}
\label{rem:restrictionTidySignaleticParallel}
Similarly to \cref{rem:restrictionMessySignaleticParallel}, for any subset $I$ of $[k]$ of cardinality~$\ell \le k$, there is a surjective restriction operad morphism ${\Res_I : \tidySignaleticParallel[k] \to \tidySignaleticParallel[\ell]}$ sending a destination vector ${\ind{p} = \destVect{\ind{p}_1, \dots, \ind{p}_k}{n}}$ to the destination subvector ${\Res_I(\ind{p})\eqdef\destVect{\ind{p}_i \mid i \in I}{n}}$.
\end{remark}

\begin{remark}
\label{rem:extensionTidySignaleticParallel}
Similarly to \cref{rem:extensionMessySignaleticParallel}, since the two operations $\op{l}$ and~${\op{l}} + {\op{r}}$ are associative in~$\tidySignaleticParallel[1]$, 
for any~$i \in [k+1]$ and~$\operation[a] \in \{\op{l}, {\op{l}} + {\op{r}}\}$, the map $\Ins^{\operation[a]}_i : \tidySignaleticParallel[k] \to \tidySignaleticParallel[k+1]$ inserting~$\operation[a]$ at position~$i$ in every nodes is an injective operad morphism.
On destination vectors, the map~$\Ins^{\op{l}}_i$ inserts a~$1$ at position~$i$, while the map~$\Ins^{{\op{l}} + {\op{r}}}_i$ is given by
\[
\Ins_i^{{\op{l}} + {\op{r}}} \big( \destVect{r_1, \dots, r_k}{n} \big) = \sum_{d \in [n]} \destVect{r_1, \dots, r_{i-1}, d, r_{i}, \dots, r_k}{n}.
\]
\end{remark}
\begin{remark}
\label{rem:multitidySignaleticParallel}
In this parallel setting, we have in fact much more freedom on the additional constraints that we impose for the tidy situation.
Consider a constraint word~$\constraint$ with $k$ letters in~$\{\op{l}, \op{r}, \odot, \star\}$.
We say that a syntax tree~$\tree$ on~$\Operations_k$ is \defn{$\constraint$-tidy parallel} if at each node, the $i$-th traffic signal either is contained in the $i$-th parallel route of~$\tree$, or points to the left if~$\constraint_i = {\op{l}}$, to the right if~$\constraint_i = {\op{r}}$, towards the $i$-th parallel route of~$\tree$ if~$\constraint_i = \odot$, or in any direction if~$\constraint_i = \star$.
The resulting $\constraint$-tidy parallel signaletic equivalence is then clearly compatible with grafting. Indeed, for two syntax trees~$\tree[t]$ and~$\tree[s]$ of arities~$p$ and~$q$ with $\constraint$-tidy parallel destination vectors~$\ind{p}$ and $\ind{q}$, the $\constraint$-tidy parallel destination vector~$\ind{r}$ of~$\tree[t] \circ_i \tree[s]$ is $0$ unless for all $j \in [k]$
\begin{itemize}
\item if $\ind{p}_j < i$, we have either~$\constraint_j = \star$, or $\constraint_j \in \{\odot, \op{l}\}$ and $\ind{q}_j = 1$, or $\constraint_j = {\op{r}}$ and $\ind{q}_j = q$,
\item if $\ind{p}_j < i$, we have either~$\constraint_j = \star$, or $\constraint_j \in \{\odot, \op{r}\}$ and $\ind{q}_j = q$, or $\constraint_j = {\op{l}}$ and $\ind{q}_j = 1$,
\end{itemize}
and if so, is equal to the messy parallel destination vector.
This defines the \defn{$\constraint$-tidy parallel signaletic operad}~$\tidySignaleticParallel[\constraint]$ as in \cref{def:tidySignaleticParallel}.
In other words, for any words~$\constraint[c], \constraint[d] \in \{\op{l}, \op{r}, \odot, \star\}$, we have
\[
\tidySignaleticParallel[\constraint] = \underset{i \in [|\constraint|]}{\scalebox{2}{$\square$}} \Dupdual_{\constraint_i}
\qqandqq
\tidySignaleticParallel[{\constraint[c] \cdot \constraint[d]}] = \tidySignaleticParallel[{\constraint[c]}] \whiteManin \tidySignaleticParallel[{\constraint[d]}].
\]
We have decided to present the tidy parallel $k$-signaletic operad with the constraint~$\constraint = {\op{l}\!^k}$ to simplify the presentation and since this will be the only possible option in series.
However, note that we will use an action of the Koszul dual of the $\op{l,r}$-tidy parallel signaletic operad in \cref{subsec:actionParallelCitalangisOperads}.
\end{remark}


\subsection{Series signaletic operads}
\label{subsec:seriesSignaleticOperads}

We now consider the series signaletic operads, and we treat separately the messy and tidy situations.


\subsubsection{Messy series signaletic operads}
\label{subsubsec:messySeriesSignaleticOperads}

Fix an integer~$k \ge 0$ and consider a syntax tree~$\tree$ on the operations~$\{\op{l}, \op{r}\}^k$ with~$n$ leaves that we label~$1, \dots, n$ from left to right. Assume now that $k$ cars~$c_1, \dots, c_k$ depart sequentially from the root of~$\tree$ and that the $j$-th car~$c_j$ always follows the indication given by the leftmost remaining letter of the traffic signal at each branching node and erases this letter. We call \defn{series routes} the paths followed by the $k$ cars in the syntax tree~$\tree$. Finally, each car~$c_j$ ends at a certain leaf labeled~$\ell_j$ and we call \defn{series destination vector} of~$\tree$ the vector~$\destVect{\ell_1, \dots, \ell_k}{n}$. See \cref{fig:traversingTreeSeries}.

\begin{figure}[h]
	\centerline{
	\begin{tikzpicture}[level/.style={sibling distance=6cm/#1, level distance = 1.2cm/sqrt(#1)}]
		\node [rectangle, draw] (root) {$\op{la,rb,lc}$}
			child {node [rectangle, draw] (l) {$\op{la,rc,r}$}
				child {node [rectangle, draw] (ll) {$\op{ra,l,l}$}
					child {node (lll) {1}}
					child {node [rectangle, draw] (llr) {$\op{la,r,r}$}
						child {node (llrl) {2}}
						child {node (llrr) {3}}
					}
				}
				child {node [rectangle, draw] (lr) {$\op{rc,r,l}$}
					child {node (lrl) {4}}
					child {node [rectangle, draw] (lrr) {$\op{rc,l,l}$}
						child {node (lrrl) {5}}
						child {node (lrrr) {6}}
					}
				}
			}
			child {node [rectangle, draw] (r) {$\op{lb,r,r}$}
				child {node [rectangle, draw] (rl) {$\op{rb,l,r}$}
					child {node (rll) {7}}
					child {node (rlr) {8}}
				}
				child {node [rectangle, draw] (rr) {$\op{r,l,r}$}
					child {node [rectangle, draw] (rrl) {$\op{r,l,l}$}
						child {node (rrll) {9}}
						child {node (rrlr) {10}}
					}
					child {node (rrr) {11}}
				}
			}
		;
		\draw[join=round, line cap=round, line width = 3pt,   red, opacity=.5, ->] ($(root) + (-.25,.5)$) -- ($(root) + (-.25,0)$) -- ($(l) + (-.25,0)$) -- ($(ll) + (-.25,0)$) -- ($(llr) + (-.25,0)$) -- ($(llrl) + (-.25,0)$);
		\draw[join=round, line cap=round, line width = 3pt,  blue, opacity=.5, ->] ($(root) + ( 0,.5)$) -- ($(root) + ( 0,0)$) -- ($(r) + (-.25,0)$) -- ($(rl) + (-.25,0)$) -- ($(rlr) + (-.25,0)$);
		\draw[join=round, line cap=round, line width = 3pt, green, opacity=.5, ->] ($(root) + ( .25,.5)$) -- ($(root) + ( .25,0)$) -- ($(l) + ( 0,0)$) -- ($(lr) + (-.25,0)$) -- ($(lrr) + (-.25,0)$) -- ($(lrrr) + (-.25,0)$);
	\end{tikzpicture}
	}
	\caption{Traversing the syntax tree in series. The series routes are marked, and the series destination vector is~$\destVect{2,8,6}{11}$.}
	\label{fig:traversingTreeSeries}
\end{figure}
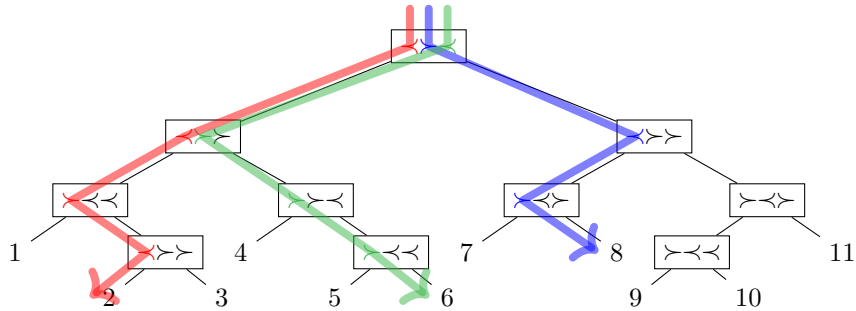

We say that two syntax trees~$\tree, \tree'$ on~$\Operations_k$ with the same arity are \defn{messy series $k$-signaletic equivalent} and we write~$\tree \approxeq^{\series} \tree'$ if they have the same series destination vector.

\begin{proposition}
The messy series $k$-signaletic equivalence is compatible with grafting of syntax trees: $\tree \circ_i \tree[s] \approxeq^\series \tree' \circ_i \tree[s]'$ for any syntax trees~$\tree \approxeq^\series \tree'$ of arity~$p$ and $\tree[s] \approxeq^\series \tree[s]'$ of arity~$q$, and any~$i \in [p]$.
\end{proposition}

\begin{proof}
Consider the $j$-th car~$c_j$ and denote by~$\ind{p}_j$ its series destination in~$\tree$ and by~$\ind{q}_j$ its series destination in~$\tree[s]$. Then its series destination~$\ind{r}_j$ in~$\tree \circ_i \tree[s]$ is given by
\begin{equation}
\label{eq:messySignaleticSeriesDestinations}
\ind{r}_j = 
\begin{cases}
\ind{p}_j & \text{if } \ind{p}_j < i, \\
\ind{p}_j + \ind{q}_{|\set{\ell \le j}{\ind{p}_\ell = i}|} - 1 & \text{if } \ind{p}_j = i, \\
\ind{p}_j + q - 1 & \text{if } \ind{p}_j > i.
\end{cases}
\qedhere
\end{equation}
\end{proof}

Here are some examples of \cref{eq:messySignaleticSeriesDestinations}: for $\ind{p} \eqdef \destVect{2, 1, 4, 2, 2}{5}$ and $\ind{q} \eqdef \destVect{3, 1, 6, 1, 2}{6}$, we have $\ind{p} \circ_1 \ind{q} = \destVect{7, 3, 9, 7, 7}{10}$, $\ind{p} \circ_2 \ind{q} = \destVect{4, 1, 9, 2, 7}{10}$, $\ind{p} \circ_3 \ind{q} = \destVect{2, 1, 9, 2, 2}{10}$, $\ind{p} \circ_4 \ind{q} = \destVect{2, 1, 6, 2, 2}{10}$, and $\ind{p} \circ_5 \ind{q} = \destVect{2, 1, 4, 2, 2}{10}$.

\begin{definition}
The \defn{messy series $k$-signaletic operad}~$\messySignaleticSeries$ is the quotient of the free operad on~${\Operations_k \eqdef \{\op{l}, \op{r}\}^k}$ by the messy series $k$-signaletic equivalence.
\end{definition}

As already mentioned, we will see in \cref{subsec:HilbertSeriesKoszulitySignaletic} that the messy series $k$-signaletic operad~$\messySignaleticSeries$ is quadratic and Koszul. We now extract from the definitions the quadratic relations which provide a presentation of~$\messySignaleticSeries$.

\begin{definition}
\label{def:messySignaleticRelationsSeries}
We call \defn{messy series $k$-signaletic relations} all the quadratic relations of~$\messySignaleticSeries$, that is all the relations of the form $\tree_1 = \tree_2$ where $\tree_1$ and $\tree_2$ are two syntax trees with two nodes (\ie of arity~$3$) sharing the same series destination vector.
\end{definition}

\begin{remark}
\label{rem:messySignaleticRelationsSeries}
These relations can be explicitly described for any~$k$ as follows.
For any destination vector~${\ind{p} \in [3]^k}$ of arity~$3$, the messy series $k$-signaletic operad~$\messySignaleticSeries$ satisfies the quadratic relation~$\tree_1 = \tree_2$ for any two syntax trees~$\tree_1$ and~$\tree_2$ of one of the following forms:

\begin{align*}
	&
	\begin{tikzpicture}[baseline=-.5cm, level/.style={sibling distance = .8cm, level distance = .8cm}]
		\node [rectangle, draw, minimum height=.6cm] {$\operation[a]_\ind{p}$}
			child {node {$1$}}
			child {node [rectangle, draw, minimum height=.6cm] {$\operation[b]_\ind{p}$}
				child {node {$2$}}
				child {node {$3$}}
			}
		;
	\end{tikzpicture}
	\;\,\text{where}\;\,
	(\operation[a]_\ind{p})_i = \begin{cases} \op{l} & \text{if } \ind{p}_i = 1 \\ \op{r} & \text{if } \ind{p}_i = 2 \\ \op{r} & \text{if } \ind{p}_i = 3 \end{cases}
	\;\,\text{and}\;\,
	(\operation[b]_\ind{p})_i = \begin{cases} \op{l} & \text{if } (\ind{p}^{\{2,3\}})_i = 2 \\ \op{r} & \text{if } (\ind{p}^{\{2,3\}})_i = 3 \end{cases}
	\;\,\text{for all } i \in [k],
\\
	\text{or}\;\, &
	\begin{tikzpicture}[baseline=-.5cm, level/.style={sibling distance = .8cm, level distance = .8cm}]
		\node [rectangle, draw, minimum height=.6cm] {$\operation[c]_\ind{p}$}
			child {node [rectangle, draw, minimum height=.6cm] {$\operation[d]_\ind{p}$}
				child {node {$1$}}
				child {node {$2$}}
			}
			child {node {$3$}}
		;
	\end{tikzpicture}
	\;\,\text{where}\;\,
	(\operation[c]_\ind{p})_i = \begin{cases} \op{l} & \text{if } \ind{p}_i = 1 \\  \op{l} & \text{if } \ind{p}_i = 2 \\ \op{r} & \text{if } \ind{p}_i = 3 \end{cases}
	\;\,\text{and}\;\,
	(\operation[d]_\ind{p})_i = \begin{cases} \op{l} & \text{if } (\ind{p}^{\{1,2\}})_i = 1 \\ \op{r} & \text{if } (\ind{p}^{\{1,2\}})_i = 2 \end{cases}
	\;\,\text{for all } i \in [k],
\end{align*}
where for any~$L \subseteq [3]$, we have denoted by~$\ind{p}^L$ the subword of~$\ind{p}$ consisting only of the letters which belong to~$L$, and for any~$x \in L$ the condition~$(\ind{p}^L)_i = x$ implicitly implies that~$\ind{p}^L$ has length at least~$i$.
Again, there are $2^{1+2k}-3^k$ independent messy series $k$-signaletic relations among syntax trees of arity~$3$ on~$\Operations_k$.
\end{remark}

\begin{example}[Messy series $0$-, $1$- and~$2$-signaletic relations]
The messy series $0$-signaletic relation is the associative relation:
\allowdisplaybreaks
\begin{gather}
\compoR{n}{n} = \compoL{n}{n}. \tag{$\messySignaleticSeries[]\,.$}\label{eq:messySignaleticSeries0}
\end{gather}
In other words, the messy series $0$-signaletic operad~$\messySignaleticSeries[0]$ is just the associative operad~$\As$.

The messy series $1$-signaletic relations are the~$5$ diassociative relations:
\begin{gather*}
\compoR{l}{l} = \compoR{l}{r} = \compoL{l}{l}, \tag{$\messySignaleticSeries[]\,1$}\label{eq:messySignaleticSeries1} \\[-.1cm]
\compoR{r}{l} = \compoL{r}{l}, \tag{$\messySignaleticSeries[]\,2$}\label{eq:messySignaleticSeries2} \\[-.1cm]
\compoR{r}{r} = \compoL{l}{r} = \compoL{r}{r}. \tag{$\messySignaleticSeries[]\,3$}\label{eq:messySignaleticSeries3}
\end{gather*}
In other words, the messy series $1$-signaletic operad~$\messySignaleticSeries[1]$ is just the diassociative operad~$\Dias$.

The messy series $2$-signaletic relations are the following~$23$ relations:
\allowdisplaybreaks
\begin{gather*}
\compoR{l,l}{l,l} = \compoR{l,l}{l,r} = \compoR{l,l}{r,l} = \compoR{l,l}{r,r} = \compoL{l,l}{l,l}, \tag{$\messySignaleticSeries[]\,11$}\label{eq:messySignaleticSeries11} \\[-.1cm]
\compoR{l,r}{l,l} = \compoR{l,r}{l,r} = \compoL{l,r}{l,l}, \tag{$\messySignaleticSeries[]\,12$}\label{eq:messySignaleticSeries12} \\[-.1cm]
\compoR{l,r}{r,l} = \compoR{l,r}{r,r} = \compoL{l,l}{l,r} = \compoL{l,r}{l,r}, \tag{$\messySignaleticSeries[]\,13$}\label{eq:messySignaleticSeries13} \\[-.1cm]
\compoR{r,l}{l,l} = \compoR{r,l}{l,r} = \compoL{r,l}{l,l}, \tag{$\messySignaleticSeries[]\,21$}\label{eq:messySignaleticSeries21} \\[-.1cm]
\compoR{r,r}{l,l} = \compoL{r,r}{l,l}, \tag{$\messySignaleticSeries[]\,22$}\label{eq:messySignaleticSeries22} \\[-.1cm]
\compoR{r,r}{l,r} = \compoL{r,l}{l,r} = \compoL{r,r}{l,r}, \tag{$\messySignaleticSeries[]\,23$}\label{eq:messySignaleticSeries23} \\[-.1cm]
\compoR{r,l}{r,l} = \compoR{r,l}{r,r} = \compoL{l,l}{r,l} = \compoL{l,r}{r,l}, \tag{$\messySignaleticSeries[]\,31$}\label{eq:messySignaleticSeries31} \\[-.1cm]
\compoR{r,r}{r,l} = \compoL{r,l}{r,l} = \compoL{r,r}{r,l}, \tag{$\messySignaleticSeries[]\,32$}\label{eq:messySignaleticSeries32} \\[-.1cm]
\compoR{r,r}{r,r} = \compoL{l,l}{r,r} = \compoL{l,r}{r,r} = \compoL{r,l}{r,r} = \compoL{r,r}{r,r}. \tag{$\messySignaleticSeries[]\,33$}\label{eq:messySignaleticSeries33}
\end{gather*}
\end{example}

\begin{remark}
\label{rem:restrictionMessySignaleticSeries}
The only subset $I \subseteq [k]$ of cardinality~$\ell$ for which the restriction~$\Res_I$ defined in \cref{rem:restrictionMessySignaleticParallel} is a surjective morphism~$\Res_I:\messySignaleticSeries[k] \mapsto \messySignaleticSeries[l]$ between messy series $k$-signaletic operads is the initial subset~$[\ell]$. We denote this morphism by~$\Res^k_\ell$. Its action on destination vectors or operations amounts to keep only the first~$\ell$ coordinates.
\end{remark}

\begin{remark}
\label{rem:extensionMessySignaleticSeries}
In contrast to \cref{rem:extensionMessySignaleticParallel}, the map~$\Ins^{\operation[a]}_i$ inserting an operation~$\operation[a] \in \{\op{l}, \op{r}\}$ at position~$i$ in all nodes of a syntax tree is not an operad morphism~$\messySignaleticSeries[k] \mapsto \messySignaleticSeries[k+1]$.
Counter-examples are given by:
\[
\compoR{l}{l} = \compoL{l}{l}
\qandq
\compoR{r}{r} = \compoL{r}{r},
\]
while
\[
\compoR{l,r}{l,r} \ne \compoL{l,r}{l,r}
\qandq
\compoR{r,l}{r,l} \ne \compoL{r,l}{r,l}.
\]
\end{remark}


\subsubsection{Tidy series signaletic operads}
\label{subsubsec:tidySeriesSignaleticOperads}

We now impose additional constraints on the trafic signals not contained in the series routes.
We say that a syntax tree~$\tree$ on~$\Operations_k$ is \defn{tidy series} if all the traffic signals not contained in its series routes point to the left.
Otherwise, we say that~$\tree$ is \defn{messy series}.
We define the \defn{tidy series destination vector} of a syntax tree~$\tree$ to be its series destination vector if~$\tree$ is tidy series, and to be $0$ if~$\tree$ is messy series.
Finally, we say that two syntax trees~$\tree, \tree'$ on~$\Operations_k$ are \defn{tidy series $k$-signaletic equivalent} and we write~$\tree \equiv^\series \tree'$ if they have the same tidy series destination vector.
In particular, all messy series syntax trees are equivalent to the zero tree~$0$.

\begin{proposition}
The tidy series $k$-signaletic equivalence is compatible with grafting of syntax trees: $\tree \circ_i \tree[s] \equiv^\series \tree' \circ_i \tree[s]'$ for any syntax trees~$\tree \equiv^\series \tree'$ of arity~$p$ and $\tree[s] \equiv^\series \tree[s]'$ of arity~$q$, and any~$i \in [p]$.
\end{proposition}

\begin{proof}
Let~$\ind{p}$ and $\ind{q}$ denote the tidy series destination vectors of~$\tree[t]$ and~$\tree[s]$ respectively.
For~$i \in [p]$, we denote by~$|\ind{p}|_i \eqdef |\set{\ell \in [k]}{\ind{p}_\ell = i}|$ the number of occurrences of~$i$ in~$\ind{p}$.
Then the tidy series destination vector~$\ind{r}$ of~$\tree[t] \circ_i \tree[s]$ is $0$ unless $\ind{q}_\ell = 1 \text{ for all } |\ind{p}|_i < \ell \le k$, and if so, is equal to the messy series destination vector, that is for all~$j \in [k]$,
\begin{equation}
\label{eq:tidySignaleticSeriesDestinations}
\ind{r}_j = 
\begin{cases}
\ind{p}_j & \text{if } \ind{p}_j < i, \\
\ind{p}_j + \ind{q}_{|\set{\ell \le j}{\ind{p}_\ell = i}|} - 1 & \text{if } \ind{p}_j = i, \\
\ind{p}_j + q - 1 & \text{if } \ind{p}_j > i.\\
\end{cases}
\qedhere
\end{equation}
\end{proof}

Here are some examples of \cref{eq:tidySignaleticSeriesDestinations}: for $\ind{p} \eqdef \destVect{2, 1, 4, 2, 4}{5}$ and $\ind{q} \eqdef \destVect{3, 5, 1, 1, 1}{6}$, we have $\ind{p} \circ_1 \ind{q} = \ind{p} \circ_3 \ind{q} = \ind{p} \circ_5 \ind{q} = 0$, $\ind{p} \circ_2 \ind{q} = \destVect{4, 1, 9, 6, 9}{10}$, $\ind{p} \circ_4 \ind{q} = \destVect{2, 1, 6, 2, 8}{10}$.

\begin{definition}
The \defn{tidy series $k$-signaletic operad}~$\tidySignaleticSeries$ is the quotient of the free operad on~${\Operations_k \eqdef \{\op{l}, \op{r}\}^k}$ by the tidy series $k$-signaletic equivalence.
\end{definition}

As already mentioned, we will see in \cref{subsec:HilbertSeriesKoszulitySignaletic} that the tidy series $k$-signaletic operad~$\tidySignaleticSeries$ is quadratic and Koszul. We now extract from the definitions the quadratic relations which provide a presentation of~$\tidySignaleticSeries$.

\begin{definition}
\label{def:tidySignaleticRelationsSeries}
We call \defn{tidy series $k$-signaletic relations} all the quadratic relations of~$\tidySignaleticSeries$, that is all the relations of the form $\tree = 0$ where $\tree$ is a messy series syntax tree with two nodes (\ie of arity~$3$), and all the relations of the form $\tree_1 = \tree_2$ where $\tree_1$ and $\tree_2$ are two tidy syntax trees with two nodes (\ie of arity~$3$) sharing the same series destination vector.
\end{definition}

\begin{remark}
\label{rem:tidySignaleticRelationsSeries}
These relations can be explicitly described for any~$k$ as follows.
First, the tidy series $k$-signaletic operad~$\tidySignaleticSeries$ satisfies the quadratic relations
\[
	\begin{tikzpicture}[baseline=-.5cm, level/.style={sibling distance = .8cm, level distance = .8cm}]
		\node [rectangle, draw, minimum height=.6cm] {$\operation[a]$}
			child {node {$1$}}
			child {node [rectangle, draw, minimum height=.6cm] {$\operation[b]$}
				child {node {$2$}}
				child {node {$3$}}
			}
		;
	\end{tikzpicture}
	\!\!= 0
	\qqandqq
	\begin{tikzpicture}[baseline=-.5cm, level/.style={sibling distance = .8cm, level distance = .8cm}]
		\node [rectangle, draw, minimum height=.6cm] {$\operation[c]$}
			child {node [rectangle, draw, minimum height=.6cm] {$\operation[d]$}
				child {node {$1$}}
				child {node {$2$}}
			}
			child {node {$3$}}
		;
	\end{tikzpicture}
	\!\!= 0
\]
for any~$\operation[a], \operation[b], \operation[c], \operation[d] \in \Operations_k$ such that~$\operation[b]_i = {\op{r}}$ for some~$i > |\operation[a]|_{\op{r}}$, and~$\operation[d]_j = {\op{r}}$ for some~$j > |\operation[c]|_{\op{l}}$.
Second, for any destination vector~${\ind{p} \in [3]^k}$ of arity~$3$, the tidy series $k$-signaletic operad~$\tidySignaleticSeries$ satisfies the quadratic relation
\[
	\begin{tikzpicture}[baseline=-.5cm, level/.style={sibling distance = .8cm, level distance = .8cm}]
		\node [rectangle, draw, minimum height=.6cm] {$\operation[a]_\ind{p}$}
			child {node {$1$}}
			child {node [rectangle, draw, minimum height=.6cm] {$\operation[b]_\ind{p}$}
				child {node {$2$}}
				child {node {$3$}}
			}
		;
	\end{tikzpicture}
	=
	\begin{tikzpicture}[baseline=-.5cm, level/.style={sibling distance = .8cm, level distance = .8cm}]
		\node [rectangle, draw, minimum height=.6cm] {$\operation[c]_\ind{p}$}
			child {node [rectangle, draw, minimum height=.6cm] {$\operation[d]_\ind{p}$}
				child {node {$1$}}
				child {node {$2$}}
			}
			child {node {$3$}}
		;
	\end{tikzpicture}
\]
where~$\operation[a]_\ind{p}, \operation[b]_\ind{p}, \operation[c]_\ind{p}, \operation[d]_\ind{p} \in \Operations_k$ are defined by
\begin{gather*}
	(\operation[a]_\ind{p})_i \eqdef \begin{cases} \op{l} & \text{if } \ind{p}_i = 1 \\ \op{r} & \text{if } \ind{p}_i = 2 \\ \op{r} & \text{if } \ind{p}_i = 3 \end{cases}
	\qquad
	(\operation[b]_\ind{p})_i \eqdef \begin{cases} \op{l} & \text{if } |\ind{p}^{\{2,3\}}| < i \\ \op{l} & \text{if } (\ind{p}^{\{2,3\}})_i = 2 \\ \op{r} & \text{if } (\ind{p}^{\{2,3\}})_i = 3 \end{cases}
	\\
	(\operation[c]_\ind{p})_i \eqdef \begin{cases} \op{l} & \text{if } \ind{p}_i = 1 \\ \op{l} & \text{if } \ind{p}_i = 2 \\ \op{r} & \text{if } \ind{p}_i = 3 \end{cases}
	\qquad
	(\operation[d]_\ind{p})_i \eqdef \begin{cases} \op{l} & \text{if } (\ind{p}^{\{1,2\}})_i = 1 \\ \op{r} & \text{if } (\ind{p}^{\{1,2\}})_i = 2 \\ \op{l} & \text{if } |\ind{p}^{\{1,2\}}| < i \end{cases}
\end{gather*}
for all~$i \in [k]$.
Again for any~$L \subseteq [3]$, we have denoted by~$\ind{p}^L$ the subword of~$\ind{p}$ consisting only of the letters which belong to~$L$, and for any~$x \in L$ the condition~$(\ind{p}^L)_i = x$ implicitly assume that~$\ind{p}^L$ has length at least~$i$.
Again, there are $2^{1+2k}-3^k$ independent tidy series $k$-signaletic relations among syntax trees of arity~$3$ on~$\Operations_k$.
\end{remark}

\begin{example}[Tidy series $0$-, $1$- and~$2$-signaletic relations]
The tidy series $0$-signaletic relation is the associative relation:
\allowdisplaybreaks
\begin{gather}
\compoR{n}{n} = \compoL{n}{n}. \tag{$\tidySignaleticSeries[]\,.$}\label{eq:tidySignaleticSeries0}
\end{gather}
In other words, the tidy series $0$-signaletic operad~$\tidySignaleticSeries[0]$ is just the associative operad~$\As$.

The tidy series $1$-signaletic relations are the~$5$ twisted dual duplicial relations:
\begin{gather*}
\compoR{l}{l} = \compoL{l}{l} \qandq \compoR{l}{r} = 0, \tag{$\tidySignaleticSeries[]\,1$}\label{eq:tidySignaleticSeries1} \\[-.1cm]
\compoR{r}{l} = \compoL{r}{l}, \tag{$\tidySignaleticSeries[]\,2$}\label{eq:tidySignaleticSeries2} \\[-.1cm]
\compoR{r}{r} = \compoL{l}{r} \qandq \compoL{r}{r} = 0. \tag{$\tidySignaleticSeries[]\,3$}\label{eq:tidySignaleticSeries3}
\end{gather*}
In other words, the tidy series $1$-signaletic operad~$\tidySignaleticSeries[1]$ is just the twisted dual duplicial operad~$\Dupdual_{\op{l}}$.

The tidy series $2$-signaletic relations are the following~$23$ relations:
\allowdisplaybreaks
\begin{gather*}
\compoR{l,l}{l,l} = \compoL{l,l}{l,l} \qandq \compoR{l,l}{l,r} = \compoR{l,l}{r,l} = \compoR{l,l}{r,r} = 0, \tag{$\tidySignaleticSeries[]\,11$}\label{eq:tidySignaleticSeries11} \\[-.1cm]
\compoR{l,r}{l,l} = \compoL{l,r}{l,l} \qandq \compoR{l,r}{l,r} = 0, \tag{$\tidySignaleticSeries[]\,12$}\label{eq:tidySignaleticSeries12} \\[-.1cm]
\compoR{l,r}{r,l} = \compoL{l,l}{l,r} \qandq \compoR{l,r}{r,r} = \compoL{l,r}{l,r} = 0, \tag{$\tidySignaleticSeries[]\,13$}\label{eq:tidySignaleticSeries13} \\[-.1cm]
\compoR{r,l}{l,l} = \compoL{r,l}{l,l} \qandq \compoR{r,l}{l,r} = 0, \tag{$\tidySignaleticSeries[]\,21$}\label{eq:tidySignaleticSeries21} \\[-.1cm]
\compoR{r,r}{l,l} = \compoL{r,r}{l,l}, \tag{$\tidySignaleticSeries[]\,22$}\label{eq:tidySignaleticSeries22} \\[-.1cm]
\compoR{r,r}{l,r} = \compoL{r,l}{l,r} \qandq \compoL{r,r}{l,r} = 0, \tag{$\tidySignaleticSeries[]\,23$}\label{eq:tidySignaleticSeries23} \\[-.1cm]
\compoR{r,l}{r,l} = \compoL{l,l}{r,l} \qandq \compoR{r,l}{r,r} = \compoL{l,r}{r,l} = 0, \tag{$\tidySignaleticSeries[]\,31$}\label{eq:tidySignaleticSeries31} \\[-.1cm]
\compoR{r,r}{r,l} = \compoL{r,l}{r,l} \qandq \compoL{r,r}{r,l} = 0, \tag{$\tidySignaleticSeries[]\,32$}\label{eq:tidySignaleticSeries32} \\[-.1cm]
\compoR{r,r}{r,r} = \compoL{l,l}{r,r} \qandq \compoL{l,r}{r,r} = \compoL{r,l}{r,r} = \compoL{r,r}{r,r} = 0. \tag{$\tidySignaleticSeries[]\,33$}\label{eq:tidySignaleticSeries33}
\end{gather*}
\end{example}

\begin{remark}
\label{rem:restrictionTidySignaleticSeries}
As in~\cref{rem:restrictionMessySignaleticSeries}, there is a surjective restriction morphism~${\Res^k_\ell : \tidySignaleticSeries[k] \mapsto \tidySignaleticSeries[l]}$.
\end{remark}

\begin{remark}
\label{rem:extensionTidySignaleticSeries}
Similarly to \cref{rem:extensionMessySignaleticSeries}, the maps~$\Ins^{\op{l}}_i$ and $\Ins^{{\op{l}}+{\op{r}}}_i$ are not operad morphisms $\tidySignaleticSeries[k] \mapsto \tidySignaleticSeries[k+1]$. For instance, the reader can check that when~$k = 1$, for any~${i \in \{1,2\}}$ and $\operation[a] \in \{{\op{l}}, {\op{l}}+{\op{r}}\}$, the morphism relation is not verified on ${\op{r}} \circ_1 {\op{l}}$.
\end{remark}


\subsection{Hilbert series, presentation and Koszulity}
\label{subsec:HilbertSeriesKoszulitySignaletic}

In this section, we give the Hilbert series and prove the quadratic presentation and Koszulity of all signaletic operads, using an approach which does not depend on whether the cars arrive in series or parallel, and whether the signals outside the routes of the cars are pointing or not to the left. We will therefore speak here about the signaletic operad and signaletic relations without further precision concerning series/parallel, nor messy/tidy. We therefore write~$\signaletic$ to denote without distinction any of the four $k$-signaletic operads~$\messySignaleticParallel$, $\tidySignaleticParallel$, $\messySignaleticSeries$, or~$\tidySignaleticSeries$.


\subsubsection{Hilbert series}
\label{subsubsec:HilberSeriesSignaletic}

By definition, the basis of~$\signaletic$ is given by syntax trees on~$\Operations_k$ modulo the $k$-signaletic equivalence, and the composition is given by grafting syntax trees. Alternatively, we can represent the basis of~$\signaletic$ using destination vectors: namely, ${\signaletic = \bigoplus_{p \ge 1} \signaletic(p)}$ where~${\signaletic(p) = \K \langle [p]^k \rangle}$ represent all possible destination vectors in a syntax tree on~$\Operations_k$ with arity~$p$, and the composition rule is described by the rules given in \cref{eq:messySignaleticParallelDestinations,eq:tidySignaleticParallelDestinations,eq:messySignaleticSeriesDestinations,eq:tidySignaleticSeriesDestinations}. Moreover, all destination vectors can clearly be obtained from a syntax tree. This shows the following statement.

\begin{corollary}
\label{coro:HilbertSignaletic}
The Hilbert series of the (messy or tidy, parallel or series) $k$-signaletic operad~$\signaletic$ is given by
\[
\Hilbert_{\signaletic}(t) = \sum_{p \ge 1} p^k t^p = \frac{t \cdot \EulPol{k}(t)}{(1-t)^{k+1}},
\]
where~$\EulPol{k}(t) \eqdef \sum\limits_{p = 0}^{k-1} \EulNum{k}{p} t^p$ and~$\EulNum{k}{p}$ is the number of permutations of~$\fS_k$ with precisely $p$ descents.
In particular, we have~$\Hilbert_{\signaletic[k+1]}(t) = t \Hilbert'_{\signaletic}(t)$.
\end{corollary}


\subsubsection{Signaletic Tamari order}
\label{subsubsec:signaleticTamariOrder}

In order to show that the $k$-signaletic operads are Koszul, we will need to orient the $k$-signaletic relations into a well-chosen rewriting system.
For this, we will be guided by the following natural generalization of the Tamari lattice of \cref{subsec:TamariLattice} to the syntax trees on~$\Operations_k$.

\begin{definition}
\label{def:signaleticTamariOrder}
The \defn{$n$-th $k$-signaletic Tamari order} is the partial order on the set~$\Syntax[\Operations_k](n)$ of syntax trees of arity $n$ defined by $\tree[s] \le \tree[t]$ if
\begin{itemize}
\item the shape of $\tree[s]$ is Tamari strictly smaller than the shape of $\tree[t]$
\item or $\tree[s]$ and $\tree[t]$ have the same shape and each letter in $\tree[s]$ is smaller than or equal to the corresponding letter in $\tree[t]$ for the order ${\op{r}} \le {\op{l}}$.
\end{itemize}
\end{definition}

In other words, the $k$-signaletic Tamari order is isomorphic to the lexicographic product of the Tamari order of \cref{def:TamariLattice} with the boolean lattice obtained as the $k(n-1)$-fold power of~${\op{r}} \le {\op{l}}$.

We will show in \cref{lem:maximumSignaleticTamari} that the following syntax trees are the maximums of the $k$-signaletic Tamari order in their signaletic equivalence classes.

\begin{definition}
\label{def:rightSignaleticComb}
A \defn{right $k$-signaletic comb} is a syntax tree on~$\Operations_k \eqdef \{\op{l}, \op{r}\}^k$ where the left child of each node is empty and each signal which is not visited by a car points to the left (\ie it is tidy).
\end{definition}

For instance, the quadratic right $k$-signaletic combs are
\medskip
\[
\compoR{l}{*} \qquad \compoR{r}{l} \qqandqq \compoR{r}{r}
\]
for the $1$-signaletic operads,
\medskip
\[
\compoR{l,l}{*,*} \; \compoR{l,r}{*,l} \; \compoR{l,r}{*,r} \;
\compoR{r,l}{l,*} \; \compoR{r,r}{l,l} \; \compoR{r,r}{l,r} \;
\compoR{r,l}{r,*} \; \compoR{r,r}{r,l} \; \compoR{r,r}{r,r}
\]
for the parallel $2$-signaletic operads, and
\medskip
\[
\compoR{l,l}{*,*} \; \compoR{l,r}{l,*} \; \compoR{l,r}{r,*} \;
\compoR{r,l}{l,*} \; \compoR{r,r}{l,l} \; \compoR{r,r}{l,r} \;
\compoR{r,l}{r,*} \; \compoR{r,r}{r,l} \; \compoR{r,r}{r,r}
\]
for the series $2$-signaletic operads.
The unvisited signals are colored in gray, and they all point to the left.

Note that the right $k$-signaletic combs are distinct in the series or parallel signaletic operad:
\medskip
\[ 
\compoR{l,r}{l,r} \ne \compoR{l,r}{r,l}
\]
are respectively the parallel and series right $2$-signaletic combs with destination vector~$\destVect{1,3}{3}$. However, the right $k$-signaletic combs are counted in both cases by the number~$p^k$ of possible signaletic destination vectors, as shown by the following statement.

\begin{lemma}
\label{lem:bijectionCombsDestinations}
The map sending the right $k$-signaletic combs to their destination vectors is~\mbox{bijective}.
\end{lemma}

\begin{proof}
For a given destination vector, the corresponding right $k$-signaletic comb is obtained by letting the $k$ cars traverse the right comb in such a way that they reach their planned destination and finally complete the unvisited signals with~$\op{l}$.
\end{proof}

For example, the following picture shows respectively the series and parallel right $k$-signaletic combs corresponding to the destination vector~$\destVect{1,4}{6}$ and the destination vector~$\destVect{1,5,3}{6}$:
\medskip
\[
	\begin{array}{r@{\quad}r@{\qquad\qquad}r@{\quad}r}
		\begin{tikzpicture}[baseline=-1.2cm, level distance=7mm, sibling distance=8mm]
			\node [rectangle, draw] {$\op{l}\op{r}$}
				child {node {$1$}}
				child {node [rectangle, draw] {$\op{*}\op{r}$}
					child {node {\phantom{1}}}
					child {node [rectangle, draw] {$\op{*}\op{r}$}
						child {node {\phantom{1}}}
						child {node [rectangle, draw] {$\op{*}\op{l}$}
							child {node {$4$}}
							child {node [rectangle, draw] {$\op{*}\op{*}$}
								child {node {\phantom{1}}}
								child {node {\phantom{1}}}
							}
						}
					}
				}
			;
		\end{tikzpicture}
		&
		\begin{tikzpicture}[baseline=-1.2cm, level distance=7mm, sibling distance=8mm]
			\node [rectangle, draw] {$\op{l}\op{r}$}
				child {node {$1$}}
				child {node [rectangle, draw] {$\op{r}\op{*}$}
					child {node {\phantom{1}}}
					child {node [rectangle, draw] {$\op{r}\op{*}$}
						child {node {\phantom{1}}}
						child {node [rectangle, draw] {$\op{l}\op{*}$}
							child {node {$4$}}
							child {node [rectangle, draw] {$\op{*}\op{*}$}
								child {node {\phantom{1}}}
								child {node {\phantom{1}}}
							}
						}
					}
				}
			;
		\end{tikzpicture}
		&
		\begin{tikzpicture}[baseline=-1.2cm, level distance=7mm, sibling distance=8mm]
			\node [rectangle, draw] {$\op{l}\op{r}\op{r}$}
				child {node {$1$}}
				child {node [rectangle, draw] {$\op{*}\op{r}\op{r}$}
					child {node {\phantom{1}}}
					child {node [rectangle, draw] {$\op{*}\op{r}\op{l}$}
						child {node {$3$}}
						child {node [rectangle, draw] {$\op{*}\op{r}\op{*}$}
							child {node {\phantom{1}}}
							child {node [rectangle, draw] {$\op{*}\op{l}\op{*}$}
								child {node {$5$}}
								child {node {\phantom{1}}}
							}
						}
					}
				}
			;
		\end{tikzpicture}
		&
		\begin{tikzpicture}[baseline=-1.2cm, level distance=7mm, sibling distance=8mm]
			\node [rectangle, draw] {$\op{l}\op{r}\op{r}$}
				child {node {$1$}}
				child {node [rectangle, draw] {$\op{r}\op{r}\op{*}$}
					child {node {\phantom{1}}}
					child {node [rectangle, draw] {$\op{r}\op{l}\op{*}$}
						child {node {$3$}}
						child {node [rectangle, draw] {$\op{r}\op{*}\op{*}$}
							child {node {\phantom{1}}}
							child {node [rectangle, draw] {$\op{l}\op{*}\op{*}$}
								child {node {$5$}}
								child {node {\phantom{1}}}
							}
						}
					}
				}
			;
		\end{tikzpicture}
		\\
		\text{parallel} & \text{series} & \text{parallel} & \text{series}
	\end{array}
\]
Again, the unvisited signals are colored in gray, and they all point to the left.

\begin{lemma}
\label{lem:maximumSignaleticTamari}
In each $k$-signaletic congruence class, the right $k$-signaletic comb is the unique $k$-signaletic Tamari maximum.
\end{lemma}

\begin{proof}
Consider a syntax tree~$\tree$ and a right $k$-signaletic comb~$\tree[c]$ with the same arity and destination vector.
If~$\tree$ is not a right comb, then the shape of~$\tree$ is Tamari smaller than the shape of~$\tree[c]$, so that~$\tree$ is $k$-signaletic Tamari smaller than~$\tree[c]$.
Otherwise, since~$\tree$ and~$\tree[c]$ have the same destination vector, the signals visited by the cars coincide in the two syntax trees.
As all unvisited signals in~$\tree[c]$ point to the left, and~${\op{r}} \le {\op{l}}$, we obtain that~$\tree$ is $k$-signaletic Tamari smaller than~$\tree[c]$.
\end{proof}


\subsubsection{Rewriting system}
\label{subsubsec:rewritingSystemSignaletic}

We now orient the $k$-signaletic relations of \cref{def:messySignaleticRelationsParallel,def:tidySignaleticRelationsParallel,def:messySignaleticRelationsSeries,def:tidySignaleticRelationsSeries} according to the signaletic Tamari order of \cref{subsubsec:signaleticTamariOrder}.

\begin{definition}
The \defn{$k$-signaletic rewriting system} is the set of rewriting rules~$(\tree[s],\tree[c])$ where~$\tree[s]$ and~$\tree[c]$ are two distinct quadratic trees with the same destination vector and~$\tree[c]$ is a right $k$-signaletic comb.
\end{definition}

\begin{lemma}
\label{lem:rewritingIncreasesSignaleticTamari}
Given any rewriting $\tree[t] \to \tree[p]$, we have $\tree[t] < \tree[p]$ in the signaletic Tamari order.
\end{lemma}

\begin{proof}
Call $(\tree[s],\tree[q])$ the rule that was applied. We distinguish two cases:
\begin{itemize}
\item If~$\tree[s]$ and~$\tree[q]$ have distinct shapes then $\tree[s]$ and $\tree[q]$ are left (resp. right) quadratic trees. Consequently, the shape of~$\tree[p]$ is obtained from the shape of~$\tree[t]$ by a right Tamari rotation so that~$\tree[t] < \tree[p]$.
\item If $\tree[s]$ and $\tree[q]$ have the same shape then they are both right quadratic trees and the rewriting changed some $\op{r}$ to $\op{l}$ in the bottom node. Again~$\tree[t] < \tree[p]$.
\qedhere
\end{itemize}
\end{proof}

\begin{corollary}
\label{coro:rewritingTermination}
The $k$-signaletic rewriting system terminates.
\end{corollary}

\begin{proof}
The $k$-signaletic rewriting system is strictly increasing for a finite order.
\end{proof}


\subsubsection{Normal forms}
\label{subsubsec:normalFormSignaletic}

According to our choice of rewriting system, the quadratic normal forms are precisely the quadratic right $k$-signaletic combs.
We aim to prove that this property holds for any arity.

\begin{lemma}
\label{lem:signaleticCombs}
Any syntax tree on~$\Operations_k \eqdef \{\op{l}, \op{r}\}^k$ can be rewritten to a right $k$-signaletic comb by the $k$-signaletic rewriting system.
\end{lemma}

\begin{proof}
First of all, if a node has an edge to a non-leaf left child, then there is a rewriting rule which rotates this edge. Therefore, any syntax tree can be rewritten to a syntax tree with the shape of a right comb. As a consequence, it is sufficient to show that any syntax tree $\tree$ whose shape is a right comb can be ultimately rewritten to the unique right $k$-signaletic comb with the same destination vector as $\tree$. The solution is to start from the root of the right comb, and to perform the rewritings from top to bottom. This replaces all the unvisited signs by left signs leading to the right $k$-signaletic comb.
\end{proof}

\begin{corollary}
\label{coro:signaleticNormalForms}
The right $k$-signaletic combs are precisely the normal forms of the $k$-signaletic quadratic rewriting system.
\end{corollary}

\begin{proof}
This follows from \cref{lem:maximumSignaleticTamari,lem:rewritingIncreasesSignaleticTamari,lem:signaleticCombs}.
\end{proof}


\subsubsection{Presentation and Koszulity}
\label{subsubsec:koszulitySignaletic}

We conclude this section with the presentation and Koszulity of the signaletic operads.

\begin{theorem}
\label{thm:signaleticOperadsQuadratic}
The $k$-signaletic operad~$\signaletic$ is quadratic and Koszul. In particular, the $k$-signaletic relations give a presentation of~$\signaletic$ and the right $k$-signaletic combs form a Poincar\'e\,--\,Birkhoff\,--\,Witt basis of~$\signaletic$.
\end{theorem}

\begin{proof}
By \cref{coro:signaleticNormalForms}, the right $k$-signaletic combs are the normal forms of the $k$-signaletic quadratic rewriting system. Since the number of right $k$-signaletic combs of arity~$p$ is the dimension of the homogeneous component of degree~$p$ of the $k$-signaletic operad, we obtain that the $k$-signaletic rewriting system is convergent. This shows that the $k$-signaletic operads are quadratic and Koszul by definition, and that the right $k$-signaletic combs form a Poincar\'e\,--\,Birkhoff\,--\,Witt basis.
\end{proof}

\begin{remark}
The result of \cref{thm:signaleticOperadsQuadratic} was already partially established for messy $k$-signaletic operads: see~\cite{Loday-dialgebras} for the diassociative operad~$\Dias$ (\ie the $1$-signaletic operad), \cite{Foissy, Vatne} for the operad~$\Quad^\koszul$ (\ie the messy parallel $2$-signaletic operad, see \cref{rem:messyCitelangisParallelOperadManin}), and \cite{Vatne} for higher black Manin products of the diassociative operad~$\Dias$ (\ie the messy parallel $k$-signaletic operads, see \cref{subsec:ManinProducts,prop:messyParallelSignaleticManinProduct}).
\end{remark}

\begin{remark}
Note that in the tidy situation, we actually used a slight extension of the notion of set-operad.
Namely, we included a zero element~$\operation[o]_d$ in~$\Operad(d)$ for each degree~$d$ such that any composition involving some~$\operation[o]_d$ results in some $\operation[o]_{d'}$.
Then the Koszulity of the set-operad is equivalent to the Koszulity of the linearized operad when $\operation[o]_d$ is actually the zero of the vector space~$\Operad(d)$.
\medskip
\end{remark}


\newpage
\section{Citelangis operads}
\label{sec:citelangisOperads}

We now introduce the citelangis operads, which are defined as the Koszul duals of the signaletic operads presented in \cref{sec:signaleticOperads}.
In this section, we describe their presentation in terms of generators and relations and discuss properties of their Hilbert series.
We show in particular that some citelangis operads where already considered in the literature: the messy parallel $2$-citelangis algebras are the quadrialgebras of~\cite{AguiarLoday, Foissy} and the messy series $k$-citelangis algebras are the $k$-twistiform algebras of~\cite{Pilaud-brickAlgebra}. 
Combinatorial models and actions of the citelangis operads will be discussed later in \cref{sec:actions}.


\subsection{Parallel citelangis operads}
\label{subsec:parallelCitelangisOperads}

We start with the Koszul duals of the parallel signaletic operads, and we treat separately the messy and tidy situations.


\subsubsection{Messy parallel citelangis operads}
\label{subsubsec:messyParallelCitelangisOperads}

We start with the messy setting.

\begin{definition}
The \defn{messy parallel $k$-citelangis operad}~$\messyCitelangisParallel$ is the Koszul dual of the messy parallel $k$-signaletic operad:
\[
\messyCitelangisParallel \eqdef (\messySignaleticParallel)^\koszul
\]
\end{definition}

From \cref{prop:messyParallelSignaleticManinProduct,prop:dendDiass}\,(ii), we derive the following statement.

\begin{proposition}
\label{prop:messyParallelCitelangisManinProduct}
For any two integers $k$ and $l$, we have
\[
\messyCitelangisParallel = \Dend^{\whiteManin k} = \Dend^{\blackManin k}
\qqandqq
\messyCitelangisParallel[k+l] = \messyCitelangisParallel[k]\blackManin\messyCitelangisParallel[l].
\]
\end{proposition}

\begin{remark}
\label{rem:messyCitelangisParallelOperadManin}
Besides the well-known associative and dendriform algebras, some special cases of messy parallel $k$-citelangis algebras were specifically studied in the literature:
\begin{itemize}
\item for~$k = 2$, the messy parallel $2$-citelangis algebras are known as quadri-algebras, introduced by M.~Aguiar and J.-L.~Loday in~\cite{AguiarLoday} and studied in~\cite{Vatne, Foissy}.
\item for~$k = 3$, the messy parallel $3$-citelangis algebras are the octo-algebras of~\cite{Leroux-octo}.
\end{itemize}
\end{remark}

By \cref{thm:signaleticOperadsQuadratic}, the messy parallel $k$-citelangis operad~$\messyCitelangisParallel$ is quadratic and Koszul.
To describe its quadratic relations, it is convenient to consider sums of operations like~${\op{m}} \eqdef {\op{l}} + {\op{r}}$.

\begin{proposition}
\label{prop:messyCitelangisParallelRelations}
Consider the $3^k$ operations~$\{\op{l}, \op{m}, \op{r}\}^k$ related by
\[
\operation[b]\!\op{m}\!\operation[b'] = \operation[b]\!\op{l}\!\operation[b'] + \operation[b]\!\op{r}\!\operation[b']
\]
for all~$\operation[b],\operation[b'] \in \{\op{l}, \op{m}, \op{r}\}^*$ with~$|\operation[b]| + |\operation[b']| = k-1$. 

For any destination vector~$\ind{p} \in [3]^k$ of arity~$3$, the messy parallel $k$-citelangis operad~$\messyCitelangisParallel$ satisfies the quadratic relation
\begin{equation} \tag{$\messyCitelangisParallel[]\,\ind{p}$}\label{eq:messyCitelangisParallelp}
	\begin{tikzpicture}[baseline=-.5cm, level/.style={sibling distance = .8cm, level distance = .8cm}]
		\node [rectangle, draw, minimum height=.6cm] {$\operation[a]_\ind{p}$}
			child {node {$1$}}
			child {node [rectangle, draw, minimum height=.6cm] {$\operation[b]_\ind{p}$}
				child {node {$2$}}
				child {node {$3$}}
			}
		;
	\end{tikzpicture}
=
	\begin{tikzpicture}[baseline=-.5cm, level/.style={sibling distance = .8cm, level distance = .8cm}]
		\node [rectangle, draw, minimum height=.6cm] {$\operation[c]_\ind{p}$}
			child {node [rectangle, draw, minimum height=.6cm] {$\operation[d]_\ind{p}$}
				child {node {$1$}}
				child {node {$2$}}
			}
			child {node {$3$}}
		;
	\end{tikzpicture}
\end{equation}
where the operations~$\operation[a]_\ind{p}, \operation[b]_\ind{p}, \operation[c]_\ind{p}, \operation[d]_\ind{p} \in \{\op{l}, \op{m}, \op{r}\}^k$ are defined by
\[
	(\operation[a]_\ind{p})_i \eqdef \begin{cases} \op{l} & \!\!\text{if } \ind{p}_i = 1 \\ \op{r} & \!\!\text{if } \ind{p}_i = 2 \\ \op{r} & \!\!\text{if } \ind{p}_i = 3 \end{cases}
	\quad
	(\operation[b]_\ind{p})_i \eqdef \begin{cases} \op{m} & \!\!\text{if } \ind{p}_i = 1 \\ \op{l} & \!\!\text{if } \ind{p}_i = 2 \\ \op{r} & \!\!\text{if } \ind{p}_i = 3 \end{cases}
	\quad
	(\operation[c]_\ind{p})_i \eqdef \begin{cases} \op{l} & \!\!\text{if } \ind{p}_i = 1 \\ \op{l} & \!\!\text{if } \ind{p}_i = 2 \\ \op{r} & \!\!\text{if } \ind{p}_i = 3 \end{cases}
	\quad
	(\operation[d]_\ind{p})_i \eqdef \begin{cases} \op{l} & \!\!\text{if } \ind{p}_i = 1 \\ \op{r} & \!\!\text{if } \ind{p}_i = 2 \\ \op{m} & \!\!\text{if } \ind{p}_i = 3 \end{cases} 
\]
for any~$i \in [k]$.
Note that \cref{eq:messyCitelangisParallelp} involves $2^{|\ind{p}|_1}$ syntax trees of~$\Syntax[\Operations_k]$ on the left hand side and $2^{|\ind{p}|_3}$ syntax trees of~$\Syntax[\Operations_k]$ on the right hand side.
We call these relations the \defn{messy parallel $k$-citelangis relations}.
\end{proposition}

\begin{proof}
This follows from~\cref{lem:koszulDual,rem:messySignaleticRelationsParallel}.
\end{proof}

\begin{example}[Messy parallel $0$-, $1$- and~$2$-citelangis relations]
\label{exm:messyCitelangisParallelRelations}
The messy parallel $0$-citelangis relation is the associative relation:
\allowdisplaybreaks
\begin{gather}
\compoR{n}{n} = \compoL{n}{n}. \tag{$\messyCitelangisParallel[]\,.$}\label{eq:messyCitelangisParallel0}
\end{gather}
In other words, the messy parallel $0$-citelangis operad~$\messyCitelangisParallel[0]$ is just the associative operad~$\As$.

The messy parallel $1$-citelangis relations are the $3$ dendriform relations:

\vspace{.4cm}
\centerline{
\begin{tabular}{c@{\qquad}c@{\qquad}c}
\(\compoR{l}{m} \!\!\overset{\displaystyle\text{($\messyCitelangisParallel[]\,1$)\label{eq:messyCitelangisParallel1}}}{=}\!\! \compoL{l}{l}\) &
\(\compoR{r}{l} \!\!\overset{\displaystyle\text{($\messyCitelangisParallel[]\,2$)\label{eq:messyCitelangisParallel2}}}{=}\!\! \compoL{r}{l}\) &
\(\compoR{r}{r} \!\!\overset{\displaystyle\text{($\messyCitelangisParallel[]\,3$)\label{eq:messyCitelangisParallel3}}}{=}\!\! \compoL{m}{r}\)
\end{tabular}
}

\noindent
where
\[
{\op{m}} = {\op{l}} + {\op{r}}.
\]
In other words, the messy parallel $1$-citelangis operad~$\messyCitelangisParallel[1]$ is just the dendriform operad~$\Dend$.

The messy parallel $2$-citelangis relations are the following $9$ relations:

\vspace{.4cm}
\centerline{
\begin{tabular}{c@{\qquad}c@{\qquad}c}
\(\compoR{l,l}{m,m} \!\!\overset{\displaystyle\text{($\messyCitelangisParallel[]\,11$)\label{eq:messyCitelangisParallel11}}}{=}\!\! \compoL{l,l}{l,l}\) &
\(\compoR{l,r}{m,l} \!\!\overset{\displaystyle\text{($\messyCitelangisParallel[]\,12$)\label{eq:messyCitelangisParallel12}}}{=}\!\! \compoL{l,r}{l,l}\) &
\(\compoR{l,r}{m,r} \!\!\overset{\displaystyle\text{($\messyCitelangisParallel[]\,13$)\label{eq:messyCitelangisParallel13}}}{=}\!\! \compoL{l,m}{l,r}\) \\[.8cm]
\(\compoR{r,l}{l,m} \!\!\overset{\displaystyle\text{($\messyCitelangisParallel[]\,21$)\label{eq:messyCitelangisParallel21}}}{=}\!\! \compoL{r,l}{l,l}\) &
\(\compoR{r,r}{l,l} \!\!\overset{\displaystyle\text{($\messyCitelangisParallel[]\,22$)\label{eq:messyCitelangisParallel22}}}{=}\!\! \compoL{r,r}{l,l}\) &
\(\compoR{r,r}{l,r} \!\!\overset{\displaystyle\text{($\messyCitelangisParallel[]\,23$)\label{eq:messyCitelangisParallel23}}}{=}\!\! \compoL{r,m}{l,r}\) \\[.8cm]
\(\compoR{r,l}{r,m} \!\!\overset{\displaystyle\text{($\messyCitelangisParallel[]\,31$)\label{eq:messyCitelangisParallel31}}}{=}\!\! \compoL{m,l}{r,l}\) &
\(\compoR{r,r}{r,l} \!\!\overset{\displaystyle\text{($\messyCitelangisParallel[]\,32$)\label{eq:messyCitelangisParallel32}}}{=}\!\! \compoL{m,r}{r,l}\) &
\(\compoR{r,r}{r,r} \!\!\overset{\displaystyle\text{($\messyCitelangisParallel[]\,33$)\label{eq:messyCitelangisParallel33}}}{=}\!\! \compoL{m,m}{r,r}\)
\end{tabular}
}

\noindent
where
\begin{gather*}
{\op{m,l}} \; = \; {\op{l,l}} + {\op{r,l}}, \qquad
{\op{m,r}} \; = \; {\op{l,r}} + {\op{r,r}}, \\
{\op{l,m}} \; = \; {\op{l,l}} + {\op{l,r}}, \qquad
{\op{r,m}} \; = \; {\op{r,l}} + {\op{r,r}}, \\
\text{and}\qquad {\op{m,m}} \; = \; {\op{l,m}} + {\op{r,m}} \; = \; {\op{m,l}} + {\op{m,r}} \; = \; {\op{l,l}} + {\op{l,r}} + {\op{r,l}} + {\op{r,r}}.
\end{gather*}
These relations are those of the $\Quad$ operad, see for instance~\cite[Definition~1]{Foissy}.
The translation between our notations and that of~\cite[Definition~1]{Foissy} is the following:
\[
{\op{l,l}} = \; \nwarrow,
\qquad
{\op{l,r}} = \; \nearrow,
\qquad
{\op{r,l}} = \; \swarrow,
\qqandqq
{\op{r,r}} = \; \searrow.
\]
\end{example}

\begin{remark}
\label{rem:restrictionMessyCitelangisParallel}
For any subset $I=\{i_1 < \dots < i_\ell\}$ of $[k]$, the dual of the morphism~$\Res_I$ defined in \cref{rem:restrictionMessySignaleticParallel} is an injective operad morphism ${\Res_I^\koszul : \messyCitelangisParallel[\ell] \to \messyCitelangisParallel[k]}$. The image~$\operation[b] \eqdef \Res_I^\koszul(\operation[a])$ of an operation~$\operation[a]$ of $\messyCitelangisParallel[\ell]$ verifies $\operation[b]_{i_p} = \operation[a]_p$ if~$p \in [\ell]$ and $\operation[b]_j = {\op{m}}$ if $j \in [k] \ssm I$.
Hence, for~$\ell \le k$, there are $\binom{k}{\ell}$ natural messy parallel $\ell$-citelangis structures in a messy parallel $k$-citelangis algebra.

In particular if $\ell=1$, a messy parallel $k$-citelangis algebra~$(\algebra, \Operations_k)$ gives $k$ different dendriform algebras~$(\algebra, \op{l}\!\!_p, \op{r}\!\!_p)$ for the operations ${\op{l}\!\!_p \eqdef \op{m}\!\!^{p-1}\!\!\op{l}\op{m}\!\!^{k-p}}$ and ${\op{r}\!\!_p \eqdef \op{m}\!\!^{p-1}\!\!\op{r}\op{m}\!\!^{k-p}}$.

Another interesting particular case is $\ell = 0$ and therefore $I = \emptyset$. One then gets a morphism of operad ${\Res_\emptyset^\koszul : \As \to \messyCitelangisParallel[k]}$. In terms of algebra, this means that a messy parallel $k$-citelangis algebra~$(\algebra, \Operations_k)$ defines a structure of associative algebra~$(\algebra, {\op{m}\!\!^k})$. This can be seen by adding up all messy parallel $k$-citelangis relations of \cref{prop:messyCitelangisParallelRelations}, obtaining
\[
	\begin{tikzpicture}[baseline=-.5cm, level 1/.style={sibling distance = .8cm, level distance = .8cm}, level 2/.style={sibling distance = .6cm, level distance = .6cm}]
		\node [rectangle, draw] {$\op{m}\!\!^k$}
			child {node {}}
			child {node [rectangle, draw] {$\op{m}\!\!^k$}
				child {node {}}
				child {node {}}
			}
		;
	\end{tikzpicture}
	=
	\begin{tikzpicture}[baseline=-.5cm, level 1/.style={sibling distance = .8cm, level distance = .8cm}, level 2/.style={sibling distance = .6cm, level distance = .6cm}]
		\node [rectangle, draw] {$\op{m}\!\!^k$}
			child {node [rectangle, draw] {$\op{m}\!\!^k$}
				child {node {}}
				child {node {}}
			}
			child {node {}}
		;
	\end{tikzpicture}
\]
Reciprocally, we say that an associative algebra~$(\algebra, \product)$ admits a \defn{messy parallel $k$-citelangis structure} if it is possible to split the product~$\product$ into~$2^k$ operations~$\Operations_k$ defining a messy parallel $k$-citelangis algebra on~$\algebra$.
\end{remark}

\begin{remark}
The dual of the morphism $\Ins^{\operation[a]}_i$ where $\operation[a] \in \{\op{l}, \op{r}\}$ defined in \cref{rem:extensionMessySignaleticParallel} is a surjective morphism~$(\Ins^{\operation[a]}_i)^\koszul : \messyCitelangisParallel[k+1] \to \messyCitelangisParallel[k]$.
It kills the operations whose $i$-th entry is not an~$\operation[a]$, and deletes the $i$-th entry in the other operations.
\end{remark}


\subsubsection{Tidy parallel citelangis operads}
\label{subsubsec:tidyParallelCitelangisOperads}

We now switch to the tidy setting.

\begin{definition}
The \defn{tidy parallel $k$-citelangis operad}~$\tidyCitelangisParallel$ is the Koszul dual of the tidy parallel $k$-signaletic operad:
\[
\tidyCitelangisParallel \eqdef (\tidySignaleticParallel)^\koszul
\]
\end{definition}

\begin{proposition}
\label{prop:tidyParallelCitelangisManinProduct}
For any two integers $k$ and $l$, we have
\[
\tidyCitelangisParallel = {\Dup_{\op{l}}}^{\blackManin k}
\qqandqq
\tidyCitelangisParallel[k+l] = \tidyCitelangisParallel[k]\blackManin\tidyCitelangisParallel[l].
\]
\end{proposition}

By \cref{thm:signaleticOperadsQuadratic}, the tidy parallel $k$-citelangis operad~$\tidyCitelangisParallel$ is quadratic and Koszul.
We now describe its quadratic relations.

\begin{proposition}
\label{prop:tidyCitelangisParallelRelations}
For any destination vector~$\ind{p} \in [3]^k$ of arity~$3$, the tidy parallel $k$-citelangis operad~$\tidyCitelangisParallel$ satisfies the quadratic relation
\begin{equation} \tag{$\tidyCitelangisParallel[]\,\ind{p}$}\label{eq:tidyCitelangisParallelp}
	\begin{tikzpicture}[baseline=-.5cm, level/.style={sibling distance = .8cm, level distance = .8cm}]
		\node [rectangle, draw, minimum height=.6cm] {$\operation[a]_\ind{p}$}
			child {node {$1$}}
			child {node [rectangle, draw, minimum height=.6cm] {$\operation[b]_\ind{p}$}
				child {node {$2$}}
				child {node {$3$}}
			}
		;
	\end{tikzpicture}
=
	\begin{tikzpicture}[baseline=-.5cm, level/.style={sibling distance = .8cm, level distance = .8cm}]
		\node [rectangle, draw, minimum height=.6cm] {$\operation[c]_\ind{p}$}
			child {node [rectangle, draw, minimum height=.6cm] {$\operation[d]_\ind{p}$}
				child {node {$1$}}
				child {node {$2$}}
			}
			child {node {$3$}}
		;
	\end{tikzpicture}
\end{equation}
where the operations~$\operation[a]_\ind{p}, \operation[b]_\ind{p}, \operation[c]_\ind{p}, \operation[d]_\ind{p} \in \{\op{l}, \op{m}, \op{r}\}^k$ are defined by
\[
	(\operation[a]_\ind{p})_i \eqdef \begin{cases} \op{l} & \!\!\text{if } \ind{p}_i = 1 \\ \op{r} & \!\!\text{if } \ind{p}_i = 2 \\ \op{r} & \!\!\text{if } \ind{p}_i = 3 \end{cases}
	\quad
	(\operation[b]_\ind{p})_i \eqdef \begin{cases} \op{l} & \!\!\text{if } \ind{p}_i = 1 \\ \op{l} & \!\!\text{if } \ind{p}_i = 2 \\ \op{r} & \!\!\text{if } \ind{p}_i = 3 \end{cases}
	\quad
	(\operation[c]_\ind{p})_i \eqdef \begin{cases} \op{l} & \!\!\text{if } \ind{p}_i = 1 \\ \op{l} & \!\!\text{if } \ind{p}_i = 2 \\ \op{r} & \!\!\text{if } \ind{p}_i = 3 \end{cases}
	\quad
	(\operation[d]_\ind{p})_i \eqdef \begin{cases} \op{l} & \!\!\text{if } \ind{p}_i = 1 \\ \op{r} & \!\!\text{if } \ind{p}_i = 2 \\ \op{l} & \!\!\text{if } \ind{p}_i = 3 \end{cases}
\]
for any~$i \in [k]$.
We call these relations the \defn{tidy parallel $k$-citelangis relations}.
\end{proposition}

\begin{proof}
This follows from~\cref{lem:koszulDual,rem:tidySignaleticRelationsParallel}.
\end{proof}

\begin{example}[Tidy parallel $0$-, $1$- and~$2$-citelangis relations]
\label{exm:tidyCitelangisParallelRelations}
The tidy parallel $0$-citelangis relation is the associative relation:
\allowdisplaybreaks
\begin{gather}
\compoR{n}{n} = \compoL{n}{n}. \tag{$\tidyCitelangisParallel[]\,.$}\label{eq:tidyCitelangisParallel0}
\end{gather}
In other words, the tidy parallel $0$-citelangis operad~$\tidyCitelangisParallel[0]$ is just the associative operad~$\As$.

The tidy parallel $1$-citelangis relations are the $3$ twisted duplicial relations:

\vspace{.4cm}
\centerline{
\begin{tabular}{c@{\qquad}c@{\qquad}c}
\(\compoR{l}{l} \!\!\overset{\displaystyle\text{($\tidyCitelangisParallel[]\,1$)\label{eq:tidyCitelangisParallel1}}}{=}\!\! \compoL{l}{l}\) &
\(\compoR{r}{l} \!\!\overset{\displaystyle\text{($\tidyCitelangisParallel[]\,2$)\label{eq:tidyCitelangisParallel2}}}{=}\!\! \compoL{r}{l}\) &
\(\compoR{r}{r} \!\!\overset{\displaystyle\text{($\tidyCitelangisParallel[]\,3$)\label{eq:tidyCitelangisParallel3}}}{=}\!\! \compoL{l}{r}\)
\end{tabular}
}

\noindent
In other words, the tidy parallel $1$-citelangis operad~$\tidyCitelangisParallel[1]$ is just the twisted duplicial operad~$\Dup_{\op{l}}$.

\pagebreak
The tidy parallel $2$-citelangis relations are the following $9$ relations:

\vspace{.4cm}
\centerline{
\begin{tabular}{c@{\qquad}c@{\qquad}c}
\(\compoR{l,l}{l,l} \!\!\overset{\displaystyle\text{($\tidyCitelangisParallel[]\,11$)\label{eq:tidyCitelangisParallel11}}}{=}\!\! \compoL{l,l}{l,l}\) &
\(\compoR{l,r}{l,l} \!\!\overset{\displaystyle\text{($\tidyCitelangisParallel[]\,12$)\label{eq:tidyCitelangisParallel12}}}{=}\!\! \compoL{l,r}{l,l}\) &
\(\compoR{l,r}{l,r} \!\!\overset{\displaystyle\text{($\tidyCitelangisParallel[]\,13$)\label{eq:tidyCitelangisParallel13}}}{=}\!\! \compoL{l,l}{l,r}\) \\[.8cm]
\(\compoR{r,l}{l,l} \!\!\overset{\displaystyle\text{($\tidyCitelangisParallel[]\,21$)\label{eq:tidyCitelangisParallel21}}}{=}\!\! \compoL{r,l}{l,l}\) &
\(\compoR{r,r}{l,l} \!\!\overset{\displaystyle\text{($\tidyCitelangisParallel[]\,22$)\label{eq:tidyCitelangisParallel22}}}{=}\!\! \compoL{r,r}{l,l}\) &
\(\compoR{r,r}{l,r} \!\!\overset{\displaystyle\text{($\tidyCitelangisParallel[]\,23$)\label{eq:tidyCitelangisParallel23}}}{=}\!\! \compoL{r,l}{l,r}\) \\[.8cm]
\(\compoR{r,l}{r,l} \!\!\overset{\displaystyle\text{($\tidyCitelangisParallel[]\,31$)\label{eq:tidyCitelangisParallel31}}}{=}\!\! \compoL{l,l}{r,l}\) &
\(\compoR{r,r}{r,l} \!\!\overset{\displaystyle\text{($\tidyCitelangisParallel[]\,32$)\label{eq:tidyCitelangisParallel32}}}{=}\!\! \compoL{l,r}{r,l}\) &
\(\compoR{r,r}{r,r} \!\!\overset{\displaystyle\text{($\tidyCitelangisParallel[]\,33$)\label{eq:tidyCitelangisParallel33}}}{=}\!\! \compoL{l,l}{r,r}\)
\end{tabular}
}
\end{example}

\begin{remark}
\label{rem:restrictionTidyCitelangisParallel}
Similarly to \cref{rem:restrictionMessyCitelangisParallel}, for any subset $I$ of $[k]$ of cardinality~$\ell \le k$, the dual morphism of~$\Res_I$ defined in \cref{rem:restrictionTidySignaleticParallel} is an injective operad morphism ${\Res_I^\koszul : \tidyCitelangisParallel[\ell] \to \tidyCitelangisParallel[k]}$.
The image~$\operation[b] \eqdef \Res_I^\koszul(\operation[a])$ of an operation~$\operation[a]$ of $\tidyCitelangisParallel[\ell]$ verifies $\operation[b]_{i_p} = \operation[a]_p$ if~$p \in [\ell]$ and $\operation[b]_j = {\op{l}}$ if $j \in [k] \ssm I$.
\end{remark}

\begin{remark}
The dual of the morphism $\Ins^{\operation[a]}_i$ where $\operation[a] \in \{\op{l}, {\op{l}} + {\op{r}}\}$ defined in \cref{rem:extensionTidySignaleticParallel} is a surjective morphism~$(\Ins^{\operation[a]}_i)^\koszul : \tidyCitelangisParallel[k+1] \to \tidyCitelangisParallel[k]$.
The map~$(\Ins^{\op{l}}_i)^\koszul$ kills the operations whose $i$-th entry is not an~$\op{l}$, and deletes the $i$-th entry in the other operations.
The map~$(\Ins^{{\op{l}} + {\op{r}}}_i)^\koszul$ just deletes the $i$-th entry in all the operations.
\end{remark}

\begin{remark}
\label{rem:multitidyCitelangisParallel}
Following \cref{rem:multitidySignaleticParallel}, for any constraint word~$\constraint \in \{\op{l}, \odot \op{r}\}^k$, we call \defn{$\constraint$-tidy parallel citelangis operad}~$\tidyCitelangisParallel[\constraint]$ the Koszul dual of the $\constraint$-tidy parallel signaletic operad~$\tidySignaleticParallel[\constraint]$.
In other words, for any words~$\constraint[c], \constraint[d] \in \{\op{l}, \odot, \op{r}\}$, we have
\[
\tidyCitelangisParallel = \underset{i \in [|\constraint|]}{\scalebox{2}{$\blacksquare$}} \Dup_{\constraint_i}
\qqandqq
\tidyCitelangisParallel[{\constraint[c] \cdot \constraint[d]}] = \tidyCitelangisParallel[{\constraint[c]}] \blackManin \tidyCitelangisParallel[{\constraint[d]}].
\]
Again, we have decided to present the tidy parallel $k$-citelangis operad with the constraint~$\constraint = {\op{l}\!^k}$ to simplify the presentation and since this will be the only possible option in series.
However, as we will use an action of the $\op{l,r}$-tidy parallel citelangis operad in \cref{subsec:actionParallelCitalangisOperads}, let us present its $9$ relations:

\vspace{.4cm}
\centerline{
\begin{tabular}{c@{\qquad}c@{\qquad}c}
\(\compoR{l,l}{l,r} \!\!\overset{\displaystyle\text{($\tidyCitelangisParallel[\op{l,r}]\,11$)\label{eq:lrtidyCitelangisParallel11}}}{=}\!\! \compoL{l,l}{l,l}\) &
\(\compoR{l,r}{l,l} \!\!\overset{\displaystyle\text{($\tidyCitelangisParallel[\op{l,r}]\,12$)\label{eq:lrtidyCitelangisParallel12}}}{=}\!\! \compoL{l,r}{l,l}\) &
\(\compoR{l,r}{l,r} \!\!\overset{\displaystyle\text{($\tidyCitelangisParallel[\op{l,r}]\,13$)\label{eq:lrtidyCitelangisParallel13}}}{=}\!\! \compoL{l,r}{l,r}\) \\[.8cm]
\(\compoR{r,l}{l,r} \!\!\overset{\displaystyle\text{($\tidyCitelangisParallel[\op{l,r}]\,21$)\label{eq:lrtidyCitelangisParallel21}}}{=}\!\! \compoL{r,l}{l,l}\) &
\(\compoR{r,r}{l,l} \!\!\overset{\displaystyle\text{($\tidyCitelangisParallel[\op{l,r}]\,22$)\label{eq:lrtidyCitelangisParallel22}}}{=}\!\! \compoL{r,r}{l,l}\) &
\(\compoR{r,r}{l,r} \!\!\overset{\displaystyle\text{($\tidyCitelangisParallel[\op{l,r}]\,23$)\label{eq:lrtidyCitelangisParallel23}}}{=}\!\! \compoL{r,r}{l,r}\) \\[.8cm]
\(\compoR{r,l}{r,r} \!\!\overset{\displaystyle\text{($\tidyCitelangisParallel[\op{l,r}]\,31$)\label{eq:lrtidyCitelangisParallel31}}}{=}\!\! \compoL{l,l}{r,l}\) &
\(\compoR{r,r}{r,l} \!\!\overset{\displaystyle\text{($\tidyCitelangisParallel[\op{l,r}]\,32$)\label{eq:lrtidyCitelangisParallel32}}}{=}\!\! \compoL{l,r}{r,l}\) &
\(\compoR{r,r}{r,r} \!\!\overset{\displaystyle\text{($\tidyCitelangisParallel[\op{l,r}]\,33$)\label{eq:lrtidyCitelangisParallel33}}}{=}\!\! \compoL{l,r}{r,r}\)
\end{tabular}
}
\end{remark}


\subsection{Series citelangis operads}
\label{subsec:seriesCitelangisOperads}

We now consider the Koszul duals of the series signaletic operads, and we treat separately the messy and tidy situations.


\subsubsection{Messy series citelangis operads}
\label{subsubsec:messySeriesCitelangisOperads}

We start with the messy setting.

\begin{definition}
The \defn{messy series $k$-citelangis operad}~$\messyCitelangisSeries$ is the Koszul dual of the messy series $k$-signaletic operad:
\[
\messyCitelangisSeries \eqdef (\messySignaleticSeries)^\koszul
\]
\end{definition}

By \cref{thm:signaleticOperadsQuadratic}, the messy series $k$-citelangis operad~$\messyCitelangisSeries$ is quadratic and Koszul.
To describe its quadratic relations, it is convenient to consider sums of operations like~${\op{m}} \eqdef {\op{l}} + {\op{r}}$.

\begin{proposition}
\label{prop:messyCitelangisSeriesRelations}
Consider the $3^k$ operations~$\{\op{l}, \op{m}, \op{r}\}^k$ related by
\[
\operation[b]\!\op{m}\!\operation[b'] = \operation[b]\!\op{l}\!\operation[b'] + \operation[b]\!\op{r}\!\operation[b']
\]
for all~$\operation[b],\operation[b'] \in \{\op{l}, \op{m}, \op{r}\}^*$ with~$|\operation[b]| + |\operation[b']| = k-1$. 

For any destination vector~$\ind{p} \in [3]^k$ of arity~$3$, the messy series $k$-citelangis operad~$\messyCitelangisSeries$ satisfies the quadratic relation
\begin{equation} \tag{$\messyCitelangisSeries[]\,\ind{p}$}\label{eq:messyCitelangisSeriesp}
	\begin{tikzpicture}[baseline=-.5cm, level/.style={sibling distance = .8cm, level distance = .8cm}]
		\node [rectangle, draw, minimum height=.6cm] (root) {$\operation[a]_\ind{p}$}
			child {node {$1$}}
			child {node [rectangle, draw, minimum height=.6cm] (r) {$\operation[b]_\ind{p}$}
				child {node {$2$}}
				child {node {$3$}}
			}
		;
	\end{tikzpicture}
=
	\begin{tikzpicture}[baseline=-.5cm, level/.style={sibling distance = .8cm, level distance = .8cm}]
		\node [rectangle, draw, minimum height=.6cm] {$\operation[c]_\ind{p}$}
			child {node [rectangle, draw, minimum height=.6cm] {$\operation[d]_\ind{p}$}
				child {node {$1$}}
				child {node {$2$}}
			}
			child {node {$3$}}
		;
	\end{tikzpicture}
\end{equation}
where the operations~$\operation[a]_\ind{p}, \operation[b]_\ind{p}, \operation[c]_\ind{p}, \operation[d]_\ind{p} \in \{\op{l}, \op{m}, \op{r}\}^k$ are defined by
\begin{gather*}
(\operation[a]_\ind{p})_i \eqdef \begin{cases} \op{l} & \text{if } \ind{p}_i = 1 \\ \op{r} & \text{if } \ind{p}_i = 2  \\ \op{r} & \text{if } \ind{p}_i = 3 \end{cases} \qquad
(\operation[b]_\ind{p})_i \eqdef \begin{cases} \op{m} & \text{if } |\ind{p}^{\{2,3\}}| < i \\ \op{l} & \text{if } (\ind{p}^{\{2,3\}})_i = 2 \\ \op{r} & \text{if } (\ind{p}^{\{2,3\}})_i = 3 \end{cases} \\
(\operation[c]_\ind{p})_i \eqdef \begin{cases} \op{l} & \text{if } \ind{p}_i = 1 \\ \op{l} & \text{if } \ind{p}_i = 2 \\ \op{r} & \text{if } \ind{p}_i = 3 \end{cases} \qquad
(\operation[d]_\ind{p})_i \eqdef \begin{cases} \op{l} & \text{if } (\ind{p}^{\{1,2\}})_i = 1 \\ \op{r} & \text{if } (\ind{p}^{\{1,2\}})_i = 2 \\ \op{m} & \text{if } |\ind{p}^{\{1,2\}}| < i \end{cases}
\end{gather*}
for any~$i \in [k]$.
Again for any~$L \subseteq [3]$, we have denoted by~$\ind{p}^L$ the subword of~$\ind{p}$ consisting only of the letters which belong to~$L$, and for any~$x \in L$ the condition~$(\ind{p}^L)_i = x$ implicitly assume that~$\ind{p}^L$ has length at least~$i$.
Note that \cref{eq:messyCitelangisSeriesp} involves $2^{|\ind{p}|_1}$ syntax trees of~$\Syntax[\Operations_k]$ on the left hand side and $2^{|\ind{p}|_3}$ syntax trees of~$\Syntax[\Operations_k]$ on the right hand side.
We call these relations the \defn{messy series $k$-citelangis relations}.
\end{proposition}

\begin{proof}
This follows from~\cref{lem:koszulDual,rem:messySignaleticRelationsSeries}.
\end{proof}

\begin{example}[Messy series $0$-, $1$- and~$2$-citelangis relations]
The messy series $0$-citelangis relation is the associative relation:
\allowdisplaybreaks
\begin{equation}
\compoR{n}{n} = \compoL{n}{n}. \tag{$\messyCitelangisSeries[]\,.$}\label{eq:messyCitelangisSeries0}
\end{equation}
In other words, the messy series $0$-citelangis operad~$\messyCitelangisSeries[0]$ is just the associative operad~$\As$.

The messy series $1$-citelangis relations are the $3$ dendriform relations:

\vspace{.4cm}
\centerline{
\begin{tabular}{c@{\qquad}c@{\qquad}c}
\(\compoR{l}{m} \!\!\overset{\displaystyle\text{($\messyCitelangisSeries[]\,1$)\label{eq:messyCitelangisSeries1}}}{=}\!\! \compoL{l}{l}\) &
\(\compoR{r}{l} \!\!\overset{\displaystyle\text{($\messyCitelangisSeries[]\,2$)\label{eq:messyCitelangisSeries2}}}{=}\!\! \compoL{r}{l}\) &
\(\compoR{r}{r} \!\!\overset{\displaystyle\text{($\messyCitelangisSeries[]\,3$)\label{eq:messyCitelangisSeries3}}}{=}\!\! \compoL{m}{r}\)
\end{tabular}
}

\noindent
where
\[
{\op{m}} = {\op{l}} + {\op{r}}.
\]
In other words, the messy series $1$-citelangis operad~$\messyCitelangisSeries[1]$ is just the dendriform operad~$\Dend$.

\pagebreak
The messy series $2$-citelangis relations are the following $9$ relations:

\vspace{.4cm}
\centerline{
\begin{tabular}{c@{\qquad}c@{\qquad}c}
\(\compoR{l,l}{m,m} \!\!\overset{\displaystyle\text{($\messyCitelangisSeries[]\,11$)\label{eq:messyCitelangisSeries11}}}{=}\!\! \compoL{l,l}{l,l}\) &
\(\compoR{l,r}{l,m} \!\!\overset{\displaystyle\text{($\messyCitelangisSeries[]\,12$)\label{eq:messyCitelangisSeries12}}}{=}\!\! \compoL{l,r}{l,l}\) &
\(\compoR{l,r}{r,m} \!\!\overset{\displaystyle\text{($\messyCitelangisSeries[]\,13$)\label{eq:messyCitelangisSeries13}}}{=}\!\! \compoL{l,m}{l,r}\) \\[.8cm]
\(\compoR{r,l}{l,m} \!\!\overset{\displaystyle\text{($\messyCitelangisSeries[]\,21$)\label{eq:messyCitelangisSeries21}}}{=}\!\! \compoL{r,l}{l,l}\) &
\(\compoR{r,r}{l,l} \!\!\overset{\displaystyle\text{($\messyCitelangisSeries[]\,22$)\label{eq:messyCitelangisSeries22}}}{=}\!\! \compoL{r,r}{l,l}\) &
\(\compoR{r,r}{l,r} \!\!\overset{\displaystyle\text{($\messyCitelangisSeries[]\,23$)\label{eq:messyCitelangisSeries23}}}{=}\!\! \compoL{r,m}{l,r}\) \\[.8cm]
\(\compoR{r,l}{r,m} \!\!\overset{\displaystyle\text{($\messyCitelangisSeries[]\,31$)\label{eq:messyCitelangisSeries31}}}{=}\!\! \compoL{l,m}{r,l}\) &
\(\compoR{r,r}{r,l} \!\!\overset{\displaystyle\text{($\messyCitelangisSeries[]\,32$)\label{eq:messyCitelangisSeries32}}}{=}\!\! \compoL{r,m}{r,l}\) &
\(\compoR{r,r}{r,r} \!\!\overset{\displaystyle\text{($\messyCitelangisSeries[]\,33$)\label{eq:messyCitelangisSeries33}}}{=}\!\! \compoL{m,m}{r,r}\)
\end{tabular}
}

\noindent
where
\begin{gather*}
{\op{m,l}} \; = \; {\op{l,l}} + {\op{r,l}}, \qquad
{\op{m,r}} \; = \; {\op{l,r}} + {\op{r,r}}, \\
{\op{l,m}} \; = \; {\op{l,l}} + {\op{l,r}}, \qquad
{\op{r,m}} \; = \; {\op{r,l}} + {\op{r,r}}, \\
\text{and}\qquad {\op{m,m}} \; = \; {\op{l,m}} + {\op{r,m}} \; = \; {\op{m,l}} + {\op{m,r}} \; = \; {\op{l,l}} + {\op{l,r}} + {\op{r,l}} + {\op{r,r}}.
\end{gather*}
\end{example}

\begin{remark}
\label{rem:restrictionMessyCitelangisSeries}
Similarly to~\cref{rem:restrictionMessyCitelangisParallel}, one can dualize \cref{rem:restrictionMessySignaleticSeries}: For any $\ell \le k$ there is a morphism $(\Res^k_\ell)^\koszul:\messyCitelangisSeries[\ell]\mapsto\messyCitelangisSeries[k]$. In terms of algebras a messy series $k$-citelangis algebra for the operations~$\{\op{l}, \op{r}\}^k$ is also a messy series $\ell$-citelangis algebra for the operations~$\set{\operation\!\op{m}\!^{k-\ell}}{\operation \in \{\op{l}, \op{r}\}^{\ell}}$.

In particular for $\ell=0$, a messy series $k$-citelangis algebra~$(\algebra, \Operations_k)$ defines a structure of associative algebra~$(\algebra, {\op{m}\!\!^k})$. This can be seen by adding up all messy series $k$-citelangis relations of \cref{prop:messyCitelangisSeriesRelations}, obtaining
\[
	\begin{tikzpicture}[baseline=-.5cm, level 1/.style={sibling distance = .8cm, level distance = .8cm}, level 2/.style={sibling distance = .6cm, level distance = .6cm}]
		\node [rectangle, draw] {$\op{m}\!\!^k$}
			child {node {}}
			child {node [rectangle, draw] {$\op{m}\!\!^k$}
				child {node {}}
				child {node {}}
			}
		;
	\end{tikzpicture}
	=
	\begin{tikzpicture}[baseline=-.5cm, level 1/.style={sibling distance = .8cm, level distance = .8cm}, level 2/.style={sibling distance = .6cm, level distance = .6cm}]
		\node [rectangle, draw] {$\op{m}\!\!^k$}
			child {node [rectangle, draw] {$\op{m}\!\!^k$}
				child {node {}}
				child {node {}}
			}
			child {node {}}
		;
	\end{tikzpicture}
\]
Reciprocally, we say that an associative algebra~$(\algebra, \product)$ admits a \defn{messy series $k$-citelangis structure} if it is possible to split the product~$\product$ into~$2^k$ operations~$\Operations_k$ defining a messy series $k$-citelangis algebra on~$\algebra$.
\end{remark}


\subsubsection{Tidy series citelangis operads}
\label{subsubsec:tidySeriesCitelangisOperads}

We now switch to the tidy setting.

\begin{definition}
The \defn{tidy series $k$-citelangis operad}~$\tidyCitelangisSeries$ is the Koszul dual of the tidy series $k$-signaletic operad:
\[
\tidyCitelangisSeries \eqdef (\tidySignaleticSeries)^\koszul
\]
\end{definition}

By \cref{thm:signaleticOperadsQuadratic}, the tidy series $k$-citelangis operad~$\tidyCitelangisSeries$ is quadratic and Koszul.
We now describe its quadratic relations.

\begin{proposition}
\label{prop:tidyCitelangisSeriesRelations}
For any destination vector~$\ind{p} \in [3]^k$ of arity~$3$, the tidy series $k$-citelangis operad~$\tidyCitelangisSeries$ satisfies the quadratic relation
\begin{equation} \tag{$\tidyCitelangisSeries[]\,\ind{p}$}\label{eq:tidyCitelangisSeriesp}
	\begin{tikzpicture}[baseline=-.5cm, level/.style={sibling distance = .8cm, level distance = .8cm}]
		\node [rectangle, draw, minimum height=.6cm] (root) {$\operation[a]_\ind{p}$}
			child {node {$1$}}
			child {node [rectangle, draw, minimum height=.6cm] (r) {$\operation[b]_\ind{p}$}
				child {node {$2$}}
				child {node {$3$}}
			}
		;
	\end{tikzpicture}
=
	\begin{tikzpicture}[baseline=-.5cm, level/.style={sibling distance = .8cm, level distance = .8cm}]
		\node [rectangle, draw, minimum height=.6cm] {$\operation[c]_\ind{p}$}
			child {node [rectangle, draw, minimum height=.6cm] {$\operation[d]_\ind{p}$}
				child {node {$1$}}
				child {node {$2$}}
			}
			child {node {$3$}}
		;
	\end{tikzpicture}
\end{equation}
where the operations~$\operation[a]_\ind{p}, \operation[b]_\ind{p}, \operation[c]_\ind{p}, \operation[d]_\ind{p} \in \{\op{l}, \op{m}, \op{r}\}^k$ are defined by
\begin{gather*}
	(\operation[a]_\ind{p})_i \eqdef \begin{cases} \op{l} & \text{if } \ind{p}_i = 1 \\ \op{r} & \text{if } \ind{p}_i = 2 \\ \op{r} & \text{if } \ind{p}_i = 3 \end{cases}
	\qquad
	(\operation[b]_\ind{p})_i \eqdef \begin{cases} \op{l} & \text{if } (\ind{p}^{\{2,3\}})_i = 1 \\ \op{l} & \text{if } (\ind{p}^{\{2,3\}})_i = 2 \\ \op{r} & \text{if } |\ind{p}^{\{2,3\}}| < i \end{cases}
	\\
	(\operation[c]_\ind{p})_i \eqdef \begin{cases} \op{l} & \text{if } \ind{p}_i = 1 \\ \op{l} & \text{if } \ind{p}_i = 2 \\ \op{r} & \text{if } \ind{p}_i = 3 \end{cases}
	\qquad
	(\operation[d]_\ind{p})_i \eqdef \begin{cases} \op{l} & \text{if } (\ind{p}^{\{1,2\}})_i = 1 \\ \op{r} & \text{if } (\ind{p}^{\{1,2\}})_i = 2 \\ \op{l} & \text{if } |\ind{p}^{\{1,2\}}| < i \end{cases}
\end{gather*}
for any~$i \in [k]$.
Again for any~$L \subseteq [3]$, we have denoted by~$\ind{p}^L$ the subword of~$\ind{p}$ consisting only of the letters which belong to~$L$, and for any~$x \in L$ the condition~$(\ind{p}^L)_i = x$ implicitly assume that~$\ind{p}^L$ has length at least~$i$.
We call these relations the \defn{tidy series $k$-citelangis relations}.
\end{proposition}

\begin{proof}
This follows from~\cref{lem:koszulDual,rem:tidySignaleticRelationsSeries}.
\end{proof}

\begin{example}[Tidy series $0$-, $1$- and~$2$-citelangis relations]
The tidy series $0$-citelangis relation is the associative relation:
\allowdisplaybreaks
\begin{equation}
\compoR{n}{n} = \compoL{n}{n}. \tag{$\tidyCitelangisSeries[]\,.$}\label{eq:tidyCitelangisSeries0}
\end{equation}
In other words, the tidy series $0$-citelangis operad~$\tidyCitelangisSeries[0]$ is just the associative operad~$\As$.

The tidy series $1$-citelangis relations are the $3$ twisted duplicial relations:

\vspace{.4cm}
\centerline{
\begin{tabular}{c@{\qquad}c@{\qquad}c}
\(\compoR{l}{l} \!\!\overset{\displaystyle\text{($\tidyCitelangisSeries[]\,1$)\label{eq:tidyCitelangisSeries1}}}{=}\!\! \compoL{l}{l}\) &
\(\compoR{r}{l} \!\!\overset{\displaystyle\text{($\tidyCitelangisSeries[]\,2$)\label{eq:tidyCitelangisSeries2}}}{=}\!\! \compoL{r}{l}\) &
\(\compoR{r}{r} \!\!\overset{\displaystyle\text{($\tidyCitelangisSeries[]\,3$)\label{eq:tidyCitelangisSeries3}}}{=}\!\! \compoL{l}{r}\)
\end{tabular}
}

\noindent
In other words, the tidy series $1$-citelangis operad~$\tidyCitelangisSeries[1]$ is just the twisted duplicial operad~$\Dup_{\op{l}}$.

The tidy series $2$-citelangis relations are the following $9$ relations:

\vspace{.4cm}
\centerline{
\begin{tabular}{c@{\qquad}c@{\qquad}c}
\(\compoR{l,l}{l,l} \!\!\overset{\displaystyle\text{($\tidyCitelangisSeries[]\,11$)\label{eq:tidyCitelangisSeries11}}}{=}\!\! \compoL{l,l}{l,l}\) &
\(\compoR{l,r}{l,l} \!\!\overset{\displaystyle\text{($\tidyCitelangisSeries[]\,12$)\label{eq:tidyCitelangisSeries12}}}{=}\!\! \compoL{l,r}{l,l}\) &
\(\compoR{l,r}{r,l} \!\!\overset{\displaystyle\text{($\tidyCitelangisSeries[]\,13$)\label{eq:tidyCitelangisSeries13}}}{=}\!\! \compoL{l,l}{l,r}\) \\[.8cm]
\(\compoR{r,l}{l,l} \!\!\overset{\displaystyle\text{($\tidyCitelangisSeries[]\,21$)\label{eq:tidyCitelangisSeries21}}}{=}\!\! \compoL{r,l}{l,l}\) &
\(\compoR{r,r}{l,l} \!\!\overset{\displaystyle\text{($\tidyCitelangisSeries[]\,22$)\label{eq:tidyCitelangisSeries22}}}{=}\!\! \compoL{r,r}{l,l}\) &
\(\compoR{r,r}{l,r} \!\!\overset{\displaystyle\text{($\tidyCitelangisSeries[]\,23$)\label{eq:tidyCitelangisSeries23}}}{=}\!\! \compoL{r,l}{l,r}\) \\[.8cm]
\(\compoR{r,l}{r,l} \!\!\overset{\displaystyle\text{($\tidyCitelangisSeries[]\,31$)\label{eq:tidyCitelangisSeries31}}}{=}\!\! \compoL{l,l}{r,l}\) &
\(\compoR{r,r}{r,l} \!\!\overset{\displaystyle\text{($\tidyCitelangisSeries[]\,32$)\label{eq:tidyCitelangisSeries32}}}{=}\!\! \compoL{r,l}{r,l}\) &
\(\compoR{r,r}{r,r} \!\!\overset{\displaystyle\text{($\tidyCitelangisSeries[]\,33$)\label{eq:tidyCitelangisSeries33}}}{=}\!\! \compoL{l,l}{r,r}\)
\end{tabular}
}
\end{example}

\begin{remark}
\label{rem:restrictionTidyCitelangisSeries}
Similarly to \cref{rem:restrictionMessyCitelangisSeries}, there is a restriction morphism~$(\Res^k_\ell)^\koszul : \tidyCitelangisSeries[k] \mapsto \tidyCitelangisSeries[l]$.
\end{remark}


\subsection{Hilbert series and normal forms}
\label{subsec:HilberSeriesNormalFormsCitelangis}

In this section, we discuss numerological properties of the citelangis operads.
As for the signaletic operads (see \cref{subsubsec:HilberSeriesSignaletic}), these properties do not really depend on whether the cars arrive in series or parallel, and whether the signals outside the routes of the cars are pointing or not to the left.
We therefore write~$\citelangis$ to denote without distinction any of the four $k$-citelangis operads~$\messyCitelangisParallel$, $\tidyCitelangisParallel$, $\messyCitelangisSeries$, or~$\tidyCitelangisSeries$.


\subsubsection{Hilbert series}
\label{subsubsec:HilberSeriesCitelangis}

We first discuss the Hilbert series of the $k$-citelangis operad~$\citelangis$, based on \cref{coro:HilbertSignaletic,subsec:Koszul}.

\begin{corollary}
\label{coro:HilbertSeriesCitelangis}
The Hilbert series of the (messy or tidy, parallel or series) $k$-citelangis operad~$\citelangis$ satisfies the equation
\begin{equation}
\label{eq:functionalEquation}
-\Hilbert_{\citelangis}(-t) \cdot \EulPol{k}(-\Hilbert_{\citelangis}(-t)) = t(1+\Hilbert_{\citelangis}(-t))^{k+1}.
\end{equation}
Therefore, the dimension~$d_k(p) \eqdef \dim(\citelangis(p))$ can be computed recursively from~$d_k(1) = 1$ by
\begin{equation}
\label{eq:recursiveFormula}
d_k(p) = \sum_{\substack{q_1, \dots, q_{k+1} \ge 0 \\ q_1 + \dots + q_{k+1} = p-1}} d_k(q_1) \cdots d_k(q_{k+1}) + \sum_{j=1}^{k-1} (-1)^{j+1} \, \EulNum{k}{j} \sum_{\substack{q_1, \dots, q_{j+1} \ge 1 \\ q_1 + \dots + q_{j+1} = p}} d_k(q_1) \cdots d_k(q_{j+1}).
\end{equation}
Alternatively, the dimension~$d_k(p) \eqdef \dim(\citelangis(p))$ is given by the summation formula
\begin{equation}
\label{eq:closedFormula}
d_k(p) = \frac{1}{p} \sum_{\substack{i, j_1, \dots, j_{k-1} \ge 0 \\ i + j' = p-1}} (-1)^{j+j'} \!\! \binom{(k+1)p}{i} \binom{p+j-1}{j} \binom{j}{j_1, \dots, j_{k-1}} \EulNum{k}{1}^{j_1} \cdots \EulNum{k}{k-1}^{j_{k-1}} \!\! ,
\end{equation}
where~$j = j_1 + \dots + j_{k-1}$ and~$j' = j_1 + \dots + (k-1) j_{k-1}$.
\end{corollary}

\begin{proof}
Since the $k$-signaletic operad~$\signaletic$ and the $k$-citelangis operad~$\citelangis$ are Koszul dual Koszul operads by definition, their Hilbert series are Lagrange inverses by \cref{thm:HilbertKoszulLagrange}:
\[
\Hilbert_{\signaletic}(-\Hilbert_{\citelangis}(-t)) = t.
\]
\cref{coro:HilbertSignaletic} thus shows the functional equation of~\cref{eq:functionalEquation}.

Using this functional equation with~$t = -u$, we obtain
\[
\Hilbert_{\citelangis}(u) = u(1+\Hilbert_{\citelangis}(u))^{k+1} + \sum_{j = 1}^{k-1} \EulNum{k}{j} (-\Hilbert_{\citelangis}(u))^{j+1}.
\]
Comparing the coefficient of~$u^p$ in the two sides of this equality yields the recursive formula of~\cref{eq:recursiveFormula}.

Finally, to obtain the summation formula of~\cref{eq:closedFormula}, we use Lagrange's theorem: since $\Hilbert_{\citelangis}(u) = u \Phi(\Hilbert_{\citelangis}(u))$ with
\[
\Phi(v) \eqdef \frac{(1+v)^{k+1}}{\EulPol{k}(-v)}
\]
we obtain by Lagrange inversion theorem that
\[
d_k(p) = [u^p] \, \Hilbert_{\citelangis}(u) = \frac{1}{p} \, [v^{p-1}] \, \Phi(v)^p = \frac{1}{p} \sum_{i = 0}^{p-1} [v^i] \, (1+v)^{(k+1)p} \cdot [v^{p-i-1}] \, \bigg( \frac{1}{\EulPol{k}(-v)} \bigg)^p .
\]
The result thus follows from the development
\begin{align*}
\bigg( \frac{1}{\EulPol{k}(-v)} \bigg)^p & = \bigg( 1 + \sum_{\ell = 1}^{k-1} \EulNum{k}{\ell} (-v)^\ell \bigg)^{-p} = \sum_{j \ge 0} (-1)^j \binom{p+j-1}{j} \bigg( \sum_{\ell = 1}^{k-1} \EulNum{k}{\ell} (-v)^\ell \bigg)^j \\
& = \sum_{j_1, \dots, j_{k-1} \ge 0} (-1)^{j+j'} \binom{p+j-1}{j} \binom{j}{j_1, \dots, j_{k-1}} \EulNum{k}{1}^{j_1} \cdots \EulNum{k}{k-1}^{j_{k-1}} v^{j'}
\end{align*}
where~$j = j_1 + \dots + j_{k-1}$ and~$j' = j_1 + \dots + (k-1) j_{k-1}$.
\end{proof}

\begin{remark}
Observe that
\[
\Phi(v) = \frac{1}{\displaystyle{\sum_{p \ge 0} p^k \, (-v)^{p-1}}} = \sum_{p \ge 0} \bigg( \sum_{\substack{\ell \ge 0 \\ c_1, \dots, c_\ell \ge 1 \\ c_1 + \dots + c_\ell = p}} (-1)^{p-\ell} (c_1+1)^k \cdots (c_\ell+1)^k \bigg) v^p.
\]
\end{remark}

\begin{example}
\label{exm:HilbertSeriesDend}
For~$k=1$, \cref{eq:recursiveFormula,eq:closedFormula} become
\[
d_1(p) = \sum_{\substack{a,b \ge 0 \\ a+b = p-1}} d_1(a) \, d_1(b)
\qqandqq
d_1(p) = \frac{1}{p} \binom{2p}{p-1} = \frac{1}{p+1} \binom{2p}{p} = C_p.
\]
For~$k=2$, we have
\begin{equation}
\label{eq:noncrossingConnectedArcDiagrams}
d_2(p) = \!\!\!\! \sum_{\substack{a,b,c \ge 0 \\ a+b+c = p-1}} \!\!\!\! d_2(a) \, d_2(b) \, d_2(c) + \sum_{\substack{a,b \ge 1 \\ a+b = p}} d_2(a) \,d_2(b)
\qandq
d_2(p) = \frac{1}{p} \sum_i \binom{3p}{i} \binom{2p-i-2}{p-i-1}.
\end{equation}
These formulas appeared as conjectures in the original paper of M.~Aguiar and J.-L.~Loday on quadri-algebras~\cite{AguiarLoday}: the recursive formula is \cite[Remark~4.7]{AguiarLoday} while the summation formula is equivalent to~\cite[Conjecture~4.2]{AguiarLoday} up to the change of variable~$j = 2n-i+1$. These conjectures follow form the work of J.\,E.~Vatne~\cite{Vatne} and L.~Foissy in~\cite{Foissy}.

For~$k=3$ we have
\[
d_3(p) = \!\!\!\! \sum_{\substack{a,b,c,d \ge 0 \\ a+b+c+d = p-1}} \!\!\!\! d_3(a) \, d_3(b) \, d_3(c) \, d_3(d) - \!\! \sum_{\substack{a,b,c \ge 1 \\ a+b+c = p}} \!\!d_3(a) \, d_3(b) \, d_3(c) + 4 \sum_{\substack{a,b \ge 1 \\ a+b = p}} d_3(a) \, d_3(b)
\]
\[
\text{and}\qquad
d_3(p) = \frac{1}{p} \sum_{i,j} (-1)^{i+j+p+1} \binom{4p}{i} \binom{p+j-1}{j} \binom{j}{p-i-j-1} 4^{2j+i-p+1}.
\]
\cref{table:dpk} gathers the values of~$d_k(p)$ for~$k \in [5]$ and~$p \in [8]$.
\begin{table}[h]
	\[
	\begin{array}{l|rrrrrrrr}
	\raisebox{-.1cm}{$k$} \backslash \, \raisebox{.1cm}{$p$}
	  & 1 &  2 &    3 &      4 &       5 &         6 &          7  &             8 \\[.1cm]
	\hline
	1 & 1 &  2 &    5 &     14 &      42 &       132 &         429 &          1430 \\
	2 & 1 &  4 &   23 &    156 &    1162 &      9192 &       75819 &        644908 \\
	3 & 1 &  8 &  101 &   1544 &   26190 &    474144 &     8975229 &     175492664 \\
	4 & 1 & 16 &  431 &  14256 &  525682 &  20731488 &   855780699 &   36512549680 \\
	5 & 1 & 32 & 1805 & 125984 & 9825222 & 820259712 & 71710602189 & 6481491238880
	\end{array}
	\]
	\caption{The values of~$d_k(p)$ for~$k \in [5]$ and~$p \in [8]$.}
	\label{table:dpk}
\end{table}
\end{example}

\begin{remark}
\label{rem:connectedCrossingFreeArcDiagrams}
For $k = 2$, the numbers $d_2(p)$ count the number of rooted non-crossing connected arc diagrams on~$p+1$ points.
To see the summation formula of~\cref{eq:noncrossingConnectedArcDiagrams}, decompose the diagrams according to whether the leftmost arc incident to the root is an isthme (then the deletion of this arc decomposes the diagram into $3$ subdiagrams with at least one node) or not (then the deletion of this arc decomposes the diagram into $2$ subdiagrams with at least two nodes).
These decompositions are illustrated in \cref{fig:decompositionArcDiagrams}.

\begin{figure}[h]
	\centerline{
	$\raisebox{-.45\height}{\includegraphics[scale=.27]{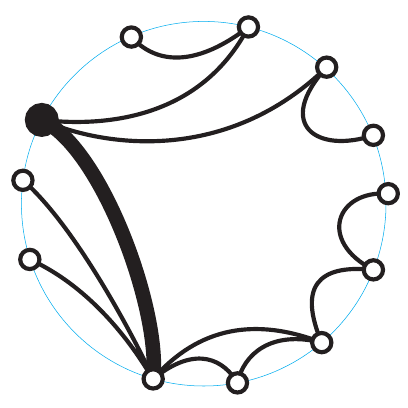}} = \raisebox{-.45\height}{\includegraphics[scale=.27]{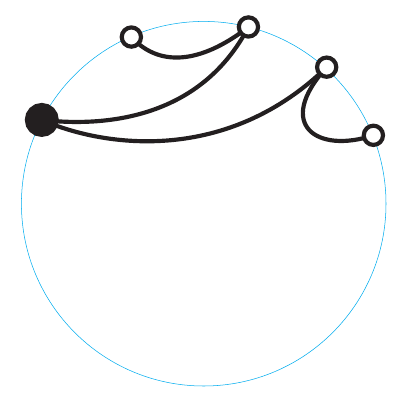}} + \raisebox{-.45\height}{\includegraphics[scale=.27]{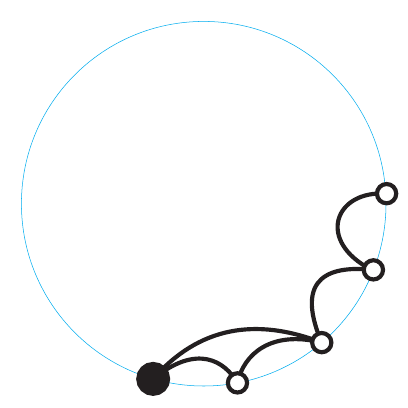}} + \raisebox{-.45\height}{\includegraphics[scale=.27]{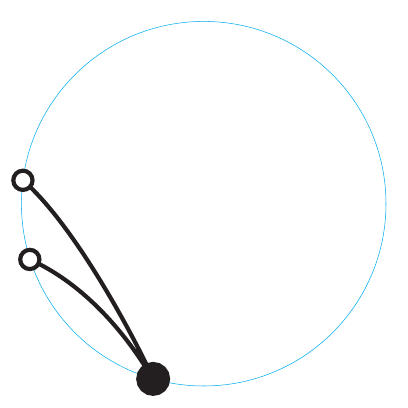}}$
	\qquad
	$\raisebox{-.45\height}{\includegraphics[scale=.27]{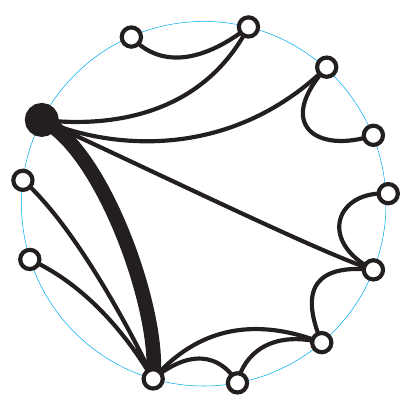}} = \raisebox{-.45\height}{\includegraphics[scale=.27]{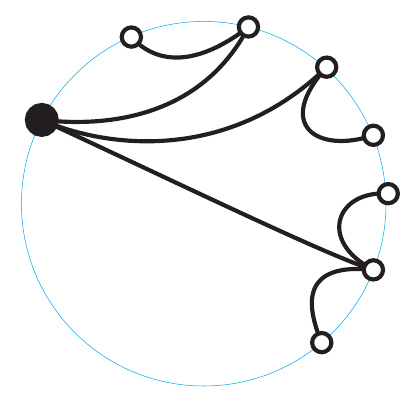}} + \raisebox{-.45\height}{\includegraphics[scale=.27]{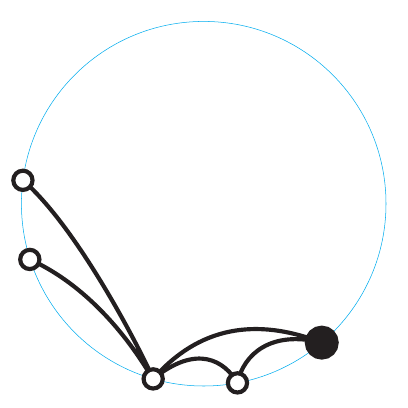}}$
	}
	\caption{Decomposition of two rooted non-crossing connected arc diagrams, whose root and leftmost arc~$\alpha$ incident to the root are bolded. We obtain three subdiagrams when $\alpha$ is an isthme (left), and two subdiagrams when~$\alpha$ is not an isthme (right). See \cref{rem:connectedCrossingFreeArcDiagrams}.}
	\label{fig:decompositionArcDiagrams}
\end{figure}
\end{remark}


\subsubsection{Rewriting system}
\label{subsubsec:rewritingSystemCitelangis}

We now orient the $k$-citelangis relations of \cref{prop:messyCitelangisParallelRelations,prop:tidyCitelangisParallelRelations,prop:messyCitelangisSeriesRelations,prop:tidyCitelangisSeriesRelations} to a convergent rewriting system.
As in \cref{subsubsec:rewritingSystemSignaletic}, we use the $k$-signaletic Tamari order defined in \cref{def:signaleticTamariOrder}, and the right $k$-signaletic combs.

\begin{definition}
\label{def:citelangisRewritingSystem}
The \defn{$k$-citelangis rewriting system} is the quadratic rewriting system where for each destination vector of~$[3]^k$, we orient the corresponding $k$-citelangis relation so that it rewrites the corresponding quadratic right $k$-signaletic comb.
\end{definition}

\begin{example}
For the destination vector~$\destVect{3,1}{}$, the corresponding $2$-citelangis relations (left) are transformed to the following rewriting rules (right):
\[
\begin{array}{c@{\qquad\longrightarrow\qquad}l}
\compoR{r,l}{r,m} \!\!\overset{\displaystyle\text{($\messyCitelangisParallel[]\,31$)}}{=}\!\! \compoL{m,l}{r,l}
&
\left( \compoR{r,l}{r,l} \quad , \quad  \compoL{l,l}{r,l} + \compoL{r,l}{r,l} - \compoR{r,l}{r,r} \right),
\\
\compoR{r,l}{r,l} \!\!\overset{\displaystyle\text{($\tidyCitelangisParallel[]\,31$)}}{=}\!\! \compoL{l,l}{r,l}
&
\left( \compoR{r,l}{r,l} \quad , \quad  \compoL{l,l}{r,l} \right),
\\
\compoR{r,l}{r,m} \!\!\overset{\displaystyle\text{($\messyCitelangisSeries[]\,31$)}}{=}\!\! \compoL{l,m}{r,l}
&
\left( \compoR{r,l}{r,l} \quad , \quad  \compoL{l,l}{r,l} + \compoL{l,r}{r,l} - \compoR{r,l}{r,r} \right),
\\
\compoR{r,l}{r,l} \!\!\overset{\displaystyle\text{($\tidyCitelangisSeries[]\,31$)}}{=}\!\! \compoL{l,l}{r,l}
&
\left( \compoR{r,l}{r,l} \quad , \quad  \compoL{l,l}{r,l} \right).
\end{array}
\]
\end{example}

Note that the $k$-citelangis rewriting system is oriented by the reverse $k$-signaletic Tamari lattice, thus opposite to the orientation for the $k$-signaletic rewriting system.
By this choice of reverse orientation, the Koszulity of the $k$-signaletic rewriting system ensures that the $k$-citelangis rewriting system also converges.


\subsubsection{Normal forms and Eulerian matrices}
\label{subsubsec:normalFormsCitelangis}

Since the $k$-citelangis rewriting system of \cref{def:citelangisRewritingSystem} rewrites precisely the quadratic right $k$-signaletic combs, its normal forms are linear combinations of the syntax trees of~$\Syntax[\Operations_k]$ that do not contain a quadratic right $k$-signaletic comb.
Note that these normal forms are the same for the tidy and the messy $k$-citelangis rewriting system.
In contrast, these normal forms are not the same in series and in parallel, but their generating series are the same.
These generating series are the Hilbert series~$\Hilbert_{\citelangis}(t)$ of the $k$-citelangis operads.

Call \defn{admissible} the syntax trees that do not contain a quadratic right $k$-signaletic comb.
This condition imposes some restrictions on the right child of each node, but none on the left child.
This observation has two immediate consequences:
\begin{enumerate}
\item We can decompose an admissible syntax tree into an admissible right comb with an admissible syntax tree (possibly empty) attached, as the left child, to each node.
\item The admissible right combs can be constructed by a simple automaton with transitions corresponding to all admissible quadratic right combs. The transition matrix of this automaton is described below.
\end{enumerate}
We explore enumerative consequences of these two points in the rest of this section.

Denote by~$\cR_k(t)$ the generating series of right combs on~$\Syntax[\Operations_k]$ (\ie syntax trees where no node has a left child) that do not contain a quadratic right $k$-signaletic comb, where the variable~$t$ counts their number of nodes.
For instance, we have
\begin{gather*}
R_1(t) = 1 + 2 t + t^2 = (1+t)^2, \\
R_2(t) = 1 + 4 t + 7 t^2 + 8 t^3 + 8 t^4 + 8 t^5 + \dots = \frac{(1+t)^3}{1-t}, \\
R_3(t) = 1 + 8 t + 37 t^2 + 144 t^3 + 540 t^4 + 2016 t^5 + \dots = \frac{(1+t)^4}{1-4t+t^2}, \\
R_4(t) = 1 + 16 t + 175 t^2 + 1760 t^3 + 17456 t^4 + 172832 t^5 + \dots = \frac{(1+t)^5}{1-11t+11t^2-t^3}.
\end{gather*}

\begin{proposition}
\label{prop:Rk}
The generating series~$\Hilbert_{\citelangis}(t)$ and~$\cR_k(t)$ satisfy $\Hilbert_{\citelangis}(t) = t \cdot \cR_k \big( \Hilbert_{\citelangis}(t) \big)$.
\end{proposition}

\begin{proof}
We have seen that we can inductively decompose any admissible syntax tree as an admissible right comb where we attach a (possibly empty) admissible syntax tree as the left child of each node.
The generating series are thus related by composition: ${\Hilbert_{\citelangis}(t) = t \cdot \cR_k \big( \Hilbert_{\citelangis}(t) \big)}$.
\end{proof}

\begin{corollary}
The generating series~$\cR_k(t)$ is given by $\cR_k(t) \cdot \EulPol{k}(-t) = (1+t)^{k+1}$.
\end{corollary}

\begin{proof}
This follows from \cref{coro:HilbertSeriesCitelangis,prop:Rk}.
\end{proof}

We now interpret $\cR_k(t)$ in terms of the coefficients of the iterated powers of the transition matrix of an automaton.
Consider the automaton whose nodes are the operations of~$\Operations_k$ and with a transition~$\operation[a] \to \operation[b]$ if the right comb with the root~$\operation[a]$ and the right child~$\operation[b]$ is admissible.
Let~$\mat{M}_k$ denote the transition matrix of this automaton, \ie the $(\Operations_k \times \Operations_k)$-matrix with coefficients $(\mat{M}_k)_{\operation[a]\operation[b]} = 1$ if there is a transition~$\operation[a] \to \operation[b]$ and~$0$ otherwise.
Since the automaton constructs the admissible right combs, we obtain the following statement.

\begin{proposition}
\label{prop:powersMk}
The generating series~$\cR_k(t)$ is given by  $\displaystyle \cR_k(t) = 1 + \sum_{p \ge 1} \sum_{\operation[a], \operation[b] \in \Operations_k} (\mat{M}_k^{p-1})_{\operation[a]\operation[b]} \, t^p.$
\end{proposition}

\begin{proof}
The $p$-th power of the transition matrix encodes all right combs of arity~$p+1$.
\end{proof}

Note that the transitions are the same in the tidy and in the messy setting.
In contrast, they differ in the series and parallel settings:
\setlength\arrayrulewidth{.3pt}\arrayrulecolor{lightgray}
\begin{enumerate}
\item For the parallel $k$-citelangis operad, the transition matrix~$\mat{M}_k^\parallel$ is the $(\Operations_k \times \Operations_k)$-matrix whose $(\operation[a], \operation[b])$-coefficient is given by
\[
(\mat{M}_k^\parallel)_{\operation[a]\operation[b]} =
\begin{cases}
1 & \text{ if there is~$i \in [k]$ such that $\operation[a]_i = {\op{l}}$ while $\operation[b]_i = {\op{r}}$,} \\
0 & \text{ otherwise.}
\end{cases}
\]
for any~$\operation[a] \in \Operations_k$ and~$\operation[b] \in \Operations_k$.
For example,
\[
\qquad\quad
\mat{M}_1^\parallel = \begin{blockarray}{c@{\;\color{lightgray}\vline\;}cc}
\op{l} & \op{r} & \\
\begin{block}{[c@{\;\color{lightgray}\vline\;}c]c@{\!}}
0 & 1 & \op{l} \\
\cline{1-3}
0 & 0 & \op{r} \\
\end{block}
\end{blockarray}
\qquad
\mat{M}_2^\parallel = \begin{blockarray}{c@{\,\color{lightgray}\vline\,}c@{\;}c@{\,\color{lightgray}\vline\,}cc}
\rotatebox{70}{$\op{l,l}$} & \rotatebox{70}{$\op{l,r}$} & \rotatebox{70}{$\op{r,l}$} & \rotatebox{70}{$\op{r,r}$} & \\
\begin{block}{[c@{\,\color{lightgray}\vline\,}c@{}c@{\,\color{lightgray}\vline\,}c]c@{\!}}
0 & 1 & 1 & 1 & \op{l,l} \\
\cline{1-5}
0 & 0 & 1 & 1 & \op{l,r} \\
0 & 1 & 0 & 1 & \op{r,l} \\
\cline{1-5}
0 & 0 & 0 & 0 & \op{r,r} \\
\end{block}
\end{blockarray}
\qquad
\mat{M}_3^\parallel = \begin{blockarray}{c@{\color{lightgray}\vline}c@{}c@{}c@{\color{lightgray}\vline}c@{}c@{}c@{\color{lightgray}\vline}cc}
\rotatebox{70}{$\op{l,l,l}$} & \rotatebox{70}{$\op{l,l,r}$} & \rotatebox{70}{$\op{l,r,l}$} & \rotatebox{70}{$\op{r,l,l}$} & \rotatebox{70}{$\op{l,r,r}$} & \rotatebox{70}{$\op{r,l,r}$} & \rotatebox{70}{$\op{r,r,l}$} & \rotatebox{70}{$\op{r,r,r}$} & \\
\begin{block}{[c@{\color{lightgray}\vline}c@{}c@{}c@{\color{lightgray}\vline}c@{}c@{}c@{\color{lightgray}\vline}c]c@{\!}}
0 & 1 & 1 & 1 & 1 & 1 & 1 & 1 & \op{l,l,l} \\
\cline{1-9}
0 & 0 & 1 & 1 & 1 & 1 & 1 & 1 & \op{l,l,r} \\
0 & 1 & 0 & 1 & 1 & 1 & 1 & 1 & \op{l,r,l} \\
0 & 1 & 1 & 0 & 1 & 1 & 1 & 1 & \op{r,l,l} \\
\cline{1-9}
0 & 0 & 0 & 1 & 0 & 1 & 1 & 1 & \op{l,r,r} \\
0 & 0 & 1 & 0 & 1 & 0 & 1 & 1 & \op{r,l,r} \\
0 & 1 & 0 & 0 & 1 & 1 & 0 & 1 & \op{r,r,l} \\
\cline{1-9}
0 & 0 & 0 & 0 & 0 & 0 & 0 & 0 & \op{r,r,r} \\
\end{block}
\end{blockarray}
\]
(Note that we have gathered together rows and columns according to the number of~$\op{r}$ signs of the corresponding operation of~$\Operations_k$ in the perspective of \cref{lem:sumsMk}.
Would we have ordered the rows and columns in the product order, we would see that~$\mat{M}_k^\parallel$ is the all $1$ matrix minus the incidence matrix of the boolean lattice.
Unfortunately, we are not able to use this observation for \cref{conj:matrices1,conj:matrices2}.)

\item For the series $k$-citelangis operad, the transition matrix~$\mat{M}_k^\series$ is the $(\Operations_k \times \Operations_k)$-matrix whose $(\operation[a], \operation[b])$-coefficient is given by
\[
(\mat{M}_k^\series)_{\operation[a]\operation[b]} =
\begin{cases}
1 & \text{ if there is~$i > |\operation[a]|_{\op{r}}$ such that $\operation[b]_i = {\op{r}}$,} \\
0 & \text{ otherwise.}
\end{cases}
\]
for any~$\operation[a] \in \Operations_k$ and~$\operation[b] \in \Operations_k$.
For example,
\[
\qquad\quad
\mat{M}_1^\series = \begin{blockarray}{c@{\;\color{lightgray}\vline\;}cc}
\op{l} & \op{r} & \\
\begin{block}{[c@{\;\color{lightgray}\vline\;}c]c@{\!}}
0 & 1 & \op{l} \\
\cline{1-3}
0 & 0 & \op{r} \\
\end{block}
\end{blockarray}
\qquad
\mat{M}_2^\series = \begin{blockarray}{c@{\,\color{lightgray}\vline\,}c@{\;}c@{\,\color{lightgray}\vline\,}cc}
\rotatebox{70}{$\op{l,l}$} & \rotatebox{70}{$\op{l,r}$} & \rotatebox{70}{$\op{r,l}$} & \rotatebox{70}{$\op{r,r}$} & \\
\begin{block}{[c@{\,\color{lightgray}\vline\,}c@{}c@{\,\color{lightgray}\vline\,}c]c@{\!}}
0 & 1 & 1 & 1 & \op{l,l} \\
\cline{1-5}
0 & 1 & 0 & 1 & \op{l,r} \\
0 & 1 & 0 & 1 & \op{r,l} \\
\cline{1-5}
0 & 0 & 0 & 0 & \op{r,r} \\
\end{block}
\end{blockarray}
\qquad
\mat{M}_3^\series = \begin{blockarray}{c@{\color{lightgray}\vline}c@{}c@{}c@{\color{lightgray}\vline}c@{}c@{}c@{\color{lightgray}\vline}cc}
\rotatebox{70}{$\op{l,l,l}$} & \rotatebox{70}{$\op{l,l,r}$} & \rotatebox{70}{$\op{l,r,l}$} & \rotatebox{70}{$\op{r,l,l}$} & \rotatebox{70}{$\op{l,r,r}$} & \rotatebox{70}{$\op{r,l,r}$} & \rotatebox{70}{$\op{r,r,l}$} & \rotatebox{70}{$\op{r,r,r}$} & \\
\begin{block}{[c@{\color{lightgray}\vline}c@{}c@{}c@{\color{lightgray}\vline}c@{}c@{}c@{\color{lightgray}\vline}c]c@{\!}}
0 & 1 & 1 & 1 & 1 & 1 & 1 & 1 & \op{l,l,l} \\
\cline{1-9}
0 & 1 & 1 & 0 & 1 & 1 & 1 & 1 & \op{l,l,r} \\
0 & 1 & 1 & 0 & 1 & 1 & 1 & 1 & \op{l,r,l} \\
0 & 1 & 1 & 0 & 1 & 1 & 1 & 1 & \op{r,l,l} \\
\cline{1-9}
0 & 1 & 0 & 0 & 1 & 1 & 0 & 1 & \op{l,r,r} \\
0 & 1 & 0 & 0 & 1 & 1 & 0 & 1 & \op{r,l,r} \\
0 & 1 & 0 & 0 & 1 & 1 & 0 & 1 & \op{r,r,l} \\
\cline{1-9}
0 & 0 & 0 & 0 & 0 & 0 & 0 & 0 & \op{r,r,r} \\
\end{block}
\end{blockarray}
\]
\setlength\arrayrulewidth{1pt}\arrayrulecolor{black}
\end{enumerate}

These matrices seem particularly relevant and closely related to Eulerian polynomials.
Let us mention in particular the following conjectures.

\begin{conjecture}
\label{conj:matrices1}
The minimal and characteristic polynomials of the matrix $\mat{M}_k^\parallel$ are
\[
\begin{array}{c@{\quad}c@{\quad}c}
\pi_{\mat{M}_k^\parallel}(t) = (-1)^{k+1} \cdot t^2 \cdot (t+1)^{k-1} \cdot \EulPol{k}(-t),
& &
\chi_{\mat{M}_k^\parallel}(t) = (-1)^{k+1} \cdot t^2 \cdot (t+1)^{\EulNum{k}{1}} \cdot \EulPol{k}(-t),
\\
\pi_{\mat{M}_k^\series}(t) = (-1)^{k+1} \cdot t^2 \cdot \EulPol{k}(-t),
& \text{and} &
\chi_{\mat{M}_k^\series}(t) = (-1)^{k+1}\cdot t^{2+\EulNum{k}{1}} \cdot \EulPol{k}(-t).
\end{array}
\]
\end{conjecture}

For example, the minimal and characteristic polynomials of the matrices above are
\[
\begin{array}{c@{\qquad}c@{\qquad}c}
\pi_{\mat{M}_1^\parallel}(t) = t^2, & \pi_{\mat{M}_2^\parallel}(t) = t^4 - t^2, & \pi_{\mat{M}_3^\parallel}(t) = t^6 - 2t^5 - 6t^4 - 2t^3 + t^2, \\
\chi_{\mat{M}_1^\parallel}(t) = t^2, & \chi_{\mat{M}_2^\parallel}(t) = t^4 - t^2, & \chi_{\mat{M}_3^\parallel}(t) = t^8 - 9 t^6 - 16 t^5 - 9 t^4 + t^2,
\end{array}
\]
and
\[
\begin{array}{c@{\qquad}c@{\qquad}c}
\pi_{\mat{M}_1^\series}(t) = t^2, & \pi_{\mat{M}_2^\series}(t) = t^3 - t^2, & \pi_{\mat{M}_3^\series}(t) = t^4 - 4 t^3 + t^2, \\
\chi_{\mat{M}_1^\series}(t) = t^2, & \chi_{\mat{M}_2^\series}(t) = t^4 - t^3, & \chi_{\mat{M}_3^\series}(t) = t^8 - 4 t^7 + t^6.
\end{array}
\]

\medskip
On the examples of matrices~$\mat{M}_k^\parallel$ and~$\mat{M}_k^\series$ above, we have gathered together rows and columns according to the number of~$\op{r}$ signs of the corresponding operation of~$\Operations_k$.
We observe that when we sum the columns of~$\mat{M}_k^\parallel$ or~$\mat{M}_k^\series$ according to these groups, all entries in a group of rows are identical.
The following lemma gives a more precise statement.

\begin{lemma}
\label{lem:sumsMk}
For~$0 \le i, j \le k$ and~$\operation[a] \in \Operations_k$ with~$|\operation[a]|_{\op{r}} = i$, we have~${\displaystyle \sum_{\substack{\operation[b] \in \Operations_k \\ |\operation[b]|_{\op{r}} = j}} (\mat{M}_k)_{\operation[a]\operation[b]} = \binom{k}{j} - \binom{i}{j}}$.
\end{lemma}

\begin{proof}
Let~$0 \le i, j \le k$ and consider~$\operation[a], \operation[b] \in \Operations_k$ with~$|\operation[a]|_{\op{r}} = i$ and~$|\operation[b]|_{\op{r}} = j$.
From the definitions of~$(\mat{M}_k^\parallel)_{\operation[a]\operation[b]}$ and~$(\mat{M}_k^\series)_{\operation[a]\operation[b]}$ above, we derive that
\begin{enumerate}
\item $(\mat{M}_k^\parallel)_{\operation[a]\operation[b]} = 0$ $\iff$ the $j$ positions of~$\op{r}$ in~$\operation[b]$ are chosen among the $i$ positions of~$\op{r}$~in~$\operation[a]$,
\item $(\mat{M}_k^\series)_{\operation[a]\operation[b]} = 0$ $\iff$ the $j$ positions of~$\op{r}$ in~$\operation[b]$ are chosen among the first $i$ positions.
\end{enumerate}
Therefore, for a fixed operation~$\operation[a] \in \Operations_k$ with~$|\operation[a]|_{\op{r}} = i$, among the~$\binom{k}{j}$ operations~$\operation[b] \in \Operations_k$ with~$|\operation[b]|_{\op{r}} = j$, there are~$\binom{i}{j}$ operations such that~$(\mat{M}_k)_{\operation[a]\operation[b]} = 0$ and~$\binom{k}{j} - \binom{i}{j}$ operations such that~${(\mat{M}_k)_{\operation[a]\operation[b]} = 1}$, in both the series and parallel situations.
The result immediately follows.
\end{proof}

This motivates to introduce the $(k+1) \times (k+1)$-matrix~$\mat{N}_k$ defined by
\[
\mat{N}_k 
\eqdef
\begin{bmatrix}
\displaystyle
\binom{k}{j} - \binom{i}{j}
\end{bmatrix}_{0 \le i,j \le k}
\]

\begin{lemma}
\label{lem:powersMkNk}
For any~$p \in \N$, we have~$\displaystyle \sum_{\operation[a], \operation[b] \in \Operations_k} (\mat{M}_k^p)_{\operation[a]\operation[b]} = \sum_{0 \le i,j \le k} (\mat{N}_k^p)_{ij} \binom{k}{i}$.
\end{lemma}

\begin{proof}
We denote by~$\mat{1}_n$ the vector with~$n$ entries all equal to~$1$, and by~$\mat{I}_n$ the $(n \times n)$-identity matrix.
Consider the vector~$\mat{B}_k$, the $\big( (k+1) \times \Operations_k \big)$-matrix~$\mat{L}_k$ and the $\big( \Operations_k \times (k+1) \big)$-matrix~$\mat{R}_k$ whose coefficients are given by
\[
(\mat{B}_k)_i = \binom{k}{i},
\qquad
(\mat{L}_k)_{i \operation[a]} = \delta_{|\operation[a]|_{\op{r}} = i} \binom{k}{i}^{-1}
\qqandqq
(\mat{R}_k)_{\operation[b] j} = \delta_{|\operation[b]|_{\op{r}} = j}.
\]
Observe that
\[
\transpose{\mat{B}_k} \cdot \mat{L}_k = \transpose{\mat{1}_{\Operations_k}}, 
\qquad
\mat{R}_k \cdot \mat{1}_{k+1} = \mat{1}_{\Operations_k}
\qqandqq
\mat{L}_k \cdot \mat{R}_k = \mat{I}_{k+1}.
\]
Moreover, \cref{lem:sumsMk} affirms that
\[
\mat{R}_k \cdot \mat{N}_k = \mat{M}_k \cdot \mat{R}_k.
\]
This imply that
\[
\mat{N}_k^p = \mat{L}_k \cdot \mat{M}_k^p \cdot \mat{R}_k
\]
for any~$p \in \N$.
Indeed, it holds for~$p = 0$ since~$\mat{L}_k \cdot \mat{R}_k = \mat{I}_{k+1}$, and we obtain by induction that~$\mat{N}_k^{p+1} = \mat{N}_k^p \cdot \mat{N}_k = \mat{L}_k \cdot \mat{M}_k^p \cdot \mat{R}_k \cdot \mat{N}_k = \mat{L}_k \cdot \mat{M}_k^p \cdot \mat{M}_k \cdot \mat{R}_k = \mat{L}_k \cdot \mat{M}_k^{p+1} \cdot \mat{R}_k$.
Therefore,
\[
\sum_{0 \le i,j \le k} \!\! (\mat{N}_k^p)_{ij} \binom{k}{i}
=
\transpose{\mat{B}_k} \cdot \mat{N}_k^p \cdot \mat{1}_{k+1}
=
\transpose{\mat{B}_k} \cdot \mat{L}_k \cdot \mat{M}_k^p \cdot \mat{R}_k \cdot \mat{1}_{k+1}
=
\transpose{\mat{1}_{\Operations_k}} \cdot \mat{M}_k^p \cdot \mat{1}_{\Operations_k}
=
\!\! \sum_{\operation[a], \operation[b] \in \Operations_k} \!\! (\mat{M}_k^p)_{\operation[a]\operation[b]}.
\qedhere
\]
\end{proof}

We thus derive from \cref{prop:powersMk,lem:powersMkNk} the following statement which provides a more efficient way to compute approximations of the generating series~$\cR_k(t)$.

\begin{corollary}
\label{coro:powersNk}
The generating series~$\cR_k(t)$ is given by  $\displaystyle \cR_k(t) = 1 + \sum_{p \ge 1} \, \sum_{0 \le i,j \le k} (\mat{N}_k^{p-1})_{ij} \binom{k}{i} t^p.$
\end{corollary}

For example, the first such matrices are given by
\[
\mat{N}_1 =
\begin{bmatrix}
	0 & 1 \\
	0 & 0
\end{bmatrix}
\qquad
\mat{N}_2 =
\begin{bmatrix}
	0 & 2 & 1 \\
	0 & 1 & 1 \\
	0 & 0 & 0
\end{bmatrix}
\qquad
\mat{N}_3 =
\begin{bmatrix}
	0 & 3 & 3 & 1 \\
	0 & 2 & 3 & 1 \\
	0 & 1 & 2 & 1 \\
	0 & 0 & 0 & 0
\end{bmatrix}
\qquad
\mat{N}_4 =
\begin{bmatrix}
	0 & 4 & 6 & 4 & 1 \\
	0 & 3 & 6 & 4 & 1 \\
	0 & 2 & 5 & 4 & 1 \\
	0 & 1 & 3 & 3 & 1 \\
	0 & 0 & 0 & 0 & 0
\end{bmatrix}
\]

Again, these matrices seem closely related to Eulerian polynomials as illustrated by the following conjecture.

\begin{conjecture}
\label{conj:matrices2}
The minimal polynomial~$\pi_{\mat{N}_k}(t)$ and the characteristic polynomial~$\chi_{\mat{N}_k}(t)$ of the matrix~$\mat{N}_k$ are multiples of the Eulerian polynomial:
\[
\pi_{\mat{N}_k}(t) = \chi_{\mat{N}_k}(t) = (-1)^{k+1} \cdot t^2 \cdot \EulPol{k}(-t).
\]
\end{conjecture}

For example, the minimal and characteristic polynomials of the matrices above are
\begin{gather*}
\pi_{\mat{N}_1}(t) = t^2, \quad \pi_{\mat{N}_2}(t) = t^3 - t^2, \quad \pi_{\mat{N}_3}(t) = t^4 - 4t^3 + t^2, \quad \pi_{\mat{N}_4}(t) = t^5 - 11t^4 + 11t^3 - t^2, \\
\chi_{\mat{N}_1}(t) = t^2, \quad \chi_{\mat{N}_2}(t) = t^3 - t^2, \quad \chi_{\mat{N}_3}(t) = t^4 - 4t^3 + t^2, \quad \chi_{\mat{N}_4}(t) = t^5 - 11t^4 + 11t^3 - t^2.
\end{gather*}


\newpage
\section{Combinatorial models and actions}
\label{sec:actions}

It is well known that the shuffle algebra can be endowed with the structure of a dendriform algebra~\cite{Loday-dialgebras} defined for the words~$xX$ and~$yY$ by
\[
xX \op{l} yY = x \big( X \shuffle yY \big)
\qqandqq
xX \op{r} yY = y \big( xX \shuffle Y \big).
\]
Replacing the shuffle by the shifted shuffle, one endows similarly the algebra~$\FQSym$ of permutations with the structure of a dendriform algebra.
The resulting dendriform algebra is known to be free~\cite{Foissy, Vong}, and the dendriform subalgebra of~$\FQSym$ generated by the permutation~$1$ is known to be the free dendriform algebra on one generator~\cite{LodayRonco}.

The objective of this section is to show similar actions of the citelangis operads defined in \cref{sec:citelangisOperads} on a certain family of permutations and posets.
Some of these actions will result in free citelangis algebras, and will thus lead to natural constructions of free citelangis algebras on one generator, providing combinatorial models to index the bases for the citelangis operads.

As the tidy and messy parallel $k$-citelangis operads are Manin powers of the dendriform and twisted duplicial operads (see \cref{subsec:ManinProducts,prop:messyParallelCitelangisManinProduct,prop:tidyParallelCitelangisManinProduct}), they naturally act on $k$-tuples of permutations.
Nevertheless, this action is not really interesting here as it does not extend in the series situation.
It is much more relevant to define actions of the parallel $k$-citelangis operads on permutations, rather than on $k$-tuple of permutations.
Such actions are already known for~$k = 1$ ($\Dend$~operad) or~$k = 2$ ($\Quad$~operad), but are unfortunately only defined in these two situations because a permutation has only two ends!
In contrast, we will define actions of the series $k$-citelangis operads for any~$k \ge 1$.

Even though it is limited to the case~$k \le 2$ and already partially known, we discuss the case of the parallel $k$-citelangis operads in \cref{subsec:actionParallelCitalangisOperads} as it serves as a prototype for the series situation developed in \cref{subsec:actionSeriesCitalangisOperads}.
Let us stress out that, while the structures, phrasings, and notations of \cref{subsec:actionParallelCitalangisOperads,subsec:actionSeriesCitalangisOperads} are voluntarily carbon copies of each other, the definitions of the (messy and tidy) parallel and series compositions of multipermutations and multiposets differ, and the parameter $k$ is restricted to~$2$ in \cref{subsec:actionParallelCitalangisOperads} and not in \cref{subsec:actionSeriesCitalangisOperads}.

We refer the reader to \cref{subsec:multiobjects} for definitions and notations on multipermutations and multiposets.

A road map through combinatorial models and actions for parallel/series messy/tidy citelangis operads is given in \cref{table:roadMap}.
Roughly speaking, both in the parallel situation with~$k = 2$ and in the series situation for arbitrary~$k$, we define an operad on suitable $k$-posets, and two operads on $k$-permutations called messy and tidy, such that
\begin{itemize}
\item the linear extensions map defines a morphism from the poset operad to the messy operad,
\item the tidy composition of two $k$-permutations is the lexicographically minimal $k$-permutation appearing in their messy composition.
\end{itemize}
This enables us to index the bases of the corresponding citanlegis operads by combinatorial objects, namely $k$-permutations satisfying certain decomposition properties. 

\begin{table}[t]
	\centerline{
	\renewcommand{\arraystretch}{1.3}
	\begin{tabular}{c|c|c}
	& Parallel (only $k = 2$) & Series (any $k$)
	\\
	\hline
	\multirow{3}{*}{\rotatebox{90}{poset\hspace{2.5cm}}}
	&
	parallel composition of bounded $2$-posets
	&
	series composition of $k$-rooted $k$-posets
	\\[-.15cm]
	&
	\cref{def:parallelBoundedPosetOperad,fig:parallelPosetOperad}
	&
	\cref{def:seriesRootedPosetOperad,fig:seriesPosetOperad}
	\\
	&
	\begin{tikzpicture}[baseline={(current bounding box.center)}, blue]
		\node (1a) at (0,.5) {$1$};
		\node (1b) at (0,1.5) {$1$};
		\node (2a) at (1,0) {$2$};
		\node (2b) at (1,2) {$2$};
		\node (3a) at (2,.5) {$3$};
		\node (3b) at (2,1.5) {$3$};
		\draw (1a) -- (1b);
		\draw (1b) -- (2b);
		\draw (2a) -- (1a);
		\draw (2a) -- (3a);
		\draw (3a) -- (3b);
		\draw (3b) -- (2b);
	\end{tikzpicture}
	$\;\posetParallelCirc{3}\;$
	\begin{tikzpicture}[baseline={(current bounding box.center)}, red]
		\node (1a) at (0,1) {$1$};
		\node (1b) at (0,2) {$1$};
		\node (2a) at (1,0) {$2$};
		\node (2b) at (1,1) {$2$};
		\draw (1a) -- (1b);
		\draw (2a) -- (1a);
		\draw (2a) -- (2b);
		\draw (2b) -- (1b);
	\end{tikzpicture}
	$\;=\;$
	\begin{tikzpicture}[baseline={(current bounding box.center)}]
		\node (1aM) at (0,.5) {\blue $1$};
		\node (1bM) at (0,2.5) {\blue $1$};
		\node (2aM) at (1,0) {\blue $2$};
		\node (2bM) at (1,3) {\blue $2$};
		\node (1aN) at (2,1.5) {\red $3$};
		\node (1bN) at (2,2.5) {\violet $3$};
		\node (2aN) at (3,.5) {\violet $4$};
		\node (2bN) at (3,1.5) {\red $4$};
		\draw [blue] (1aM) -- (1bM);
		\draw [red] (1aN) -- (1bN);
		\draw [blue] (1bM) -- (2bM);
		\draw [blue] (1bN) -- (2bM);
		\draw [blue] (2aM) -- (1aM);
		\draw [blue] (2aM) -- (2aN);
		\draw [red] (2aN) -- (1aN);
		\draw [red] (2aN) -- (2bN);
		\draw [red] (2bN) -- (1bN);
	\end{tikzpicture}
	&
	\begin{tikzpicture}[baseline={(current bounding box.center)},level/.style={sibling distance=1cm, level distance = .7cm}, blue]
		\node {$3$} [grow' = up]
			child {node {$3$}
				child {node {$1$}
					child {node {$2$}
						child {node {$1$}}
						child {node {$2$}}
					}
				}
				child {node {$4$}
					child {node {$4$}}
				}
			}
		;
	\end{tikzpicture}
	$\;\posetSeriesCirc{1}\;$
	\begin{tikzpicture}[baseline={(current bounding box.center)},level/.style={sibling distance=1cm, level distance = .7cm}, red]
		\node {$2$} [grow' = up]
			child {node {$1$}
				child {node {$1$}}
				child {node {$2$}}
			}
		;
	\end{tikzpicture}
	$\;=\;$
	\begin{tikzpicture}[baseline={(current bounding box.center)},level/.style={sibling distance=1cm, level distance = .7cm}, blue]
		\node {$4$} [grow' = up]
			child {node {$4$}
				child {node {$1$}
					child {node {\violet $3$}
						child {node {$1$}}
						child [violetEdge] {node {\violet $2$}
							child [redEdge] {node {\red $2$}}
							child [redEdge] {node {\red $3$}}						
						}
					}
				}
				child {node {$5$}
					child {node {$5$}}
				}
			}
		;
	\end{tikzpicture}
	\\[.5cm]
	&
	$\LinExt \Big\downarrow$\cref{prop:morphismParallelLinExt}
	&
	$\LinExt \Big\downarrow$\cref{prop:morphismSeriesLinExt}
	\\[.5cm]
	\multirow{3}{*}{\rotatebox{90}{messy}}
	&
	messy parallel composition of $2$-permutations
	&
	messy series composition of $k$-permutations
	\\[-.15cm]
	&
	\cref{def:compositionMessyParallelZinbiel}
	&
	\cref{def:compositionMessySeriesZinbiel}
	\\
	&
	$3{\blue1422}3{\blue41} \messyParallelCirc{3} 3{\red1312}2 =$
	&
	${\blue31}2{\blue3}2{\blue144} \messySeriesCirc{2} 31{\red3122} =$
	\\[-.1cm]
	&
	$5{\red3534}{\blue1622}4{\blue61} + \dots + 5{\blue1622}{\red3534}4{\blue61}$
	&
	${\blue51}4{\blue5}2{\red4233}{\blue1}{\blue66} + \dots +  {\blue5}3{\blue454}1{\blue66}{\red3122}$
	\\[.5cm]
	&
	$\LexMin\Big\downarrow$\cref{lem:morphismParallelLexMin}\hspace*{.8cm}
	&
	$\LexMin\Big\downarrow$\cref{lem:morphismSeriesLexMin}\hspace*{.8cm}
	\\[.5cm]
	\multirow{3}{*}{\rotatebox{90}{tidy}}
	&
	tidy parallel composition of $2$-permutations
	&
	tidy series composition of $k$-permutations
	\\[-.15cm]
	&
	\cref{prop:combinatorialModelCompositionsTidyParallelCitelangis,def:compositionTidyParallelZinbiel}
	&
	\cref{prop:combinatorialModelCompositionsTidySeriesCitelangis,def:compositionTidySeriesZinbiel}
	\\
	&
	$3{\blue1}|{\blue422}3{\blue41} \tidyParallelCirc{3} 3{\red1312}2 = 5{\blue1}{\red3534}{\blue622}4{\blue61}$
	&
	${\blue31}2{\blue3}2{\blue1}|{\blue44} \tidySeriesCirc{2} 31{\red3122} = {\blue51}4{\blue5}2{\blue1}{\red4233}{\blue66}$
	\\[.5cm]
	\multirow{3}{*}{\rotatebox{90}{bases\hspace*{-1cm}}}
	&
	fully bounded cuttable $2$-permutations
	&
	fully $k$-rooted cuttable $k$-permutations
	\\[-.15cm]
	&
	\cref{def:fullyBoundedCuttable,prop:fullyBoundedCuttableBasisMessy,prop:fullyBoundedCuttableBasisTidy}
	&
	\cref{def:fullykrootedCuttable,prop:fullyRootedCuttableBasisMessy,prop:fullyRootedCuttableBasisTidy}
	\end{tabular}
	}
	\caption{A road map through combinatorial models and actions for parallel/series messy/tidy citelangis operads.}
	\label{table:roadMap}
\end{table}


\subsection{Actions of parallel citelangis operads}
\label{subsec:actionParallelCitalangisOperads}

We now discuss the action of parallel $2$-citelangis operads~$\tidyCitelangisParallel[\op{l,r}]$ and~$\messyCitelangisParallel[2]$ on certain $2$-permutations and $2$-posets.
We start with some combinatorial considerations on certain $2$-permutations.


\subsubsection{Bounded cuts in \texorpdfstring{$2$}{2}-permutations}
\label{subsubsec:boundedpermutations}

We introduce first an intruiguing class of $2$-permutations.
These permutations will be instrumental in studying the action of the tidy parallel $2$-citelangis operad, but we also believe that they deserve further study for their own sake.

\begin{definition}
\label{def:boundedCutMultipermutation}
Let~$\sigma \in \Perm_2(n)$ be a $2$-permutation of degree~$n$ and~$\gamma \in [n-1]$.
We say that~$\gamma$ is a \defn{bounded cut} of~$\sigma$ if we can write~$\sigma = f \mu \nu l$ where~$f$ and~$l$ are the first and last letter of~$\sigma$ respectively, while~$\mu$ and~$\nu$ are words such that~$\mu_i \le \gamma < \nu_j$ for all~$i \in [|\mu|]$ and~$j \in [|\nu|]$.
We denote by~$\bcuts(\sigma)$ the set of bounded cuts of~$\sigma$.
We say that the $2$-permutation~$\sigma$ is \defn{bounded cuttable} if it admits a bounded cut, and \defn{bounded uncuttable} if it admits no bounded cut.
\end{definition}

For example, $1$ is a bounded cut of the $2$-permutation~$31123424$, while the $2$-permutation~$31421324$ is bounded uncuttable.
Note that there is no condition on the sizes of~$\mu$ and~$\nu$ in \cref{def:boundedCutMultipermutation} (they can be empty, or of sizes which are not multiples of~$2$).
We now observe that bounded cuts behave properly with restrictions in the sense of of \cref{def:restrictionMultipermutation}.

\begin{lemma}
\label{lem:boundedCutsRestriction}
Consider~$L \subseteq [n]$ and~$\gamma \in [\min(L), \max(L)-1]$, and let~$\gamma^{|L} \eqdef |[\gamma] \cap L|$ denote the number of elements of~$L$ between~$1$ and~$\gamma$.
If~$\gamma$ is a bounded cut of a $2$-permutation~$\sigma \in \Perm_2(n)$, then $\gamma^{|L}$ is a bounded cut of its restriction~$\sigma^{|L}$.
\end{lemma}

\begin{proof}
Since~$\gamma$ is a bounded cut of~$\sigma$, we can write~$\sigma = f \mu \nu l$ with~$\mu_i \le \gamma < \nu_j$ for all~${i \in [|\mu|]}$ and~$j \in [|\nu|]$.
Then~$\sigma^{|L} = f^{|L} \mu^{|L} \nu^{|L} l^{|L}$.
Moreover~$\mu^{|L}_i \le \gamma^{|L} < \nu^{|L}_j$ for any~$i \in [|\mu^{|L}|]$ and~${j \in [|\nu^{|L}|]}$.
Therefore, $\gamma^{|L}$ is a bounded cut of~$\sigma^{|L}$.
\end{proof}

Note that the reverse statement is wrong.
Consider for instance the $2$-permutation~$\sigma = 213123$ and the subset~$L = \{1,2\}$.
Then~$1$ is not a bounded cut of~$\sigma$, but $1^{|L} = 1$ is a bounded cut of~$\sigma^{|L} = 2112$.

\begin{proposition}
\label{prop:equivalenceFullyBoundedCuttable}
The following conditions are equivalent for a $2$-permutation of degree~$n$:
\begin{enumerate}[(i)]
\item its restriction to any interval of~$[n]$ of size at least $2$ is bounded cuttable,
\item its restriction to any subset of~$[n]$ of size at least $2$ is bounded cuttable.
\end{enumerate}
\end{proposition}

\begin{proof}
Assume that~$\sigma \in \Perm_2(n)$ satisfies~(i).
Let~$L \subseteq [n]$ with~$|L| \ge 2$.
Since~$|L| \ge 2$, we have~$\min(L) < \max(L)$ so that the restriction~$\sigma^{|[\min(L), \max(L)]}$ admits a bounded cut~$\gamma$ with~$\min(L) \le \gamma < \max(L)$.
By \cref{rem:relationOperatorsMultipermutations}, the restriction~$\sigma^{|L}$ is just the restriction of $\sigma^{|[\min(L), \max(L)]}$ to~$\bar L \eqdef \set{\ell-\min(L)+1}{\ell \in L}$.
By \cref{lem:boundedCutsRestriction}, $\sigma^{|L}$ admits a bounded cut~$\gamma^{|\bar L}$ with~${1 \le \gamma^{|\bar L} \le |L|}$.
Therefore, $\sigma$ satisfies~(ii).
The reverse implication is obvious.
\end{proof}

\begin{definition}
\label{def:fullyBoundedCuttable}
A $2$-permutation is \defn{fully bounded cuttable} if it satisfies the equivalent conditions of \cref{prop:equivalenceFullyBoundedCuttable}.
\end{definition}

For example, the $2$-permutation~$31123424$ is bounded cuttable but not fully bounded cuttable (since the restriction to the interval~$[2,4]$ is not bounded cuttable).
In contrast, the $2$-permu\-tation~$5211332454$ is fully bounded cuttable.
Note that~$11$ is fully bounded cuttable as there is no subset of size at least~$2$.

\medskip
There is also a characterization of fully bounded cuttable $2$-permutations in terms of pattern avoidance.
We start by the following technical observation.

\begin{lemma}
\label{lem:boundedUncuttableImpliesPattern}
A bounded uncuttable $2$-permutation of degree at least~$2$ contains a pattern~$b \cdot c \cdot a \cdot b'$ with~$a \le b, b' \le c$.
\end{lemma}

\begin{proof}
Suppose by contradiction that a $2$-permutation~$\sigma = \sigma_1 \dots \sigma_{2n}$ contains no such pattern.

Assume first that~$\sigma_1 = \sigma_{2n} \defeq v$.
Then for any~$u \le v < w$, both values~$u$ must appear before both values~$w$, otherwise we would have the pattern~$v \cdot w \cdot u \cdot v$.
Therefore, both~$v-1$ and~$v$ (resp.~$1$, resp.~$n-1$) are bounded cuts of~$\sigma$ if~$1 < v < n$ (resp.~if~$v = 1$, resp.~if~$v = n$).

Assume now that~$r \eqdef \sigma_1 < \sigma_{2n} \defeq s$.
Let~$p$ denote the position of the other~$s$ of~$\sigma$, \ie $p < 2n$ and~$\sigma_p = s$.
We distinguish two cases:
\begin{enumerate}[(i)]
\item Assume first that no value of~$\sigma$ appears both before and after the position~$p$. This implies that for~$w > s$, both values~$w$ appear after the position~$p$, as otherwise we would have a forbidden pattern~$w \cdot w \cdot s \cdot s$. Let~$v$ be the minimal value that appears after the position~$p$. Note that~$v > 1$ as otherwise~$r \cdot s \cdot 1 \cdot s$ is a forbidden pattern. We claim that~$v-1$ is a bounded cut of~$\sigma$. Indeed, for any~$u < v$, both values~$u$ appear before the position~$p$ by definition. Moreover, for any~$w \ge v$ distinct from~$s$, both values~$w$ appear after the position~$p$, as otherwise we would have~$v \le w < s$ and thus the forbidden pattern~$w \cdot s \cdot v \cdot s$.
\item Assume now that there is a value~$t$ that appears both before and after the position~$p$. Note that it imposes that~$s < t$ as otherwise we would have the forbidden pattern~$t \cdot s \cdot t \cdot s$. Let~$q$ denote the position of the first value~$t$. We can assume without loss of generality that $q$ is the minimal position of a value that appears both before and after~$p$. Let~$v$ be the minimal value that appears after the position~$q$. Note that~$v > 1$ as otherwise~$r \cdot t \cdot 1 \cdot s$ is a forbidden pattern. We claim that~$v-1$ is a bounded cut of~$\sigma$. Indeed, for any~$u < v$, both values~$u$ appear before the position~$q$ by definition. Moreover, for any~$w \ge v$ distinct from~$s$, both values~$w$ appear after the position~$q$. Otherwise, by minimality of the position~$q$, the second value~$w$ could not be on the right of~$p$. Therefore, either~$w > s$ and we have the forbidden pattern~$w \cdot w \cdot s \cdot s$, or~$v \le w < s < t$ and we have the forbidden pattern~$w \cdot t \cdot v \cdot s$.
\end{enumerate}

Finally, the case~$\sigma_1 > \sigma_{2n}$ is very similar to the previous one and left to the reader.
\end{proof}

\begin{proposition}
\label{prop:fullyBoundedCuttablePatterns}
A $2$-permutation is fully bounded cuttable if and only if it avoids the pattern ${b \cdot c \cdot a \cdot b'}$ with~$a \le b , b' \le c$.
\end{proposition}

\begin{proof}
There is nothing to prove for the $2$-permutation~$11$.
If a $2$-permutation~$\sigma$ contains a pattern~${b \cdot c \cdot a \cdot b'}$ with~$a \le b , b' \le c$, then its restriction~$\sigma^{|[a,c]}$ is bounded uncuttable, so that $\sigma$ is not fully bounded cuttable.
Conversely, if~$\sigma$ avoids the pattern~${b \cdot c \cdot a \cdot b'}$ with~$a \le b , b' \le c$, then all its restrictions do, so that they are all bounded cuttable by \cref{lem:boundedUncuttableImpliesPattern}.
\end{proof}

We derive in particular the following observation from \cref{prop:fullyBoundedCuttablePatterns}.

\begin{corollary}
\label{coro:fullyBoundedCuttableStructure}
If~$\sigma \in \Perm_2(n)$ is fully bounded cuttable and~$i \in [n]$, the factor of~$\sigma$ located in between the two occurences of~$i$ decomposes into~$\mu \nu$ where~$\mu_p < i < \nu_q$ for all~$p \in [|\mu|]$ and~$q \in [|\nu|]$.
\end{corollary}

\pagebreak
Finally, we observe that fully bounded cuttable $2$-permutations form a pattern class.

\begin{theorem}
\label{thm:fullyBoundedCuttablepermutationClass}
The set of fully bounded cuttable $2$-permutations is a $2$-permutation class: for any fully bounded cuttable $2$-permutation~${\sigma \in \Perm_2(n)}$ and any~$L \subseteq [n]$, the restriction~$\sigma^{|L}$ is fully bounded cuttable.
\end{theorem}

\begin{proof}
Assume that~$\sigma \in \Perm_2(n)$ is fully bounded cuttable and that~$L = \{\ell_1, \dots, \ell_{|L|}\} \subseteq [n]$.
For any~$X \subseteq [|L|]$, the restriction~$(\sigma^{|L})^{|X}$ coincides with the restriction~$\sigma^{|\set{\ell_x}{x \in X}}$ by \cref{rem:relationOperatorsMultiposets}, and is thus bounded cuttable.
Therefore, $\sigma^{|L}$ is fully bounded cuttable.

Another approach would be via \cref{prop:fullyBoundedCuttablePatterns} since patterns are preserved by restriction.
\end{proof}


\subsubsection{Action of~\texorpdfstring{$\tidyCitelangisParallel[\prec\!\succ]$}{TCit<>} on words and permutations}
\label{subsubsec:actionTidyParallelSignaletic}

We show in this section that $\FQSym_2$ can be endowed with a $\op{l,r}$-tidy parallel citelangis structure, and that the resulting $\op{l,r}$-tidy parallel citelangis algebra is free. Therefore, the free $\op{l,r}$-tidy parallel citelangis subalgebra generated by the $2$-permutation~$11$ provides a combinatorial model for the basis for the $\op{l,r}$-tidy parallel citelangis operad.

\paraul{Action on words}
We first observe that for any non-empty alphabet~$\alphabet$, the free algebra~$\alphabet^{\ge 2}$ can be endowed with the structure of a $\op{l,r}$-tidy parallel citelangis algebra.

\begin{definition}
\label{def:tidyCitelangisParallelActionWords}
We define the action of the four $\op{l,r}$-tidy parallel citelangis operations of $\op{l,l}$, $\op{l,r}$, $\op{r,l}$ and $\op{r,r}$ on any two words~$xXx'$ and~$yYy'$ by
\begin{gather*}
xXx' \op{l,l} yYy' = xXyYy'x'
\\
xXx' \op{l,r} yYy' = xXx'yYy'
\\
xXx' \op{r,l} yYy' = yxXYy'x'
\\
xXx' \op{r,r} yYy' = yxXx'Yy'.
\end{gather*}
In other words, we choose the first and last letters of the result among the first and last letters of~$xXx'$ or~$yYy'$ depending on the operation, and we concatenate the remaining factors of~$xXx'$~and~$yYy'$.
\end{definition}

\begin{proposition}
\label{prop:tidyCitelangisParallelWords}
The free algebra~$\alphabet^{\ge 2}$, endowed with the operations of \cref{def:tidyCitelangisParallelActionWords}, defines a $\op{l,r}$-tidy parallel citelangis algebra.
The concatenation product~$\cdot$ of~$\alphabet^{\ge 2}$ is given by~${\op{l,r}}$.
\end{proposition}

\begin{proof}
One immediately checks that our operations on words indeed satisfy the $9$ $\op{l,r}$-tidy parallel $2$-citelangis relations given in \cref{rem:multitidyCitelangisParallel}.
\end{proof}

\paraul{Action on permutations}
Replacing the concatenation by the shifted concatenation, one endows similarly the algebra~$\FQSym_2$ of $2$-permutations with the structure of a $\op{l,r}$-tidy parallel citelangis algebra.
This can be rephrased as follows.

\begin{definition}
\label{def:tidyCitelangisParallelActionPermutations}
For any operation~${\operation \in \Operations_2}$ and any two $2$-permutations~$\mu$ and~$\nu$ of degree~$m$ and~$n$ respectively, $\mu \, \operation \, \nu$ is the $2$-permutation~$\pi$ appearing in~$\mu \shiftedShuffle \nu$ such that
\begin{itemize}
\item the first and last entries of~$\pi$ are determined by~$\operation$:
\[
\pi_1 = \begin{cases} \mu_1 & \text{if } \operation_1 = {\op{l}} \, , \\ \nu_1 + m & \text{if } \operation_1 = {\op{r}} \, , \end{cases}
\qqandqq
\pi_{m+n} = \begin{cases} \mu_m & \text{if } \operation_2 = {\op{l}} \, , \\ \nu_n + m & \text{if } \operation_2 = {\op{r}} \, , \end{cases}
\]

\item the remaining entries of~$\pi$ are the concatenation of the remaining entries of~$\mu$ and the remaining entries of~$\nu[m]$.
\end{itemize}
\end{definition}

For example, for~$\mu = {\blue 42132413}$ and~$\nu = {\red 2121}$ in~$\FQSym_2$ we have~$\mu \op{l,l} \nu = {\blue 4213241}{\red 6565}{\blue3}$, $\mu \op{l,r} \nu = {\blue 42132413}{\red 6565}$, $\mu \op{r,l} \nu = {\red 6}{\blue 4213241}{\red 565}{\blue3}$ and~$\mu \op{r,r} \nu = {\red 6}{\blue 42132413}{\red 565}$.
The following statement is immediate from \cref{prop:tidyCitelangisParallelWords}.

\begin{proposition}
\label{prop:tidyCitelangisParallelMultipermutations}
The algebra~$(\FQSym_2, \bar\cdot)$, endowed with the
operations of \cref{def:tidyCitelangisParallelActionPermutations}, defines a $\op{l,r}$-tidy parallel citelangis algebra.
The concatenation product~$\bar\cdot$ of~$\FQSym_2$ is given~by~${\op{l,r}}$.
\end{proposition}

The goal of this section is to show that the $\op{l,r}$-tidy parallel citelangis algebra~$\FQSym_2$ is free.
To manipulate this $\op{l,r}$-tidy parallel citelangis algebra, we consider the evaluations of syntax trees of~$\Syntax[\Operations_2]$ in~$\FQSym_2$.
See \cref{fig:tidyEvalParallel} for an illustration.

\begin{definition}
\label{def:tidyEvalParallel}
Denote by~$\tidyEvalPermParallel(\tree ; \sigma_1, \dots, \sigma_p)$ the evaluation of a syntax tree~$\tree \in \Syntax[\Operations_2]$ of arity~$p$ on $p$ $2$-permutations~${\sigma_1, \dots, \sigma_p}$ of~$\Perm_2$ using the $\op{l,r}$-tidy parallel citelangis structure of~$\FQSym_2$. 
The \defn{tidy parallel permutation evaluation} of~$\tree$ is then~$\tidyEvalPermParallel(\tree) \eqdef \tidyEvalPermParallel(\tree ; 11, \dots, 11)$.
We extend by linearity $\tree$ to the elements of~$\Free[\Operations_2]$ on the one hand, and $\sigma_1, \dots, \sigma_p$ to the elements of~$\FQSym_2$ on the other hand.
\end{definition}

\begin{figure}[t]
	\centerline{$
	\tidyEvalPermParallel \left(
		\begin{tikzpicture}[baseline={([yshift=-.8ex]current bounding box.center)}, level/.style={sibling distance=18mm/#1, level distance = 1cm/sqrt(#1)}]
			\node [rectangle, draw] {$\op{r,r}$}
				child {node [rectangle, draw] {$\op{r,r}$}
					child {node {11}}
					child {node [rectangle, draw] {$\op{l,l}$}
						child {node {22}}
						child {node {33}}
					}
				}
				child {node [rectangle, draw] {$\op{r,l}$}
					child {node {44}}
					child {node {55}}
				}
			;
		\end{tikzpicture}
	\right)
	=
	5 \cdot \tidyEvalPermParallel \left(
		\begin{tikzpicture}[baseline={([yshift=-.8ex]current bounding box.center)}, level 1/.style={sibling distance = 1cm, level distance = .7cm}, level 2/.style={sibling distance = .5cm, level distance = .6cm}]
			\node [rectangle, draw] {$\op{r,r}$}
				child {node {11}}
				child {node [rectangle, draw] {$\op{l,l}$}
					child {node {22}}
					child {node {33}}
				}
			;
		\end{tikzpicture}	
	\right) \cdot \tidyEvalPermParallel \left(
		\begin{tikzpicture}[baseline={([yshift=-.8ex]current bounding box.center)}, level/.style={sibling distance=18mm/(#1+2), level distance = 1cm/sqrt(#1+2)}]
			\node [rectangle, draw] {$\perp\perp$}
				child {node {4}}
				child {node {5}}
			;
		\end{tikzpicture}	
	\right) \cdot 4
	=
	5 \cdot 211332 \cdot 45 \cdot 4
	$}
	\caption{Illustration of \cref{rem:tidyEvalParallel}.}
	\label{fig:tidyEvalParallel}
\end{figure}

\begin{remark}
\label{rem:tidyEvalParallel}
Let us rephrase algorithmically \cref{def:tidyEvalParallel}.
For this, we generalize the evaluation to partial syntax trees, \ie trees whose nodes are labeled by operations with $2$ letters among~$\{\op{l}, \op{r}, \perp\}$ and whose leaves are labeled by words.
The evaluation~$\tidyEvalPermParallel(\tree[s])$ of such a partial syntax tree~$\tree[s]$ is defined inductively as follows.
If~$\tree[s]$ is a leaf labeled by a word~$\ind{w}$, then~$\tidyEvalPermParallel(\tree[s]) = \ind{w}$.
Otherwise, $\tidyEvalPermParallel(\tree[s])$ is obtained as follows.
If the first (resp.~second) letter of the root of~$\tree[s]$ is~$\perp$, then we let~$f = \varepsilon$ (resp.~$l = \varepsilon$).
Otherwise, we let a car traverse the partial syntax tree~$\tree[s]$.
This car follows and replaces by a~$\perp$ the first (resp.~second) letter of each signal it traverses, and finally arrives at a leaf where it reads and erases the first (resp.~last) letter.
Let~$f$ (resp.~$l$) be the letter read by the first (resp.~second) car at its destination.
After both cars have reached their destinations, the signal at the root is~$\perp\perp$.
We are left with the two partial syntax trees~$\tree[l]$ and~$\tree[r]$ (where some letters of the signals in the nodes and of the words in the leaves have been erased by the cars).
The evaluation~$\tidyEvalPermParallel(\tree[s])$ is then obtained inductively by
\[
\tidyEvalPermParallel(\tree[s]) = f \cdot \tidyEvalPermParallel(\tree[l]) \cdot \tidyEvalPermParallel(\tree[r]) \cdot l.
\]
Finally, for a syntax tree~$\tree$ of arity~$p$ and $2$-permutations~$\sigma_1 \in \Perm_2(n_1), \dots, \sigma_p \in \Perm_2(n_p)$, the evaluation $\tidyEvalPermParallel(\tree ; \sigma_1, \dots, \sigma_p)$ is the evaluation of the partial syntax tree~$\tree[s]$ obtained from~$\tree$ by putting the permutation~$\sigma_i[n_1 + \dots + n_{i-1}]$ at the $i$-th leaf for all~$i \in [p]$.
In particular, for~$\tidyEvalPermParallel(\tree) = \tidyEvalPermParallel(\tree ; 11, \dots, 11)$, the $i$-th leaf is labeled by the word~$ii$.
See \cref{fig:tidyEvalParallel}.
\end{remark}

As $\FQSym_2$ is a $\op{l,r}$-tidy parallel citelangis algebra, the tidy parallel permutation evaluation is preserved by the $\op{l,r}$-tidy parallel citelangis relations of \cref{rem:multitidyCitelangisParallel}.
Thus, $\tidyEvalPermParallel(\tree[s] ; F_1, \dots, F_p) = \tidyEvalPermParallel(\tree[t] ; F_1, \dots, F_p)$ for any~$\tree[s], \tree[t] \in \Free[\Operations_2](p)$ which are equivalent modulo the $\op{l,r}$-tidy parallel citelangis relations, and for any~$F_1, \dots, F_p \in \FQSym_2$.
The objective of this section is to show the reciprocal statement.
The proof is based on bounded cuts in $2$-permutations introduced in \cref{def:boundedCutMultipermutation}.

\paraul{Tidy parallel permutation evaluations and bounded cuts}
Our next two lemmas state that the bounded cuts of a $2$-permutation~$\rho$ precisely correspond to its decompositions of the form~$\rho = \sigma \, \operation \, \tau$, where~$\operation \in \Operations_2$.
Their proofs immediately follow from \cref{def:tidyCitelangisParallelActionPermutations} and are thus left to the reader.

\begin{lemma}
\label{lem:operationImpliesBoundedCut}
For any $2$-permutations~$\sigma \in \Perm_2(m)$ and~$\tau \in \Perm_2(n)$, and any operation~${\operation \in \Operations_2}$, the degree~$m$ of~$\sigma$ is a bounded cut of~$\sigma \, \operation \, \tau$.
\end{lemma}

\begin{lemma}
\label{lem:boundedCutImpliesOperation}
For any $2$-permutation~$\rho \in \Perm_2(\ell)$ and any bounded cut~$\gamma \in \bcuts(\rho)$, there is a unique~${\operation \in \Operations_2}$ (defined by $\operation_i \eqdef {\op{l}}$ if~$\rho_i \le \gamma$ and~$\operation_i \eqdef {\op{r}}$ if~$\rho_i > \gamma$) such that~${\rho = \rho^{|[\gamma]} \, \operation \, \rho^{|[\ell] \ssm [\gamma]}}$.
\end{lemma}

\begin{remark}
\label{rem:algoTidyEvalParallel}
\cref{lem:boundedCutImpliesOperation} gives an inductive algorithm to compute all decompositions of a given $2$-permutation~$\rho$ as an evaluation of the form~$\rho = \tidyEvalPermParallel(\tree ; \sigma_1, \dots, \sigma_p)$.
Namely, $\rho$ admits
\begin{itemize}
\item the trivial evaluation~$\rho = \tidyEvalPermParallel(\one ; \rho)$, where~$\one$ is the unit syntax tree with no node and a single leaf, and 
\item the evaluation~$\rho = \tidyEvalPermParallel(\tree ; \sigma_1, \dots, \sigma_l, \tau_1, \dots, \tau_r)$, for any bounded cut~$\gamma \in \bcuts(\rho)$ and any evaluations~$\rho^{|[\gamma]} = \tidyEvalPermParallel(\tree[l] ; \sigma_1, \dots, \sigma_l)$ and $\rho^{|[\ell] \ssm [\gamma]} = \tidyEvalPermParallel(\tree[r] ; \tau_1, \dots, \tau_r)$, where~$\tree$ is the syntax tree with root~$\operation \in \Operations_2$ defined by \cref{lem:boundedCutImpliesOperation} and with subtrees~$\tree[l]$ and~$\tree[r]$.
\end{itemize}
This algorithm implies the existence of decompositions of the form~$\rho = \tidyEvalPermParallel(\tree ; \sigma_1, \dots, \sigma_p)$ where~$\sigma_1, \dots, \sigma_p$ are bounded uncuttable.
In fact, we can even impose the position of the first bounded cut.
\end{remark}

\begin{corollary}
\label{coro:boundedCutImpliesSyntaxTree}
For any $2$-permutation~$\rho \in \Perm_2$ and any bounded cut~$\gamma \in \bcuts(\rho)$, there exists a syntax tree~$\tree$ of arity~$p$ with left subtree of arity~$l$ and bounded uncuttable $2$-permutations ${\sigma_1 \in \Perm_2(n_1), \dots, \sigma_p \in \Perm_2(n_p)}$ such that~$\rho = \tidyEvalPermParallel(\tree ; \sigma_1, \dots, \sigma_p)$ and~$\gamma = n_1 + \dots + n_l$.
\end{corollary}

We now characterize the $2$-permutations~$\rho$ that admit a decomposition of the form~${\rho = \tidyEvalPermParallel(\tree)}$.

\begin{proposition}
\label{prop:characterizationTidyParallelPermutationEvaluations}
The tidy parallel permutation evaluations of the syntax trees of~$\Syntax[\Operations_2]$ are precisely the fully bounded cuttable $2$-permutations.
\end{proposition}

\begin{proof}
Consider first a $2$-permutation~$\rho = \tidyEvalPermParallel(\tree)$ with~$\tree \in \Syntax[\Operations_2](\ell)$.
We prove by induction on~$\ell$ that~$\rho$ is fully bounded cuttable.
If~$\ell = 1$, there is nothing to prove.
Assume that~$\ell \ge 2$ and let~$1 \le a < b \le \ell$.
Let~$\tree[l]$ and~$\tree[r]$ denote the left and right subtrees of~$\tree$, and let~$\gamma$ be the arity of~$\tree[l]$, so that~$\rho^{|[\gamma]} = \tidyEvalPermParallel(\tree[l])$ and~$\rho^{|[\ell] \ssm [\gamma]} = \tidyEvalPermParallel(\tree[r])$.
We distinguish three cases:
\begin{itemize}
\item Assume that~$b \le \gamma$. Since~$\rho^{|[\ell]} = \tidyEvalPermParallel(\tree[l])$ is fully bounded cuttable by induction hypothesis and bounded cuts are preserved by restriction by \cref{lem:boundedCutsRestriction}, we obtain that~$\rho^{|[a,b]} = (\rho^{|[\gamma]})^{|[a,b]}$ is bounded cuttable.
\item Assume that~$\gamma \le a$. The argument is similar since~$\rho^{|[a,b]} = (\rho^{|[\ell] \ssm [\gamma]})^{|[a-\gamma,b-\gamma]}$.
\item Assume finally that~$a < \gamma < b$. By \cref{lem:operationImpliesBoundedCut}, $\gamma$ is a bounded cut of~$\rho$. Therefore, $\gamma-a$ is a bounded cut of~$\rho^{|[a,b]}$ by \cref{lem:boundedCutsRestriction}.
\end{itemize}

Conversely, consider now a fully bounded cuttable $2$-permutation~$\rho \in \Perm_2(\ell)$.
Similarly to \cref{rem:algoTidyEvalParallel}, we prove by induction on~$\ell$ that~$\rho$ is the tidy parallel permutation evaluation of a syntax tree.
If~$\ell = 1$, then~$\rho = \tidyEvalPermParallel(\one)$.
If~$\ell \ge 2$, then~$\rho$ admits at least one bounded cut~$\gamma$ by assumption.
Moreover, $\rho^{|[\gamma]}$ and~$\rho^{|[\ell] \ssm [\gamma]}$ are both fully bounded cuttable by \cref{thm:fullyBoundedCuttablepermutationClass}.
By induction, we obtain that~$\rho^{|[\gamma]} = \tidyEvalPermParallel(\tree[l])$ and~$\rho^{|[\ell] \ssm [\gamma]} = \tidyEvalPermParallel(\tree[r])$.
Then~$\rho = \tidyEvalPermParallel(\tree)$, where~$\tree$ is the syntax tree with root~$\operation \in \Operations_2$ defined by \cref{lem:boundedCutImpliesOperation} and with subtrees~$\tree[l]$~and~$\tree[r]$.
\end{proof}

\paraul{Freeness}
Our objective is now to prove that the decompositions of a given $2$-permutation~$\rho$ provides by \cref{coro:boundedCutImpliesSyntaxTree} are all equivalent up to the $\op{l,r}$-tidy parallel citelangis relations of \cref{rem:multitidyCitelangisParallel}.
Our first step is to understand the evaluations of a quadratic syntax tree on three permutations.
We start from a simple observation, which again immediately follows from \cref{def:tidyCitelangisParallelActionPermutations}.
Recall that a syntax tree is $\op{l,r}$-tidy parallel when the first (resp.~last) traffic signal at each node not contained in its first (resp.~last) parallel route point to the left (resp.~right).

\begin{lemma}
\label{lem:boundedCutIffTidy}
For any $2$-permutations~$\rho \in \Perm_2(\ell)$, $\sigma \in \Perm_2(m)$ and~$\tau \in \Perm_2(n)$, and any operations~$\operation[a], \operation[b], \operation[a]', \operation[b]' \in \Operations_2$, we have
\begin{itemize}
\item $\ell + m$ is a bounded cut of~$\tidyEvalPermParallel\Big( \!\!
	\begin{tikzpicture}[baseline=-.5cm, level 1/.style={sibling distance = .8cm, level distance = .7cm}, level 2/.style={sibling distance = .6cm, level distance = .5cm}]
		\node [rectangle, draw, minimum height=.5cm] {$\operation[a]$}
			child {node {}}
			child {node [rectangle, draw, minimum height=.5cm] {$\operation[b]$}
				child {node {}}
				child {node {}}
			}
		;
	\end{tikzpicture}
\!\! ; \rho, \sigma, \tau \Big)$ if and only if
	\begin{tikzpicture}[baseline=-.5cm, level 1/.style={sibling distance = .8cm, level distance = .7cm}, level 2/.style={sibling distance = .6cm, level distance = .5cm}]
		\node [rectangle, draw, minimum height=.5cm] {$\operation[a]$}
			child {node {}}
			child {node [rectangle, draw, minimum height=.5cm] {$\operation[b]$}
				child {node {}}
				child {node {}}
			}
		;
	\end{tikzpicture}
is $\op{l,r}$-tidy~parallel,
\item $\ell$ is a bounded cut of~$\tidyEvalPermParallel \Big( 
	\begin{tikzpicture}[baseline=-.5cm, level 1/.style={sibling distance = .8cm, level distance = .7cm}, level 2/.style={sibling distance = .6cm, level distance = .5cm}]
		\node [rectangle, draw, minimum height=.5cm] {$\operation[a]'$}
			child {node [rectangle, draw, minimum height=.5cm] {$\operation[b]'$}
				child {node {}}
				child {node {}}
			}
			child {node {}}
		;
	\end{tikzpicture}
; \rho, \sigma, \tau \Big)$ if and only if
	\begin{tikzpicture}[baseline=-.5cm, level 1/.style={sibling distance = .8cm, level distance = .7cm}, level 2/.style={sibling distance = .6cm, level distance = .5cm}]
		\node [rectangle, draw, minimum height=.5cm] {$\operation[a]'$}
			child {node [rectangle, draw, minimum height=.5cm] {$\operation[b]'$}
				child {node {}}
				child {node {}}
			}
			child {node {}}
		;
	\end{tikzpicture}
is $\op{l,r}$-tidy parallel.
\end{itemize}
\end{lemma}

\begin{lemma}
\label{lem:uniqueTidyParallelCitelangisPermutationEvaluationQuadratic}
For any $2$-permutations~$\rho, \sigma, \tau \in \Perm_2$, and any operations~$\operation[a], \operation[b], \operation[a]', \operation[b]' \in \Operations_2$, if
\[
	\tidyEvalPermParallel \Big( 
	\begin{tikzpicture}[baseline=-.5cm, level 1/.style={sibling distance = .8cm, level distance = .7cm}, level 2/.style={sibling distance = .6cm, level distance = .5cm}]
		\node [rectangle, draw, minimum height=.5cm] {$\operation[a]$}
			child {node {}}
			child {node [rectangle, draw, minimum height=.5cm] {$\operation[b]$}
				child {node {}}
				child {node {}}
			}
		;
	\end{tikzpicture}
	; \rho, \sigma, \tau \Big)
	=
	\tidyEvalPermParallel \Big( 
	\begin{tikzpicture}[baseline=-.5cm, level 1/.style={sibling distance = .8cm, level distance = .7cm}, level 2/.style={sibling distance = .6cm, level distance = .5cm}]
		\node [rectangle, draw, minimum height=.5cm] {$\operation[a]'$}
			child {node [rectangle, draw, minimum height=.5cm] {$\operation[b]'$}
				child {node {}}
				child {node {}}
			}
			child {node {}}
		;
	\end{tikzpicture}
	; \rho, \sigma, \tau \Big),
\]
then
\(
	\begin{tikzpicture}[baseline=-.5cm, level 1/.style={sibling distance = .8cm, level distance = .7cm}, level 2/.style={sibling distance = .6cm, level distance = .5cm}]
		\node [rectangle, draw, minimum height=.5cm] {$\operation[a]$}
			child {node {}}
			child {node [rectangle, draw, minimum height=.5cm] {$\operation[b]$}
				child {node {}}
				child {node {}}
			}
		;
	\end{tikzpicture}
	=
	\begin{tikzpicture}[baseline=-.5cm, level 1/.style={sibling distance = .8cm, level distance = .7cm}, level 2/.style={sibling distance = .6cm, level distance = .5cm}]
		\node [rectangle, draw, minimum height=.5cm] {$\operation[a]'$}
			child {node [rectangle, draw, minimum height=.5cm] {$\operation[b]'$}
				child {node {}}
				child {node {}}
			}
			child {node {}}
		;
	\end{tikzpicture}
\)
is a $\op{l,r}$-tidy parallel citelangis relation.
\end{lemma}

\begin{proof}
Let~$\pi$ denote the $2$-permutation obtained by these evaluations, and let $\ell$, $m$ and~$n$ denote the degrees of~$\rho$, $\sigma$ and~$\tau$ respectively.
By \cref{lem:operationImpliesBoundedCut}, we obtain that~$\ell$ and~$\ell + m$ are bounded cuts of~$\pi$.
By \cref{lem:boundedCutIffTidy}, we therefore derive that the two syntax trees are $\op{l,r}$-tidy parallel.
Moreover, the parallel destination vector of the two syntax trees are both given by the first and last letters of~$\pi$ where we replace all letters between~$1$ and~$\ell$ by~$1$, all letters between~$\ell + 1$ and~$\ell + m$ by~$2$, and all letters between~$\ell+m+1$ and~$\ell+m+n$ by~$3$.
Since they are $\op{l,r}$-tidy parallel and have the same parallel destination vector, the two syntax trees form a $\op{l,r}$-tidy parallel citelangis relation.
\end{proof}

We now prove that, up to the $\op{l,r}$-tidy parallel citelangis relations, any $2$-permutation can be obtained in a unique way as the evaluation of a syntax tree on bounded uncuttable $2$-permutations.

\begin{proposition}
\label{prop:uniqueTidyParallelCitelangisPermutationEvaluation}
For any syntax trees~$\tree , \tree' \in \Syntax[\Operations_2]$ of arity~$p$ and~$p'$ respectively, and any bounded uncuttable $2$-permuta\-tions~$\sigma_1, \dots, \sigma_p$, $\sigma'_1, \dots, \sigma'_{p'} \in \Perm_2$, if
\[
\tidyEvalPermParallel(\tree ; \sigma_1, \dots, \sigma_p) = \tidyEvalPermParallel(\tree' ; \sigma'_1, \dots, \sigma'_{p'}),
\]
then $p = p'$, $\tree = \tree'$ modulo the $\op{l,r}$-tidy parallel citelangis relations and~$\sigma_i = \sigma_i'$ for all~$i \in [p]$.
\end{proposition}

\begin{proof}
Let~$\pi = \tidyEvalPermParallel(\tree ; \sigma_1, \dots, \sigma_p) = \tidyEvalPermParallel(\tree' ; \sigma'_1, \dots, \sigma'_{p'})$ and let~$n$ denote its degree.

We prove the result by induction on~$p$.
If~$p = 1$, then~$\pi$ is bounded uncuttable, so that~$p' = 1$, $\tree = \tree'$ is the only syntax tree of arity~$1$, and~$\sigma_1 = \sigma_1' = \pi$.

Assume now that~$p > 1$.
Let~$\operation[a]$ be the root of~$\tree$, let~$\tree[l]$ and~$\tree[r]$ be its left and right subtrees, let~$l$ be the arity of~$\tree[l]$ and let~$\gamma$ be the corresponding bounded cut of~$\pi$.
We thus have~${\tidyEvalPermParallel(\tree[l] ; \sigma_1, \dots, \sigma_l) = \pi^{|[\gamma]}}$ and~$\tidyEvalPermParallel(\tree[r] ; \sigma_{l+1}, \dots, \sigma_p) = \pi^{|[n] \ssm [\gamma]}$.
Define~$\operation[a]', \tree[l]', \tree[r]', l'$ and~$\gamma'$ similarly for~$\tree'$.
These notations are illustrated below:

\begin{center}
	\begin{tikzpicture}[baseline={([yshift=-.8ex]current bounding box.center)}, level/.style={sibling distance = .5cm, level distance = .7cm}]
		\node {$\tree$}
			child {node {$\sigma_1$}}
			child [edge from parent/.style={}] {node {$\dots$}}
			child {node {$\sigma_p$}}
		;
	\end{tikzpicture}
	=
	\begin{tikzpicture}[baseline={([yshift=-.8ex]current bounding box.center)}, level 1/.style={sibling distance = 2cm, level distance = .8cm}, level 2/.style={sibling distance = .6cm, level distance = .7cm}]
		\node [rectangle, draw, minimum height=.5cm] {$\operation[a]$}
			child {node {$\tree[l]$}
				child {node {$\sigma_1$}}
				child [edge from parent/.style={}] {node {$\dots$}}
				child {node {$\sigma_l$}}
			}
			child {node {$\tree[r]$}
				child {node {$\sigma_{l+1}$}}
				child [edge from parent/.style={}] {node {$\dots$}}
				child {node {$\sigma_p$}}
			}
		;
	\end{tikzpicture}
	\qquad\text{and}\qquad
	\begin{tikzpicture}[baseline={([yshift=-.8ex]current bounding box.center)}, level/.style={sibling distance = .5cm, level distance = .7cm}]
		\node {$\tree'$}
			child {node {$\sigma'_1$}}
			child [edge from parent/.style={}] {node {$\dots$}}
			child {node {$\sigma'_{p'}$}}
		;
	\end{tikzpicture}
	=
	\begin{tikzpicture}[baseline={([yshift=-.8ex]current bounding box.center)}, level 1/.style={sibling distance = 2cm, level distance = .8cm}, level 2/.style={sibling distance = .6cm, level distance = .7cm}]
		\node [rectangle, draw, minimum height=.5cm] {$\operation[a]'$}
			child {node {$\tree[l]'$}
				child {node {$\sigma'_1$}}
				child [edge from parent/.style={}] {node {$\dots$}}
				child {node {$\sigma'_{l'}$}}
			}
			child {node {$\tree[r]'$}
				child {node {$\sigma'_{l+1}$}}
				child [edge from parent/.style={}] {node {$\dots$}}
				child {node {$\sigma'_{p'}$}}
			}
		;
	\end{tikzpicture}
\end{center}

Assume first that~$\gamma = \gamma'$. Then~$\operation[a] = \operation[a]'$, $\tidyEvalPermParallel(\tree[l] ; \sigma_1, \dots, \sigma_l) = \pi^{|[\gamma]} = \tidyEvalPermParallel(\tree[l]' ; \sigma'_1, \dots, \sigma'_{l'})$ and~$\tidyEvalPermParallel(\tree[r] ; \sigma_{l+1}, \dots, \sigma_p) = \pi^{|[n] \ssm [\gamma]} = \tidyEvalPermParallel(\tree[r]' ; \sigma'_{l'+1}, \dots, \sigma'_{p'})$ by \cref{lem:boundedCutImpliesOperation}.
By induction hypothesis, the first equality ensures that~$l = l'$, that~$\tree[l] = \tree[l]'$ modulo the $\op{l,r}$-tidy parallel citelangis relations and that~${\sigma_i = \sigma'_i}$ for all~$i \in [l]$, while the second equality ensures that~$p-l = p'-l'$, that~$\tree[r] = \tree[r]'$ modulo the $\op{l,r}$-tidy parallel citelangis relations and that~$\sigma_i = \sigma'_i$ for all~$i \in [p] \ssm [l]$.
We thus conclude that~$p = p'$, that~$\tree = \tree'$ modulo the $\op{l,r}$-tidy parallel citelangis relations, and that~$\sigma_i = \sigma'_i$ for all~$i \in [p]$.

Assume now without loss of generality that~$\gamma < \gamma'$.
Consider the $2$-permutations~$\rho \eqdef \pi^{|[\gamma]}$, $\sigma \eqdef \pi^{|[\gamma'] \ssm [\gamma]}$ and~$\tau \eqdef \pi^{|[n] \ssm [\gamma']}$.
Since~$\gamma'$ is a bounded cut of~$\pi$ larger than~$\gamma$, \cref{lem:boundedCutsRestriction} ensures that~$\gamma'-\gamma$ is a bounded cut of~$\pi^{|[n] \ssm [\gamma]}$.
By \cref{coro:boundedCutImpliesSyntaxTree} and induction hypothesis, there exists a syntax tree~$\tree[s]$ of arity~$p-l$ with root~$\operation[b]$, left subtree~$\tree[u]$ of arity~$u$ and right subtree~$\tree[v]$, such that~$\tree[l]' = \tree[s]'$ modulo the $\op{l,r}$-tidy parallel citelangis relations and~$\tidyEvalPermParallel(\tree[s] ; \sigma_{l+1}, \dots, \sigma_p) = \pi^{|[n] \ssm [\gamma]}$, $\tidyEvalPermParallel(\tree[u] ; \sigma_{l+1}, \dots, \sigma_{l+u}) = \sigma$ and $\tidyEvalPermParallel(\tree[v] ; \sigma_{l+u+1}, \dots, \sigma_p) = \tau$.
Similarly, since~$\gamma$ is a bounded cut of~$\pi$ smaller than~$\gamma'$, \cref{lem:boundedCutsRestriction} ensures that~$\gamma$ is a bounded cut of~$\pi^{|[\gamma']}$.
By \cref{coro:boundedCutImpliesSyntaxTree} and induction hypothesis, there exists a syntax tree~$\tree[s]'$ of arity~$l'$, with root~$\operation[b]'$, left subtree~$\tree[u]'$ of arity~$u'$ and right subtree~$\tree[v]'$, such that~$\tree[l]' = \tree[s]'$ modulo the $\op{l,r}$-tidy parallel citelangis relations and $\tidyEvalPermParallel(\tree[s]' ; \sigma'_1, \dots, \sigma'_{l'}) = \pi^{|[\gamma']}$, $\tidyEvalPermParallel(\tree[u]' ; \sigma'_1, \dots, \sigma'_{u'}) = \rho$ and $\tidyEvalPermParallel(\tree[v]' ; \sigma_{u'+1}, \dots, \sigma_{l'}) = \sigma$.
Since
\[
\begin{array}{r@{\;}c@{\;}l}
\tidyEvalPermParallel(\tree[l] ; \sigma_1, \dots, \sigma_l) & = \rho = & \tidyEvalPermParallel(\tree[u]' ; \sigma'_1, \dots, \sigma'_{u'}), \\
\tidyEvalPermParallel(\tree[u] ; \sigma_{l+1}, \dots, \sigma_{l+u}) & = \sigma = & \tidyEvalPermParallel(\tree[v]' ; \sigma_{u'+1}, \dots, \sigma_{l'}), \\
\text{and}\quad
\tidyEvalPermParallel(\tree[v] ; \sigma_{l+u+1}, \dots, \sigma_p) & = \tau = & \tidyEvalPermParallel(\tree[r]' ; \sigma'_{l'+1}, \dots, \sigma'_{p'}),
\end{array}
\]
we obtain by three applications of the induction hypothesis that~$p = p'$ and~$\sigma_i = \sigma'_i$ for all~$i \in [p]$.
Moreover, we have
\[
	\tidyEvalPermParallel \Big( 
	\begin{tikzpicture}[baseline=-.5cm, level 1/.style={sibling distance = .8cm, level distance = .7cm}, level 2/.style={sibling distance = .6cm, level distance = .5cm}]
		\node [rectangle, draw, minimum height=.5cm] {$\operation[a]$}
			child {node {}}
			child {node [rectangle, draw, minimum height=.5cm] {$\operation[b]$}
				child {node {}}
				child {node {}}
			}
		;
	\end{tikzpicture}
	; \rho, \sigma, \tau \Big)
	= \pi =
	\tidyEvalPermParallel \Big( 
	\begin{tikzpicture}[baseline=-.5cm, level 1/.style={sibling distance = .8cm, level distance = .7cm}, level 2/.style={sibling distance = .6cm, level distance = .5cm}]
		\node [rectangle, draw, minimum height=.5cm] {$\operation[a]'$}
			child {node [rectangle, draw, minimum height=.5cm] {$\operation[b]'$}
				child {node {}}
				child {node {}}
			}
			child {node {}}
		;
	\end{tikzpicture}
	; \rho, \sigma, \tau \Big).
\]
By \cref{lem:uniqueTidyParallelCitelangisPermutationEvaluationQuadratic}, we conclude that
\(
	\begin{tikzpicture}[baseline=-.5cm, level 1/.style={sibling distance = .8cm, level distance = .7cm}, level 2/.style={sibling distance = .6cm, level distance = .5cm}]
		\node [rectangle, draw, minimum height=.5cm] {$\operation[a]$}
			child {node {}}
			child {node [rectangle, draw, minimum height=.5cm] {$\operation[b]$}
				child {node {}}
				child {node {}}
			}
		;
	\end{tikzpicture}
	=
	\begin{tikzpicture}[baseline=-.5cm, level 1/.style={sibling distance = .8cm, level distance = .7cm}, level 2/.style={sibling distance = .6cm, level distance = .5cm}]
		\node [rectangle, draw, minimum height=.5cm] {$\operation[a]'$}
			child {node [rectangle, draw, minimum height=.5cm] {$\operation[b]'$}
				child {node {}}
				child {node {}}
			}
			child {node {}}
		;
	\end{tikzpicture}
\)
is a $\op{l,r}$-tidy parallel citelangis relation, and thus that~$\tree = \tree'$ up to $\op{l,r}$-tidy parallel citelangis relations.
\end{proof}

\begin{remark}
\label{rem:algoTidyEvalParallelNormalForms}
From \cref{rem:algoTidyEvalParallel,prop:uniqueTidyParallelCitelangisPermutationEvaluation}, we derive an inductive algorithm to compute the decomposition of a given $2$-permutation~$\rho$ as an evaluation of the form~$\rho = \tidyEvalPermParallel(\tree ; \sigma_1, \dots, \sigma_p)$, where~$\tree$ is in normal form in the $\op{l,r}$-tidy parallel citelangis rewriting system of \cref{subsubsec:rewritingSystemCitelangis} and~$\sigma_1, \dots, \sigma_p$ are bounded uncuttable $2$-permutations.
Namely, 
\begin{itemize}
\item if~$\rho$ is bounded uncuttable, then~$\rho = \tidyEvalPermParallel(\one ; \rho)$, where~$\one$ is the unit syntax tree with no node and a single leaf, and
\item otherwise, $\rho = \tidyEvalPermParallel(\tree ; \sigma_1, \dots, \sigma_l, \tau_1, \dots, \tau_r)$, where~$\gamma$ be the rightmost bounded cut of~$\rho$, $\rho^{|[\gamma]} = \tidyEvalPermParallel(\tree[l] ; \sigma_1, \dots, \sigma_l)$ and $\rho^{|[\ell] \ssm [\gamma]} = \tidyEvalPermParallel(\tree[r] ; \tau_1, \dots, \tau_r)$ are such that~$\tree[l], \tree[r]$ are in normal form and~$\sigma_1, \dots, \sigma_l, \tau_1, \dots, \tau_r$ are bounded uncuttable, and~$\tree$ is the syntax tree with root~$\operation \in \Operations_2$ defined by \cref{lem:boundedCutImpliesOperation} and with subtrees~$\tree[l]$ and~$\tree[r]$.
\end{itemize}
\end{remark}

\cref{prop:uniqueTidyParallelCitelangisPermutationEvaluation} proves the main result of this section.

\begin{theorem}
\label{thm:freeTidyCitelangisParallel}
The $\op{l,r}$-tidy parallel citelangis algebra~$\FQSym_2$ is free on bounded uncuttable $2$-permutations.
\end{theorem}

\paraul{A complete combinatorial model for~$\tidyCitelangisParallel[\op{l,r}]$}
By \cref{thm:freeTidyCitelangisParallel}, the $\op{l,r}$-tidy parallel citelangis operad~$\tidyCitelangisParallel[\op{l,r}]$ can be fully understood from the $\op{l,r}$-tidy parallel citelangis subalgebra of~$\FQSym_2$ generated by the $2$-permutation~$11$.
We close this section with a completely explicit combinatorial model for this algebra.
We first obtain from \cref{prop:characterizationTidyParallelPermutationEvaluations} a combinatorial model for the operations of~$\tidyCitelangisParallel[\op{l,r}]$.

\begin{proposition}
\label{prop:combinatorialModelTidyParallelCitelangis}
The tidy parallel permutation evaluation~$\tree \mapsto \tidyEvalPermParallel(\tree)$ is a graded bijection from the $\op{l,r}$-tidy parallel equivalence classes of syntax trees of~$\Syntax[\Operations_2]$ to the fully bounded cuttable $2$-permutations.
\end{proposition}

\begin{proof}
This map~$\tree \mapsto \tidyEvalPermParallel(\tree)$ is surjective on fully bounded cuttable $2$-permutations by \cref{prop:characterizationTidyParallelPermutationEvaluations}.
It is compatible with the $\op{l,r}$-tidy parallel citelangis relations by \cref{prop:tidyCitelangisParallelMultipermutations}.
Finally, it is bijective since the $\op{l,r}$-tidy parallel citelangis algebra~$\FQSym_2$ is free on uncuttable $2$-permutations by \cref{thm:freeTidyCitelangisParallel}.
\end{proof}

Therefore, the fully bounded cuttable $2$-permutations can be thought of as a basis of the $\op{l,r}$-tidy parallel citelangis operad~$\tidyCitelangisParallel[\op{l,r}]$.
Through the bijection of \cref{prop:characterizationTidyParallelPermutationEvaluations}, we can thus define the compositions of the $\op{l,r}$-tidy parallel citelangis operad~$\tidyCitelangisParallel[\op{l,r}]$ directly on fully bounded cuttable $2$-permutations.

\begin{definition}
\label{def:combinatorialModelCompositionsTidyParallelCitelangis}
For any integers~$i \le m$ and~$n$, and any two fully bounded cuttable $2$-permutations $\sigma = \tidyEvalPermParallel(\tree[s]) \in \Perm_2(m)$ and $\tau = \tidyEvalPermParallel(\tree[t]) \in \Perm_2(n)$, we define the \defn{tidy parallel $i$-th composition} as
\[
\sigma \tidyParallelCirc{i} \tau \eqdef \tidyEvalPermParallel(\tree[s] \circ_i \tree[t]) \in \Perm_2(m+n-1).
\]
\end{definition}

The following statement provides a direct combinatorial description of the composition~$\tidyParallelCirc{i}$ of the $\op{l,r}$-tidy parallel citelangis operad~$\tidyCitelangisParallel[\op{l,r}]$ on fully bounded cuttable $2$-permutations.

\begin{proposition}
\label{prop:combinatorialModelCompositionsTidyParallelCitelangis}
Let~$\sigma \in \Perm_2(m)$ and~$\tau \in \Perm_2(n)$ be two fully bounded cuttable $2$-permutations and~$i \in [m]$.
Write~$\sigma = \lambda \, i \, \mu \, \nu \, i \, \omega$, where~$\mu_p < i < \nu_q$ for all~$p \in [|\mu|]$ and~$q \in [|\nu|]$ according to \cref{coro:fullyBoundedCuttableStructure}, and write~$\tau = f \, \theta \, l$ where $f$ and $l$ are its first and last letters.
Then
\[
\begin{array}{c@{\; = \;}l@{\,}l@{\,}l@{\,}l@{\,}l@{\,}l@{\,}l}
\sigma \tidyParallelCirc{i} \tau
& \lambda[i,n] & f[i-1] & \mu[i,n] & \theta[i-1] & \nu[i,n] & l[i-1] & \omega[i,n] \\
& \lambda[i,n] & f[i-1] & \mu      & \theta[i-1] & \nu[n-1] & l[i-1] & \omega[i,n].
\end{array}
\]
\end{proposition}

\begin{proof}
Consider arbitrary syntax trees~$\tree[s]$ and~$\tree[t]$ such that~$\sigma = \tidyEvalPermParallel(\tree[s])$ and~$\tau = \tidyEvalPermParallel(\tree[t])$.
By \cref{rem:tidyEvalParallel}, $\tidyEvalPermParallel(\tree[s] \circ_i \tree[t])$ is obtained inductively by letting~$2$ cars traverse~$\tree[s] \circ_i \tree[t]$ in parallel and recording the position of their arrival.
The cars that arrive at position~$i$ in~$\tree[s]$ thus continue their journey through~$\tree[t]$.
Therefore, the first and last values~$i$ in~$\sigma$ are replaced by the first and last values of~$\tau$, and the remaining values of~$\tau$ are placed at the only possible position in~$\sigma \tidyParallelCirc{i} \tau$.
\end{proof}

Here are some examples of tidy parallel compositions on fully bounded cuttable $2$-permutations:
\begin{gather*}
  {\blue6}1|1{\blue432325546} \tidyParallelCirc{1} 2{\red1123}3 = {\blue8}2{\red1123}3{\blue654547768}, \\
  {\blue61143}2|{\blue3}2{\blue5546} \tidyParallelCirc{2} 2{\red1123}3 = {\blue81165}3{\red2234}{\blue5}4{\blue7768}, \\
  {\blue6114}3{\blue2}|3{\blue25546} \tidyParallelCirc{3} 2{\red1123}3 = {\blue8116}4{\blue2}{\red3345}5{\blue27768}, \\
  {\blue611}4{\blue3232}|{\blue55}4{\blue6} \tidyParallelCirc{4} 2{\red1123}3 = {\blue811}5{\blue3232}{\red4456}{\blue77}6{\blue8}, \\
  {\blue61143232}5|5{\blue46} \tidyParallelCirc{5} 2{\red1123}3 = {\blue81143232}6{\red5567}7{\blue48}, \\
  6{\blue1143232554}|6 \tidyParallelCirc{6} 2{\red1123}3 = 7{\blue1143232554}{\red6678}8,
\end{gather*}
where we have marked by a vertical bar the separation $\mu|\nu$.

\begin{remark}
\label{rem:tidyParallelZinbiel}
Motivated by \cref{prop:combinatorialModelCompositionsTidyParallelCitelangis}, we will extend in \cref{subsubsec:parallelZinbiel} the tidy parallel compositions of \cref{def:combinatorialModelCompositionsTidyParallelCitelangis} from fully bounded cuttable $2$-permutations to all $2$-permutations.
\end{remark}


\subsubsection{Action of~\texorpdfstring{$\messyCitelangisParallel[2]$}{MCit2} on words and multipermutations}
\label{subsubsec:actionMessyParallelSignaleticPermutations}

As observed by M.~Aguiar and J.-L.~Loday in~\cite{AguiarLoday}, $\FQSym_2$ can also be endowed with a messy parallel $2$-citelangis algebra structure (\ie quadri-algebra structure).
In this section, we recall the construction and we revisit the proof that the resulting messy parallel $2$-citelangis algebra is free~\cite{Foissy, Vong}.

\paraul{Action on words}
The shuffle subalgebra~$\Shuffle^{\ge 2}$ can be endowed with the structure of a parallel $2$-citelangis algebra (or quadri-algebra~\cite{AguiarLoday}) defined as follows.

\begin{definition}
\label{def:messyCitelangisParallelActionWords}
We define the action of the four messy parallel $2$-citelangis operations~$\op{l,l}$, $\op{l,r}$, $\op{r,l}$ and~$\op{r,r}$ on any two words~$xXx'$ and~$yYy'$ by
\begin{gather*}
xXx' \op{l,l} yYy' = x \big( X \shuffle yYy' \big) x'
\\
xXx' \op{l,r} yYy' = x \big( Xx' \shuffle yY \big) y'
\\
xXx' \op{r,l} yYy' = y \big( xX \shuffle Yy' \big) x'
\\
xXx' \op{r,r} yYy' = y \big( xXx' \shuffle Y \big) y'.
\end{gather*}
In other words, we choose the first and last letters of the results among the first and last letters of~$xXx'$ or~$yYy'$ depending on the operation, and we shuffle the remaining factors of~$xXx'$~and~$yYy'$.
\end{definition}

\begin{proposition}
\label{prop:messyCitelangisParallelWords}
The shuffle algebra~$\Shuffle^{\ge 2}$, endowed with the operations of \cref{def:messyCitelangisParallelActionWords}, defines a messy parallel $2$-citelangis algebra (or quadri-algebra).
The shuffle product~$\shuffle$ of~$\Shuffle^{\ge 2}$ is given by~${\op{m,m}} = {\op{l,l}} + {\op{l,r}} + {\op{r,l}} + {\op{r,r}}$.
\end{proposition}

\pagebreak
\begin{proof}
One immediately checks that our operations on words indeed satisfy the $9$ messy parallel $2$-citelangis relations given in \cref{exm:messyCitelangisParallelRelations}.
\end{proof}

\paraul{Action on permutations}
Replacing the shuffle by the shifted shuffle, one endows similarly the algebra~$\FQSym_2$ of $2$-permutations with the structure of a messy parallel $2$-citelangis algebra.
This was already observed by M.~Aguiar and J.-L.~Loday in~\cite{AguiarLoday}.
This can be rephrased as follows.

\begin{definition}
\label{def:messyCitelangisParallelActionPermutations}
For any operation~$\operation \in \Operations_2$ and any two $2$-permutations~$\mu$ and~$\nu$ of degree~$m$ and~$n$ respectively, $\mu \, \operation \, \nu$ is the sum of all $2$-permutations~$\pi$ appearing in~$\mu \shiftedShuffle \nu$ such that
\[
\pi_1 = \begin{cases} \mu_1 & \text{if } \operation_1 = {\op{l}} \, , \\ \nu_1 + m & \text{if } \operation_1 = {\op{r}} \, , \end{cases}
\qqandqq
\pi_{m+n} = \begin{cases} \mu_m & \text{if } \operation_2 = {\op{l}} \, , \\ \nu_n + m & \text{if } \operation_2 = {\op{r}} \, . \end{cases}
\]
\end{definition}

For example, for~$\mu = {\blue 321312}$ and~$\nu = {\red 213231}$ in~$\FQSym_2$ we have~$\mu \op{l,l} \nu = {\blue 3} ({\blue 2131} \shuffle {\red 546564}) {\blue 2}$, $\mu \op{l,r} \nu = {\blue 3} ({\red 21312} \shuffle {\blue 54656}) {\red 4}$, $\mu \op{r,l} \nu = {\red 5} ({\blue 32131} \shuffle {\red 46564}) {\blue 2}$ and $\mu \op{r,r} \nu = {\red 5} ({\blue 321312} \shuffle {\red 4656}) {\red 4}$.
The next statement is immediate from \cref{prop:messyCitelangisParallelWords}.

\begin{proposition}[\cite{AguiarLoday}]
\label{prop:messyCitelangisParallelMultipermutations}
The algebra~$(\FQSym_2, \shiftedShuffle)$, endowed with the operations of \cref{def:messyCitelangisParallelActionPermutations}, defines a messy parallel $2$-citelangis algebra (or quadri-algebra).
The shifted shuffle product~$\shiftedShuffle$ of~$\FQSym_2$ is given~by~$\op{m,m}$.
\end{proposition}

Similarly to \cref{def:tidyEvalParallel}, we consider the evaluations of syntax trees of~$\Syntax[\Operations_2]$ in~$\FQSym_2$ to manipulate this messy parallel $2$-citelangis algebra.
See \cref{fig:LinExtEvalPosetParallel}.

\begin{definition}
\label{def:messyEvalParallel}
Denote by~$\messyEvalPermParallel(\tree ; \sigma_1, \dots, \sigma_p)$ the evaluation of a syntax tree~$\tree \in \Syntax[\Operations_2]$ of arity~$p$ on $p$ $2$-permutations~${\sigma_1, \dots, \sigma_p}$ of~$\Perm_2$ using the messy parallel $2$-citelangis structure of~$\FQSym_2$.
The \defn{messy parallel permutation evaluation} of~$\tree$ is then~$\messyEvalPermParallel(\tree) \eqdef \messyEvalPermParallel(\tree ; 11, \dots, 11)$.
We extend by linearity $\tree$ to the elements of~$\Free[\Operations_2]$ on the one hand, and $\sigma_1, \dots, \sigma_p$ to the elements of~$\FQSym_2$ on the other hand.
\end{definition}

\begin{remark}
\label{rem:messyEvalParallel}
Let us rephrase algorithmically \cref{def:messyEvalParallel}.
For this, we generalize the evaluation to partial syntax trees as in \cref{rem:tidyEvalParallel}, where nodes are labeled by operations with $2$ letters among~${\{\op{l}, \op{r}, \perp\}}$ and leaves are labeled by any words.
The messy parallel permutation evaluation~$\messyEvalPermParallel(\tree[s])$ of such a partial syntax tree~$\tree[s]$ is defined inductively by
\begin{itemize}
\item $\messyEvalPermParallel(\tree[s]) = \ind{w}$ if $\tree[s]$ is a leaf labeled by the word~$\ind{w}$,
\item $\messyEvalPermParallel(\tree[s]) = f \cdot \big( \messyEvalPermParallel(\tree[l]) \shuffle \messyEvalPermParallel(\tree[r]) \big) \cdot l$ otherwise, where~$f$ and~$l$ are the letters read by two cars in parallel on~$\tree[s]$, and $\tree[l]$ and~$\tree[r]$ are the two partial syntax trees left after these two cars traversed~$\tree[s]$ while replacing by~$\perp$ the letters they use in the signal in the nodes and erasing the letters they read in the words in the leaves as described in \cref{rem:tidyEvalParallel}.
\end{itemize}
Finally, for a syntax tree~$\tree$ of arity~$p$ and $2$-permutations~$\sigma_1 \in \Perm_2(n_1), \dots, \sigma_p \in \Perm_2(n_p)$, the evaluation $\messyEvalPermParallel(\tree ; \sigma_1, \dots, \sigma_p)$ is the evaluation of the partial syntax tree~$\tree[s]$ obtained from~$\tree$ by putting the permutation~$\sigma_i[n_1 + \dots + n_{i-1}]$ at the $i$-th leaf for all~$i \in [p]$.
In particular, for~$\messyEvalPermParallel(\tree) = \messyEvalPermParallel(\tree ; 11, \dots, 11)$, the $i$-th root is labeled by the word~$ii$.
See \cref{fig:LinExtEvalPosetParallel}.
\end{remark}

\paraul{Freeness}
It turns out that the tidy and messy parallel permutation evaluations are related by triangularity for the lexicographic order.

\begin{definition}
For a homogeneous element~$F \in \FQSym_2$, we denote by~$\LexMin(F)$ the lexicographic minimal $2$-permutation with a non-zero coefficient in~$F$.
\end{definition}

\begin{lemma}
\label{lem:LexMinParallel}
For any~$\tree \in \Syntax[\Operations_2](p)$ and~$\sigma_1, \dots, \sigma_p \in \Perm_2$, we have
\[
\tidyEvalPermParallel(\tree ; \sigma_1, \dots, \sigma_p) = \LexMin \big( \messyEvalPermParallel(\tree ; \sigma_1, \dots, \sigma_p) \big).
\]
In particular, $\tidyEvalPermParallel(\tree) = \LexMin \big( \messyEvalPermParallel(\tree) \big)$.
\end{lemma}

\begin{proof}
The proof works by induction using the descriptions of~$\tidyEvalPermParallel$ in \cref{rem:tidyEvalParallel} and of~$\messyEvalPermParallel$ in \cref{rem:messyEvalParallel}.
Indeed, for any partial syntax tree~$\tree[s]$, we have
\begin{itemize}
\item if $\tree[s]$ is a leaf labeled by the word~$\ind{w}$, then
\(
\tidyEvalPermParallel(\tree[s]) = \ind{w} = \messyEvalPermParallel(\tree[s]),
\)
\item otherwise, if and $\tree[l]$ and~$\tree[r]$ are the two partial syntax trees left after $2$ cars traversed~$\tree[s]$ in parallel while erasing the first and last letters of the signals in the nodes and of first and last letters of the words in the leaves as described in \cref{rem:tidyEvalParallel}, then
\begin{align*}
\tidyEvalPermParallel(\tree[s])
& = f \cdot \tidyEvalPermParallel(\tree[l]) \cdot \tidyEvalPermParallel(\tree[r]) \cdot l
  = f \cdot \LexMin \big( \messyEvalPermParallel(\tree[l]) \big) \cdot \LexMin \big( \messyEvalPermParallel(\tree[r]) \big) \cdot l \\
& = \LexMin \big( f \cdot \big( \messyEvalPermParallel(\tree[l]) \shuffle \messyEvalPermParallel(\tree[r]) \big)  \cdot l\big)
  = \LexMin \big( \messyEvalPermParallel(\tree[s]) \big).
\end{align*}
\end{itemize}
The result then follows by application on the syntax tree~$\tree[s]$ obtained from~$\tree$ by putting the permutation~$\sigma_i[n_1 + \dots + n_{i-1}]$ at the $i$-th leaf for all~$i \in [p]$.
\end{proof}

\begin{remark}
Since~$\tidyEvalPermParallel$ is increasing for the lexicographic order, \cref{lem:LexMinParallel} extends to $\FQSym_2$.
Namely, for any~$\tree \in \Syntax[\Operations_2](p)$ and homogeneous elements~${F_1, \dots, F_p \in \FQSym_2}$,
\[
\LexMin \big( \messyEvalPermParallel(\tree ; F_1, \dots, F_p) \big) = \tidyEvalPermParallel \big( \tree;  \LexMin(F_1), \dots, \LexMin(F_p) \big).
\]
\end{remark}

\cref{lem:LexMinParallel} enables us to revisit the following result of~\cite{Foissy, Vong}.

\begin{theorem}[\cite{Foissy, Vong}]
\label{thm:freeMessyCitelangisParallel}
The messy parallel $2$-citelangis algebra~$\FQSym_2$ (\aka quadri-algebra) is free on bounded uncuttable $2$-permutations.
\end{theorem}

\begin{proof}
Consider the $\op{l,r}$-tidy and the messy parallel $2$-citelangis rewriting systems defined in \cref{subsubsec:rewritingSystemCitelangis}.
As seen in \cref{subsubsec:normalFormsCitelangis}, these two rewriting systems have the same normal forms.
Consider two such normal forms~$\tree, \tree'$ of arity~$p$ and~$p'$ respectively, and some bounded uncuttable $2$-permutations~$\sigma_1, \dots, \sigma_p, \sigma'_1, \dots, \sigma'_{p'}$, such that~$\messyEvalPermParallel(\tree ; \sigma_1, \dots, \sigma_p) = \messyEvalPermParallel(\tree' ; \sigma'_1, \dots, \sigma'_{p'})$.
It then follows from \cref{lem:LexMinParallel} that $\tidyEvalPermParallel(\tree ; \sigma_1, \dots, \sigma_p) = \tidyEvalPermParallel(\tree' ; \sigma'_1, \dots, \sigma'_{p'})$.
\cref{prop:uniqueTidyParallelCitelangisPermutationEvaluation} then implies that~$p = p'$, that ${\sigma_i = \sigma'_i}$ for all~$i \in [p]$, and that~$\tree$ and~$\tree'$ are tidy parallel $2$-citelangis equivalent.
As we assumed that they were in normal form, we obtain that~$\tree = \tree'$.
The result follows.
\end{proof}

\begin{remark}
By \cref{thm:freeMessyCitelangisParallel}, the messy parallel $2$-citelangis operad (or $\Quad$ operad) can be fully understood from the messy parallel $2$-citelangis subalgebra of~$\FQSym_2$ generated by the $2$-permutation~$11$.
However, while the fully bounded cuttable $2$-permutations still provide a combinatorial model for the basis of this messy parallel $2$-citelangis algebra similarly to \cref{prop:combinatorialModelTidyParallelCitelangis}, the compositions~$\messyParallelCirc{i}$ of the messy parallel $2$-citelangis operad~$\messyCitelangisParallel$ on fully bounded cuttable $2$-permutations are more intricate than \cref{prop:combinatorialModelCompositionsTidyParallelCitelangis}.
\end{remark}


\subsubsection{Operations of~\texorpdfstring{$\messyCitelangisParallel[2]$}{MCit2} on bounded \texorpdfstring{$2$}{2}-posets}
\label{subsubsec:actionMessyParallelSignaleticBoundedPosets}

In this section, we observe that the result of any messy parallel permutation evaluation is the sum of all linear extensions of a well-chosen $2$-poset.
This observation allows us to encode the messy parallel permutation evaluation by an alternative combinatorial model and to study directly on this model the action of the messy parallel $2$-citelangis operad~$\messyCitelangisParallel[2]$.
It also motivates the introduction of the parallel $2$-poset operad that will be studied in \cref{subsubsec:parallelPosetOperad}.

\paraul{Operations on bounded $2$-posets and parallel poset evaluations}
We first define some operations on the following $2$-posets.

\begin{definition}
A $2$-poset~$\le_M$ is \defn{bounded} if it admits a unique minimal element~${\min(\le_M)}$ and a unique maximal element~${\max(\le_M)}$.
We then denote by~$\le_{M_*}$ the poset induced by~$\le_M$ on~${M_* \eqdef M \ssm \{\min(\le_M)\}}$ and by~$\le_{M^*}$ the poset induced by~$\le_M$ on~$M^* \eqdef M \ssm \{\max(\le_M)\}$.
\end{definition}

\begin{definition}
\label{def:messyCitelangisParallelActionPosets}
Consider an operation~$\operation \in \Operations_2 \eqdef \{\op{l,l},\op{l,r},\op{r,l},\op{r,r}\}$ and two bounded $2$-posets~$\le_M$ and~$\le_N$ of degrees~$m$ and~$n$ respectively.
We define a bounded $2$-poset ${\le_P} \eqdef {\le_M} \; \operation \; {\le_N}$~by:
\begin{itemize}
\item $P \eqdef M \sqcup N[m]$ where $N[m]$ is the $2$-set~$(m+1)^{\{2\}} \dots (m+n)^{\{2\}}$ obtained from~$N$ by shifting each element by the degree~$m$ of~$M$ as in \cref{def:shiftPoset},
\item in $\le_P$, the elements of $M$ (respectively of~$N[m]$) are ordered by~$\le_M$ (resp.~by~$\le_{N[m]}$), and the only other cover relations are $\{\min(\le_M), \min(\le_{N[m]})\}$ whose order is given by the first traffic signal of~$\operation$, and $\{\max(\le_M), \max(\le_{N[m]})\}$ whose order is given by the second traffic signal of~$\operation$. Formally, the comparison~$x \le_P y$ holds for~$x, y \in P$ if and only if one of the following statements holds:
	\begin{itemize}
	\item $x \in M$, $y \in M$ and $x \le_M y$, or
	\item $x \in N[m]$, $y \in N[m]$ and $x \le_{N[m]} y$, or
	\item $x = \min(\le_M)$, $y \in N[m]$ and $\operation_1 = {\op{l}}$, or
	\item $x = \min(\le_{N[m]})$, $y \in M$ and $\operation_1 = {\op{r}}$, or
	\item $x \in N[m]$, $y = \max(\le_M)$ and $\operation_2 = {\op{l}}$, or
	\item $x \in M$, $y = \max(\le_{N[m]})$ and $\operation_2 = {\op{r}}$.
	\end{itemize}
\end{itemize}
\end{definition}

\begin{remark}
\label{rem:sumsPosetsParallel}
\cref{def:messyCitelangisParallelActionPosets} can be conveniently rephrased in terms of ordered sums and disjoint unions of posets presented in \cref{def:disjointUnionOrderedSum}.
Indeed,
\begin{align*}
{\le_M} \op{l,l} {\le_N} & = \{\min(\le_M)\} + \big( {\le_{M_*^*}} \sqcup {\le_{N[m]}} \big) + \{\max(\le_M)\}, \\
{\le_M} \op{l,r} {\le_N} & = \{\min(\le_M)\} + \big( {\le_{M_*}} \sqcup {\le_{N[m]^*}} \big) + \{\max(\le_{N[m]})\}, \\
{\le_M} \op{r,l} {\le_N} & = \{\min(\le_{N[m]})\} + \big( {\le_{M^*}} \sqcup {\le_{N[m]_*}} \big) + \{\max(\le_M)\}, \\
{\le_M} \op{r,r} {\le_N} & = \{\min(\le_{N[m]})\} + \big( {\le_M} \sqcup {\le_{N[m]_*^*}} \big) + \{\max(\le_{N[m]})\},
\end{align*}
where~$m = |M|$.
\end{remark}

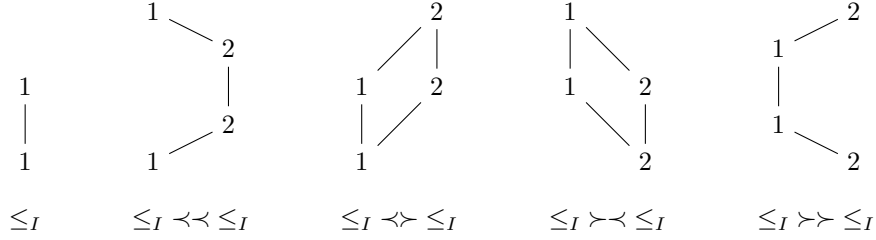
\begin{figure}
	\centerline{\begin{tabular}{c@{\hspace{1.2cm}}c@{\hspace{1.2cm}}c@{\hspace{1.2cm}}c@{\hspace{1.2cm}}c}
		\begin{tikzpicture}
			\node (1a) at (0,0) {$1$};
			\node (1b) at (0,1) {$1$};
			\draw (1a) -- (1b);
		\end{tikzpicture}
		&
		\begin{tikzpicture}
			\node (1a) at (0,0) {$1$};
			\node (1b) at (0,2) {$1$};
			\node (2a) at (1,.5) {$2$};
			\node (2b) at (1,1.5) {$2$};
			\draw (1a) -- (2a);
			\draw (2a) -- (2b);
			\draw (2b) -- (1b);
		\end{tikzpicture}
		&
		\begin{tikzpicture}
			\node (1a) at (0,0) {$1$};
			\node (1b) at (0,1) {$1$};
			\node (2a) at (1,1) {$2$};
			\node (2b) at (1,2) {$2$};
			\draw (1a) -- (1b);
			\draw (1a) -- (2a);
			\draw (1b) -- (2b);
			\draw (2a) -- (2b);
		\end{tikzpicture}
		&
		\begin{tikzpicture}
			\node (1a) at (0,1) {$1$};
			\node (1b) at (0,2) {$1$};
			\node (2a) at (1,0) {$2$};
			\node (2b) at (1,1) {$2$};
			\draw (1a) -- (1b);
			\draw (2a) -- (1a);
			\draw (2a) -- (2b);
			\draw (2b) -- (1b);
		\end{tikzpicture}
		&
		\begin{tikzpicture}
			\node (1a) at (0,.5) {$1$};
			\node (1b) at (0,1.5) {$1$};
			\node (2a) at (1,0) {$2$};
			\node (2b) at (1,2) {$2$};
			\draw (1a) -- (1b);
			\draw (1b) -- (2b);
			\draw (2a) -- (1a);
		\end{tikzpicture}
		\\[.2cm]
		${\le_I}$ &
		${\le_I} \op{l,l} {\le_I}$ &
		${\le_I} \op{l,r} {\le_I}$ &
		${\le_I} \op{r,l} {\le_I}$ &
		${\le_I} \op{r,r} {\le_I}$
	\end{tabular}}
	\caption{Four parallel operations on bounded $2$-posets. See \cref{def:messyCitelangisParallelActionPosets}.}
	\label{fig:exmMessyCitelangisParallelActionPosets}
\end{figure}

\begin{figure}
	\centerline{$
	\begin{tikzpicture}[baseline={([yshift=-.8ex]current bounding box.center)}, blue]
		\node (1a) at (0,.5) {$1$};
		\node (1b) at (0,1.5) {$1$};
		\node (2a) at (1,0) {$2$};
		\node (2b) at (1,2) {$2$};
		\node (3a) at (2,.5) {$3$};
		\node (3b) at (2,1.5) {$3$};
		\draw (1a) -- (1b);
		\draw (1b) -- (2b);
		\draw (2a) -- (1a);
		\draw (2a) -- (3a);
		\draw (3a) -- (3b);
		\draw (3b) -- (2b);
	\end{tikzpicture}
	\op{r,r}
	\begin{tikzpicture}[baseline={([yshift=-.8ex]current bounding box.center)}, red]
		\node (1a) at (0,1) {$1$};
		\node (1b) at (0,2) {$1$};
		\node (2a) at (1,0) {$2$};
		\node (2b) at (1,1) {$2$};
		\draw (1a) -- (1b);
		\draw (2a) -- (1a);
		\draw (2a) -- (2b);
		\draw (2b) -- (1b);
	\end{tikzpicture}
	=
	\begin{tikzpicture}[baseline={([yshift=-.8ex]current bounding box.center)}]
		\node (1aM) at (0,1.5) {\blue $1$};
		\node (1bM) at (0,2.5) {\blue $1$};
		\node (2aM) at (1,1) {\blue $2$};
		\node (2bM) at (1,3) {\blue $2$};
		\node (3aM) at (2,1.5) {\blue $3$};
		\node (3bM) at (2,2.5) {\blue $3$};
		\node (1aN) at (3,2) {\red $4$};
		\node (1bN) at (3,3.5) {\red $4$};
		\node (2aN) at (4,0.5) {\red $5$};
		\node (2bN) at (4,2) {\red $5$};
		\draw [blue] (1aM) -- (1bM);
		\draw [red] (1aN) -- (1bN);
		\draw [blue] (1bM) -- (2bM);
		\draw [blue] (2aM) -- (1aM);
		\draw [blue] (2aM) -- (3aM);
		\draw [red] (2aN) -- (1aN);
		\draw (2aN) -- (2aM);
		\draw [red] (2aN) -- (2bN);
		\draw (2bM) -- (1bN);
		\draw [red] (2bN) -- (1bN);
		\draw [blue] (3bM) -- (2bM);
		\draw [blue] (3aM) -- (3bM);
	\end{tikzpicture}
	=
	\EvalPosetParallel \left(
		\begin{tikzpicture}[baseline={([yshift=-.8ex]current bounding box.center)}, level/.style={sibling distance=18mm/#1, level distance = 1cm/sqrt(#1)}]
			\node [rectangle, draw] {$\op{r,r}$}
				child {node [rectangle, draw] {$\op{r,r}$}
					child {node {}}
					child {node [rectangle, draw] {$\op{l,l}$}
						child {node {}}
						child {node {}}
					}
				}
				child {node [rectangle, draw] {$\op{r,l}$}
					child {node {}}
					child {node {}}
				}
			;
		\end{tikzpicture}
	\right)
	$}
	\caption{Another example of parallel operation on bounded $2$-posets. See \cref{def:messyCitelangisParallelActionPosets}.}
	\label{fig:exmMessyCitelangisParallelActionPosetsBis}
\end{figure}
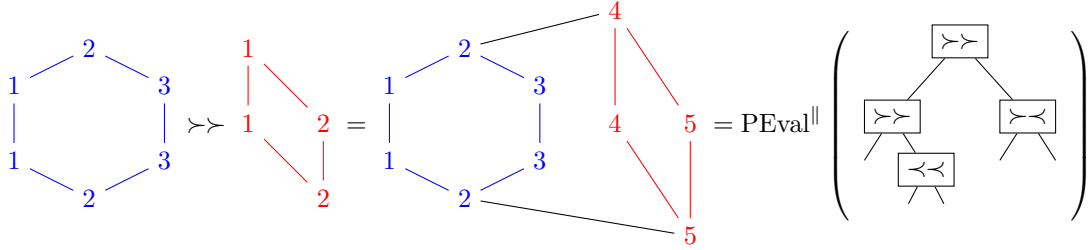

For example, let~$\le_I$ denote the $2$-chain on~$1^{\{2\}}$ represented in \cref{fig:exmMessyCitelangisParallelActionPosets}\,(left).
Then the four bounded $2$-posets~${\le_I} \; \operation \; {\le_I}$, for~${\operation \in \Operations_2}$, are represented in \cref{fig:exmMessyCitelangisParallelActionPosets}\,(left).
See \cref{fig:exmMessyCitelangisParallelActionPosetsBis} for another example.
The next statement clearly follows from \cref{rem:sumsPosetsParallel}.

\begin{lemma}
For any operation~$\operation \in \Operations_2$ and any two bounded $2$-posets~$\le_M$ and~$\le_N$ of degrees~$m$ and~$n$ respectively, ${\le_M \; \operation \; \le_N}$ is a bounded $2$-poset of degree~$m+n$.
\end{lemma}

This statement allows to define the evaluation of a syntax tree on posets.
See \cref{fig:exmMessyCitelangisParallelActionPosetsBis}.

\begin{definition}
\label{def:posetParallel}
Denote by~$\EvalPosetParallel(\tree ; \le_{M_1}, \dots, \le_{M_p})$ the evaluation of a syntax tree~${\tree \in \Syntax[\Operations_2]}$ of arity~$p$ on $p$ bounded $2$-posets~${\le_{M_1}, \dots, \le_{M_p}}$ using the operations of \cref{def:messyCitelangisParallelActionPosets}. 
The \defn{parallel poset evaluation} of~$\tree \in \Syntax[\Operations_2]$ is then~$\EvalPosetParallel(\tree) \eqdef \EvalPosetParallel(\tree ; \le_I, \dots, \le_I)$, where~$\le_I$ is the $2$-chain on~$1^{\{2\}}$.
\end{definition}

\begin{remark}
\label{rem:lattices}
It is not difficult to check that, if~$\le_M$ and~$\le_N$ are lattices, then~${\le_M} \, \operation \, {\le_N}$ is a lattice for any operation~$\operation \in \Operations_2$.
Therefore, for any~$\tree \in \Syntax[\Operations_2]$, the parallel poset evaluation~$\EvalPosetParallel(\tree)$ is a lattice $2$-poset.
We will characterize the parallel poset evaluations later in \cref{prop:characterizationParallelPosetEvaluations}.
\end{remark}

\paraul{Bounded cuts in $2$-posets}
We now aim at characterizing parallel poset evaluations.
As in \cref{subsubsec:actionTidyParallelSignaletic}, this understanding goes through bounded cuts in bounded $2$-posets.

\begin{definition}
\label{def:boundedCutPoset}
Let~$\le_M$ be a bounded $2$-poset of degree~$n$ and~$\gamma \in [n-1]$.
We say that $\gamma$ is a \defn{bounded cut} of~$\le_M$ if we can decompose~$M$ into~$M = \{\min(\le_M), \max(\le_M)\} \sqcup L \sqcup R$ such that, for all~$\ell \in L$ and~$r \in R$, we have~$\ell \le \gamma < r$ and~$\ell$ and~$r$ are incomparable for~$\le_M$.
We denote by~$\bcuts(\le_M)$ the set of bounded cuts of~$\le_M$.
We say that the bounded $2$-poset~$\le_M$ is \defn{bounded cuttable} if it admits a bounded cut, and \defn{bounded uncuttable} if it admits no bounded cut.
\end{definition}

The following statements are similar to \cref{lem:boundedCutsRestriction,prop:equivalenceFullyBoundedCuttable}.

\begin{lemma}
\label{lem:boundedCutsRestrictionPoset}
Consider~$L \subseteq [m]$ and~$\gamma \in [\min(L), \max(L)-1]$, and let~$\gamma^{|L} \eqdef |[\gamma] \cap L|$ denote the number of elements of~$L$ between~$1$ and~$\gamma$.
If~$\gamma$ is a bounded cut of a bounded $2$-poset~$\le_M$ of degree~$m$, then $\gamma^{|L}$ is a bounded cut of its restriction~$\le_{M^{|L}}$.
\end{lemma}

\begin{proposition}
\label{prop:equivalenceFullyBoundedCuttablePoset}
The following conditions are equivalent for a bounded $2$-poset of degree~$n$:
\begin{enumerate}[(i)]
\item its restriction to any interval of~$[n]$ of size at least $2$ is bounded cuttable,
\item its restriction to any subset of~$[n]$ of size at least $2$ is bounded cuttable.
\end{enumerate}
\end{proposition}

\begin{definition}
\label{def:fullyBoundedCuttablePoset}
A bounded $2$-poset is \defn{fully bounded cuttable} if it satisfies the equivalent conditions of \cref{prop:equivalenceFullyBoundedCuttablePoset}.
\end{definition}

Note that the only multiposet on~$\{1,1\}$ is fully bounded cuttable as there is no subset of size at least~$2$.
The following statements are similar to \cref{lem:operationImpliesBoundedCut,lem:boundedCutImpliesOperation}.

\begin{lemma}
\label{lem:operationImpliesBoundedCutPoset}
For any bounded $2$-posets~$\le_M$ and~$\le_N$, and any operation~${\operation \in \Operations_2}$, the degree~$m$ of~$\le_M$ is a bounded cut of~$\le_M \, \operation \, \le_N$.
\end{lemma}

\begin{lemma}
\label{lem:boundedCutImpliesOperationPoset}
For any bounded $2$-poset~$\le_P$ of degree~$p$ and any bounded cut~$\gamma \in \bcuts(\le_P)$, there is a unique~${\operation \in \Operations_2}$ (defined by $\operation_1 \eqdef {\op{l}}$ if~$\min(\le_P) \le \gamma$ while~$\operation_1 \eqdef {\op{r}}$ if~$\min(\le_P) > \gamma$, and $\operation_2 \eqdef {\op{l}}$ if~$\max(\le_P) \le \gamma$ while~$\operation_2 \eqdef {\op{r}}$ if~$\max(\le_P) > \gamma$) such that~${{\le_P} = {\le_{P^{|[\gamma]}}} \, \operation \, {\le_{P^{|[p] \ssm [\gamma]}}}}$.
\end{lemma}

\begin{remark}
\label{rem:algoPosetParallel}
Similarly to \cref{rem:algoTidyEvalParallel}, observe that \cref{lem:boundedCutImpliesOperationPoset} gives an inductive algorithm to compute all decompositions of a given bounded $2$-poset~$\le_P$ as an evaluation of the form ${{\le_P} = \EvalPosetParallel(\tree ; \le_{M_1}, \dots, \le_{M_p})}$.
Namely, $\le_P$ admits
\begin{itemize}
\item the trivial evaluation~$\le_0 = \EvalPosetParallel(\one ; \le_0)$, where~$\one$ is the unit syntax tree with no node and a single leaf, and 
\item the evaluation~${\le_P} = \EvalPosetParallel(\tree ; \le_{M_1}, \dots, \le_{M_l}, \le_{N_1}, \dots, \le_{N_r})$, for any~${\gamma \in \bcuts(\le_P)}$ and any~${\le_{P^{|[\gamma]}}} = \EvalPosetParallel(\tree[l] ; \le_{M_1}, \dots, \le_{M_l})$ and ${\le_{P^{|[\ell] \ssm [\gamma]}}} = \EvalPosetParallel(\tree[r] ; \le_{N_1}, \dots, \le_{N_r})$, where~$\tree$ is the syntax tree with root~$\operation \in \Operations_2$ defined by \cref{lem:boundedCutImpliesOperationPoset} and with subtrees~$\tree[l]$ and~$\tree[r]$.
\end{itemize}
This algorithm implies the existence of decompositions of the form~${\le_P} = \EvalPosetParallel(\tree ; \le_{M_1}, \dots, \le_{M_p})$ where~$\le_{M_1}, \dots, \le_{M_p}$ are bounded uncuttable.
\end{remark}

Note that not all bounded $2$-posets are obtained by evaluating syntax trees on~$\Operations_2$.
For instance, the evaluation the syntax trees of arity~$2$ only produces the four bounded $2$-posets of \cref{fig:exmMessyCitelangisParallelActionPosets}, thus only two of the six linear bounded $2$-posets of degree~$2$.
We now characterize the bounded $2$-posets which are parallel poset evaluations of syntax trees.

\begin{proposition}
\label{prop:characterizationParallelPosetEvaluations}
The parallel poset evaluations of the syntax trees of~$\Syntax[\Operations_2]$ are precisely the fully bounded cuttable $2$-posets.
\end{proposition}

\begin{proof}
The proof is exactly the same as \cref{prop:characterizationTidyParallelPermutationEvaluations}.
Consider first a parallel poset evaluation~${\le_P} = \EvalPosetParallel(\tree)$ with~$\tree \in \Syntax[\Operations_2](\ell)$.
We prove by induction on~$\ell$ that~$\le_P$ is fully bounded cuttable.
If~$\ell = 1$, there is nothing to prove.
Assume that~$\ell \ge 2$ and let~$1 \le a < b \le \ell$.
Let~$\tree[l]$ and~$\tree[r]$ denote the left and right subtrees of~$\tree$, and let~$\gamma$ be the arity of~$\tree[l]$, so that~${\le_{P^{|[\gamma]}}} = \EvalPosetParallel(\tree[l])$ and~$\le_{M^{|[\ell] \ssm [\gamma]}} = \EvalPosetParallel(\tree[r])$.
We distinguish three cases:
\begin{itemize}
\item Assume that~$b \le \gamma$. Since~$\le_{M^{|[\ell]}} = \EvalPosetParallel(\tree[l])$ is fully bounded cuttable by induction hypothesis and bounded cuts are preserved by restriction by \cref{lem:boundedCutsRestrictionPoset}, we obtain that~${\le_{P^{|[a,b]}}} = (\le_{P^{|[\gamma]}})^{|[a,b]}$ is bounded cuttable.
\item Assume that~$\gamma \le a$. The argument is similar since~${\le_{P^{|[a,b]}}} = (\le_{P^{|[\ell] \ssm [\gamma]}})^{|[a-\gamma,b-\gamma]}$.
\item Assume finally that~$a < \gamma < b$. By \cref{lem:operationImpliesBoundedCutPoset}, $\gamma$ is a bounded cut of~$\le_P$. Therefore, $\gamma-a$ is a bounded cut of~$\le_{P^{|[a,b]}}$ by \cref{lem:boundedCutsRestrictionPoset}.
\end{itemize}

Conversely, consider now a fully bounded cuttable $2$-poset~$\le_P$ of degree~$\ell$.
Similarly to \cref{rem:algoPosetParallel}, we prove by induction on~$\ell$ that~$\le_P$ is the parallel poset evaluation of a syntax tree.
If~$\ell = 1$, then~${\le_P} = \EvalPosetParallel(\one)$.
If~$\ell \ge 2$, then~$\le_P$ admits at least one bounded cut~$\gamma$ by assumption.
Moreover, $\le_{P^{|[\gamma]}}$ and~$\le_{P^{|[\ell] \ssm [\gamma]}}$ are both fully bounded cuttable by \cref{lem:boundedCutsRestrictionPoset}.
By induction, we obtain that~${\le_{P^{|[\gamma]}}} = \EvalPosetParallel(\tree[l])$ and~${\le_{P^{|[\ell] \ssm [\gamma]}}} = \EvalPosetParallel(\tree[r])$.
Then~${\le_P} = \EvalPosetParallel(\tree)$, where~$\tree$ is the syntax tree with root~$\operation \in \Operations_2$ defined by \cref{lem:boundedCutImpliesOperation} and with subtrees~$\tree[l]$ and~$\tree[r]$.
\end{proof}

\paraul{Connections between~$\tidyEvalPermParallel$, $\messyEvalPermParallel$ and~$\EvalPosetParallel$}
We have seen in \cref{lem:LexMinParallel} that $\tidyEvalPermParallel(\tree) = \LexMin \big( \messyEvalPermParallel(\tree) \big)$ for any syntax tree~$\tree \in \Syntax[\Operations_2]$.
We now compare with the parallel poset evaluation~$\EvalPosetParallel(\tree)$.

For a multiposet~$\le_M$, we let
\[
\LinExt(\le_M) \eqdef \sum_{\mu \in \linearExtensions(\le_M)} \mu
\]
be the sum of all linear extensions of~$\le_M$ (meaning the formal sum in~$\K\Perm_2(|M|)$).
We obtain the following statement, illustrated in \cref{fig:LinExtEvalPosetParallel}.

\begin{lemma}
\label{lem:LinExtEvalPosetParallel}
For any syntax tree~$\tree \in \Syntax[\Operations_2]$, we have~$\messyEvalPermParallel(\tree) = \LinExt \big( \EvalPosetParallel(\tree) \big)$.
\end{lemma}

\begin{proof}
\cref{rem:sumsPosetsParallel,lem:disjointUnionOrderedSumLinearExtensions} prove that
\[
\LinExt(\le_M) \; \operation \, \LinExt(\le_N) = \LinExt({\le_M} \; \operation \; {\le_N}).
\]
for any operation~$\operation\in\Operations_2$ and any bounded $2$-posets~$\le_M$ and~$\le_N$.
The statement then follows by induction on syntax trees.
\end{proof}

\begin{figure}[t]
	\centerline{$
	\messyEvalPermParallel \left(
		\begin{tikzpicture}[baseline={([yshift=-.8ex]current bounding box.center)}, level/.style={sibling distance=18mm/#1, level distance = 1cm/sqrt(#1)}]
			\node [rectangle, draw] {$\op{r,r}$}
				child {node [rectangle, draw] {$\op{r,r}$}
					child {node {}}
					child {node [rectangle, draw] {$\op{l,l}$}
						child {node {}}
						child {node {}}
					}
				}
				child {node [rectangle, draw] {$\op{r,l}$}
					child {node {}}
					child {node {}}
				}
			;
		\end{tikzpicture}
	\right)
	=
	\LinExt \left(
		\begin{tikzpicture}[baseline={([yshift=-.8ex]current bounding box.center)}]
			\node (1aM) at (0,1.5) {$1$};
			\node (1bM) at (0,2.5) {$1$};
			\node (2aM) at (1,1) {$2$};
			\node (2bM) at (1,3) {$2$};
			\node (3aM) at (2,1.5) {$3$};
			\node (3bM) at (2,2.5) {$3$};
			\node (1aN) at (3,2) {$4$};
			\node (1bN) at (3,3.5) {$4$};
			\node (2aN) at (4,0.5) {$5$};
			\node (2bN) at (4,2) {$5$};
			\draw (1aM) -- (1bM);
			\draw (1aN) -- (1bN);
			\draw (1bM) -- (2bM);
			\draw (2aM) -- (1aM);
			\draw (2aM) -- (3aM);
			\draw (2aN) -- (1aN);
			\draw (2aN) -- (2aM);
			\draw (2aN) -- (2bN);
			\draw (2bM) -- (1bN);
			\draw (2bN) -- (1bN);
			\draw (3bM) -- (2bM);
			\draw (3aM) -- (3bM);
		\end{tikzpicture}
	\right)
	=
	\begin{array}{c@{\;}c@{\;}c@{\;}c@{\;}c@{\;}c}
		5211332454 & + & 5211332544 & + \\
		5211334254 & + & 5211335244 & + \\
		5211334524 & + & \dots\dots     \\
		\ldots\ldots & + & 5542331124
	\end{array}
	$}
	\caption{Illustration of \cref{lem:LinExtEvalPosetParallel}.}
	\label{fig:LinExtEvalPosetParallel}
\end{figure}

Conversely, we can obtain~$\EvalPosetParallel(\tree)$ from $\messyEvalPermParallel(\tree)$ by intersecting all linear orders corresponding to the $2$-permutations that appear in~$\messyEvalPermParallel(\tree)$.

\medskip
We now describe the connection between $\tidyEvalPermParallel(\tree)$ and~$\EvalPosetParallel(\tree)$.
By \cref{lem:LinExtEvalPosetParallel,lem:LexMinParallel}, the tidy parallel permutation evaluation~$\tidyEvalPermParallel(\tree)$ is the lexicographically minimal linear extension of~$\EvalPosetParallel(\tree)$.
It can thus be obtained from~$\EvalPosetParallel(\tree)$ by repeatedly reading and deleting the minimal source.

Conversely, there are several syntax trees~$\tree$ whose parallel poset evaluation~$\EvalPosetParallel(\tree)$ has the same lexicographically minimal linear extension.
Given a bounded cuttable $2$-permutation~$\sigma$, we can find all $2$-posets~$\EvalPosetParallel(\tree)$ whose lexicographically minimal linear extension is~$\sigma$ by first using \cref{rem:tidyEvalParallel} to compute all possible syntax trees~$\tree$ such that~$\sigma = \tidyEvalPermParallel(\tree)$, and then applying \cref{def:messyCitelangisParallelActionPosets,def:posetParallel} to get their parallel poset evaluations~$\EvalPosetParallel(\tree)$.

\paraul{Relations among parallel poset evaluations}
We now study which syntax trees evaluate to the same bounded $2$-poset.
Observe first the following quadratic relation:
\begin{equation}
\label{eq:relationPosetsParallel}
	\EvalPosetParallel \left( \compoL{r,r}{l,l} \right)
	= 
	\begin{tikzpicture}[baseline={([yshift=-.8ex]current bounding box.center)}]
		\node (1a) at (-1,.5) {$1$};
		\node (1b) at (-1,1.5) {$1$};
		\node (2a) at (0,0) {$2$};
		\node (2b) at (0,2) {$2$};
		\draw (1a) -- (1b);
		\draw (1b) -- (2b);
		\draw (2a) -- (1a);
	\end{tikzpicture}
	\hspace{-.3cm}\op{l,l}
	\begin{tikzpicture}[baseline={([yshift=-.8ex]current bounding box.center)}]
		\node (1a) at (0,0) {$1$};
		\node (1b) at (0,1) {$1$};
		\draw (1a) -- (1b);
	\end{tikzpicture}
	=
	\begin{tikzpicture}[baseline={([yshift=-.8ex]current bounding box.center)}]
		\node (1a) at (0,.5) {$1$};
		\node (1b) at (0,1.5) {$1$};
		\node (2a) at (1,0) {$2$};
		\node (2b) at (1,2) {$2$};
		\node (3a) at (2,.5) {$3$};
		\node (3b) at (2,1.5) {$3$};
		\draw (1a) -- (1b);
		\draw (1b) -- (2b);
		\draw (2a) -- (1a);
		\draw (2a) -- (3a);
		\draw (3a) -- (3b);
		\draw (3b) -- (2b);
	\end{tikzpicture}
	=
	\begin{tikzpicture}[baseline={([yshift=-.8ex]current bounding box.center)}]
		\node (1a) at (0,0) {$1$};
		\node (1b) at (0,1) {$1$};
		\draw (1a) -- (1b);
	\end{tikzpicture}
	\op{r,r}\hspace{-.3cm}
	\begin{tikzpicture}[baseline={([yshift=-.8ex]current bounding box.center)}]
		\node (1a) at (0,0) {$1$};
		\node (1b) at (0,2) {$1$};
		\node (2a) at (1,.5) {$2$};
		\node (2b) at (1,1.5) {$2$};
		\draw (1a) -- (2a);
		\draw (2a) -- (2b);
		\draw (2b) -- (1b);
	\end{tikzpicture}
	=
	\EvalPosetParallel \left( \compoR{r,r}{l,l} \right)
\end{equation}

\warning
The operations of \cref{def:messyCitelangisParallelActionPosets} on bounded $2$-posets do not satisfy the other messy parallel $2$-citelangis relations, and thus do not define a messy parallel $2$-citelangis algebra.
In fact, we will now show that \cref{eq:relationPosetsParallel} is the only relation that leads to the same parallel poset evaluations.

\begin{lemma}
\label{lem:uniqueRelationParallelPosetOperad}
For any operations~$\operation, \operation' \in \Operations_2$ and any bounded $2$-posets~${\le_M}, {\le_N}, {\le_{M'}}, {\le_{N'}}$, we have~${\le_M} \; \operation \; {\le_N} = {\le_{M'}} \; \operation' \; {\le_{N'}}$ if and only if 
\begin{itemize}
\item either~$\operation = \operation'$, ${\le_M} = {\le_{M'}}$ and ${\le_N} = {\le_{N'}}$, 
\item or~$\operation = {\op{l,l}}$ and~$\operation' = {\op{r,r}}$, and there exists a bounded $2$-poset~$\le_O$ such that ${{\le_M} = {\le_{M'}} \op{r,r} {\le_O}}$ and~${{\le_{N'}} = {\le_O} \op{l,l} {\le_N}}$.
\end{itemize}
\end{lemma}

\begin{proof}
The ``if'' direction is immediate as the second case was already observed in \cref{eq:relationPosetsParallel}.

The ``only if'' direction relies on the fact that for any~$\operation \in \Operations_2$ and any two bounded $2$-posets~$\le_M$ and~$\le_N$, we can easily identify the positions of the minimums and maximums of~$\le_M$ and~$\le_N$ in the $2$-poset~${\le_P} \eqdef {\le_M} \; \operation \; {\le_N}$.
For this, let~$a \eqdef \min(\le_P)$ and~$t \eqdef \max(\le_P)$, let~$b$ and~$c$ denote the minimal and maximal values covering~$a$ in~$\le_P$ and let~$r$ and~$s$ denote the minimal and maximal values covered by~$t$ in~$\le_P$.
Then the $2$-posets~$\le_M$ and~$\le_N$ are (up to the label shift) the sub-$2$-posets of~$\le_P$ induced by $\tilde M$ and $\tilde N$ depending on $\operation$ as follows:
\begin{itemize}
\item if~$\operation = {\op{l,l}}$ then $\tilde M \eqdef [a,t] \ssm [c,s]$ and~$\tilde N \eqdef [c,s] = [a,s] \ssm \{a\} = [c,t] \ssm \{t\}$,
\item if~$\operation = {\op{l,r}}$ then $\tilde M \eqdef [a,r]$ and~$\tilde N \eqdef [c,t]$,
\item if~$\operation = {\op{r,l}}$ then $\tilde M \eqdef [b,t]$ and~$\tilde N \eqdef [a,s]$,
\item if~$\operation = {\op{r,r}}$ then $\tilde M \eqdef [b,r] = [a,r] \ssm \{a\} = [b,t] \ssm \{t\}$ and~$\tilde N \eqdef [a,t] \ssm [b,r]$,
\end{itemize}
where~$[x,y]$ denotes the interval between~$x$ and~$y$ in~$\le_P$.
We now argue the ``only if'' direction of the statement by case analysis:
\begin{itemize}
\item Case $\operation = \operation'$: We immediately obtain that~${\le_M} = {\le_{M'}}$ and~${\le_N} = {\le_{N'}}$ since the description above enables us to reconstruct~$\le_M$ and~$\le_N$ from~${\le_M} \; \operation \; {\le_N}$ when we know~$\operation$.
\item Case~$\operation = {\op{l,r}}$ and~$\operation' = {\op{r,l}}$\,: Impossible since~$\min({\le_M} \op{l,r} {\le_N}) < \max({\le_M} \op{l,r} {\le_N})$ while~$\min({\le_{M'}} \op{r,l} {\le_{N'}}) > \max({\le_{M'}} \op{r,l} {\le_{N'}})$ (where~$<$ and~$>$ denote the natural order on the integer values).
\item Case~$\operation = {\op{l,r}}$ and~$\operation' = {\op{l,l}}$\,: Impossible since the above discussion would imply~${\tilde N = [c,t]}$ and~${\tilde N' = [c,t] \ssm \{t\}}$ (up to the label shift), contradicting the fact that both~$\tilde N$ and~$\tilde N'$ should be even (because we work with $2$-posets).
\item Case~$\operation = {\op{l,l}}$ and~$\operation' = {\op{r,r}}$\,: Let~$\le_O$ denote the bounded sub-$2$-posets of~$\le_P$ induced by~$\tilde O \eqdef [a,t] \ssm ([b,r] \cup [c,s])$. Since~$\le_N$ is the sub-$2$-poset of~$\le_P$ induced by~$\tilde N \eqdef [c,s]$ and~$\le_{M'}$ is the sub-$2$-poset of~$\le_P$ induced by~$\tilde N' \eqdef [b,r]$, we obtain that~${{\le_M} = {\le_{M'}} \op{r,r} {\le_O}}$ and~${{\le_{N'}} = {\le_O} \op{l,l} {\le_N}}$.
\item All other cases are symmetric.
\end{itemize}
An alternative argument would be to show that for any bounded $2$-posets~$\le_M$ and~$\le_N$, the $2$-posets ${\le_M} \op{l,r} {\le_N}$ and~${\le_M} \op{r,l} {\le_N}$ have a unique cut and to apply \cref{lem:operationImpliesBoundedCutPoset,lem:boundedCutImpliesOperationPoset} to prove that the only remaining option is that~$\operation = {\op{l,l}}$ and~$\operation' = {\op{r,r}}$.
\end{proof}

\begin{remark}
\cref{lem:uniqueRelationParallelPosetOperad} fails for multiposets, as we used a parity argument on $2$-posets.
\end{remark}

Finally, the next statement is similar to \cref{prop:uniqueTidyParallelCitelangisPermutationEvaluation}.

\begin{proposition}
\label{prop:uniqueParallelPosetEvaluation}
For any syntax trees~$\tree , \tree' \in \Syntax[\Operations_2]$ of arity~$p$ and~$p'$ respectively, and any bounded uncuttable $2$-posets~$\le_{M_1}, \dots, \le_{M_p}$, $\le_{M'_1}, \dots, \le_{M'_{p'}}$, if
\[
\EvalPosetParallel(\tree ; \le_{M_1}, \dots, \le_{M_p}) = \EvalPosetParallel(\tree' ; \le_{M'_1}, \dots, \le_{M'_{p'}}),
\]
then $p = p'$, $\tree = \tree'$ modulo rewritings using \cref{eq:relationPosetsParallel} and~$\le_{M_i} = \le_{M'_i}$ for all~$i \in [p]$.
\end{proposition}

\begin{proof}
Immediate by induction from \cref{lem:uniqueRelationParallelPosetOperad}.
\end{proof}

\paraul{Parallel poset evaluations of parallel $2$-citelangis normal forms}
As the operations of \cref{def:messyCitelangisParallelActionPosets} on bounded $2$-posets do not satisfy all the messy parallel $2$-citelangis relations, the parallel poset evaluations of all syntax trees result to too many different $2$-posets.
To obtain an alternative combinatorial model for the messy parallel $2$-citelangis operad, we thus need to restrict our attention to parallel poset evaluations of normal forms of the messy parallel $2$-citelangis rewriting system described in \cref{subsubsec:rewritingSystemCitelangis}.

\begin{definition}
\label{def:normalPosetParallel}
A \defn{parallel normal} $2$-poset is a $2$-poset that can be obtained as the parallel poset evaluation of a normal form of the messy parallel $2$-citelangis rewriting system.
\end{definition}

In particular, by \cref{rem:lattices,prop:characterizationParallelPosetEvaluations}, a parallel normal $2$-poset is a fully bounded cuttable lattice.
Besides these conditions, there seems to be no simple characterization of the parallel normal $2$-posets.

There is however a simple algorithm to decide whether a given fully bounded cuttable $2$-poset~$\le_P$ is parallel normal.
Namely, compute its lexicographically minimal linear extension~$\sigma$, then the normal form~$\tree$ of the $\op{l,r}$-tidy parallel citelangis rewriting system such that~$\sigma = \tidyEvalPermParallel(\tree)$ by \cref{rem:algoTidyEvalParallelNormalForms}, and check whether~${\le_P} = \EvalPosetParallel(\tree)$.

Finally, one can also consider the list of all parallel poset evaluations of the $8$ quadratic parallel $2$-signaletic right combs distinct from that of \cref{eq:relationPosetsParallel}.
These posets are given by

\medskip
\centerline{
\begin{tabular}{c@{\;}c@{\;}c@{\;}c@{\;}c@{\;}c@{\;}c@{\;}c}
	\begin{tikzpicture}[baseline={([yshift=-.8ex]current bounding box.center)}, yscale=.7]
		\node (1a) at (0,0) {$1$};
		\node (1b) at (0,5) {$1$};
		\node (2a) at (0,1) {$2$};
		\node (2b) at (0,4) {$2$};
		\node (3a) at (0,2) {$3$};
		\node (3b) at (0,3) {$3$};
		\draw (1a) -- (2a);
		\draw (2a) -- (3a);
		\draw (3a) -- (3b);
		\draw (3b) -- (2b);
		\draw (2b) -- (1b);
	\end{tikzpicture}
	&
	\begin{tikzpicture}[baseline={([yshift=-.8ex]current bounding box.center)}, yscale=.7]
		\node (1a) at (0,0) {$1$};
		\node (1b) at (-.5,2) {$1$};
		\node (2a) at (.5,1) {$2$};
		\node (2b) at (0,4) {$2$};
		\node (3a) at (.5,2) {$3$};
		\node (3b) at (.5,3) {$3$};
		\draw (1a) -- (1b);
		\draw (1a) -- (2a);
		\draw (1b) -- (2b);
		\draw (2a) -- (3a);
		\draw (3a) -- (3b);
		\draw (3b) -- (2b);
	\end{tikzpicture}
	&
	\begin{tikzpicture}[baseline={([yshift=-.8ex]current bounding box.center)}, yscale=.7]
		\node (1a) at (0,0) {$1$};
		\node (1b) at (-.8,1.5) {$1$};
		\node (2a) at (0,1) {$2$};
		\node (2b) at (0,2) {$2$};
		\node (3a) at (.8,2) {$3$};
		\node (3b) at (0,3) {$3$};
		\draw (1a) -- (1b);
		\draw (1a) -- (2a);
		\draw (1b) -- (3b);
		\draw (2a) -- (2b);
		\draw (2a) -- (3a);
		\draw (2b) -- (3b);
		\draw (3a) -- (3b);
	\end{tikzpicture}
	&
	\begin{tikzpicture}[baseline={([yshift=-.8ex]current bounding box.center)}, yscale=.7]
		\node (1a) at (-.5,2) {$1$};
		\node (1b) at (0,4) {$1$};
		\node (2a) at (0,0) {$2$};
		\node (2b) at (.5,3) {$2$};
		\node (3a) at (.5,1) {$3$};
		\node (3b) at (.5,2) {$3$};
		\draw (1a) -- (1b);
		\draw (2a) -- (1a);
		\draw (2a) -- (3a);
		\draw (2b) -- (1b);
		\draw (3a) -- (3b);
		\draw (3b) -- (2b);
	\end{tikzpicture}
	&
	\begin{tikzpicture}[baseline={([yshift=-.8ex]current bounding box.center)}, yscale=.7]
		\node (1a) at (-.8,1.5) {$1$};
		\node (1b) at (0,3) {$1$};
		\node (2a) at (0,1) {$2$};
		\node (2b) at (0,2) {$2$};
		\node (3a) at (0,0) {$3$};
		\node (3b) at (.8,1) {$3$};
		\draw (1a) -- (1b);
		\draw (2a) -- (2b);
		\draw (2b) -- (1b);
		\draw (3a) -- (1a);
		\draw (3a) -- (2a);
		\draw (3a) -- (3b);
		\draw (3b) -- (2b);
	\end{tikzpicture}
	&
	\begin{tikzpicture}[baseline={([yshift=-.8ex]current bounding box.center)}, yscale=.7]
		\node (1a) at (-.7,1) {$1$};
		\node (1b) at (-.7,2) {$1$};
		\node (2a) at (0,0) {$2$};
		\node (2b) at (0,1.5) {$2$};
		\node (3a) at (.8,1.5) {$3$};
		\node (3b) at (0,3) {$3$};
		\draw (1a) -- (1b);
		\draw (1b) -- (3b);
		\draw (2a) -- (1a);
		\draw (2a) -- (2b);
		\draw (2a) -- (3a);
		\draw (2b) -- (3b);
		\draw (3a) -- (3b);
	\end{tikzpicture}
	&
	\begin{tikzpicture}[baseline={([yshift=-.8ex]current bounding box.center)}, yscale=.7]
		\node (1a) at (-.7,1) {$1$};
		\node (1b) at (-.7,2) {$1$};
		\node (2a) at (0,1.5) {$2$};
		\node (2b) at (0,3) {$2$};
		\node (3a) at (0,0) {$3$};
		\node (3b) at (.8,1.5) {$3$};
		\draw (1a) -- (1b);
		\draw (1b) -- (2b);
		\draw (2a) -- (2b);
		\draw (3a) -- (1a);
		\draw (3a) -- (2a);
		\draw (3a) -- (3b);
		\draw (3b) -- (2b);
	\end{tikzpicture}
	&
	\begin{tikzpicture}[baseline={([yshift=-.8ex]current bounding box.center)}, yscale=.7]
		\node (1a) at (-.5,1) {$1$};
		\node (1b) at (-.5,2) {$1$};
		\node (2a) at (.5,1) {$2$};
		\node (2b) at (.5,2) {$2$};
		\node (3a) at (0,0) {$3$};
		\node (3b) at (0,3) {$3$};
		\draw (1a) -- (1b);
		\draw (1b) -- (3b);
		\draw (2a) -- (2b);
		\draw (2b) -- (3b);
		\draw (3a) -- (1a);
		\draw (3a) -- (2a);
	\end{tikzpicture}
	\\[2cm]
    \compoR{l,l}{l,l}
	&
    \compoR{l,r}{l,l}
	&
    \compoR{l,r}{l,r}
	&
    \compoR{r,l}{l,l}
	&
    \compoR{r,l}{r,l}
	&
    \compoR{r,r}{l,r}
	&
    \compoR{r,r}{r,l}
	&
    \compoR{r,r}{r,r}
\end{tabular}
}

\medskip
\noindent
were we represented below each $2$-poset the parallel $2$-signaletic comb from which it was evaluated.
A fully bounded cuttable $2$-poset is parallel normal if it does not contain any of these forbidden $2$-posets as patterns.
The meaning of patterns in bounded $2$-posets should be clear with the poset parallel compositions defined in \cref{def:parallelBoundedPosetOperad}.
Here, we prefer to illustrate intuitively the notion on an example:
\[
	\EvalPosetParallel \left(
		\begin{tikzpicture}[baseline=-.5cm, level 1/.style={sibling distance = .8cm, level distance = .6cm}, level 2/.style={sibling distance = .8cm, level distance = .6cm}]
			\node [rectangle, draw, fill=red!50] {$\op{l,r}$}
				child {node {\blue $11$}}
				child {node [rectangle, draw, fill=red!50] {$\op{l,r}$}
					child {node {\cyan $22$}}
					child {node {\green $33$}}
				}
			;
		\end{tikzpicture}
	\right)
	=
	\begin{tikzpicture}[baseline={([yshift=-.8ex]current bounding box.center)}, yscale=.7]
		\node (1a) at (0,0) {\blue $1$};
		\node (1b) at (-.8,1.5) {\blue $1$};
		\node (2a) at (0,1) {\cyan $2$};
		\node (2b) at (0,2) {\cyan $2$};
		\node (3a) at (.8,2) {\green $3$};
		\node (3b) at (0,3) {\green $3$};
		\draw[blue] (1a) -- (1b);
		\draw (1a) -- (2a);
		\draw (1b) -- (3b);
		\draw[cyan] (2a) -- (2b);
		\draw (2a) -- (3a);
		\draw (2b) -- (3b);
		\draw[green] (3a) -- (3b);
	\end{tikzpicture}
\]
is a pattern in
\[
	\EvalPosetParallel \left(
		\begin{tikzpicture}[baseline={([yshift=-.8ex]current bounding box.center)}, level/.style={sibling distance=1cm, level distance = .6cm}, level 1/.style={sibling distance=1.5cm}, level 4/.style={sibling distance=1.5cm}]
			\node [rectangle, draw] {$\op{r,l}$}
				child {node [rectangle, draw] {$\op{r,r}$}
					child {node {}}
					child {node [rectangle, draw] {$\op{l,l}$}
					child {node [rectangle, draw, fill=red!50] {$\op{l,r}$}
						child {node [rectangle, draw, fill=blue!50] {$\op{r,r}$}
							child {node {}}
							child {node {}}
						}
						child {node [rectangle, draw, fill=red!50] {$\op{l,r}$}
							child {node [rectangle, draw, fill=cyan!50] {$\op{l,r}$}
								child {node {}}
								child {node {}}
							}
							child {node [rectangle, draw, fill=green!50] {$\op{r,l}$}
								child {node {}}
								child {node {}}
							}
						}
					}
						child {node {}}
					}
				}
				child {node [rectangle, draw] {$\op{r,l}$}
					child {node {}}
					child {node {}}
				}
			;
		\end{tikzpicture}
	\right)
	=
	\begin{tikzpicture}[baseline={([yshift=-.8ex]current bounding box.center)}]
		\node (1a) at (.5,2) {$1$};
		\node (1b) at (.5,4) {$1$};
		\node (2a) at (1.7,2) {\blue $2$};
		\node (2b) at (1.7,3) {\blue $2$};
		\node (3a) at (3,1) {\blue $3$};
		\node (3b) at (1.7,4) {\blue $3$};
		\node (4a) at (3,2) {\cyan $4$};
		\node (4b) at (2.5,3) {\cyan $4$};
		\node (5a) at (3.5,3) {\cyan $5$};
		\node (5b) at (3,4) {\cyan $5$};
		\node (6a) at (4,4) {\green $6$};
		\node (6b) at (4,5) {\green $6$};
		\node (7a) at (5,3) {\green $7$};
		\node (7b) at (5,4) {\green $7$};
		\node (8a) at (6,2) {$8$};
		\node (8b) at (6,4) {$8$};
		\node (9a) at (7,2) {$9$};
		\node (9b) at (7,4) {$9$};
		\node (10a) at (8,0) {$10$};
		\node (10b) at (8,2) {$10$};
		\draw (1a) -- (1b);
		\draw (1b) -- (6b);
		\draw[blue] (2a) -- (2b);
		\draw[blue] (2b) -- (3b);
		\draw (3a) -- (1a);
		\draw[blue] (3a) -- (2a);
		\draw (3a) -- (4a);
		\draw (3a) -- (8a);
		\draw (3b) -- (6b);
		\draw[cyan] (4a) -- (4b);
		\draw[cyan] (4a) -- (5a);
		\draw (4a) -- (7a);
		\draw[cyan] (4b) -- (5b);
		\draw[cyan] (5a) -- (5b);
		\draw (5b) -- (6b);
		\draw[green] (6a) -- (6b);
		\draw[green] (7a) -- (6a);
		\draw[green] (7a) -- (7b);
		\draw[green] (7b) -- (6b);
		\draw (8a) -- (8b);
		\draw (8b) -- (6b);
		\draw (9a) -- (9b);
		\draw (9b) -- (6b);
		\draw (10a) -- (3a);
		\draw (10a) -- (9a);
		\draw (10a) -- (10b);
		\draw (10b) -- (9b);
	\end{tikzpicture}
\]


\subsubsection{Parallel poset operad}
\label{subsubsec:parallelPosetOperad}

As observed above, the operations~$\Operations_2$ on bounded $2$-posets defined in \cref{def:messyCitelangisParallelActionPosets} do not satisfy all messy parallel $2$-citelangis relations.
However, they satisfy a subset of these relations which in turn defines an operad on posets.
This operad can be seen as a suboperad of an operad on all bounded $2$-posets defined by the following composition rules, illustrated in \cref{fig:parallelPosetOperad}.

\begin{definition}
\label{def:parallelBoundedPosetOperad}
Let~$\le_M$ and~$\le_N$ be two bounded $2$-posets of degrees~$m$ and~$n$ respectively, and let~$i \in [m]$.
Let~$i_1$ and~$i_2$ denote the two copies of~$i$ in~$M$ such that~$i_1 <_M i_2$.
We define the \defn{parallel composition}~${\le_P} \eqdef {\le_M} \posetParallelCirc{i} {\le_N}$ by:
\begin{itemize}
\item $P \eqdef 1^{\{2\}} \dots (m+n-1)^{\{2\}}$,
\item $\le_P$ is obtained by inserting~$\le_N$ in~$\le_M$ by placing~$\min(\le_N)$ at~$i_1$ and~$\max(\le_N)$ at~$i_2$, and performing the appropriate shift. More precisely, using the notations of \cref{def:shiftMultiset,def:shiftMultisetPlace}, the comparison $x \le_P y$ holds for~$x,y \in P$ if and only if one of the following statements holds:
	\begin{itemize}
	\item $x \in M[i,n]$, $y \in M[i,n]$ and~$\overline{x} \le_M \overline{y}$, or
	\item $x \in N[i-1]$, $y \in N[i-1]$ and~$\overline{x} \le_N \overline{y}$, or
	\item $x \in M[i,n]$, $y \in N[i-1]$ and~$\overline{x} \le_M i_1$, or
	\item $x \in N[i-1]$, $y \in M[i,n]$ and~$i_2 \le_M \overline{y}$.
	\end{itemize}
\end{itemize}
\end{definition}

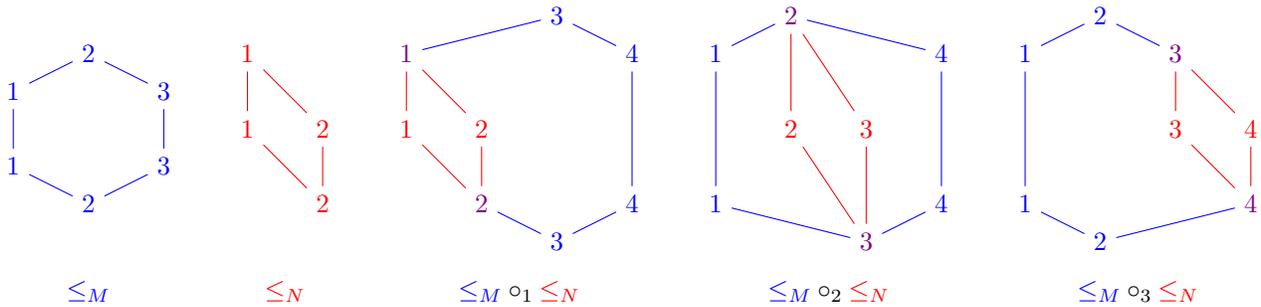
\begin{figure}[h]
	\centerline{
	\begin{tabular}{c@{\qquad}c@{\qquad}c@{\qquad}c@{\qquad}c}
	\begin{tikzpicture}[baseline={([yshift=-.8ex]current bounding box.center)}, blue]
		\node (1a) at (0,.5) {$1$};
		\node (1b) at (0,1.5) {$1$};
		\node (2a) at (1,0) {$2$};
		\node (2b) at (1,2) {$2$};
		\node (3a) at (2,.5) {$3$};
		\node (3b) at (2,1.5) {$3$};
		\draw (1a) -- (1b);
		\draw (1b) -- (2b);
		\draw (2a) -- (1a);
		\draw (2a) -- (3a);
		\draw (3a) -- (3b);
		\draw (3b) -- (2b);
	\end{tikzpicture}
	&
	\begin{tikzpicture}[baseline={([yshift=-.8ex]current bounding box.center)}, red]
		\node (1a) at (0,1) {$1$};
		\node (1b) at (0,2) {$1$};
		\node (2a) at (1,0) {$2$};
		\node (2b) at (1,1) {$2$};
		\draw (1a) -- (1b);
		\draw (2a) -- (1a);
		\draw (2a) -- (2b);
		\draw (2b) -- (1b);
	\end{tikzpicture}
	&
	\begin{tikzpicture}[baseline={([yshift=-.8ex]current bounding box.center)}]
		\node (2aM) at (2,0) {\blue $3$};
		\node (2bM) at (2,3) {\blue $3$};
		\node (3aM) at (3,.5) {\blue $4$};
		\node (3bM) at (3,2.5) {\blue $4$};
		\node (1aN) at (0,1.5) {\red $1$};
		\node (1bN) at (0,2.5) {\violet $1$};
		\node (2aN) at (1,.5) {\violet $2$};
		\node (2bN) at (1,1.5) {\red $2$};
		\draw [red] (1aN) -- (1bN);
		\draw [blue] (1bN) -- (2bM);
		\draw [blue] (2aM) -- (2aN);
		\draw [blue] (2aM) -- (3aM);
		\draw [red] (2aN) -- (1aN);
		\draw [red] (2aN) -- (2bN);
		\draw [red] (2bN) -- (1bN);
		\draw [blue] (3aM) -- (3bM);
		\draw [blue] (3bM) -- (2bM);
	\end{tikzpicture}
	&
	\begin{tikzpicture}[baseline={([yshift=-.8ex]current bounding box.center)}]
		\node (1aM) at (0,.5) {\blue $1$};
		\node (1bM) at (0,2.5) {\blue $1$};
		\node (3aM) at (3,.5) {\blue $4$};
		\node (3bM) at (3,2.5) {\blue $4$};
		\node (1aN) at (1,1.5) {\red $2$};
		\node (1bN) at (1,3) {\violet $2$};
		\node (2aN) at (2,0) {\violet $3$};
		\node (2bN) at (2,1.5) {\red $3$};
		\draw [blue] (1aM) -- (1bM);
		\draw [red] (1aN) -- (1bN);
		\draw [blue] (1bM) -- (1bN);
		\draw [blue] (2aN) -- (1aM);
		\draw [red] (2aN) -- (1aN);
		\draw [red] (2aN) -- (2bN);
		\draw [blue] (2aN) -- (3aM);
		\draw [red] (2bN) -- (1bN);
		\draw [blue] (3aM) -- (3bM);
		\draw [blue] (3bM) -- (1bN);
	\end{tikzpicture}
	&
	\begin{tikzpicture}[baseline={([yshift=-.8ex]current bounding box.center)}]
		\node (1aM) at (0,.5) {\blue $1$};
		\node (1bM) at (0,2.5) {\blue $1$};
		\node (2aM) at (1,0) {\blue $2$};
		\node (2bM) at (1,3) {\blue $2$};
		\node (1aN) at (2,1.5) {\red $3$};
		\node (1bN) at (2,2.5) {\violet $3$};
		\node (2aN) at (3,.5) {\violet $4$};
		\node (2bN) at (3,1.5) {\red $4$};
		\draw [blue] (1aM) -- (1bM);
		\draw [red] (1aN) -- (1bN);
		\draw [blue] (1bM) -- (2bM);
		\draw [blue] (1bN) -- (2bM);
		\draw [blue] (2aM) -- (1aM);
		\draw [blue] (2aM) -- (2aN);
		\draw [red] (2aN) -- (1aN);
		\draw [red] (2aN) -- (2bN);
		\draw [red] (2bN) -- (1bN);
	\end{tikzpicture}
	\\[1.7cm]
	$\blue \le_M$ & $\red \le_N$ & ${\blue \le_M} \posetParallelCirc{1} {\red \le_N}$ & ${\blue \le_M} \posetParallelCirc{2} {\red \le_N}$ & ${\blue \le_M} \posetParallelCirc{3} {\red \le_N}$
	\end{tabular}
	}
	\caption{Examples of compositions in the parallel $2$-poset operad. See \cref{def:parallelBoundedPosetOperad}.}
	\label{fig:parallelPosetOperad}
\end{figure}

The following statement is left to the reader.

\begin{lemma}
For any two bounded $2$-posets~$\le_M$ and~$\le_N$ of degrees~$m$ and~$n$ respectively, and any~$i \in [m]$, the parallel composition ${\le_M} \posetParallelCirc{i} {\le_N}$ is a bounded $2$-poset of degree~$m+n-1$.
\end{lemma}

\begin{proposition}
The parallel compositions~$\posetParallelCirc{i}$ of \cref{def:parallelBoundedPosetOperad} define an operad structure~$\boundedPosParallel[2]$ on bounded $2$-posets.
\end{proposition}

\begin{proof}
We start with the series axiom (see \cref{subsec:operads}).
Let $\le_P$, $\le_Q$ and $\le_R$ be three bounded $2$-posets of respective arity $p$, $q$ and $r$ and $i\in[p]$, $j\in[q]$. Then applying the definition, one sees that $({\le_P} \posetParallelCirc{i} {\le_Q}) \posetParallelCirc{j+q-1} {\le_R}$ and $({\le_P} \posetParallelCirc{j} {\le_R}) \posetParallelCirc{i} {\le_Q}$ are both equal to $\le_S$ defined by the following:
\begin{itemize}
\item Set $S \eqdef 1^{\{2\}} \dots (p+q+r-2)^{\{2\}}$, and denote by~$P' \eqdef P[i,q+r-1]$, by~$Q' \eqdef Q[j,r][i]$ and by~$R' \eqdef R[i+j-1]$. For $x$ in $Q'$, we still denote $\overline{x}$ the corresponding unshifted element in $Q'$ (which should be denoted $\overline{\overline{x}}$, according to \cref{def:shiftMultisetPlace}). Finally, let~$i_1, i_2$ and $j_1, j_2$ denote the respective coies of $i$ and $j$ in $P$ and $Q$.
\item $x \le_S y$ holds for~$x,y \in S$ if and only if one of the following statements holds:
    \begin{itemize}
    \item $x \in T'$, $y \in T'$ and~$\overline{x} \le_{T} \overline{y}$ for $T \in \{P, Q, R\}$, or
    \item $x \in P'$, $y \in Q'\cup R'$ and~$\overline{x} <_{P} i_1$, or
    \item $x \in Q'\cup R'$, $y \in P'$ and~$i_2 <_{P} \overline{y}$, or
    \item $x \in Q'$, $y \in R'$ and~$\overline{x} <_{Q} j_1$, or
    \item $x \in R'$, $y \in Q'$ and~$j_2 <_{Q} \overline{y}$.
    \end{itemize}
\end{itemize}
This is illustrated in the following picture:

\bigskip
\centerline{$
	\left(
	\begin{tikzpicture}[baseline={([yshift=-.8ex]current bounding box.center)}, xscale=.3, yscale=.5]
		\draw[fill=green!30] (3,0) -- (0,3) -- (3,6) -- (6,3) -- (3,0);
		\node at (3,1.2) {\blue $i_1$};
		\node at (3,4.8) {\blue $i_2$};
	\end{tikzpicture}
	\posetParallelCirc{\blue i}
	\begin{tikzpicture}[baseline={([yshift=-.8ex]current bounding box.center)}, xscale=.3, yscale=.5]
		\draw[fill=blue!30] (3,1) -- (1,3) -- (3,5) -- (5,3) -- (3,1);
		\node at (3,2.2) {\red $j_1$};
		\node at (3,3.8) {\red $j_2$};
	\end{tikzpicture}
	\right)
	\posetParallelCirc{{\blue i}+{\red j}-1}
	\begin{tikzpicture}[baseline={([yshift=-.8ex]current bounding box.center)}, xscale=.3, yscale=.5]
		\draw[fill=red!30] (3,2) -- (2,3) -- (3,4) -- (4,3) -- (3,2);
	\end{tikzpicture}
	=
	\begin{tikzpicture}[baseline={([yshift=-.8ex]current bounding box.center)}, xscale=.3, yscale=.5]
		\draw[fill=green!30] (3,0) -- (0,3) -- (3,6) -- (6,3) -- (3,0);
		\draw[fill=blue!30] (3,1) -- (1,3) -- (3,5) -- (5,3) -- (3,1);
		\draw[fill=red!30] (3,2) -- (2,3) -- (3,4) -- (4,3) -- (3,2);
	\end{tikzpicture}
	=	
	\begin{tikzpicture}[baseline={([yshift=-.8ex]current bounding box.center)}, xscale=.3, yscale=.5]
		\draw[fill=green!30] (3,0) -- (0,3) -- (3,6) -- (6,3) -- (3,0);
		\node at (3,1.2) {\blue $i_1$};
		\node at (3,4.8) {\blue $i_2$};
	\end{tikzpicture}
	\posetParallelCirc{\blue i}
	\left(
	\begin{tikzpicture}[baseline={([yshift=-.8ex]current bounding box.center)}, xscale=.3, yscale=.5]
		\draw[fill=blue!30] (3,1) -- (1,3) -- (3,5) -- (5,3) -- (3,1);
		\node at (3,2.2) {\red $j_1$};
		\node at (3,3.8) {\red $j_2$};
	\end{tikzpicture}
	\posetParallelCirc{\red j}
	\begin{tikzpicture}[baseline={([yshift=-.8ex]current bounding box.center)}, xscale=.3, yscale=.5]
		\draw[fill=red!30] (3,2) -- (2,3) -- (3,4) -- (4,3) -- (3,2);
	\end{tikzpicture}
	\right)
$}

\bigskip
\noindent
A similar computation can be done for parallel axiom (see \cref{subsec:operads}).
We leave the details to the reader, but provide the corresponding illustration:

\bigskip
\centerline{$
	\left(
	\begin{tikzpicture}[baseline={([yshift=-.8ex]current bounding box.center)}, xscale=.3, yscale=.5]
		\draw[fill=green!30] (4.5,0) -- (0,2) -- (4.5,4) -- (9,2) -- (4.5,0);
		\node at (3,1.2) {\blue $i_1$};
		\node at (3,2.8) {\blue $i_2$};
		\node at (6,1.2) {\red $j_1$};
		\node at (6,2.8) {\red $j_2$};
	\end{tikzpicture}
	\posetParallelCirc{\blue i}
	\begin{tikzpicture}[baseline={([yshift=-.8ex]current bounding box.center)}, xscale=.3, yscale=.5]
		\draw[fill=blue!30] (3,1) -- (2,2) -- (3,3) -- (4,2) -- (3,1);
	\end{tikzpicture}
	\right)
	\posetParallelCirc{{\red j}+{\blue q}-1}
	\begin{tikzpicture}[baseline={([yshift=-.8ex]current bounding box.center)}, xscale=.3, yscale=.5]
		\draw[fill=red!30] (6,1) -- (5,2) -- (6,3) -- (7,2) -- (6,1);
	\end{tikzpicture}
	=
	\begin{tikzpicture}[baseline={([yshift=-.8ex]current bounding box.center)}, xscale=.3, yscale=.5]
		\draw[fill=green!30] (4.5,0) -- (0,2) -- (4.5,4) -- (9,2) -- (4.5,0);
		\draw[fill=blue!30] (3,1) -- (2,2) -- (3,3) -- (4,2) -- (3,1);
		\draw[fill=red!30] (6,1) -- (5,2) -- (6,3) -- (7,2) -- (6,1);
	\end{tikzpicture}
	=	
	\left(
	\begin{tikzpicture}[baseline={([yshift=-.8ex]current bounding box.center)}, xscale=.3, yscale=.5]
		\draw[fill=green!30] (4.5,0) -- (0,2) -- (4.5,4) -- (9,2) -- (4.5,0);
		\node at (3,1.2) {\blue $i_1$};
		\node at (3,2.8) {\blue $i_2$};
		\node at (6,1.2) {\red $j_1$};
		\node at (6,2.8) {\red $j_2$};
	\end{tikzpicture}
	\posetParallelCirc{\red j}
	\begin{tikzpicture}[baseline={([yshift=-.8ex]current bounding box.center)}, xscale=.3, yscale=.5]
		\draw[fill=red!30] (6,1) -- (5,2) -- (6,3) -- (7,2) -- (6,1);
	\end{tikzpicture}
	\right)
	\posetParallelCirc{\blue i}
	\begin{tikzpicture}[baseline={([yshift=-.8ex]current bounding box.center)}, xscale=.3, yscale=.5]
		\draw[fill=blue!30] (3,1) -- (2,2) -- (3,3) -- (4,2) -- (3,1);
	\end{tikzpicture}
$}
\end{proof}

The following statements connect this operad~$\boundedPosParallel[2]$ with the parallel poset evaluation of \cref{def:posetParallel}.

\begin{lemma}
\label{lem:parallelPosetOperad}
For any operation~${\operation \in \Operations_2}$ and any bounded $2$-posets~$\le_M$ and~$\le_N$, we have
\[
{\le_M} \; \operation \; {\le_N} = \EvalPosetParallel(\operation) \, \posetParallelCirc{} \, ({\le_M}, {\le_N}),
\]
where the posets~$\EvalPosetParallel(\operation) \eqdef {\le_I} \; \operation \; {\le_I}$ are illustrated in \cref{fig:exmMessyCitelangisParallelActionPosets}.
\end{lemma}

\begin{proof}
By \cref{def:parallelBoundedPosetOperad}, the parallel composition~$\EvalPosetParallel(\operation) \posetParallelCirc{} ({\le_M}, {\le_N})$ places $\min(\le_M)$ at~$1_1$, $\max(\le_M)$ at~$1_2$, $\min(\le_N)$ at~$2_1$, and~$\max(\le_N)$ at~$2_2$ in the $2$-poset~$\EvalPosetParallel(\operation)$.
This coincides with the result of the operation~${\le_M} \; \operation \; {\le_N}$ described in \cref{def:messyCitelangisParallelActionPosets,rem:sumsPosetsParallel}.
\end{proof}

\cref{lem:parallelPosetOperad} enables us to interpret the operations on bounded $2$-posets of the previous section as a suboperad of~$\boundedPosParallel[2]$.
We say that two syntax trees~$\tree, \tree'$ on~$\Operations_2$ with the same arity are \defn{parallel poset equivalent} and we write~$\tree \simeq^{\parallel} \tree'$ if they have the same parallel poset evaluation.

\begin{theorem}
\label{thm:parallelPosetOperad}
The parallel poset equivalence is compatible with the grafting of syntax trees: for any syntax trees~$\tree \simeq^\parallel \tree'$ of arity~$p$ and~$\tree[s] \simeq^\parallel \tree[s]'$ of arity~$q$ and~$i \in [p]$, we have~$\tree \circ_i \tree[s] \simeq^\parallel \tree' \circ_i \tree[s]'$.
Therefore, there exists a \defn{parallel $2$-poset operad}
\[
\posParallel[2] \eqdef \EvalPosetParallel \big( \Syntax[\Operations_2] \big).
\]
\end{theorem}

\begin{proof}
By induction on \cref{lem:parallelPosetOperad}, the map~$\EvalPosetParallel$ coincides with the unique operad morphism from the free operad~$\Free[\Operations_2]$ to~$\boundedPosParallel[2]$ that sends a generator~$\operation \in \Operations_2$ to~$\EvalPosetParallel(\operation)$.
Therefore, $\posParallel[2]$ is the suboperad of~$\boundedPosParallel[2]$ generated by~$\bigset{\EvalPosetParallel(\operation)}{\operation \in \Operations_2}$.
\end{proof}

By construction, the parallel $2$-poset operad~$\posParallel[2]$ satisfies the relation of \cref{eq:relationPosetsParallel}.
\cref{prop:uniqueParallelPosetEvaluation} implies that it is the only relation in~$\posParallel[2]$.

\begin{theorem}
\label{thm:presentationParallelPosetOperad}
The parallel $2$-poset operad~$\posParallel[2]$ is generated by~$\Operations_2$ with the unique relation
\[
\compoL{r,r}{l,l} = \compoR{r,r}{l,l}.
\]
\end{theorem}

In fact, the parallel $2$-poset operad~$\posParallel[2](\tree)$ is actually isomorphic to the series $2$-poset operad~$\posSeries[2](\tree)$ that will be defined for arbitrary~$k \ge 1$ in \cref{subsubsec:seriesPosetOperad}.
We refer in particular to \cref{subsubsec:seriesPosetOperad} for enumerative properties of~$\posParallel[2](\tree)$.


\subsubsection{Parallel \texorpdfstring{$2$}{2}-Zinbiel operads}
\label{subsubsec:parallelZinbiel}

Recall that the dendriform operad is the suboperad of the non-symmetric Zinbiel operad generated by ${\op{l}} \eqdef 12$ and ${\op{r}} \eqdef 21$, see \cref{subsubsec:LeibnizZinbieOperads}.
It is therefore natural to look for an operad $\messyZinbielParallel[2]$ on $2$-permutations which contains the operad~$\messyCitelangisParallel[2]$ as a suboperad generated in degree~$2$ and which moreover closes the following left square of operad morphisms.
We will also find a tidy version $\tidyZinbielParallel[\op{l,r}]$, closing the right square where the dashed arrows are not operad morphisms, but bijections of normal forms.

\begin{center}
\begin{tikzcd}[column sep=2cm, row sep=.3cm]
	\phantom{|}\boundedPosParallel[2]\phantom{|} \arrow[r, "\LinExt",  two heads]
	& \phantom{|}\messyZinbielParallel[2]\phantom{|} \arrow[r, dashrightarrow, "\LexMin", hook, two heads]
	& \phantom{|}\tidyZinbielParallel[\op{l,r}]\phantom{|} \\
	\phantom{|}\posParallel[2]\phantom{|} \arrow[r, "\LinExt", two heads] \arrow[u, hook]
	& \phantom{|}\messyCitelangisParallel[2]\phantom{|} \arrow[r, dashrightarrow, "\LexMin", hook, two heads] \arrow[u, hook]
	& \phantom{|}\tidyCitelangisParallel[\op{l,r}]\phantom{|} \arrow[u, hook]
\end{tikzcd}
\end{center}

\paraul{Messy parallel $2$-Zinbiel operad}
We start with the messy version.

\begin{definition}
\label{def:compositionMessyParallelZinbiel}
Let~$\sigma \in \Perm_2(m)$ and~$\tau \in \Perm_2(n)$ be two $2$-permutations and~$i \in [m]$.
Write~$\sigma = \lambda \, i \, \mu \, i \, \omega$ and $\tau = f \, \theta \, l$ where $f$ and $l$ are the first and last letter of~$\tau$ respectively.
Then the \defn{messy parallel $i$-th composition} of $\sigma$ and $\mu$ is defined by
\[
\sigma \messyParallelCirc{i} \tau = \lambda[i,n] \, f[i-1] \, (\mu[i,n] \shuffle \theta[i-1]) \, l[i-1] \, \omega[i,n].
\]
We extends this definition by linearity.
\end{definition}

Here are some examples of messy parallel compositions on $2$-permutations:
\begin{align*}
  {\blue3}1{\blue42234}1 \messyParallelCirc{1} 3{\red1312}2 = \; &
  {\blue5}3{\red1312}{\blue64456}2 +
  {\blue5}3{\red131}{\blue6}{\red2}{\blue4456}2 +
  {\blue5}3{\red131}{\blue64}{\red2}{\blue456}2 + \dots +
  {\blue5}3{\blue6}{\red1}{\blue4}{\red3}{\blue4}{\red1}{\blue56}{\red2}2
  \\ & + \cdots (\text{$126$~terms}) \cdots +
  {\blue5}3{\blue6445}{\red1}{\blue6}{\red312}2 + 
  {\blue5}3{\blue64456}{\red1312}2, \\
  {\blue314}22{\blue341} \messyParallelCirc{2} 3{\red1312}2 = \; & {\blue516}4{\red2423}3{\blue561}, \\
  3{\blue1422}3{\blue41} \messyParallelCirc{3} 3{\red1312}2 = \; &
  5{\red3534}{\blue1622}4{\blue61} +
  5{\red353}{\blue1}{\red4}{\blue622}4{\blue61} +
  5{\red353}{\blue16}{\red4}{\blue22}4{\blue61} + \dots +
  5{\blue1}{\red35}{\blue62}{\red34}{\blue2}4{\blue61}
  \\ & + \cdots (\text{$70$~terms}) \cdots +
  5{\blue162}{\red3}{\blue2}{\red534}4{\blue61} +
  5{\blue1622}{\red3534}4{\blue61}, \\
  {\blue31}4{\blue223}4{\blue1} \messyParallelCirc{4} 3{\red1312}2 = \; &
  {\blue31}6{\red4645}{\blue223}5{\blue1} +
  {\blue31}6{\red464}{\blue2}{\red5}{\blue23}5{\blue1} +
  {\blue31}6{\red464}{\blue22}{\red5}{\blue3}5{\blue1} + \dots
  + {\blue31}6{\blue2}{\red464}{\blue2}{\red5}{\blue3}5{\blue1}
  \\ & + \cdots (\text{$35$~terms}) \cdots +
  {\blue31}6{\blue22}{\red4}{\blue3}{\red645}5{\blue1} + 
  {\blue31}6{\blue223}{\red4645}5{\blue1}.
\end{align*}

The following two lemmas relate the messy parallel composition~$\messyParallelCirc{i}$ of $2$-permutations of \cref{def:compositionMessyParallelZinbiel} and the parallel composition~$\posetParallelCirc{i}$ of bounded $2$-posets of \cref{def:parallelBoundedPosetOperad}.
They play the same role for the composition as \cref{lem:LinExtEvalPosetParallel} played for the operations.

\begin{lemma}
\label{lem:morphismParallelPermPos}
For any two $2$-permutations~$\sigma \in \Perm_2(m)$ and~$\tau \in \Perm_2(n)$ and~$i \in [m]$,
\[
\sigma \messyParallelCirc{i} \tau = \LinExt \big( {\le_\sigma} \posetParallelCirc{i} {\le_\tau} \big).
\]
\end{lemma}

\begin{proof}
As in \cref{def:compositionMessyParallelZinbiel}, write~$\sigma = \lambda \, i \, \mu \, i \, \omega$ and $\tau = f \, \theta \, l$ where $f$ and $l$ are the first and last letter of~$\tau$ respectively.
From the description of \cref{def:parallelBoundedPosetOperad}, the poset~${\le_\sigma} \posetParallelCirc{i} {\le_\tau}$ decomposes into the ordered sum
\[
{\le_{\sigma[i,n]}} + \big\{ f[i-1] \big\} + \big( {\le_{\sigma[i,n]}} \, \sqcup \, {\le_{\tau[i-1]}} \big) + \big\{ l[i-1] \big\} + {\le_{\sigma[i,n]}}.
\]
Hence, the statement follows from \cref{lem:disjointUnionOrderedSumLinearExtensions}.
\end{proof}

\begin{lemma}
\label{lem:morphismParallelPermPos2}
For any two bounded $2$-posets~${\le_M} \in \boundedPosParallel[2](m)$ and~${\le_N} \in \boundedPosParallel[2](n)$ and~$i \in [m]$,
\[
\LinExt(\le_M) \messyParallelCirc{i} \LinExt(\le_N) = \LinExt({\le_M} \posetParallelCirc{i} {\le_N}).
\]
\end{lemma}

\begin{proof}
Fix three integers $m$, $n$, and $i\in[m]$.
Remark that if $\nu$ is a permutation appearing in the messy parallel composition $\sigma \messyParallelCirc{i} \tau$ for some $\sigma \in \Perm_2(m)$ and~$\tau \in \Perm_2(n)$, one can recover uniquely these two permutations from $\nu$.
As a consequence, $\LinExt(\le_M) \messyParallelCirc{i} \LinExt(\le_N)$ cannot have multiplicities, so that we can argue by double inclusion.

Observe from \cref{def:parallelBoundedPosetOperad} that the inclusion of relations~${\le_M} \subseteq {\le'_M}$ and~${\le_N} \subseteq {\le'_N}$ imply the inclusion~${{\le_M} \posetParallelCirc{i} {\le_N}} \subseteq {{\le'_M} \posetParallelCirc{i} {\le'_N}}$.
For~$\sigma \in \LinExt(\le_M)$ and~$\tau \in \LinExt(\le_N)$, we have~${\le_M} \subseteq {\le_\sigma}$ and~${\le_N} \subseteq {\le_\tau}$, hence~${{\le_M} \posetParallelCirc{i} {\le_N}} \subseteq {{\le_\sigma} \posetParallelCirc{i} {\le_\sigma}}$.
We thus get~$\LinExt({\le_\sigma} \posetParallelCirc{i} {\le_\tau}) \subseteq \LinExt({\le_M} \posetParallelCirc{i} {\le_N})$.
As~$\LinExt({\le_\sigma} \posetParallelCirc{i} {\le_\tau}) = \sigma \messyParallelCirc{i} \tau$ by \cref{lem:morphismParallelPermPos}, we obtain that
\[
\LinExt(\le_M) \messyParallelCirc{i} \LinExt(\le_N)  \; = \bigcup_{\substack{\sigma \in \LinExt(\le_M) \\ \tau \in \LinExt(\le_N)}} \sigma \messyParallelCirc{i} \tau \; \subseteq \; \LinExt({\le_M} \posetParallelCirc{i} {\le_N}).
\]

Conversely, starting from any linear extension~$\rho$ of~${\le_M} \posetParallelCirc{i} {\le_N}$, restricting~$\rho$  to~$M[i,n]$ and~$N[i-1]$ and renumbering respectively yields linear extensions~$\sigma$ of~$\le_M$ and~$\tau$ of~$\le_N$.
A straightforward case analysis of on \cref{def:parallelBoundedPosetOperad} shows that~$\rho$ appears in~$\sigma \messyParallelCirc{i} \tau$.
\end{proof}

As a consequence, \cref{def:compositionMessyParallelZinbiel} actually defines an operad composition.

\begin{proposition}
\label{prop:morphismParallelLinExt}
The family $\big( \K\Perm_2(n) \big)_{n>0}$ endowed with the messy parallel composition rules~$\messyParallelCirc{i}$ of \cref{def:compositionMessyParallelZinbiel} defines an operad $\messyZinbielParallel[2]$, called the \defn{messy parallel $2$-Zinbiel operad}. Moreover, $\LinExt$ is a surjective operad morphism from $\boundedPosParallel[2]$ to $\messyZinbielParallel[2]$.
\end{proposition}

\begin{proof}
Thanks to \cref{lem:morphismParallelPermPos,lem:morphismParallelPermPos2}, we have
\[
\sigma \messyParallelCirc{i} (\tau \messyParallelCirc{j} \mu) = \LinExt \big( {\le_\sigma} \posetParallelCirc{i} ({\le_\tau} \posetParallelCirc{j} {\le_\mu}) \big).
\]
Using similar equalities for the other compositions, we prove the operad axioms.
The morphism property is just \cref{lem:morphismParallelPermPos}.
The surjectivity follows from $\LinExt({\le_\sigma}) = \sigma$.
\end{proof}

\begin{proposition}
\label{prop:fullyBoundedCuttableBasisMessy}
The operad $\messyCitelangisParallel[2]$ is the suboperad of $\messyZinbielParallel[2]$ generated by the four elements $1221$, $1122+1212$, $2121+2211$ and $2112$.
\end{proposition}

\begin{proof}
The generators are given by the linear extensions of the bounded $2$-posets of \cref{fig:exmMessyCitelangisParallelActionPosets}:
\[
\begin{array}{l@{\;=\;}l@{\quad}c@{\quad}l@{\;=\;}l}
\LinExt({\le_I} \op{l,l} {\le_I}) & 1221,
&& \LinExt({\le_I} \op{l,r} {\le_I}) & 1122+1212, \\
\LinExt({\le_I} \op{r,l} {\le_I}) & 2121+2211
& \text{and}
& \LinExt({\le_I} \op{r,r} {\le_I}) & 2112.
\end{array}
\qedhere
\]
\end{proof}

\paraul{Tidy parallel $2$-Zinbiel operad}
The goal of this section is to generalize to any $2$-permutations the composition formula of~\cref{prop:combinatorialModelCompositionsTidyParallelCitelangis} for fully bounded cuttable $2$-permutations, as suggested by \cref{rem:tidyParallelZinbiel}.

\begin{definition}
\label{def:compositionTidyParallelZinbiel}
Let~$\sigma \in \Perm_2(m)$ and~$\tau \in \Perm_2(n)$ be two $2$-permutations and~$i \in [m]$.
Write~$\sigma = \lambda \, i \, \mu \, \nu \, i \, \omega$, where $\mu$ is the maximal factor such that~$\mu_p < i$ for all~$p \in [|\mu|]$.
Therefore either $\nu$ is empty or $i < \nu_1$. Write moreover~$\tau = f \, \theta \, l$ where $f$ and $l$ are the first and last letter of~$\tau$ respectively.
Then the \defn{tidy parallel $i$-th composition} of $\sigma$ and $\mu$ is defined by
\[
\begin{array}{c@{\;}c@{\;}l@{\,}l@{\,}l@{\,}l@{\,}l@{\,}l@{\,}l}
\sigma \tidyParallelCirc{i} \tau
& \eqdef & \lambda[i,n] & f[i-1] & \mu[i,n] & \theta[i-1] & \nu[i,n] & l[i-1] & \omega[i,n] \\
& =      & \lambda[i,n] & f[i-1] & \mu      & \theta[i-1] & \nu[i,n] & l[i-1] & \omega[i,n].
\end{array}
\]
\end{definition}

Note that this is the same composition formula as in~\cref{prop:combinatorialModelCompositionsTidyParallelCitelangis}.
Here are some examples of tidy parallel compositions on $2$-permutations:
\begin{gather*}{1}
  {\blue3}1|{\blue42234}1 \tidyParallelCirc{1} 3{\red1312}2 = {\blue5}3{\red1312}{\blue64456}2, \\
  {\blue314}2|2{\blue341} \tidyParallelCirc{2} 3{\red1312}2 = {\blue516}4{\red2423}3{\blue561}, \\
  3{\blue1}|{\blue422}3{\blue41} \tidyParallelCirc{3} 3{\red1312}2 = 5{\blue1}{\red3534}{\blue622}4{\blue61}, \\
  {\blue31}4{\blue223}|4{\blue1} \tidyParallelCirc{4} 3{\red1312}2 = {\blue31}6{\blue223}{\red4645}5{\blue1},
\end{gather*}
where we have marked by a vertical bar the separation $\mu|\nu$.

Similarly to \cref{lem:LexMinParallel}, there is a relation between the tidy parallel composition~$\tidyParallelCirc{i}$ of \cref{def:compositionTidyParallelZinbiel} and the messy parallel composition~$\tidyParallelCirc{i}$ of \cref{def:compositionMessyParallelZinbiel}.

\begin{lemma}
\label{lem:morphismParallelLexMin}
For any two $2$-permutations~$\sigma \in \Perm_2(m)$ and~$\tau \in \Perm_2(n)$ and~$i \in [m]$,
\[
\sigma \tidyParallelCirc{i} \tau = \LexMin(\sigma \messyParallelCirc{i} \tau).
\]
\end{lemma}

As a consequence, the tidy parallel composition rules of \cref{def:compositionTidyParallelZinbiel} define an operad.

\begin{proposition}
\label{prop:fullyBoundedCuttableBasisTidy}
The family $\big( \Perm_2(n) \big)_{n>0}$ endowed with the tidy parallel composition rules~$\tidyParallelCirc{i}$ of~\cref{def:compositionTidyParallelZinbiel} defines a non-symmetric set operad~$\tidyZinbielParallel[\op{l,r}]$, called the \defn{tidy parallel $2$-Zinbiel operad}. Moreover~$\tidyCitelangisParallel[\op{l,r}]$ is the suboperad of~$\tidyZinbielParallel[\op{l,r}]$ given by fully bounded cuttable $2$-permutations.
\end{proposition}


\subsection{Actions of series citelangis operads}
\label{subsec:actionSeriesCitalangisOperads}

We now discuss the action of series $k$-citelangis operads~$\tidyCitelangisSeries$ and~$\messyCitelangisSeries$ on certain $k$-permutations and $k$-posets.
Our presentation closely follows the prototype given by the actions of the parallel $k$-citelangis operads presented in \cref{subsec:actionParallelCitalangisOperads}.
More precisely, the structure, phrasing and notations here are a carbon copy of that of \cref{subsec:actionParallelCitalangisOperads}.
However, all definitions differ, in particular that of the compositions of multiposets and multipermutations.
In particular, substituing $k = 2$ in this section does not recover \cref{subsec:actionParallelCitalangisOperads}.
We start with some combinatorial considerations on certain $k$-permutations.


\subsubsection{\texorpdfstring{$k$}{k}-rooted cuts in \texorpdfstring{$k$}{k}-permutations}
\label{subsubsec:krootedkpermutations}

We introduce first an intruiguing class of $k$-permutations.
These permutations will be instrumental in studying the action of the tidy series $k$-citelangis operad, but we also believe that they would deserve further study for their own sake.

\begin{definition}
\label{def:krootedCutMultipermutation}
Let~$\sigma \in \Perm_k(n)$ be a $k$-permutation of degree~$n$ and~$\gamma \in [n-1]$.
We say that~$\gamma$ is a \defn{$k$-rooted cut} of~$\sigma$ if we can write~$\sigma = \mu \nu \omega$ where~$\mu$, $\nu$ and~$\omega$ are words such that~$|\mu| = k$ and~$\nu_i \le \gamma < \omega_j$ for all~$i \in [|\nu|]$ and~$j \in [|\omega|]$.
We denote by~$\kcuts(\sigma)$ the set of $k$-rooted cuts of~$\sigma$.
We say that the $k$-permutation~$\sigma$ is \defn{$k$-rooted cuttable} if it admits a $k$-rooted cut, and \defn{$k$-rooted uncuttable} if it admits no $k$-rooted cut.
\end{definition}

For example, $3$ is a $2$-rooted cut of the $2$-permutation~$31213244$, while the $2$-permutation~$31421324$ is $2$-rooted uncuttable.
Note that there is no condition on the sizes of~$\nu$ and~$\omega$ in \cref{def:krootedCutMultipermutation} (they can be empty, or of sizes which are not multiples of~$k$).
In particular, it is convenient to observe that we could change the condition~$|\mu| = k$ by~$|\mu| \le k$.
Indeed, if~$|\mu| < k$, one can always recover the situation where~$|\mu| = k$ by transferring the $k-|\mu|$ letters at the beginning of~$\nu$ to the end of~$\mu$.
We now observe that $k$-rooted cuts behave properly with restrictions in the sense of \cref{def:restrictionMultipermutation}.

\begin{lemma}
\label{lem:krootedCutsRestriction}
Consider~$L \subseteq [n]$ and~$\gamma \in [\min(L), \max(L)-1]$, and let~$\gamma^{|L} \eqdef |[\gamma] \cap L|$ denote the number of elements of~$L$ between~$1$ and~$\gamma$.
If~$\gamma$ is a $k$-rooted cut of a $k$-permutation~$\sigma \in \Perm_k(n)$, then $\gamma^{|L}$ is a $k$-rooted cut of its restriction~$\sigma^{|L}$.
\end{lemma}

\begin{proof}
Since~$\gamma$ is a $k$-rooted cut of~$\sigma$, we can write~$\sigma = \mu \nu \omega$ with~$|\mu| = k$ and~$\nu_i \le \gamma < \omega_j$ for all~$i \in [|\nu|]$ and~$j \in [|\omega|]$.
Then~$\sigma^{|L} = \mu^{|L} \nu^{|L} \omega^{|L}$.
Moreover~$|\mu^{|L}| \le k$ and~$\nu^{|L}_i \le \gamma^{|L} < \omega^{|L}_j$ for any~$i \in [|\nu^{|L}|]$ and~$j \in [|\omega^{|L}|]$.
Therefore, $\gamma^{|L}$ is a $k$-rooted cut of~$\sigma^{|L}$.
\end{proof}

Note that the reverse statement is wrong.
Consider for instance the $2$-permutation~$\sigma = 213123$ and the subset~$L = \{1,2\}$.
Then~$1$ is not a $2$-rooted cut of~$\sigma$, but $1^{|L} = 1$ is a $2$-rooted cut of~$\sigma^{|L} = 2112$.

\begin{proposition}
\label{prop:equivalenceFullykrootedCuttable}
The following conditions are equivalent for a $k$-permutation of degree~$n$:
\begin{enumerate}[(i)]
\item its restriction to any interval of~$[n]$ of size at least $2$ is $k$-rooted cuttable,
\item its restriction to any subset of~$[n]$ of size at least $2$ is $k$-rooted cuttable.
\end{enumerate}
\end{proposition}

\begin{proof}
Assume that~$\sigma \in \Perm_k(n)$ satisfies~(i).
Let~$L \subseteq [n]$ with~$|L| \ge 2$.
Since~$|L| \ge 2$, we have~$\min(L) < \max(L)$ so that the restriction~$\sigma^{|[\min(L), \max(L)]}$ admits a $k$-rooted cut~$\gamma$ with~$\min(L) \le \gamma < \max(L)$.
By \cref{rem:relationOperatorsMultipermutations}, the restriction~$\sigma^{|L}$ is just the restriction of $\sigma^{|[\min(L), \max(L)]}$ to~$\bar L \eqdef \set{\ell-\min(L)+1}{\ell \in L}$.
By \cref{lem:krootedCutsRestriction}, $\sigma^{|L}$ admits a $k$-rooted cut~$\gamma^{|\bar L}$ with~${1 \le \gamma^{|\bar L} \le |L|}$.
Therefore, $\sigma$ satisfies~(ii).
The reverse implication is obvious.
\end{proof}

\begin{definition}
\label{def:fullykrootedCuttable}
A $k$-permutation is \defn{fully $k$-rooted cuttable} if it satisfies the equivalent conditions of \cref{prop:equivalenceFullykrootedCuttable}.
\end{definition}

For example, the $2$-permutation~$31213244$ is $2$-rooted cuttable but not fully $2$-rooted cuttable (since the restriction to the interval~$[1,3]$ is not $2$-rooted cuttable).
In contrast, the $2$-permutation~$363121244556$ is fully $2$-rooted cuttable.
Note that~$1^k$ is fully $k$-rooted cuttable as there is no subset of size at least~$2$.

\medskip
When~$k = 1$, we drop $1$- in $1$-rooted, and just say fully rooted cuttable permutations.
They coincide with the usual Catalan permutations avoiding the pattern~$231$.
Recall that a $231$ pattern in a permutation is a subword~$b \cdot c \cdot a$ of three (non-necessarily consecutive) letters such that~${a < b < c}$.
Let us start with the following simple observation.

\begin{lemma}
\label{lem:rootedUncuttableImpliesPattern}
A rooted uncuttable permutation of degree at least~$2$ contains a pattern~$231$.
\end{lemma}

\begin{proof}
Let~$\sigma$ be a rooted uncuttable permutation and let~$b$ be its first letter.
Since~$b$ is not a rooted cut of~$\sigma$, there is~$a < b < c$ such that~$c$ appears before~$a$ in~$\sigma$.
Therefore, $b \cdot c \cdot a$ is a $231$ pattern in~$\sigma$.
\end{proof}

\begin{proposition}
A permutation is fully rooted cuttable if and only if it avoids the pattern~$231$.
\end{proposition}

\begin{proof}
There is nothing to prove for the permutation~$1$.
If a permutation~$\sigma$ contains a $231$-pattern~$b \cdot c \cdot a$, then its restriction~$\sigma^{|[a,c]}$ is uncuttable, so that $\sigma$ is not fully rooted cuttable.
Conversely, if~$\sigma$ avoids the pattern~$231$, then all its restrictions do, so that they are all rooted cuttable by \cref{lem:rootedUncuttableImpliesPattern}.
\end{proof}

Similar statements hold when~$k = 2$.

\begin{lemma}
\label{lem:2rootedUncuttableImpliesPattern}
A $2$-rooted uncuttable $2$-permutation of degree at least~$2$ contains a pattern~$b \cdot b' \cdot c \cdot a$ with~$a \le b, b' \le c$.
\end{lemma}

\begin{proof}
Suppose by contradiction that a $2$-permutation~$\sigma = \sigma_1 \dots \sigma_{2n}$ contains no such pattern.

Assume first that~$\sigma_1 = \sigma_2 \defeq v$.
Then for any~$u \le v < w$, both values~$u$ must appear before both values~$w$, otherwise we would have the pattern~$v \cdot v \cdot w \cdot u$.
Therefore, both~$v-1$ and~$v$ (resp.~$1$, resp.~$n-1$) are $2$-rooted cuts of~$\sigma$ if~$1 < v < n$ (resp.~if~$v = 1$, resp.~if~$v = n$).

Assume now that~$\sigma_1 \ne \sigma_2$, and let~$r \eqdef \min(\sigma_1, \sigma_2)$ and~$s \eqdef \max(\sigma_1, \sigma_2)$.
Let~$p$ denote the position of the second~$s$ of~$\sigma$, \ie $p > 2$ and~$\sigma_p = s$.
We distinguish two cases:
\begin{enumerate}[(i)]
\item Assume first that no value of~$\sigma$ appears both before and after the position~$p$. This implies that for~$w > s$, both values~$w$ appear after the position~$p$, as otherwise we would have a forbidden pattern~$s \cdot w \cdot w \cdot s$. Let~$v$ be the minimal value that appears after the position~$p$. Note that~$v > 1$ as otherwise~$\{r,s\} \cdot s \cdot 1$ is a forbidden pattern. We claim that~$v-1$ is a $2$-rooted cut of~$\sigma$. Indeed, for any~$u < v$, both values~$u$ appear before the position~$p$ by definition. Moreover, for any~$w \ge v$ distinct from~$s$, both values~$w$ appear after the position~$p$, as otherwise we would have~$v \le w < s$ and thus the forbidden pattern~$s \cdot w \cdot s \cdot v$.
\item Assume now that there is a value~$t$ that appears both before and after the position~$p$. Note that it imposes that~$s < t$ as otherwise we would have the forbidden pattern~$s \cdot t \cdot s \cdot t$. Let~$q$ denote the position of the first value~$t$. We can assume without loss of generality that $q$ is the minimal position of a value that appears both before and after~$p$. Let~$v$ be the minimal value that appears after the position~$q$. Note that~$v > 1$ as otherwise~$\{r,s\} \cdot t \cdot 1$ is a forbidden pattern. We claim that~$v-1$ is a $2$-rooted cut of~$\sigma$. Indeed, for any~$u < v$, both values~$u$ appear before the position~$q$ by definition. Moreover, for any~$w \ge v$ distinct from~$s$, both values~$w$ appear after the position~$q$. Otherwise, by minimality of the position~$q$, the second value~$w$ could not be on the right of~$p$. Therefore, either $w > s$ and we have the forbidden pattern~$s \cdot w \cdot w \cdot s$, or~$v \le w < s < t$ and we have the forbidden pattern~$s \cdot w \cdot t \cdot v$.
\qedhere
\end{enumerate}
\end{proof}

\pagebreak
\begin{proposition}
A $2$-permutation is fully $2$-rooted cuttable if and only if it avoids the pattern ${b \cdot b' \cdot c \cdot a}$ with~$a \le b , b' \le c$.
\end{proposition}

\begin{proof}
There is nothing to prove for the $2$-permutation~$11$.
If a $2$-permutation~$\sigma$ contains a pattern~${b \cdot b' \cdot c \cdot a}$ with~$a \le b , b' \le c$, then its restriction~$\sigma^{|[a,c]}$ is $2$-rooted uncuttable, so that $\sigma$ is not fully $2$-rooted cuttable.
Conversely, if~$\sigma$ avoids the pattern~${b \cdot b' \cdot c \cdot a}$ with~$a \le b , b' \le c$, then all its restrictions do, so that they are all $2$-rooted cuttable by \cref{lem:2rootedUncuttableImpliesPattern}.
\end{proof}

In contrast, for $k > 2$, there is no clear characterization of fully $k$-rooted cuttable $k$-permutations in terms of pattern avoidance.
The following statement provides a necessary and a sufficient pattern avoiding conditions, although these two conditions do not match.
The necessary condition~\eqref{item:fullykrootedCuttablePatternsNecessary} was considered for permutations in~\cite{Pilaud-brickAlgebra}.

\begin{proposition}
\label{prop:fullykrootedCuttablePatterns}
Let~$\sigma$ be a $k$-permutation.
\begin{enumerate}
\item \label{item:fullykrootedCuttablePatternsNecessary}
If~$\sigma$ if fully $k$-rooted cuttable, it contains no pattern~$b_1 \cdots b_k \cdot c \cdot a$ with~$a \le b_i \le c$~for~$i \in [k]$,
\item \label{item:fullykrootedCuttablePatternsSufficient}
If~$\sigma$ is not fully $k$-rooted cuttable, it contains a pattern~$a_1 \cdots a_k \cdot b \cdot a$ with~$a < b$ and~$a_i \le b$ for~$i \in [k]$, and a pattern~$b_1 \cdots b_k \cdot b \cdot a$ with~$a < b$ and~$a \le b_i$ for~$i \in [k]$.
\end{enumerate}
\end{proposition}

\begin{proof}
For~\eqref{item:fullykrootedCuttablePatternsNecessary}, assume by contradiction that~$\sigma$ contains a pattern~$b_1 \cdots b_k \cdot c \cdot a$ with~${a \le b_i \le c}$ for~$i \in [k]$.
Then in its restriction~$\sigma^{|[a,c]}$ to the interval~$[a,c]$, the letters~$a$ and~$c$ appear after the first~$k$ letters.
It follows that~$\sigma^{|[a,c]}$ is $k$-rooted uncuttable, so that~$\sigma$ is not fully $k$-rooted cuttable.

For~\eqref{item:fullykrootedCuttablePatternsSufficient}, assume that~$\sigma$ is not fully $k$-rooted cuttable.
Let~$\alpha < \beta$ be such that~$\sigma^{|[\alpha, \beta]}$ is not cuttable.
Denote by~$a_1, \dots, a_k$ the first $k$ letters of~$\sigma$ in~$[\alpha, \beta]$ and by~$b$ the first occurrence of~$\beta$ in~$\sigma$.
Since~$\beta$ is not a $k$-rooted cut of~$\sigma^{|[\alpha, \beta]}$, there is a letter~$a$ after~$b$ in~$\sigma$ with~$a < b$.
We have thus found a pattern~$a_1 \cdots a_k \cdot b \cdot a$ with~$a < b$ and~$a_i \le b$ for~$i \in [k]$.
The proof for the other pattern is symmetric considering the last occurrence of~$\alpha$ in~$\sigma$.
\end{proof}

\begin{remark}
In \cref{prop:fullykrootedCuttablePatterns}, observe that
\begin{itemize}
\item the necessary condition~\eqref{item:fullykrootedCuttablePatternsNecessary} is not sufficient: for instance, the $3$-permutation~$113213223$ is $3$-rooted uncuttable while it contains no pattern~$b_1 \cdot b_2 \cdot b_3 \cdot c \cdot a$ with~$a \le b_i \le c$~for~$i \in [3]$.
\item the sufficient condition~\eqref{item:fullykrootedCuttablePatternsSufficient} is not necessary: for instance, the permutation~$1432$ is fully $1$-rooted cuttable (it avoids~$231$) but contains a pattern~$a_1 \cdot b \cdot a$ with~$a < b$ and~$a_1 \le b$ (consider~$143$) and a pattern~$b_1 \cdot b \cdot a$ with~$a < b$ and~$a \le b_1$ (consider~$432$).
\end{itemize}
\end{remark}

We derive in particular the following observation from \cref{prop:fullykrootedCuttablePatterns}\,\eqref{item:fullykrootedCuttablePatternsNecessary}.

\begin{corollary}
\label{coro:fullykrootedCuttableStructure}
If~$\sigma \in \Perm_k(n)$ is fully $k$-rooted cuttable and~$i \in [n]$, the suffix of~$\sigma$ located after the last occurence of~$i$ decomposes into~$\mu \nu$ where~$\mu_p < i < \nu_q$ for all~$p \in [|\mu|]$ and~$q \in [|\nu|]$.
\end{corollary}

Finally, we observe that fully $k$-rooted cuttable $k$-permutations form a pattern class.

\begin{theorem}
\label{thm:fullykrootedCuttablepermutationClass}
The set of fully $k$-rooted cuttable $k$-permutations is a $k$-permutation class: for any fully $k$-rooted cuttable $k$-permutation~${\sigma \in \Perm_k(n)}$ and any~$L \subseteq [n]$, the restriction~$\sigma^{|L}$ is fully $k$-rooted cuttable.
\end{theorem}

\begin{proof}
Assume that~$\sigma \in \Perm_k(n)$ is fully $k$-rooted cuttable and that~$L = \{\ell_1, \dots, \ell_{|L|}\} \subseteq [n]$.
For any~$X \subseteq [|L|]$, the restriction~$(\sigma^{|L})^{|X}$ coincides with the restriction~$\sigma^{|\set{\ell_x}{x \in X}}$ by \cref{rem:relationOperatorsMultiposets}, and is thus $k$-rooted cuttable.
Therefore, $\sigma^{|L}$ is fully $k$-rooted cuttable.
\end{proof}


\subsubsection{Action of~\texorpdfstring{$\tidyCitelangisSeries$}{TCitk} on words and permutations}
\label{subsubsec:actionTidySeriesSignaletic}

We show in this section that $\FQSym_k$ can be endowed with a tidy series $k$-citelangis structure, and that the resulting tidy series $k$-citelangis algebra is free. Therefore, the free tidy series $k$-citelangis subalgebra generated by the $k$-permutation~$1^{\{k\}}$ provides a combinatorial model for the basis for the tidy series $k$-citelangis operad.

\pagebreak
\paraul{Action on words}
We first show that the free algebra~$\alphabet^{\ge k}$ can be endowed with the structure of a tidy series $k$-citelangis algebra.

\begin{definition}
\label{def:tidyCitelangisSeriesActionWords}
For a tidy series $k$-citelangis operation~$\operation \in \{\op{l}, \op{r}\}^*$ and two words~$X$ and~$Y$ such that~${|\operation| \le \min(|X|, |Y|)}$, we define inductively
\[
X \; \operation \; Y =
\begin{cases}
X Y & \text{if } \operation = \varepsilon, \\
x (\underline{X} \; \underline{\operation} \; Y) & \text{if } \operation = \; \op{l} \!\underline{\operation} \text{ and } X = x\underline{X}, \\
y (X \; \underline{\operation} \; \underline{Y}) & \text{if } \operation = \; \op{r} \!\underline{\operation} \text{ and } Y = y\underline{Y}.
\end{cases}
\]
In other words, we choose the first letters of the result among the first letters of~$X$ or~$Y$ depending on the operation~$\operation$, and we then concatenate the remaining suffixes of~$X$ and~$Y$.
\end{definition}

For example, when~$k = 1$, the two operations are given for the words~$xX$ and~$yY$ by
\[
x X \op{l} y Y = x X y Y
\qqandqq
x X \op{r} y Y = y x X Y.
\]
In particular, the concatenation is given by~${\cdot} = {\op{l}}$.
When~$k = 2$, the four operations are given for the words~$x_1 x_2 X$ and~$y_1 y_2 Y$ by
\begin{gather*}
x_1 x_2 X \op{l,l} y_1 y_2 Y = x_1 x_2 X y_1 y_2 Y,
\\
x_1 x_2 X \op{l,r} y_1 y_2 Y = x_1 y_1 x_2 X y_2 Y,
\\
x_1 x_2 X \op{r,l} y_1 y_2 Y = y_1 x_1 x_2 X y_2 Y,
\\
x_1 x_2 X \op{r,r} y_1 y_2 Y = y_1 y_2 x_1 x_2 X Y.
\end{gather*}
Again, the concatenation is given by~${\cdot} = {\op{l,l}}$.

\begin{proposition}
\label{prop:tidyCitelangisSeriesWords}
The free algebra~$\alphabet^{\ge k}$, endowed with the operations of \cref{def:tidyCitelangisSeriesActionWords}, defines a tidy series $k$-citelangis algebra.
The concatenation product~$\cdot$ of~$\alphabet^{\ge k}$ is given by~$\op{l}\!\!^k$.
\end{proposition}

\begin{proof}
Consider a destination vector~${\ind{p} \in [3]^k}$ of arity~$3$, and the corresponding tidy series $k$-citelangis operations~${\operation[a]_\ind{p}, \operation[b]_\ind{p}, \operation[c]_\ind{p}, \operation[d]_\ind{p} \in \{\op{l}, \op{r}\}^k}$ defined in \cref{prop:tidyCitelangisSeriesRelations}. Then for any words~$X,Y,Z$, the words that appear in
\[
X \, \operation[a]_\ind{p} \, (Y \, \operation[b]_\ind{p} \, Z)
\qqandqq
(X \, \operation[d]_\ind{p} \, Y) \, \operation[c]_\ind{p} \, Z
\]
is the word in~$X \shuffle Y \shuffle Z$ such that
\begin{itemize}
\item for~$i \in [k]$, its $i$-th letter was taken from the word~$X$ if~$\ind{p}_i = 1$, from the word~$Y$ if~$\ind{p}_i = 2$ and from the word~$Z$ if~$\ind{p}_i = 3$, and 
\item its remaining letters are taken first from~$X$, then from~$Y$ and finally from~$Z$.
\end{itemize}
Therefore our operations on words indeed satisfy the tidy series $k$-citelangis relation \eqref{eq:tidyCitelangisSeriesp}.
\end{proof}

\paraul{Action on permutations}
Replacing the concatenation by the shifted concatenation, one endows similarly the algebra~$\FQSym_k$ of $k$-permutations with the structure of a tidy series $k$-citelangis algebra.
This can be rephrased as follows.

\begin{definition}
\label{def:tidyCitelangisSeriesActionPermutations}
For any operation~${\operation \in \Operations_k}$ and any two $k$-permutations~$\mu$ and~$\nu$ of degree~$m$ and~$n$ respectively, $\mu \, \operation \, \nu$ is the $k$-permutation~$\pi$ appearing in~$\mu \shiftedShuffle \nu$ such that
\begin{itemize}
\item for all~$i \in [k]$, we have~$\pi_i \le m$ if~$\operation_i = {\op{l}}$ while~$\pi_i > m$ if~$\operation_i = {\op{r}}$,
\item the remaining entries of~$\pi$ are the concatenation of the remaining entries of~$\mu$ and the remaining entries of~$\nu[m]$.
\end{itemize}
\end{definition}

For example, for~$\mu = {\blue 321312132}$ and~$\nu = {\red 221211}$ in~$\FQSym_3$ we have~$\mu \op{l,r,r} \nu = {\blue 3}{\red 55}{\blue 21312132}{\red 4544}$ and~$\mu \op{r,l,l} \nu = {\red 5}{\blue 321312132}{\red 54544}$.
The following statement is immediate from \cref{prop:tidyCitelangisSeriesWords}.

\begin{proposition}
\label{prop:tidyCitelangisSeriesMultipermutations}
The algebra~$(\FQSym_k, \bar\cdot)$, endowed with the operations of \cref{def:tidyCitelangisSeriesActionPermutations}, defines a tidy series $k$-citelangis algebra.
The concatenation product~$\bar\cdot$ of~$\FQSym_k$ is given by~$\op{l}\!\!^k$.
\end{proposition}

The goal of this section is to show that the tidy series $k$-citelangis algebra~$\FQSym_k$ is free.
To manipulate this tidy series $k$-citelangis algebra, we consider the evaluations of syntax trees of~$\Syntax[\Operations_k]$ in~$\FQSym_k$.
See \cref{fig:tidyEvalSeries} for an illustration.

\begin{definition}
\label{def:tidyEvalSeries}
Denote by~$\tidyEvalPermSeries(\tree ; \sigma_1, \dots, \sigma_p)$ the evaluation of a syntax tree~$\tree \in \Syntax[\Operations_k]$ of arity~$p$ on $p$ $k$-permutations~${\sigma_1, \dots, \sigma_p}$ of~$\Perm_k$ using the tidy series $k$-citelangis structure of~$\FQSym_k$. 
The \defn{tidy series permutation evaluation} of~$\tree$ is then~$\tidyEvalPermSeries(\tree) \eqdef \tidyEvalPermSeries(\tree ; 1^{\{k\}}, \dots, 1^{\{k\}})$.
We extend by linearity $\tree$ to the elements of~$\Free[\Operations_k]$ on the one hand, and $\sigma_1, \dots, \sigma_p$ to the elements of~$\FQSym_k$ on the other hand.
\end{definition}

\begin{figure}[t]
	\centerline{$
	\tidyEvalPermSeries \left(
		\begin{tikzpicture}[baseline={([yshift=-.8ex]current bounding box.center)}, level 1/.style={sibling distance = 1.8cm, level distance = 1cm}, level 2/.style={sibling distance = 1cm, level distance = .7cm}, level 3/.style={sibling distance = .5cm, level distance = .6cm}]
			\node [rectangle, draw] {$\op{l,r}$}
				child {node [rectangle, draw] {$\op{r,r}$}
					child {node [rectangle, draw] {$\op{l,r}$}
						child {node {11}}
						child {node {22}}
					}
					child {node [rectangle, draw] {$\op{l,l}$}
						child {node {33}}
						child {node {44}}
					}
				}
				child {node [rectangle, draw] {$\op{r,l}$}
					child {node {55}}
					child {node {66}}
				}
			;
		\end{tikzpicture}
	\right)
	=
	36 \cdot \tidyEvalPermSeries \left(
		\begin{tikzpicture}[baseline={([yshift=-.8ex]current bounding box.center)}, level 1/.style={sibling distance = 1cm, level distance = .7cm}, level 2/.style={sibling distance = .5cm, level distance = .6cm}]
			\node [rectangle, draw] {$\op{r}$}
				child {node [rectangle, draw] {$\op{l,r}$}
					child {node {11}}
					child {node {22}}
				}
				child {node [rectangle, draw] {$\op{l}$}
					child {node {3}}
					child {node {44}}
				}
			;
		\end{tikzpicture}	
	\right) \cdot \tidyEvalPermSeries \left(
		\begin{tikzpicture}[baseline={([yshift=-.8ex]current bounding box.center)}, level/.style={sibling distance=18mm/(#1+2), level distance = 1cm/sqrt(#1+2)}]
			\node [rectangle, draw] {$\op{l}$}
				child {node {55}}
				child {node {6}}
			;
		\end{tikzpicture}	
	\right)
	=
	36 \cdot 3121244 \cdot 556
	$}
	\caption{Illustration of \cref{rem:tidyEvalSeries}.}
	\label{fig:tidyEvalSeries}
\end{figure}

\begin{remark}
\label{rem:tidyEvalSeries}
Let us rephrase algorithmically \cref{def:tidyEvalSeries}.
For this, we generalize the evaluation to partial syntax trees, \ie trees whose nodes are labeled by operations with at most~$k$ letters~$\{\op{l}, \op{r}\}$ and whose leaves are labeled by any words.
The evaluation~$\tidyEvalPermSeries(\tree[s])$ of such a partial syntax tree~$\tree[s]$ is defined inductively as follows.
If~$\tree[s]$ is a leaf labeled by a word~$\ind{w}$, then~$\tidyEvalPermSeries(\tree[s]) = \ind{w}$.
Otherwise, if~$\tree[s]$ has a root with $\ell$ signals, $\tidyEvalPermSeries(\tree[s])$ is obtained as follows.
Let $\ell$ cars traverse the partial syntax tree~$\tree[s]$ in series.
The $j$-th car follows and erases the first letter of each signal it traverses, and finally arrives at a leaf where it reads and erases the first letter~$w_j$.
Let~$\ind{w} = w_1 \dots w_\ell$ denote the word formed by the letters read by the $\ell$ cars at their destination.
After all cars have reached their destinations, all letters of the signal at the root have been erased.
We are left with the two partial syntax trees~$\tree[l]$ and~$\tree[r]$ (where some letters of the signals in the nodes and of the words in the leaves have been erased by the $\ell$ cars).
The evaluation~$\tidyEvalPermSeries(\tree[s])$ is then obtained inductively by
\[
\tidyEvalPermSeries(\tree[s]) = \ind{w} \cdot \tidyEvalPermSeries(\tree[l]) \cdot \tidyEvalPermSeries(\tree[r]).
\]
Finally, for a syntax tree~$\tree$ of arity~$p$ and $k$-permutations~$\sigma_1 \in \Perm_k(n_1), \dots, \sigma_p \in \Perm_k(n_p)$, the evaluation $\tidyEvalPermSeries(\tree ; \sigma_1, \dots, \sigma_p)$ is the evaluation of the partial syntax tree~$\tree[s]$ obtained from~$\tree$ by putting the permutation~$\sigma_i[n_1 + \dots + n_{i-1}]$ at the $i$-th leaf for all~$i \in [p]$.
In particular, for~$\tidyEvalPermSeries(\tree) = \tidyEvalPermSeries(\tree ; 1^{\{k\}}, \dots, 1^{\{k\}})$, the $i$-th leaf is labeled by the word~$i^{\{k\}}$.
See \cref{fig:tidyEvalSeries}.
\end{remark}

As $\FQSym_k$ is a tidy series $k$-citelangis algebra, the tidy series permutation evaluation is preserved by the tidy series $k$-citelangis relations of \cref{prop:tidyCitelangisSeriesRelations}.
Thus, $\tidyEvalPermSeries(\tree[s] ; F_1, \dots, F_p) = \tidyEvalPermSeries(\tree[t] ; F_1, \dots, F_p)$ for any~$\tree[s], \tree[t] \in \Free[\Operations_k](p)$ which are equivalent modulo the tidy series $k$-citelangis relations, and for any~$F_1, \dots, F_p \in \FQSym_k$.
The objective of this section is to show the reciprocal statement.
The proof is based on $k$-rooted cuts in $k$-permutations introduced in \cref{def:krootedCutMultipermutation}.

\paraul{Tidy series permutation evaluations and $k$-rooted cuts}
Our next two lemmas state that the $k$-rooted cuts of a $k$-permutation~$\rho$ precisely correspond to its decompositions of the form~$\rho = \sigma \, \operation \, \tau$, where~$\operation \in \Operations_2$.
Their proofs immediately follow from \cref{def:tidyCitelangisSeriesActionPermutations} and are thus left to the reader.

\begin{lemma}
\label{lem:operationImplieskrootedCut}
For any $k$-permutations~$\sigma \in \Perm_k(m)$ and~$\tau \in \Perm_k(n)$, and any operation~${\operation \in \Operations_k}$, the degree~$m$ of~$\sigma$ is a $k$-rooted cut of~$\sigma \, \operation \, \tau$.
\end{lemma}

\begin{lemma}
\label{lem:krootedCutImpliesOperation}
For any $k$-permutation~$\rho \in \Perm_k(\ell)$ and any $k$-rooted cut~$\gamma \in \kcuts(\rho)$, there is a unique~${\operation \in \Operations_k}$ (defined by $\operation_i \eqdef {\op{l}}$ if~$\rho_i \le \gamma$ and~$\operation_i \eqdef {\op{r}}$ if~$\rho_i > \gamma$) such that~${\rho = \rho^{|[\gamma]} \, \operation \, \rho^{|[\ell] \ssm [\gamma]}}$.
\end{lemma}

\begin{remark}
\label{rem:algoTidyEvalSeries}
\cref{lem:krootedCutImpliesOperation} gives an inductive algorithm to compute all decompositions of a given $k$-permutation~$\rho$ as an evaluation of the form~$\rho = \tidyEvalPermSeries(\tree ; \sigma_1, \dots, \sigma_p)$.
Namely, $\rho$ admits
\begin{itemize}
\item the trivial evaluation~$\rho = \tidyEvalPermSeries(\one ; \rho)$, where~$\one$ is the unit syntax tree with no node and a single leaf, and 
\item the evaluation~$\rho = \tidyEvalPermSeries(\tree ; \sigma_1, \dots, \sigma_l, \tau_1, \dots, \tau_r)$, for any $k$-rooted cut~$\gamma \in \kcuts(\rho)$ and any evaluations~$\rho^{|[\gamma]} = \tidyEvalPermSeries(\tree[l] ; \sigma_1, \dots, \sigma_l)$ and $\rho^{|[\ell] \ssm [\gamma]} = \tidyEvalPermSeries(\tree[r] ; \tau_1, \dots, \tau_r)$, where~$\tree$ is the syntax tree with root~$\operation \in \Operations_k$ defined by \cref{lem:krootedCutImpliesOperation} and with subtrees~$\tree[l]$ and~$\tree[r]$.
\end{itemize}
This algorithm implies the existence of decompositions of the form~$\rho = \tidyEvalPermSeries(\tree ; \sigma_1, \dots, \sigma_p)$ where~$\sigma_1, \dots, \sigma_p$ are $k$-rooted uncuttable.
In fact, we can even impose the position of the first $k$-rooted cut.
\end{remark}

\begin{corollary}
\label{coro:krootedCutImpliesSyntaxTree}
For any $k$-permutation~$\rho \in \Perm_k$ and any $k$-rooted cut~$\gamma \in \kcuts(\rho)$, there exists a syntax tree~$\tree$ of arity~$p$ with left subtree of arity~$l$ and $k$-rooted uncuttable $k$-permutations ${\sigma_1 \in \Perm_k(n_1), \dots, \sigma_p \in \Perm_k(n_p)}$ such that~$\rho = \tidyEvalPermSeries(\tree ; \sigma_1, \dots, \sigma_p)$ and~$\gamma = n_1 + \dots + n_l$.
\end{corollary}

We now characterize the $k$-permutations~$\rho$ that admit a decomposition of the form~${\rho = \tidyEvalPermSeries(\tree)}$.

\begin{proposition}
\label{prop:characterizationTidySeriesPermutationEvaluations}
The tidy series permutation evaluations of the syntax trees of~$\Syntax[\Operations_k]$ are precisely the fully $k$-rooted cuttable $k$-permutations.
\end{proposition}

\begin{proof}
Consider first a $k$-permutation~$\rho = \tidyEvalPermSeries(\tree)$ with~$\tree \in \Syntax[\Operations_k](\ell)$.
We prove by induction on~$\ell$ that~$\rho$ is fully $k$-rooted cuttable.
If~$\ell = 1$, there is nothing to prove.
Assume that~$\ell \ge 2$ and let~$1 \le a < b \le \ell$.
Let~$\tree[l]$ and~$\tree[r]$ denote the left and right subtrees of~$\tree$, and let~$\gamma$ be the arity of~$\tree[l]$, so that~$\rho^{|[\gamma]} = \tidyEvalPermSeries(\tree[l])$ and~$\rho^{|[\ell] \ssm [\gamma]} = \tidyEvalPermSeries(\tree[r])$.
We distinguish three cases:
\begin{itemize}
\item Assume that~$b \le \gamma$. Since~$\rho^{|[\ell]} = \tidyEvalPermSeries(\tree[l])$ is fully $k$-rooted cuttable by induction hypothesis and $k$-rooted cuts are preserved by restriction by \cref{lem:krootedCutsRestriction}, we obtain that~$\rho^{|[a,b]} = (\rho^{|[\gamma]})^{|[a,b]}$ is $k$-rooted cuttable.
\item Assume that~$\gamma \le a$. The argument is similar since~$\rho^{|[a,b]} = (\rho^{|[\ell] \ssm [\gamma]})^{|[a-\gamma,b-\gamma]}$.
\item Assume finally that~$a < \gamma < b$. By \cref{lem:operationImplieskrootedCut}, $\gamma$ is a $k$-rooted cut of~$\rho$. Therefore, $\gamma-a$ is a $k$-rooted cut of~$\rho^{|[a,b]}$ by \cref{lem:krootedCutsRestriction}.
\end{itemize}

Conversely, consider now a fully $k$-rooted cuttable $k$-permutation~$\rho \in \Perm_k(\ell)$.
Similarly to \cref{rem:algoTidyEvalSeries}, we prove by induction on~$\ell$ that~$\rho$ is the tidy series permutation evaluation of a syntax tree.
If~$\ell = 1$, then~$\rho = \tidyEvalPermSeries(\one)$.
If~$\ell \ge 2$, then~$\rho$ admits at least one $k$-rooted cut~$\gamma$ by assumption.
Moreover, $\rho^{|[\gamma]}$ and~$\rho^{|[\ell] \ssm [\gamma]}$ are both fully $k$-rooted cuttable by \cref{thm:fullykrootedCuttablepermutationClass}.
By induction, we obtain that~$\rho^{|[\gamma]} = \tidyEvalPermSeries(\tree[l])$ and~$\rho^{|[\ell] \ssm [\gamma]} = \tidyEvalPermSeries(\tree[r])$.
Then~$\rho = \tidyEvalPermSeries(\tree)$, where~$\tree$ is the syntax tree with root~$\operation \in \Operations_k$ defined by \cref{lem:krootedCutImpliesOperation} and with subtrees~$\tree[l]$ and~$\tree[r]$.
\end{proof}

\paraul{Freeness}
Our objective is now to prove that the decompositions of a given $k$-permutation~$\rho$ provided by \cref{coro:krootedCutImpliesSyntaxTree} are all equivalent up to the tidy series $k$-citelangis relations of \cref{prop:tidyCitelangisSeriesRelations}.
Our first step is to understand the evaluations of a quadratic syntax tree on three permutations.
We start from a simple observation, which again immediately follows from \cref{def:tidyCitelangisSeriesActionPermutations}.
Recall that a syntax tree is tidy series when all the traffic signals not contained in its series routes point to the left.

\begin{lemma}
\label{lem:krootedCutIffTidy}
For any $k$-permutations~$\rho \in \Perm_k(\ell)$, $\sigma \in \Perm_k(m)$ and~$\tau \in \Perm_k(n)$, and any operations~$\operation[a], \operation[b], \operation[a]', \operation[b]' \in \Operations_k$, we have
\begin{itemize}
\item $\ell + m$ is a $k$-rooted cut of~$\tidyEvalPermSeries\Big( 
	\begin{tikzpicture}[baseline=-.5cm, level 1/.style={sibling distance = .8cm, level distance = .7cm}, level 2/.style={sibling distance = .6cm, level distance = .5cm}]
		\node [rectangle, draw, minimum height=.5cm] {$\operation[a]$}
			child {node {}}
			child {node [rectangle, draw, minimum height=.5cm] {$\operation[b]$}
				child {node {}}
				child {node {}}
			}
		;
	\end{tikzpicture}
; \rho, \sigma, \tau \Big)$ if and only if
	\begin{tikzpicture}[baseline=-.5cm, level 1/.style={sibling distance = .8cm, level distance = .7cm}, level 2/.style={sibling distance = .6cm, level distance = .5cm}]
		\node [rectangle, draw, minimum height=.5cm] {$\operation[a]$}
			child {node {}}
			child {node [rectangle, draw, minimum height=.5cm] {$\operation[b]$}
				child {node {}}
				child {node {}}
			}
		;
	\end{tikzpicture}
is tidy series,
\item $\ell$ is a $k$-rooted cut of~$\tidyEvalPermSeries \Big( 
	\begin{tikzpicture}[baseline=-.5cm, level 1/.style={sibling distance = .8cm, level distance = .7cm}, level 2/.style={sibling distance = .6cm, level distance = .5cm}]
		\node [rectangle, draw, minimum height=.5cm] {$\operation[a]'$}
			child {node [rectangle, draw, minimum height=.5cm] {$\operation[b]'$}
				child {node {}}
				child {node {}}
			}
			child {node {}}
		;
	\end{tikzpicture}
; \rho, \sigma, \tau \Big)$ if and only if
	\begin{tikzpicture}[baseline=-.5cm, level 1/.style={sibling distance = .8cm, level distance = .7cm}, level 2/.style={sibling distance = .6cm, level distance = .5cm}]
		\node [rectangle, draw, minimum height=.5cm] {$\operation[a]'$}
			child {node [rectangle, draw, minimum height=.5cm] {$\operation[b]'$}
				child {node {}}
				child {node {}}
			}
			child {node {}}
		;
	\end{tikzpicture}
is tidy series.
\end{itemize}
\end{lemma}

\begin{lemma}
\label{lem:uniqueTidySeriesCitelangisPermutationEvaluationQuadratic}
For any $k$-permutations~$\rho, \sigma, \tau \in \Perm_k$, and any operations~$\operation[a], \operation[b], \operation[a]', \operation[b]' \in \Operations_k$, if
\[
	\tidyEvalPermSeries \Big( 
	\begin{tikzpicture}[baseline=-.5cm, level 1/.style={sibling distance = .8cm, level distance = .7cm}, level 2/.style={sibling distance = .6cm, level distance = .5cm}]
		\node [rectangle, draw, minimum height=.5cm] {$\operation[a]$}
			child {node {}}
			child {node [rectangle, draw, minimum height=.5cm] {$\operation[b]$}
				child {node {}}
				child {node {}}
			}
		;
	\end{tikzpicture}
	; \rho, \sigma, \tau \Big)
	=
	\tidyEvalPermSeries \Big( 
	\begin{tikzpicture}[baseline=-.5cm, level 1/.style={sibling distance = .8cm, level distance = .7cm}, level 2/.style={sibling distance = .6cm, level distance = .5cm}]
		\node [rectangle, draw, minimum height=.5cm] {$\operation[a]'$}
			child {node [rectangle, draw, minimum height=.5cm] {$\operation[b]'$}
				child {node {}}
				child {node {}}
			}
			child {node {}}
		;
	\end{tikzpicture}
	; \rho, \sigma, \tau \Big),
\]
then
\(
	\begin{tikzpicture}[baseline=-.5cm, level 1/.style={sibling distance = .8cm, level distance = .7cm}, level 2/.style={sibling distance = .6cm, level distance = .5cm}]
		\node [rectangle, draw, minimum height=.5cm] {$\operation[a]$}
			child {node {}}
			child {node [rectangle, draw, minimum height=.5cm] {$\operation[b]$}
				child {node {}}
				child {node {}}
			}
		;
	\end{tikzpicture}
	=
	\begin{tikzpicture}[baseline=-.5cm, level 1/.style={sibling distance = .8cm, level distance = .7cm}, level 2/.style={sibling distance = .6cm, level distance = .5cm}]
		\node [rectangle, draw, minimum height=.5cm] {$\operation[a]'$}
			child {node [rectangle, draw, minimum height=.5cm] {$\operation[b]'$}
				child {node {}}
				child {node {}}
			}
			child {node {}}
		;
	\end{tikzpicture}
\)
is a tidy series $k$-citelangis relation.
\end{lemma}

\begin{proof}
Let~$\pi$ denote the $k$-permutation obtained by these evaluations, and let $\ell$, $m$ and~$n$ denote the degrees of~$\rho$, $\sigma$ and~$\tau$ respectively.
By \cref{lem:operationImplieskrootedCut}, we obtain that~$\ell$ and~$\ell + m$ are $k$-rooted cuts of~$\pi$.
By \cref{lem:krootedCutIffTidy}, we therefore derive that the two syntax trees are tidy series.
Moreover, the series destination vector of the two syntax trees are both given by the first $k$ letters of~$\pi$ where we replace all letters between~$1$ and~$\ell$ by~$1$, all letters between~$\ell + 1$ and~$\ell + m$ by~$2$, and all letters between~$\ell+m+1$ and~$\ell+m+n$ by~$3$.
Since they are tidy series and have the same series destination vector, the two syntax trees form a tidy series $k$-citelangis relation.
\end{proof}

We now prove that, up to the tidy series $k$-citelangis relations, any $k$-permutation can be obtained in a unique way as the evaluation of a syntax tree on $k$-rooted uncuttable $k$-permutations.

\begin{proposition}
\label{prop:uniqueTidySeriesCitelangisPermutationEvaluation}
For any syntax trees~$\tree , \tree' \in \Syntax[\Operations_k]$ of arity~$p$ and~$p'$ respectively, and any $k$-rooted uncuttable $k$-permuta\-tions~$\sigma_1, \dots, \sigma_p$, $\sigma'_1, \dots, \sigma'_{p'} \in \Perm_k$, if
\[
\tidyEvalPermSeries(\tree ; \sigma_1, \dots, \sigma_p) = \tidyEvalPermSeries(\tree' ; \sigma'_1, \dots, \sigma'_{p'}),
\]
then $p = p'$, $\tree = \tree'$ modulo the tidy series $k$-citelangis relations and~$\sigma_i = \sigma_i'$ for all~$i \in [p]$.
\end{proposition}

\begin{proof}
Let~$\pi = \tidyEvalPermSeries(\tree ; \sigma_1, \dots, \sigma_p) = \tidyEvalPermSeries(\tree' ; \sigma'_1, \dots, \sigma'_{p'})$ and let~$n$ denote its degree.

We prove the result by induction on~$p$.
If~$p = 1$, then~$\pi$ is $k$-rooted uncuttable, so that~$p' = 1$, $\tree = \tree'$ is the only syntax tree of arity~$1$, and~$\sigma_1 = \sigma_1' = \pi$.

Assume now that~$p > 1$.
Let~$\operation[a]$ be the root of~$\tree$, let~$\tree[l]$ and~$\tree[r]$ be its left and right subtrees, let~$l$ be the arity of~$\tree[l]$ and let~$\gamma$ be the corresponding $k$-rooted cut of~$\pi$.
We thus have~${\tidyEvalPermSeries(\tree[l] ; \sigma_1, \dots, \sigma_l) = \pi^{|[\gamma]}}$ and~$\tidyEvalPermSeries(\tree[r] ; \sigma_{l+1}, \dots, \sigma_p) = \pi^{|[n] \ssm [\gamma]}$.
Define~$\operation[a]', \tree[l]', \tree[r]', l'$ and~$\gamma'$ similarly for~$\tree'$.
These notations are illustrated below:

\begin{center}
	\begin{tikzpicture}[baseline={([yshift=-.8ex]current bounding box.center)}, level/.style={sibling distance = .5cm, level distance = .7cm}]
		\node {$\tree$}
			child {node {$\sigma_1$}}
			child [edge from parent/.style={}] {node {$\dots$}}
			child {node {$\sigma_p$}}
		;
	\end{tikzpicture}
	=
	\begin{tikzpicture}[baseline={([yshift=-.8ex]current bounding box.center)}, level 1/.style={sibling distance = 2cm, level distance = .8cm}, level 2/.style={sibling distance = .6cm, level distance = .7cm}]
		\node [rectangle, draw, minimum height=.5cm] {$\operation[a]$}
			child {node {$\tree[l]$}
				child {node {$\sigma_1$}}
				child [edge from parent/.style={}] {node {$\dots$}}
				child {node {$\sigma_l$}}
			}
			child {node {$\tree[r]$}
				child {node {$\sigma_{l+1}$}}
				child [edge from parent/.style={}] {node {$\dots$}}
				child {node {$\sigma_p$}}
			}
		;
	\end{tikzpicture}
	\qquad\text{and}\qquad
	\begin{tikzpicture}[baseline={([yshift=-.8ex]current bounding box.center)}, level/.style={sibling distance = .5cm, level distance = .7cm}]
		\node {$\tree'$}
			child {node {$\sigma'_1$}}
			child [edge from parent/.style={}] {node {$\dots$}}
			child {node {$\sigma'_{p'}$}}
		;
	\end{tikzpicture}
	=
	\begin{tikzpicture}[baseline={([yshift=-.8ex]current bounding box.center)}, level 1/.style={sibling distance = 2cm, level distance = .8cm}, level 2/.style={sibling distance = .6cm, level distance = .7cm}]
		\node [rectangle, draw, minimum height=.5cm] {$\operation[a]'$}
			child {node {$\tree[l]'$}
				child {node {$\sigma'_1$}}
				child [edge from parent/.style={}] {node {$\dots$}}
				child {node {$\sigma'_{l'}$}}
			}
			child {node {$\tree[r]'$}
				child {node {$\sigma'_{l+1}$}}
				child [edge from parent/.style={}] {node {$\dots$}}
				child {node {$\sigma'_{p'}$}}
			}
		;
	\end{tikzpicture}
\end{center}

Assume first that~$\gamma = \gamma'$. Then~$\operation[a] = \operation[a]'$, $\tidyEvalPermSeries(\tree[l] ; \sigma_1, \dots, \sigma_l) = \pi^{|[\gamma]} = \tidyEvalPermSeries(\tree[l]' ; \sigma'_1, \dots, \sigma'_{l'})$ and~$\tidyEvalPermSeries(\tree[r] ; \sigma_{l+1}, \dots, \sigma_p) = \pi^{|[n] \ssm [\gamma]} = \tidyEvalPermSeries(\tree[r]' ; \sigma'_{l'+1}, \dots, \sigma'_{p'})$ by \cref{lem:krootedCutImpliesOperation}.
By induction hypothesis, the first equality ensures that~$l = l'$, that~$\tree[l] = \tree[l]'$ modulo the tidy series $k$-citelangis relations and that~${\sigma_i = \sigma'_i}$ for all~$i \in [l]$, while the second equality ensures that~$p-l = p'-l'$, that~$\tree[r] = \tree[r]'$ modulo the tidy series $k$-citelangis relations and that~$\sigma_i = \sigma'_i$ for all~$i \in [p] \ssm [l]$.
We thus conclude that~$p = p'$, that~$\tree = \tree'$ modulo the tidy series $k$-citelangis relations, and that~$\sigma_i = \sigma'_i$ for all~$i \in [p]$.

Assume now without loss of generality that~$\gamma < \gamma'$.
Consider the $k$-permutations~$\rho \eqdef \pi^{|[\gamma]}$, $\sigma \eqdef \pi^{|[\gamma'] \ssm [\gamma]}$ and~$\tau \eqdef \pi^{|[n] \ssm [\gamma']}$.
Since~$\gamma'$ is a $k$-rooted cut of~$\pi$ larger than~$\gamma$, \cref{lem:krootedCutsRestriction} ensures that~$\gamma'-\gamma$ is a $k$-rooted cut of~$\pi^{|[n] \ssm [\gamma]}$.
By \cref{coro:krootedCutImpliesSyntaxTree} and induction hypothesis, there exists a syntax tree~$\tree[s]$ of arity~$p-l$ with root~$\operation[b]$, left subtree~$\tree[u]$ of arity~$u$ and right subtree~$\tree[v]$, such that~$\tree[l]' = \tree[s]'$ modulo the tidy series $k$-citelangis relations and~$\tidyEvalPermSeries(\tree[s] ; \sigma_{l+1}, \dots, \sigma_p) = \pi^{|[n] \ssm [\gamma]}$, $\tidyEvalPermSeries(\tree[u] ; \sigma_{l+1}, \dots, \sigma_{l+u}) = \sigma$ and $\tidyEvalPermSeries(\tree[v] ; \sigma_{l+u+1}, \dots, \sigma_p) = \tau$.
Similarly, since~$\gamma$ is a $k$-rooted cut of~$\pi$ smaller than~$\gamma'$, \cref{lem:krootedCutsRestriction} ensures that~$\gamma$ is a $k$-rooted cut of~$\pi^{|[\gamma']}$.
By \cref{coro:krootedCutImpliesSyntaxTree} and induction hypothesis, there exists a syntax tree~$\tree[s]'$ of arity~$l'$, with root~$\operation[b]'$, left subtree~$\tree[u]'$ of arity~$u'$ and right subtree~$\tree[v]'$, such that~$\tree[l]' = \tree[s]'$ modulo the tidy series $k$-citelangis relations and $\tidyEvalPermSeries(\tree[s]' ; \sigma'_1, \dots, \sigma'_{l'}) = \pi^{|[\gamma']}$, $\tidyEvalPermSeries(\tree[u]' ; \sigma'_1, \dots, \sigma'_{u'}) = \rho$ and $\tidyEvalPermSeries(\tree[v]' ; \sigma_{u'+1}, \dots, \sigma_{l'}) = \sigma$.
Since
\[
\begin{array}{r@{\;}c@{\;}l}
\tidyEvalPermSeries(\tree[l] ; \sigma_1, \dots, \sigma_l) & = \rho = & \tidyEvalPermSeries(\tree[u]' ; \sigma'_1, \dots, \sigma'_{u'}), \\
\tidyEvalPermSeries(\tree[u] ; \sigma_{l+1}, \dots, \sigma_{l+u}) & = \sigma = & \tidyEvalPermSeries(\tree[v]' ; \sigma_{u'+1}, \dots, \sigma_{l'}), \\
\text{and}\quad
\tidyEvalPermSeries(\tree[v] ; \sigma_{l+u+1}, \dots, \sigma_p) & = \tau = & \tidyEvalPermSeries(\tree[r]' ; \sigma'_{l'+1}, \dots, \sigma'_{p'}),
\end{array}
\]
we obtain by three applications of the induction hypothesis that~$p = p'$ and~$\sigma_i = \sigma'_i$ for all~$i \in [p]$.
Moreover, we have
\[
	\tidyEvalPermSeries \Big( 
	\begin{tikzpicture}[baseline=-.5cm, level 1/.style={sibling distance = .8cm, level distance = .7cm}, level 2/.style={sibling distance = .6cm, level distance = .5cm}]
		\node [rectangle, draw, minimum height=.5cm] {$\operation[a]$}
			child {node {}}
			child {node [rectangle, draw, minimum height=.5cm] {$\operation[b]$}
				child {node {}}
				child {node {}}
			}
		;
	\end{tikzpicture}
	; \rho, \sigma, \tau \Big)
	= \pi =
	\tidyEvalPermSeries \Big( 
	\begin{tikzpicture}[baseline=-.5cm, level 1/.style={sibling distance = .8cm, level distance = .7cm}, level 2/.style={sibling distance = .6cm, level distance = .5cm}]
		\node [rectangle, draw, minimum height=.5cm] {$\operation[a]'$}
			child {node [rectangle, draw, minimum height=.5cm] {$\operation[b]'$}
				child {node {}}
				child {node {}}
			}
			child {node {}}
		;
	\end{tikzpicture}
	; \rho, \sigma, \tau \Big).
\]
By \cref{lem:uniqueTidySeriesCitelangisPermutationEvaluationQuadratic}, we conclude that
\(
	\begin{tikzpicture}[baseline=-.5cm, level 1/.style={sibling distance = .8cm, level distance = .7cm}, level 2/.style={sibling distance = .6cm, level distance = .5cm}]
		\node [rectangle, draw, minimum height=.5cm] {$\operation[a]$}
			child {node {}}
			child {node [rectangle, draw, minimum height=.5cm] {$\operation[b]$}
				child {node {}}
				child {node {}}
			}
		;
	\end{tikzpicture}
	=
	\begin{tikzpicture}[baseline=-.5cm, level 1/.style={sibling distance = .8cm, level distance = .7cm}, level 2/.style={sibling distance = .6cm, level distance = .5cm}]
		\node [rectangle, draw, minimum height=.5cm] {$\operation[a]'$}
			child {node [rectangle, draw, minimum height=.5cm] {$\operation[b]'$}
				child {node {}}
				child {node {}}
			}
			child {node {}}
		;
	\end{tikzpicture}
\)
is a tidy series $k$-citelangis relation, and thus that~$\tree = \tree'$ up to tidy series $k$-citelangis relations.
\end{proof}

\begin{remark}
\label{rem:algoTidyEvalSeriesNormalForms}
From \cref{rem:algoTidyEvalSeries,prop:uniqueTidySeriesCitelangisPermutationEvaluation}, we derive an inductive algorithm to compute the decomposition of a given $k$-permutation~$\rho$ as an evaluation of the form~$\rho = \tidyEvalPermSeries(\tree ; \sigma_1, \dots, \sigma_p)$, where~$\tree$ is in normal form in the tidy series $k$-citelangis rewriting system of \cref{subsubsec:rewritingSystemCitelangis} and~$\sigma_1, \dots, \sigma_p$ are $k$-rooted uncuttable $k$-permutations.
Namely, 
\begin{itemize}
\item if~$\rho$ is $k$-rooted uncuttable, then~$\rho = \tidyEvalPermSeries(\one ; \rho)$, where~$\one$ is the unit syntax tree with no node and a single leaf, and
\item otherwise, $\rho = \tidyEvalPermSeries(\tree ; \sigma_1, \dots, \sigma_l, \tau_1, \dots, \tau_r)$, where~$\gamma$ be the rightmost $k$-rooted cut of~$\rho$, $\rho^{|[\gamma]} = \tidyEvalPermSeries(\tree[l] ; \sigma_1, \dots, \sigma_l)$ and $\rho^{|[\ell] \ssm [\gamma]} = \tidyEvalPermSeries(\tree[r] ; \tau_1, \dots, \tau_r)$ are such that~$\tree[l], \tree[r]$ are in normal form and~$\sigma_1, \dots, \sigma_l, \tau_1, \dots, \tau_r$ are $k$-rooted uncuttable, and~$\tree$ is the syntax tree with root~$\operation \in \Operations_k$ defined by \cref{lem:krootedCutImpliesOperation} and with subtrees~$\tree[l]$ and~$\tree[r]$.
\end{itemize}
\end{remark}

\cref{prop:uniqueTidySeriesCitelangisPermutationEvaluation} proves the main result of this section.

\begin{theorem}
\label{thm:freeTidyCitelangisSeries}
The tidy series $k$-citelangis algebra~$\FQSym_k$ is free on $k$-rooted uncuttable $k$-permutations.
\end{theorem}

\paraul{A complete combinatorial model for~$\tidyCitelangisSeries$}
By \cref{thm:freeTidyCitelangisSeries}, the tidy series $k$-citelangis operad~$\tidyCitelangisSeries$ can be fully understood from the tidy series $k$-citelangis subalgebra of~$\FQSym_k$ generated by the $k$-permutation~$1^{\{k\}}$.
We close this section with a completely explicit combinatorial model for this algebra.
We first obtain from \cref{prop:characterizationTidySeriesPermutationEvaluations} a combinatorial model for the operations of~$\tidyCitelangisSeries$.

\begin{proposition}
\label{prop:combinatorialModelTidySeriesCitelangis}
The tidy series permutation evaluation~$\tree \mapsto \tidyEvalPermSeries(\tree)$ is a graded bijection from the tidy series equivalence classes of syntax trees of~$\Syntax[\Operations_k]$ to the fully $k$-rooted cuttable $k$-permutations.
\end{proposition}

\begin{proof}
This map~$\tree \mapsto \tidyEvalPermSeries(\tree)$ is surjective on fully $k$-rooted cuttable $k$-permutations by \cref{prop:characterizationTidySeriesPermutationEvaluations}.
It is compatible with the tidy series $k$-citelangis relations by \cref{prop:tidyCitelangisSeriesMultipermutations}.
Finally, it is bijective since the tidy series $k$-citelangis algebra~$\FQSym_k$ is free on uncuttable $k$-permutations by \cref{thm:freeTidyCitelangisSeries}.
\end{proof}

Therefore, the fully $k$-rooted cuttable $k$-permutations can be thought of as a basis of the tidy series $k$-citelangis operad~$\tidyCitelangisSeries$.
Through the bijection of \cref{prop:characterizationTidySeriesPermutationEvaluations}, we can thus define the compositions of the tidy series $k$-citelangis operad~$\tidyCitelangisSeries$ directly on fully $k$-rooted cuttable $k$-permutations.

\begin{definition}
\label{def:combinatorialModelCompositionsTidySeriesCitelangis}
For any integers~$i \le m$ and~$n$, and any two fully $k$-rooted cuttable $k$-permutations $\sigma = \tidyEvalPermSeries(\tree[s]) \in \Perm_k(m)$ and $\tau = \tidyEvalPermSeries(\tree[t]) \in \Perm_k(n)$, we define the \defn{tidy series $i$-th composition} as
\[
\sigma \tidySeriesCirc{i} \tau \eqdef \tidyEvalPermSeries(\tree[s] \circ_i \tree[t]) \in \Perm_k(m+n-1).
\]
\end{definition}

The following statement provides a direct combinatorial description of the composition~$\tidySeriesCirc{i}$ of the tidy series $k$-citelangis operad~$\tidyCitelangisSeries$ on fully $k$-rooted cuttable $k$-permutations.

\begin{proposition}
\label{prop:combinatorialModelCompositionsTidySeriesCitelangis}
Let~$\sigma \in \Perm_k(m)$ and~$\tau \in \Perm_k(n)$ be two fully $k$-rooted cuttable $k$-permutations and~$i \in [m]$.
Write~$\sigma = \omega_0 \, i \, \omega_1 \, i \, \omega_2 \cdots \omega_{k-1} \, i \, \omega_k$, where~$\omega_k = \mu \, \nu$ with~$\mu_p < i < \nu_q$ for all~$p \in [|\mu|]$ and~$q \in [|\nu|]$ according to \cref{coro:fullykrootedCuttableStructure}, and write~$\tau = \tau_1 \, \cdots \, \tau_k \, \theta$.
Then
\[
\begin{array}{c@{\; = \;}l@{\,}l@{\,}l@{\,}l@{\,}l@{\,}l@{\,}l@{\,}l@{\,}l@{\,}l}
\sigma \tidySeriesCirc{i} \tau
& \omega_0[i,n] & \tau_1[i-1] & \omega_1[i,n] & \tau_2[i-1] & \omega_2[i,n] \cdots \omega_{k-1}[i,n] & \tau_k[i-1] & \mu[i,n] & \theta[i-1] & \nu[i,n] \\
& \omega_0[i,n] & \tau_1[i-1] & \omega_1[i,n] & \tau_2[i-1] & \omega_2[i,n] \cdots \omega_{k-1}[i,n] & \tau_k[i-1] & \mu      & \theta[i-1] & \nu[n-1].
\end{array}
\]
\end{proposition}

\begin{proof}
Consider arbitrary syntax trees~$\tree[s]$ and~$\tree[t]$ such that~$\sigma = \tidyEvalPermSeries(\tree[s])$ and~$\tau = \tidyEvalPermSeries(\tree[t])$.
By \cref{rem:tidyEvalSeries}, $\tidyEvalPermSeries(\tree[s] \circ_i \tree[t])$ is obtained inductively by letting~$k$ cars traverse~$\tree[s] \circ_i \tree[t]$ in series and recording the position of their arrival.
The cars that arrive at position~$i$ in~$\tree[s]$ thus continue their journey through~$\tree[t]$.
Therefore, the $k$ values~$i$ in~$\sigma$ are replaced by the first~$k$ values of~$\tau$, and the remaining values of~$\tau$ are placed at the only possible position in~$\sigma \tidySeriesCirc{i} \tau$.
\end{proof}

Here are some examples of tidy series compositions on fully $2$-rooted cuttable $2$-permutations:
\begin{gather*}
  {\blue66}1{\blue4}1|{\blue4322355} \tidySeriesCirc{1} 23{\red2113} = {\blue88}2{\blue6}3{\red2113}{\blue6544577}, \\
  {\blue6614143}2{\blue}2|{\blue355} \tidySeriesCirc{2} 23{\red2113} = {\blue8816165}3{\blue}4{\red3224}{\blue577}, \\
  {\blue661414}3{\blue22}3|{\blue55} \tidySeriesCirc{3} 23{\red2113} = {\blue881616}4{\blue22}5{\red4335}{\blue77}, \\
  {\blue661}4{\blue1}4{\blue3223}|{\blue55} \tidySeriesCirc{4} 23{\red2113} = {\blue881}5{\blue1}6{\blue3223}{\red5446}{\blue77}, \\
  {\blue6614143223}5{\blue}5| \tidySeriesCirc{5} 23{\red2113} = {\blue8814143223}6{\blue}7{\red6557}{\blue}, \\
  {\blue}6{\blue}6{\blue1414322355}| \tidySeriesCirc{6} 23{\red2113} = {\blue}7{\blue}8{\blue1414322355}{\red7668},
\end{gather*}
where we have marked by a vertical bar the separation $\mu|\nu$.

\begin{remark}
\label{rem:tidySeriesZinbiel}
Motivated by \cref{prop:combinatorialModelCompositionsTidySeriesCitelangis}, we will extend in \cref{subsubsec:seriesZinbiel} the tidy series compositions of \cref{def:combinatorialModelCompositionsTidySeriesCitelangis} from fully $k$-rooted cuttable $k$-permutations to all $k$-permutations.
\end{remark}


\subsubsection{Action of~\texorpdfstring{$\messyCitelangisSeries$}{MCitk} on words and permutations}
\label{subsubsec:actionMessySeriesSignaleticPermutations}

The objective of this section is to show that $\FQSym_k$ can also be endowed with a messy series $k$-citelangis algebra structure, generalizing the dendriform algebra structure on~$\FQSym_1$. Moreover, the resulting messy series $k$-citelangis algebra is free. Therefore, the free messy series $k$-citelangis subalgebra generated by the $k$-permutation~$1^{\{k\}}$ provides a combinatorial model for the basis for the messy series $k$-citelangis operad.

\paraul{Action on words}
We first show that the shuffle algebra~$\Shuffle^{\ge k}$ can be endowed with the structure of a messy series $k$-citelangis algebra.

\begin{definition}
\label{def:messyCitelangisSeriesActionWords}
For a messy series $k$-citelangis operation~$\operation \in \{\op{l}, \op{r}\}^*$ and two words~$X$ and~$Y$ such that~${|\operation| \le \min(|X|, |Y|)}$, we define
\[
X \; \operation \; Y =
\begin{cases}
X \shuffle Y & \text{if } \operation = \varepsilon, \\
x (\underline{X} \; \underline{\operation} \; Y) & \text{if } \operation = \; \op{l} \!\underline{\operation} \text{ and } X = x\underline{X}, \\
y (X \; \underline{\operation} \; \underline{Y}) & \text{if } \operation = \; \op{r} \!\underline{\operation} \text{ and } Y = y\underline{Y}.
\end{cases}
\]
In other words, we consider the shuffle of~$X$ and~$Y$, except that the~$i$-th letter of~$X \, \operation \, Y$ is forced to belong to~$X$ (resp.~to~$Y$) if the $i$-th letter of~$\operation$ is~$\op{l}$ (resp.~is~$\op{r}$). 
\end{definition}

For example, when~$k = 1$, the two operations are given for the words~$xX$ and~$yY$ by
\[
x X \op{l} y Y = x \big( X \shuffle y Y \big)
\qqandqq
x X \op{r} y Y = y \big( x X \shuffle Y \big).
\]
In particular, the shuffle product is given by~${\shuffle} = {\op{m}} \eqdef {\op{l}} + {\op{r}}$.
When~$k = 2$, the four operations are given for the words~$x_1 x_2 X$ and~$y_1 y_2 Y$ by
\begin{gather*}
x_1 x_2 X \op{l,l} y_1 y_2 Y = x_1 x_2 \big( X \shuffle y_1 y_2 Y \big),
\\
x_1 x_2 X \op{l,r} y_1 y_2 Y = x_1 y_1 \big( x_2 X \shuffle y_2 Y \big),
\\
x_1 x_2 X \op{r,l} y_1 y_2 Y = y_1 x_1 \big( x_2 X \shuffle y_2 Y \big),
\\
x_1 x_2 X \op{r,r} y_1 y_2 Y = y_1 y_2 \big( x_1 x_2 X \shuffle Y \big).
\end{gather*}
Again, the shuffle product is given by~${\shuffle} = {\op{m,m}} \eqdef {\op{l,l}} + {\op{l,r}} + {\op{r,l}} + {\op{r,r}}$.

\begin{proposition}[\cite{Pilaud-brickAlgebra}]
\label{prop:messyCitelangisSeriesWords}
The shuffle algebra~$\Shuffle^{\ge k}$, endowed with the operations of \cref{def:messyCitelangisSeriesActionWords}, defines a messy series $k$-citelangis algebra.
The shuffle product~$\shuffle$ of~$\Shuffle^{\ge k}$ is given by~$\op{m}\!\!^k \eqdef \sum_{\operation \in \Operations_k} \operation$.
\end{proposition}

\begin{proof}
Consider a destination vector~${\ind{p} \in [3]^k}$ of arity~$3$, and the corresponding messy series $k$-citelangis operations~${\operation[a]_\ind{p}, \operation[b]_\ind{p}, \operation[c]_\ind{p}, \operation[d]_\ind{p} \in \{\op{l}, \op{m}, \op{r}\}^k}$ defined in \cref{prop:messyCitelangisSeriesRelations}. Then for any words~$X,Y,Z$, the words that appear in
\[
X \, \operation[a]_\ind{p} \, (Y \, \operation[b]_\ind{p} \, Z)
\qqandqq
(X \, \operation[d]_\ind{p} \, Y) \, \operation[c]_\ind{p} \, Z
\]
are precisely the words in~$X \shuffle Y \shuffle Z$ such that for~$i \in [k]$, their $i$-th letter was taken from the word~$X$ if~$\ind{p}_i = 1$, from the word~$Y$ if~$\ind{p}_i = 2$ and from the word~$Z$ if~$\ind{p}_i = 3$.
Therefore our operations on words indeed satisfy the messy series $k$-citelangis relation \eqref{eq:messyCitelangisSeriesp}.
\end{proof}

\paraul{Action on permutations}
Replacing the shuffle by the shifted shuffle, one endows similarly the algebra~$\FQSym_k$ of $k$-permutations with the structure of a messy series $k$-citelangis algebra.
This can be rephrased as follows.

\begin{definition}
\label{def:messyCitelangisSeriesActionPermutations}
For any operation~${\operation \in \Operations_k}$ and any two $k$-permutations~$\mu$ and~$\nu$ of degree~$m$ and~$n$ respectively, $\mu \, \operation \, \nu$ is the sum of all $k$-permutations~$\pi$ appearing in~$\mu \shiftedShuffle \nu$ such that for all~$i \in [k]$, we have~$\pi_i \le m$ if~$\operation_i = {\op{l}}$ while~$\pi_i > m$ if~$\operation_i = {\op{r}}$.
\end{definition}

For example, for~$\mu = {\blue 321312132}$ and~$\nu = {\red 221211}$ we have~$\mu \op{l,r,r} \nu = {\blue 3}{\red 55}({\blue 21312132} \shuffle {\red 4544})$ and~$\mu \op{r,l,l} \nu = {\red 5}{\blue 32}({\blue 1312132} \shuffle {\red 54544})$.
The next statement is immediate from \cref{prop:messyCitelangisSeriesWords}.

\begin{proposition}[\cite{Pilaud-brickAlgebra}]
\label{prop:messyCitelangisSeriesMultipermutations}
The algebra~$(\FQSym_k, \shiftedShuffle)$, endowed with the operations of \cref{def:messyCitelangisSeriesActionPermutations}, defines a messy series $k$-citelangis algebra.
The shifted shuffle product~$\shiftedShuffle$ of~$\FQSym_k$ is given~by~$\op{m}\!\!^k$.
\end{proposition}

Similarly to \cref{def:tidyEvalSeries}, we consider the evaluations of syntax trees of~$\Syntax[\Operations_k]$ in~$\FQSym_k$ to manipulate this messy series $k$-citelangis algebra.
See \cref{fig:LinExtEvalPosetSeries}.

\begin{definition}
\label{def:messyEvalSeries}
Denote by~$\messyEvalPermSeries(\tree ; \sigma_1, \dots, \sigma_p)$ the evaluation of a syntax tree~$\tree \in \Syntax[\Operations_k]$ of arity~$p$ on $p$ $k$-permutations~${\sigma_1, \dots, \sigma_p}$ of~$\Perm_k$ using the messy series $k$-citelangis structure of~$\FQSym_k$. 
The \defn{messy series permutation evaluation} of~$\tree$ is then~$\messyEvalPermSeries(\tree) \eqdef \messyEvalPermSeries(\tree ; 1^{\{k\}}, \dots, 1^{\{k\}})$.
We extend by linearity $\tree$ to the elements of~$\Free[\Operations_k]$ on the one hand, and $\sigma_1, \dots, \sigma_p$ to the elements of~$\FQSym_k$ on the other hand.
\end{definition}

\begin{remark}
\label{rem:messyEvalSeries}
Let us rephrase algorithmically \cref{def:messyEvalSeries}.
For this, we generalize the evaluation to partial syntax trees as in \cref{rem:tidyEvalSeries}, where nodes are labeled by operations with at most~$k$ letters~${\{\op{l}, \op{r}\}}$ and leaves are labeled by any words.
The messy series permutation evaluation~$\messyEvalPermSeries(\tree[s])$ of such a partial syntax tree~$\tree[s]$ is defined inductively by
\begin{itemize}
\item $\messyEvalPermSeries(\tree[s]) = \ind{w}$ if $\tree[s]$ is a leaf labeled by the word~$\ind{w}$,
\item $\messyEvalPermSeries(\tree[s]) = \ind{w} \cdot \big( \messyEvalPermSeries(\tree[l]) \shuffle \messyEvalPermSeries(\tree[r]) \big)$ if $\tree[s]$ has a root with $\ell$ signals, and $\tree[l]$ and~$\tree[r]$ are the two partial syntax trees left after $\ell$ cars traversed~$\tree[s]$ in series while erasing the first letters of the signals in the nodes and of the words in the leaves as described in \cref{rem:tidyEvalSeries}.
\end{itemize}
Finally, for a syntax tree~$\tree$ of arity~$p$ and $k$-permutations~$\sigma_1 \in \Perm_k(n_1), \dots, \sigma_p \in \Perm_k(n_p)$, the evaluation $\messyEvalPermSeries(\tree ; \sigma_1, \dots, \sigma_p)$ is the evaluation of the partial syntax tree~$\tree[s]$ obtained from~$\tree$ by putting the permutation~$\sigma_i[n_1 + \dots + n_{i-1}]$ at the $i$-th leaf for all~$i \in [p]$.
In particular, for~$\tidyEvalPermSeries(\tree) = \tidyEvalPermSeries(\tree ; 1^{\{k\}}, \dots, 1^{\{k\}})$, the $i$-th root is labeled by the word~$i^{\{k\}}$.
See \cref{fig:LinExtEvalPosetSeries}.
\end{remark}

\paraul{Freeness}
It turns out that the tidy and messy series permutation evaluations are related by triangularity for the lexicographic order.

\begin{definition}
For a homogeneous element~$F \in \FQSym_k$, we denote by~$\LexMin(F)$ the lexicographic minimal $k$-permutation with a non-zero coefficient in~$F$.
\end{definition}

\begin{lemma}
\label{lem:LexMinSeries}
For any~$\tree \in \Syntax[\Operations_k](p)$ and~$\sigma_1, \dots, \sigma_p \in \Perm_k$, we have
\[
\tidyEvalPermSeries(\tree ; \sigma_1, \dots, \sigma_p) = \LexMin \big( \messyEvalPermSeries(\tree ; \sigma_1, \dots, \sigma_p) \big).
\]
In particular, $\tidyEvalPermSeries(\tree) = \LexMin \big( \messyEvalPermSeries(\tree) \big)$.
\end{lemma}

\begin{proof}
The proof works by induction using the descriptions of~$\tidyEvalPermSeries$ in \cref{rem:tidyEvalSeries} and of~$\messyEvalPermSeries$ in \cref{rem:messyEvalSeries}.
Indeed, for any partial syntax tree~$\tree[s]$, we have
\begin{itemize}
\item if $\tree[s]$ is a leaf labeled by the word~$\ind{w}$, then
\(
\tidyEvalPermSeries(\tree[s]) = \ind{w} = \messyEvalPermSeries(\tree[s]),
\)
\item if $\tree[s]$ has a root with $\ell$ signals, and $\tree[l]$ and~$\tree[r]$ are the two partial syntax trees left after $\ell$ cars traversed~$\tree[s]$ in series while erasing the first letters of the signals in the nodes and of the words in the leaves as described in \cref{rem:tidyEvalSeries}, then
\begin{align*}
\tidyEvalPermSeries(\tree[s])
& = \ind{w} \cdot \tidyEvalPermSeries(\tree[l]) \cdot \tidyEvalPermSeries(\tree[r])
  = \ind{w} \cdot \LexMin \big( \messyEvalPermSeries(\tree[l]) \big) \cdot \LexMin \big( \messyEvalPermSeries(\tree[r]) \big) \\
& = \LexMin \big( \ind{w} \cdot \big( \messyEvalPermSeries(\tree[l]) \shuffle \messyEvalPermSeries(\tree[r]) \big) \big)
  = \LexMin \big( \messyEvalPermSeries(\tree[s]) \big).
\end{align*}
\end{itemize}
The result then follows by application on the syntax tree~$\tree[s]$ obtained from~$\tree$ by putting the permutation~$\sigma_i[n_1 + \dots + n_{i-1}]$ at the $i$-th leaf for all~$i \in [p]$.
\end{proof}

\begin{remark}
Fix~$\tree \in \Syntax[\Operations_k](p)$.
Observe that denoting $\le$ the lexicographic order on \mbox{$k$-per}\-mutations, we have~$\sigma_1 \le \sigma'_1, \dots, \sigma_p \le \sigma'_p$ implies $\tidyEvalPermSeries(\tree ; \sigma_1, \dots, \sigma_p) \le \tidyEvalPermSeries(\tree ; \sigma'_1, \dots, \sigma'_p)$.
Hence, \cref{lem:LexMinSeries} extends to $\FQSym_k$: for any homogeneous elements~${F_1, \dots, F_p \in \FQSym_k}$,
\[
\LexMin \big( \messyEvalPermSeries(\tree ; F_1, \dots, F_p) \big) = \tidyEvalPermSeries \big( \tree;  \LexMin(F_1), \dots, \LexMin(F_p) \big).
\]
\end{remark}

We derive from \cref{lem:LexMinSeries} the main result of this section.

\begin{theorem}
\label{thm:freeMessyCitelangisSeries}
The messy series $k$-citelangis algebra~$\FQSym_k$ is free on $k$-rooted uncuttable $k$-permutations.
\end{theorem}

\begin{proof}
Consider the tidy and the messy series $k$-citelangis rewriting systems defined in \cref{subsubsec:rewritingSystemCitelangis}.
As seen in \cref{subsubsec:normalFormsCitelangis}, these two rewriting systems have the same normal forms.
Consider two such normal forms~$\tree, \tree'$ of arity~$p$ and~$p'$ respectively, and some $k$-rooted uncuttable $k$-permutations~$\sigma_1, \dots, \sigma_p, \sigma'_1, \dots, \sigma'_{p'}$, such that~$\messyEvalPermSeries(\tree ; \sigma_1, \dots, \sigma_p) = \messyEvalPermSeries(\tree' ; \sigma'_1, \dots, \sigma'_{p'})$.
It then follows from \cref{lem:LexMinSeries} that $\tidyEvalPermSeries(\tree ; \sigma_1, \dots, \sigma_p) = \tidyEvalPermSeries(\tree' ; \sigma'_1, \dots, \sigma'_{p'})$.
\cref{prop:uniqueTidySeriesCitelangisPermutationEvaluation} then implies that~$p = p'$, that ${\sigma_i = \sigma'_i}$ for all~$i \in [p]$, and that~$\tree$ and~$\tree'$ are tidy series $k$-citelangis equivalent.
As we assumed that they were in normal form, we obtain that~$\tree = \tree'$.
The result follows.
\end{proof}

\begin{remark}
By \cref{thm:freeMessyCitelangisSeries}, the messy series $k$-citelangis operad can be fully understood from the messy series $k$-citelangis subalgebra of~$\FQSym_k$ generated by the $k$-permutation~$1^{\{k\}}$.
However, while the fully $k$-rooted cuttable $k$-permutations still provide a combinatorial model for the basis of this messy series $k$-citelangis algebra similarly to \cref{prop:combinatorialModelTidySeriesCitelangis}, the compositions~$\messySeriesCirc{i}$ of the messy series $k$-citelangis operad~$\messyCitelangisSeries$ on fully $k$-rooted cuttable $k$-permutations are more intricate than \cref{prop:combinatorialModelCompositionsTidySeriesCitelangis}.
\end{remark}


\subsubsection{Operations of~\texorpdfstring{$\messyCitelangisSeries$}{MCitk} on \texorpdfstring{$k$}{k}-rooted \texorpdfstring{$k$}{k}-posets}
\label{subsubsec:actionMessyParallelSignaleticRootedPosets}

In this section, we observe that the result of any messy series permutation evaluation is the sum of all linear extensions of a well-chosen $k$-poset.
This observation allows us to encode the messy series permutation evaluation by an alternative combinatorial model and to study directly on this model the action of the messy series $k$-citelangis operad~$\messyCitelangisSeries$.
It also motivates the introduction of the series $k$-poset operad that will be studied in \cref{subsubsec:seriesPosetOperad}.

\paraul{Operations on $k$-rooted $k$-posets and series poset evaluations}
We first define some operations on the following $k$-posets.

\begin{definition}
A multiposet~$\le_M$ is \defn{$k$-rooted} if it contains a chain (\ie a totally ordered) submultiposet of size~$k$ whose elements are smaller than any other element.
For~$j \le k$, we denote by~$\min_j(\le_M)$ the~$j$-th element of this chain, by~$\Root_j(\le_M)$ the chain formed by the first $j$ elements of this chain, and by~$\le_{M_{*j}}$ the submultiposet of~$\le_M$ induced by~$M_{*j} \eqdef M \ssm \Root_j(\le_M)$.
\end{definition}

\begin{definition}
\label{def:messyCitelangisSeriesActionPosets}
Consider an operation~$\operation \in \Operations_k \eqdef \{\op{l}, \op{r}\}^k$ and two $k$-rooted $k$-posets~$\le_M$ and~$\le_N$ of degrees~$m$ and~$n$ respectively.
For~$j \le k$, denote by~$\ell_j$ (resp.~$r_j$) the number of $\op{l}$ (resp.~$\op{r}$) signals among the first~$j$ signals of~$\operation$.
In particular, $\ell \eqdef \ell_k$ (resp.~$r \eqdef r_k$) is the number of $\op{l}$ (resp.~$\op{r}$) signals in~$\operation$.
We define a $k$-rooted $k$-poset ${\le_P} \eqdef {\le_M} \; \operation \; {\le_N}$~by:
\begin{itemize}
\item $P \eqdef M \sqcup N[m]$ where $N[m]$ is the $k$-set~$(m+1)^{\{k\}} \dots (m+n)^{\{k\}}$ obtained from~$N$ by shifting each element by the degree~$m$ of~$M$ as in \cref{def:shiftPoset},
\item in $\le_P$, the elements of $M$ (respectively of~$N[m]$) are ordered by~$\le_M$ (resp.~by~$\le_{N[m]}$), and the only other relations are such that~$\Root_\ell(\le_M) \, \operation \, \Root_r(\le_N)$ is a $k$-root of~$\le_P$. Formally, the comparison~$x \le_P y$ holds for~$x, y \in P$ if and only if one of the following statements holds:
	\begin{itemize}
	\item $x \in M$, $y \in M$ and $x \le_M y$, or
	\item $x \in N[m]$, $y \in N[m]$ and $x \le_{N[m]} y$, or
	\item $x = \min_{\ell_j}(\le_M)$ and~$y \in N[m]_{*r_j}$ for $j \le k$ such that~$\operation_j = {\op{l}}$, or
	\item $x \in M_{*\ell_j}$ and~$y = \min_{r_j}(\le_N[m])$ for $j \le k$ such that~$\operation_j = {\op{r}}$.
	\end{itemize}
\end{itemize}
\end{definition}

\begin{remark}
\label{rem:sumsPosetsSeries}
\cref{def:messyCitelangisSeriesActionPosets} can be conveniently rephrased in terms of ordered sums and disjoint unions of posets presented in \cref{def:disjointUnionOrderedSum}.
Indeed,
\[
{\le_M} \; \operation \; {\le_N} = \big( \Root_\ell(\le_M) \; \operation \; \Root_r(\le_N) \big) + \big( {\le_{M_{*\ell}}} \sqcup {\le_{N[m]_{*r}}} \big),
\]
where~$m = |M|$, and $\ell$ and~$r$ are the number of~$\op{l}$ and~$\op{r}$ signals in~$\operation$.
\end{remark}

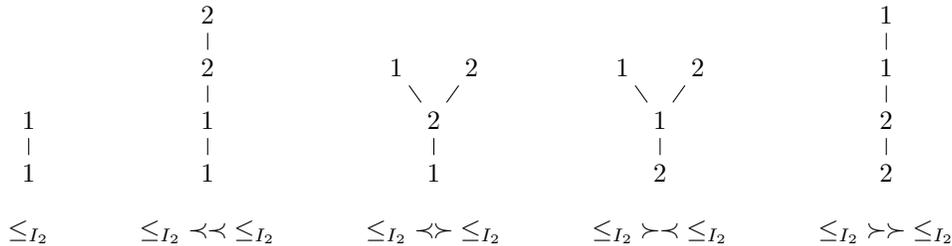
\begin{figure}[b]
	\centerline{\begin{tabular}{c@{\hspace{1.2cm}}c@{\hspace{1.2cm}}c@{\hspace{1.2cm}}c@{\hspace{1.2cm}}c}
		\begin{tikzpicture}[level/.style={sibling distance=1cm, level distance = .7cm}]
			\node {$1$} [grow' = up]
				child {node {$1$}}
			;
		\end{tikzpicture}
		&
		\begin{tikzpicture}[level/.style={sibling distance=1cm, level distance = .7cm}]
			\node {$1$} [grow' = up]
				child {node {$1$}
					child {node {$2$}
						child {node {$2$}}
					}
				}
			;
		\end{tikzpicture}
		&
		\begin{tikzpicture}[level/.style={sibling distance=1cm, level distance = .7cm}]
			\node {$1$} [grow' = up]
				child {node {$2$}
					child {node {$1$}}
					child {node {$2$}}
				}
			;
		\end{tikzpicture}
		&
		\begin{tikzpicture}[level/.style={sibling distance=1cm, level distance = .7cm}]
			\node {$2$} [grow' = up]
				child {node {$1$}
					child {node {$1$}}
					child {node {$2$}}
				}
			;
		\end{tikzpicture}
		&
		\begin{tikzpicture}[level/.style={sibling distance=1cm, level distance = .7cm}]
			\node {$2$} [grow' = up]
				child {node {$2$}
					child {node {$1$}
						child {node {$1$}}
					}
				}
			;
		\end{tikzpicture}
		\\[.2cm]
		${\le_{I_2}}$ &
		${\le_{I_2}} \op{l,l} {\le_{I_2}}$ &
		${\le_{I_2}} \op{l,r} {\le_{I_2}}$ &
		${\le_{I_2}} \op{r,l} {\le_{I_2}}$ &
		${\le_{I_2}} \op{r,r} {\le_{I_2}}$
	\end{tabular}}
	\caption{Four series operations on $2$-rooted $2$-posets. See \cref{def:messyCitelangisSeriesActionPosets}.}
	\label{fig:exmMessyCitelangisSeriesActionPosets}
\end{figure}

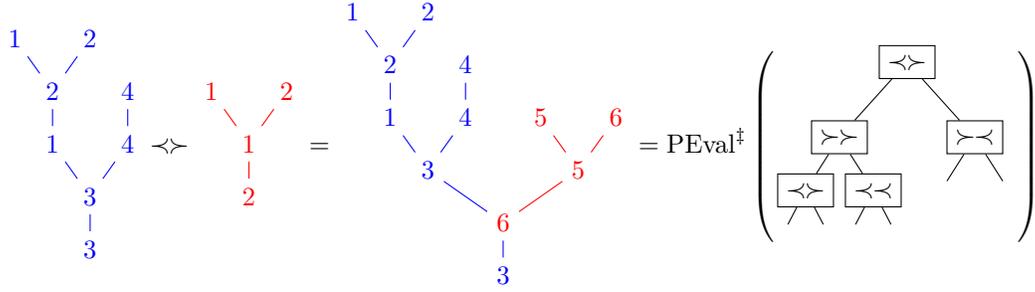
\begin{figure}[t]
	\centerline{$
	\begin{tikzpicture}[baseline={([yshift=-.8ex]current bounding box.center)}, level/.style={sibling distance=1cm, level distance = .7cm}, blue]
		\node {$3$} [grow' = up]
			child {node {$3$}
				child {node {$1$}
					child {node {$2$}
						child {node {$1$}}
						child {node {$2$}}
					}
				}
				child {node {$4$}
					child {node {$4$}}
				}
			}
		;
	\end{tikzpicture}
	\op{l,r}
	\begin{tikzpicture}[baseline={([yshift=-.8ex]current bounding box.center)}, level/.style={sibling distance=1cm, level distance = .7cm}, red]
		\node {$2$} [grow' = up]
			child {node {$1$}
				child {node {$1$}}
				child {node {$2$}}
			}
		;
	\end{tikzpicture}
	=
	\begin{tikzpicture}[baseline={([yshift=-.8ex]current bounding box.center)}, level 2/.style={sibling distance=2cm}, level/.style={sibling distance=1cm, level distance = .7cm}]
		\node {\blue $3$} [grow' = up]
			child [blueNode] {node {\red $6$}
				child [blueNode] {node {$3$}
					child {node {$1$}
						child {node {$2$}
							child {node {$1$}}
							child {node {$2$}}
						}
					}
					child {node {$4$}
						child {node {$4$}}
					}
				}
				child [redNode] {node {$5$}
					child {node {$5$}}
					child {node {$6$}}
				}
			}
		;
	\end{tikzpicture}	
	=
	\EvalPosetSeries \left(
		\begin{tikzpicture}[baseline={([yshift=-.8ex]current bounding box.center)}, level/.style={sibling distance=18mm/#1, level distance = 1cm/sqrt(#1)}]
			\node [rectangle, draw] {$\op{l,r}$}
				child {node [rectangle, draw] {$\op{r,r}$}
					child {node [rectangle, draw] {$\op{l,r}$}
						child {node {}}
						child {node {}}
					}
					child {node [rectangle, draw] {$\op{l,l}$}
						child {node {}}
						child {node {}}
					}
				}
				child {node [rectangle, draw] {$\op{r,l}$}
					child {node {}}
					child {node {}}
				}
			;
		\end{tikzpicture}
	\right)
	$}
	\caption{Another example of series operation on $2$-rooted $2$-posets. See \cref{def:messyCitelangisSeriesActionPosets}.}
	\label{fig:exmMessyCitelangisSeriesActionPosetsBis}
\end{figure}

For example, let~$\le_{I_k}$ denote the only $k$-chain on~$1^{\{k\}}$ represented in \cref{fig:exmMessyCitelangisSeriesActionPosets}\,(left).
Then for any operation~$\operation \in \Operations_k$ with~$\ell$ signals~$\op{l}$ and~$r$ signals~$\op{r}$, the $k$-poset~$\le_{I_k}$ forms a ``Y'' where the bottom branch is the chain~$1^{\{\ell\}} \, \operation \, 2^{\{r\}}$, the left branch is the chain on~$1^{\{k-\ell\}}$, and the right branch is the chain on~$2^{\{k-r\}}$.
This is illustrated when~$k = 2$ in \cref{fig:exmMessyCitelangisSeriesActionPosets}.
See \cref{fig:exmMessyCitelangisSeriesActionPosetsBis} for another example.
The next statement clearly follows from \cref{rem:sumsPosetsSeries}.

\begin{lemma}
\label{lem:operationskrootedPosets}
For any operation~$\operation \in \Operations_k$ and any two $k$-rooted $k$-posets~$\le_M$ and~$\le_N$ of degrees~$m$ and~$n$ respectively, ${\le_M \; \operation \; \le_N}$ is a $k$-rooted $k$-poset of degree~$m+n$.
\end{lemma}

This statement allows to define the evaluation of a syntax tree on posets.
See \cref{fig:exmMessyCitelangisSeriesActionPosetsBis}.

\begin{definition}
\label{def:posetSeries}
Denote by~$\EvalPosetSeries(\tree ; \le_{M_1}, \dots, \le_{M_p})$ the evaluation of a syntax tree~${\tree \in \Syntax[\Operations_k]}$ of arity~$p$ on $p$ $k$-rooted $k$-posets~${\le_{M_1}, \dots, \le_{M_p}}$ using the operations of \cref{def:messyCitelangisSeriesActionPosets}. 
The \defn{series poset evaluation} of~$\tree$ is then~$\EvalPosetSeries(\tree) \eqdef \EvalPosetSeries(\tree ; \le_{I_k}, \dots, \le_{I_k})$, where~$\le_{I_k}$ is the $k$-chain on~$1^{\{k\}}$.
\end{definition}

\begin{remark}
\label{rem:intervalLabeled}
It is not difficult to check that, for any operation~$\operation \in \Operations_k$ and any two $k$-rooted $k$-posets~$\le_M$ and~$\le_N$,
\begin{itemize}
\item if~$\le_M$ and~$\le_N$ are trees, then~${\le_M} \, \operation \, {\le_N}$ is a tree (see \cref{def:treeMultiposet}), and
\item if~$\le_M$ and~$\le_N$ are interval labelled, then~${\le_M} \, \operation \, {\le_N}$ is interval labeled (see \cref{def:intervalLabelled}).
\end{itemize}
Therefore, for any~$\tree \in \Syntax[\Operations_k]$, the series poset evaluation~$\EvalPosetSeries(\tree)$ is a $k$-rooted interval labelled tree.
We will characterize the series poset evaluations later in \cref{prop:characterizationSeriesPosetEvaluations}.
\end{remark}

\paraul{$k$-rooted cuts in $k$-posets}
We now aim at characterizing series poset evaluations.
As in \cref{subsubsec:actionTidySeriesSignaletic}, this understanding goes through $k$-rooted cuts in $k$-rooted $k$-posets.

\begin{definition}
\label{def:krootedCutPoset}
Let~$\le_M$ be a $k$-rooted $k$-poset of degree~$n$ and~$\gamma \in [n-1]$.
We say that $\gamma$ is a \defn{$k$-rooted cut} of~$\le_M$ if we can decompose~$M$ into~$M = \Root_k(\le_M) \sqcup L \sqcup R$ such that, for all~$\ell \in L$ and~$r \in R$, we have~$\ell \le \gamma < r$ and~$\ell$ and~$r$ are incomparable for~$\le_M$.
We denote by~$\kcuts(\le_M)$ the set of $k$-rooted cuts of~$\le_M$.
We say that the $k$-rooted $k$-poset~$\le_M$ is \defn{$k$-rooted cuttable} if it admits a $k$-rooted cut, and \defn{$k$-rooted uncuttable} if it admits no $k$-rooted cut.
\end{definition}

The following statements are similar to \cref{lem:krootedCutsRestriction,prop:equivalenceFullykrootedCuttable}.

\begin{lemma}
\label{lem:krootedCutsRestrictionPoset}
Consider~$L \subseteq [m]$ and~$\gamma \in [\min(L), \max(L)-1]$, and let~$\gamma^{|L} \eqdef |[\gamma] \cap L|$ denote the number of elements of~$L$ between~$1$ and~$\gamma$.
If~$\gamma$ is a $k$-rooted cut of a $k$-rooted $k$-poset~$\le_M$ of degree~$m$, then $\gamma^{|L}$ is a $k$-rooted cut of its restriction~$\le_{M^{|L}}$.
\end{lemma}

\begin{proposition}
\label{prop:equivalenceFullykrootedCuttablePoset}
The following conditions are equivalent for a $k$-rooted $k$-poset of degree~$n$:
\begin{enumerate}[(i)]
\item its restriction to any interval of~$[n]$ of size at least $2$ is $k$-rooted cuttable,
\item its restriction to any subset of~$[n]$ of size at least $2$ is $k$-rooted cuttable.
\end{enumerate}
\end{proposition}

\begin{definition}
\label{def:fullykrootedCuttablePoset}
A $k$-rooted $k$-poset is \defn{fully $k$-rooted cuttable} if it satisfies the equivalent conditions of \cref{prop:equivalenceFullykrootedCuttablePoset}.
\end{definition}

Note that the only multiposet on~$1^k$ is fully $k$-rooted cuttable as there is no subset of size at least~$2$.
The following statements are similar to \cref{lem:operationImplieskrootedCut,lem:krootedCutImpliesOperation}.

\begin{lemma}
\label{lem:operationImplieskrootedCutPoset}
For any $k$-rooted $k$-posets~$\le_M$ and~$\le_N$, and any operation~${\operation \in \Operations_k}$, the degree~$m$ of~$\le_M$ is a $k$-rooted cut of~$\le_M \, \operation \, \le_N$.
\end{lemma}

\begin{lemma}
\label{lem:krootedCutImpliesOperationPoset}
For any $k$-rooted $k$-poset~$\le_P$ of degree~$p$ and any $k$-rooted cut~$\gamma \in \kcuts(\le_P)$, there is a unique~${\operation \in \Operations_k}$ (defined by $\operation_i \eqdef {\op{l}}$ if~$\Root_i(\le_P) \le \gamma$ and~$\operation_i \eqdef {\op{r}}$ if~$\Root_i(\le_P) > \gamma$) such that~${{\le_P} = {\le_{P^{|[\gamma]}}} \, \operation \, {\le_{P^{|[p] \ssm [\gamma]}}}}$.
\end{lemma}

\begin{remark}
\label{rem:algoPosetSeries}
Similarly to \cref{rem:algoTidyEvalSeries}, observe that \cref{lem:krootedCutImpliesOperationPoset} gives an inductive algorithm to compute all decompositions of a given $k$-rooted $k$-poset~$\le_P$ as an evaluation of the form ${{\le_P} = \EvalPosetSeries(\tree ; \le_{M_1}, \dots, \le_{M_p})}$.
Namely, $\le_P$ admits
\begin{itemize}
\item the trivial evaluation~$\le_0 = \EvalPosetSeries(\one ; \le_0)$, where~$\one$ is the unit syntax tree with no node and a single leaf, and 
\item the evaluation~${\le_P} = \EvalPosetSeries(\tree ; \le_{M_1}, \dots, \le_{M_l}, \le_{N_1}, \dots, \le_{N_r})$, for any~${\gamma \in \kcuts(\le_P)}$ and any~${\le_{P^{|[\gamma]}}} = \EvalPosetSeries(\tree[l] ; \le_{M_1}, \dots, \le_{M_l})$ and ${\le_{P^{|[\ell] \ssm [\gamma]}}} = \EvalPosetSeries(\tree[r] ; \le_{N_1}, \dots, \le_{N_r})$, where~$\tree$ is the syntax tree with root~$\operation \in \Operations_k$ defined by \cref{lem:krootedCutImpliesOperationPoset} and with subtrees~$\tree[l]$ and~$\tree[r]$.
\end{itemize}
This algorithm implies the existence of decompositions of the form~${\le_P} = \EvalPosetSeries(\tree ; \le_{M_1}, \dots, \le_{M_p})$ where~$\le_{M_1}, \dots, \le_{M_p}$ are $k$-rooted uncuttable.
\end{remark}

Note that not all $k$-rooted $k$-posets are obtained by evaluating syntax trees on~$\Operations_k$.
For instance, the evaluation the syntax trees on~$\Operations_2$ of arity~$2$ only produces the four $2$-rooted $2$-posets of \cref{fig:exmMessyCitelangisSeriesActionPosets}, thus only two of the six linear $2$-rooted $2$-posets of degree~$2$.
We now characterize the $k$-rooted $k$-posets which are series poset evaluations of syntax trees.

\begin{proposition}
\label{prop:characterizationSeriesPosetEvaluations}
The series poset evaluations of the syntax trees of~$\Syntax[\Operations_k]$ are precisely the fully $k$-rooted cuttable $k$-posets.
\end{proposition}

\begin{proof}
The proof is exactly the same as \cref{prop:characterizationTidySeriesPermutationEvaluations}.
Consider first a series poset evaluation~${\le_P} = \EvalPosetSeries(\tree)$ with~$\tree \in \Syntax[\Operations_k](\ell)$.
We prove by induction on~$\ell$ that~$\le_P$ is fully $k$-rooted cuttable.
If~$\ell = 1$, there is nothing to prove.
Assume that~$\ell \ge 2$ and let~$1 \le a < b \le \ell$.
Let~$\tree[l]$ and~$\tree[r]$ denote the left and right subtrees of~$\tree$, and let~$\gamma$ be the arity of~$\tree[l]$, so that~${\le_{P^{|[\gamma]}}} = \EvalPosetSeries(\tree[l])$ and~$\le_{M^{|[\ell] \ssm [\gamma]}} = \EvalPosetSeries(\tree[r])$.
We distinguish three cases:
\begin{itemize}
\item Assume that~$b \le \gamma$. Since~$\le_{M^{|[\ell]}} = \EvalPosetSeries(\tree[l])$ is fully $k$-rooted cuttable by induction hypothesis and $k$-rooted cuts are preserved by restriction by \cref{lem:krootedCutsRestrictionPoset}, we obtain that~${\le_{P^{|[a,b]}}} = (\le_{P^{|[\gamma]}})^{|[a,b]}$ is $k$-rooted cuttable.
\item Assume that~$\gamma \le a$. The argument is similar since~${\le_{P^{|[a,b]}}} = (\le_{P^{|[\ell] \ssm [\gamma]}})^{|[a-\gamma,b-\gamma]}$.
\item Assume finally that~$a < \gamma < b$. By \cref{lem:operationImplieskrootedCutPoset}, $\gamma$ is a $k$-rooted cut of~$\le_P$. Therefore, $\gamma-a$ is a $k$-rooted cut of~$\le_{P^{|[a,b]}}$ by \cref{lem:krootedCutsRestrictionPoset}.
\end{itemize}

Conversely, consider now a fully $k$-rooted cuttable $k$-poset~$\le_P$ of degree~$\ell$.
Similarly to \cref{rem:algoPosetSeries}, we prove by induction on~$\ell$ that~$\le_P$ is the series poset evaluation of a syntax tree.
If~$\ell = 1$, then~${\le_P} = \EvalPosetSeries(\one)$.
If~$\ell \ge 2$, then~$\le_P$ admits at least one $k$-rooted cut~$\gamma$ by assumption.
Moreover, $\le_{P^{|[\gamma]}}$ and~$\le_{P^{|[\ell] \ssm [\gamma]}}$ are both fully $k$-rooted cuttable by \cref{lem:krootedCutsRestrictionPoset}.
By induction, we obtain that~${\le_{P^{|[\gamma]}}} = \EvalPosetSeries(\tree[l])$ and~${\le_{P^{|[\ell] \ssm [\gamma]}}} = \EvalPosetSeries(\tree[r])$.
Then~${\le_P} = \EvalPosetSeries(\tree)$, where~$\tree$ is the syntax tree with root~$\operation \in \Operations_k$ defined by \cref{lem:krootedCutImpliesOperation} and with subtrees~$\tree[l]$ and~$\tree[r]$.
\end{proof}

\paraul{Connections between~$\tidyEvalPermSeries$, $\messyEvalPermSeries$ and~$\EvalPosetSeries$}
We have seen in \cref{lem:LexMinSeries} that $\tidyEvalPermSeries(\tree) = \LexMin \big( \messyEvalPermSeries(\tree) \big)$ for any syntax tree~$\tree \in \Syntax[\Operations_k]$.
We now compare with the series poset evaluation~$\EvalPosetSeries(\tree)$.

First, as in \cref{subsubsec:actionMessyParallelSignaleticBoundedPosets}, we get back the messy series permutation evaluation by projecting posets to the sum of their linear extensions.
Recall that for a $k$-poset~$\le_M$, we let
\[
\LinExt(\le_M) \eqdef \sum_{\mu \in \linearExtensions(\le_M)} \mu
\]
be the sum of all linear extensions of~$\le_M$ (meaning the formal sum in~$\K\Perm_2(|M|)$).
We obtain the following statement, illustrated in \cref{fig:LinExtEvalPosetSeries}.

\begin{lemma}
\label{lem:LinExtEvalPosetSeries}
For any syntax tree~$\tree \in \Syntax[\Operations_k]$, we have~$\messyEvalPermSeries(\tree) = \LinExt \big( \EvalPosetSeries(\tree) \big)$.
\end{lemma}

\begin{proof}
\cref{rem:sumsPosetsSeries,lem:disjointUnionOrderedSumLinearExtensions} prove that
\[
\LinExt(\le_M) \; \operation \, \LinExt(\le_N) = \LinExt({\le_M} \; \operation \; {\le_N})
\]
for any operation~$\operation\in\Operations_k$ and any $k$-rooted $k$-posets~$\le_M$ and~$\le_N$.
The statement then follows by induction on syntax trees.
\end{proof}

\begin{figure}[t]
	\centerline{$
	\messyEvalPermSeries \left(
		\begin{tikzpicture}[baseline={([yshift=-.8ex]current bounding box.center)}, level/.style={sibling distance=18mm/#1, level distance = 1cm/sqrt(#1)}]
			\node [rectangle, draw] {$\op{l,r}$}
				child {node [rectangle, draw] {$\op{r,r}$}
					child {node [rectangle, draw] {$\op{l,r}$}
						child {node {}}
						child {node {}}
					}
					child {node [rectangle, draw] {$\op{l,l}$}
						child {node {}}
						child {node {}}
					}
				}
				child {node [rectangle, draw] {$\op{r,l}$}
					child {node {}}
					child {node {}}
				}
			;
		\end{tikzpicture}
	\right)
	=
	\LinExt \left(
		\begin{tikzpicture}[baseline={([yshift=-.8ex]current bounding box.center)}, level 2/.style={sibling distance=2cm}, level/.style={sibling distance=1cm, level distance = .7cm}]
			\node {$3$} [grow' = up]
				child {node {$6$}
					child {node {$3$}
						child {node {$1$}
							child {node {$2$}
								child {node {$1$}}
								child {node {$2$}}
							}
						}
						child {node {$4$}
							child {node {$4$}}
						}
					}
					child {node {$5$}
						child {node {$5$}}
						child {node {$6$}}
					}
				}
			;
		\end{tikzpicture}	
	\right)
	=
	\begin{array}{c@{\;}c@{\;}c@{\;}c@{\;}c@{\;}c}
		363121244556 & + & 363121244565 & + \\
		363121245456 & + & 363121245465 & + \\
		363121245546 & + & \dots\dots       \\
		\ldots\ldots & + & 365653441221
	\end{array}
	$}
	\caption{Illustration of \cref{lem:LinExtEvalPosetSeries}.}
	\label{fig:LinExtEvalPosetSeries}
\end{figure}

Conversely, we can obtain~$\EvalPosetSeries(\tree)$ from $\messyEvalPermSeries(\tree)$ by intersecting all linear orders corresponding to the $k$-permutations that appear in~$\messyEvalPermSeries(\tree)$.

\medskip
We now describe the connection between $\tidyEvalPermSeries(\tree)$ and~$\EvalPosetSeries(\tree)$.
By \cref{lem:LinExtEvalPosetSeries,lem:LexMinSeries}, the tidy series permutation evaluation~$\tidyEvalPermSeries(\tree)$ is the lexicographically minimal linear extension of~$\EvalPosetSeries(\tree)$.
It can thus be obtained from~$\EvalPosetSeries(\tree)$ by repeatedly reading and deleting the minimal source.

Conversely, there are several syntax trees~$\tree$ whose series poset evaluation~$\EvalPosetSeries(\tree)$ has the same lexicographically minimal linear extension.
Given a $k$-rooted cuttable $k$-permutation~$\sigma$, we can find all $k$-posets~$\EvalPosetSeries(\tree)$ whose lexicographically minimal linear extension is~$\sigma$ by first using \cref{rem:tidyEvalSeries} to compute all possible syntax trees~$\tree$ such that~$\sigma = \tidyEvalPermSeries(\tree)$, and then applying \cref{def:messyCitelangisSeriesActionPosets,def:posetSeries} to get their series poset evaluations~$\EvalPosetSeries(\tree)$.

\paraul{Relations among series poset evaluations}
We now study which syntax trees evaluate to the same $k$-rooted $k$-poset.
Observe first the following quadratic relation:

\begin{equation}
\label{eq:relationPosetsSeries}
\EvalPosetSeries \left( 
	\begin{tikzpicture}[baseline=-.5cm, level 1/.style={sibling distance = 1cm, level distance = .7cm}, level 2/.style={sibling distance = .8cm, level distance = .6cm}]
		\node [rectangle, draw] {$\op{l}\!\!^k$}
			child {node [rectangle, draw] {$\op{r}\!\!^k$}
				child {node {}}
				child {node {}}
			}
			child {node {}}
		;
	\end{tikzpicture}
\right)
=
\begin{tikzpicture}[baseline=1cm, level/.style={sibling distance=1cm, level distance = .7cm}]
	\node {$2$} [grow' = up]
		child [dotted, thick] {node {$2$}
			child [solid] {node {$1$}
				child [dotted, thick] {node {$1$}}
			}
			child [solid] {node {$3$}
				child [dotted, thick] {node {$3$}}
			}
		}
	;
\end{tikzpicture}
=
\EvalPosetSeries \left(
	\begin{tikzpicture}[baseline=-.5cm, level 1/.style={sibling distance = 1cm, level distance = .7cm}, level 2/.style={sibling distance = .8cm, level distance = .6cm}]
		\node [rectangle, draw] {$\op{r}\!\!^k$}
			child {node {}}
			child {node [rectangle, draw] {$\op{l}\!\!^k$}
				child {node {}}
				child {node {}}
			}
		;
	\end{tikzpicture}
\right).
\end{equation}

\warning
The operations of \cref{def:messyCitelangisSeriesActionPosets} on $k$-rooted $k$-posets do not satisfy the other messy series $k$-citelangis relations, and thus do not define a messy series $k$-citelangis algebra.
In fact, we will now show that \cref{eq:relationPosetsSeries} is the only relation that leads to the same series poset evaluations.

\begin{lemma}
\label{lem:uniqueRelationSeriesPosetOperad}
For any operations~$\operation, \operation' \in \Operations_k$ and any $k$-rooted $k$-posets~${\le_M}, {\le_N}, {\le_{M'}}, {\le_{N'}}$, we have~${\le_M} \; \operation \; {\le_N} = {\le_{M'}} \; \operation' \; {\le_{N'}}$ if and only if 
\begin{itemize}
\item either~$\operation = \operation'$, ${\le_M} = {\le_{M'}}$ and ${\le_N} = {\le_{N'}}$, 
\item or~$\operation = {\op{l}\!^k}$ and~$\operation' = {\op{r}\!^k}$, and there exists a $k$-rooted $k$-poset~$\le_O$ such that ${{\le_M} = {\le_{M'}} \op{r}\!^k \, {\le_O}}$ and~${{\le_{N'}} = {\le_O} \op{l}\!^k \, {\le_N}}$.
\end{itemize}
\end{lemma}

\begin{proof}
The ``if'' direction is immediate as the second case was already observed in \cref{eq:relationPosetsSeries}.

For the ``only if'' direction, consider a $k$-rooted $k$-poset~${\le_P} \eqdef {\le_M} \; \operation \; {\le_N}$ obtained from an operation~$\operation \in \Operations_k$ and two $k$-rooted $k$-posets~$\le_M$ and~$\le_N$, and let~$p$ be the degree of~$\le_P$.

When~$\operation \notin \{\op{l}\!^k, \op{r}\!^k\}$, the element~$\min_k(\le_P)$ is covered by precisely two elements, and all values of~$\le_P$ appear outside of its $k$-root.
Therefore, $\le_P$ has a unique $k$-rooted cut~$\gamma$, so that~$\operation$, $\le_M$ and~$\le_N$ are all determined by \cref{lem:operationImplieskrootedCutPoset,lem:krootedCutImpliesOperationPoset}.

Observe now that when~$\operation = {\op{l}\!^k}$, we have~${\le_M} = {\le_{P^{|[p] \ssm R}}}$ and~${\le_N} = {\le_{P^{|R}}}$ where~$R$ be the upper order ideal of~$\le_P$ generated by the maximal value covering~$\min_k(\le_P)$.
Similarly, when~$\operation = {\op{r}\!^k}$, we have~${\le_M} = {\le_{P^{|L}}}$ and~${\le_N} = {\le_{P^{|[p] \ssm L}}}$ where~$L$ be the upper order ideal of~$\le_P$ generated by the minimal value covering~$\min_k(\le_P)$.
Finally, these two situations are simultaneously possible only in the case when~${{\le_M} = {\le_{M'}} \op{r}\!^k \, {\le_O}}$ and~${{\le_{N'}} = {\le_O} \op{l}\!^k \, {\le_N}}$ where~$\le_{M'}$, $\le_{N'}$ and~$\le_O$ are the sub-$k$-posets of~$\le_P$ respectively induced by~$L$, $R$ and~$P \ssm (L \cup R)$.
\end{proof}

\begin{remark}
\cref{lem:uniqueRelationSeriesPosetOperad} fails for multiposets, as we used that the multiplicity of each value in a $k$-poset is at least the cardinality~$k$ of the root chain.
\end{remark}

Finally, the next statement is similar to \cref{prop:uniqueTidySeriesCitelangisPermutationEvaluation}.

\begin{proposition}
\label{prop:uniqueSeriesPosetEvaluation}
For any syntax trees~$\tree , \tree' \in \Syntax[\Operations_k]$ of arity~$p$ and~$p'$ respectively, and any $k$-rooted uncuttable $k$-posets~$\le_{M_1}, \dots, \le_{M_p}$, $\le_{M'_1}, \dots, \le_{M'_{p'}}$, if
\[
\tidyEvalPermSeries(\tree ; \le_{M_1}, \dots, \le_{M_p}) = \tidyEvalPermSeries(\tree' ; \le_{M'_1}, \dots, \le_{M'_{p'}}),
\]
then $p = p'$, $\tree = \tree'$ modulo rewritings using \cref{eq:relationPosetsSeries} and~$\le_{M_i} = \le_{M'_i}$ for all~$i \in [p]$.
\end{proposition}

\begin{proof}
Immediate by induction from \cref{lem:uniqueRelationSeriesPosetOperad}.
\end{proof}

\paraul{Series poset evaluations of series $k$-citelangis normal forms}
As the operations of \cref{def:messyCitelangisSeriesActionPosets} on $k$-rooted $k$-posets do not satisfy all the messy series $k$-citelangis relations, the series poset evaluations of all syntax trees result to too many different $k$-posets.
To obtain an alternative combinatorial model for the messy series $k$-citelangis operad, we thus need to restrict our attention to series poset evaluations of normal forms of the messy series $k$-citelangis rewriting system described in \cref{subsubsec:rewritingSystemCitelangis}.

\begin{definition}
\label{def:normalPosetSeries}
A \defn{series normal} $k$-poset is a $k$-poset that can be obtained as the series poset evaluation of a normal form of the messy series $k$-citelangis rewriting system.
\end{definition}

In particular, by \cref{rem:intervalLabeled,prop:characterizationSeriesPosetEvaluations}, a series normal $k$-poset is an interval labeled and fully $k$-rooted cuttable tree.

\begin{example}
\label{exm:willowPosets}
When~$k = 1$, the series normal $1$-posets are just interval labeled trees such that:
\begin{itemize}
\item any node has at most one smaller child, and
\item if a node is larger than its parent, then it is larger than all its children.
\end{itemize}
We call them \defn{willow posets}.
These posets are illustrated in \cref{fig:TamariLatticeWillowPosets} for~$n = 4$.
In each poset, we have colored the increasing edges in red and the decreasing ones in blue.
The conditions defining willow posets translate to the following color rule: blue edges cannot appear in parallel (no node appears as the bottom of two blue covering relations) while red edges cannot appear in series (no node appears both at the bottom and at the top of red covering relations).

\begin{figure}[t]
	\centerline{
    \begin{tikzpicture}[scale=2, inner sep=2pt]
    	\node (a) at (4.2,1) {
    		\begin{tikzpicture}[level/.style={sibling distance=.6cm, level distance = .7cm}]
    			\node {$4$} [grow' = up]
    				child [blueEdge] {node {$3$}
    					child [blueEdge] {node {$2$}
    						child [blueEdge] {node {$1$}}
    					}
    				}
    			;
    		\end{tikzpicture}	
    	};
    	\node (b) at (3,2) {
            \begin{tikzpicture}[level/.style={sibling distance=.6cm, level distance = .7cm}]
            	\node {$4$} [grow' = up]
            		child [blueEdge] {node {$3$}
            			child [blueEdge] {node {$1$}
            				child [redEdge] {node {$2$}}
            			}
            		}
            	;
            \end{tikzpicture}	
    	};
    	\node (c) at (2,3) {
            \begin{tikzpicture}[level/.style={sibling distance=.6cm, level distance = .7cm}]
            	\node {$4$} [grow' = up]
            		child [blueEdge] {node {$1$}
            			child [redEdge] {node {$3$}
            				child [blueEdge] {node {$2$}}
            			}
            		}
            	;
            \end{tikzpicture}	
    	};
    	\node (d) at (1,4) {
            \begin{tikzpicture}[level/.style={sibling distance=.6cm, level distance = .7cm}]
            	\node {$4$} [grow' = up]
            		child [blueEdge] {node {$1$}
            			child [redEdge] {node {$2$}}
            			child [redEdge] {node {$3$}}
            		}
            	;
            \end{tikzpicture}	
    	};
    	\node (e) at (2,5) {
            \begin{tikzpicture}[level/.style={sibling distance=.6cm, level distance = .7cm}]
            	\node {$1$} [grow' = up]
            		child [redEdge] {node {$4$}
            			child [blueEdge] {node {$2$}
            				child [redEdge] {node {$3$}}
            			}
            		}
            	;
            \end{tikzpicture}	
    	};
    	\node (f) at (3,6) {
            \begin{tikzpicture}[level/.style={sibling distance=.6cm, level distance = .7cm}]
            	\node {$1$} [grow' = up]
            		child [redEdge] {node {$2$}}
            		child [redEdge] {node {$4$}
            			child [blueEdge] {node {$3$}}
            		}
            	;
            \end{tikzpicture}	
    	};
    	\node (g) at (4.2,7) {
            \begin{tikzpicture}[level/.style={sibling distance=.6cm, level distance = .7cm}]
            	\node {$1$} [grow' = up]
            		child [redEdge] {node {$2$}}
            		child [redEdge] {node {$3$}}
            		child [redEdge] {node {$4$}}
            	;
            \end{tikzpicture}	
    	};
    	\node (h) at (3.5,3.1) {
            \begin{tikzpicture}[level/.style={sibling distance=.6cm, level distance = .7cm}]
            	\node {$4$} [grow' = up]
            		child [blueEdge] {node {$2$}
            			child [blueEdge] {node {$1$}}
            			child [redEdge] {node {$3$}}
            		}
            	;
            \end{tikzpicture}	
    	};
    	\node (i) at (3,4) {
            \begin{tikzpicture}[level/.style={sibling distance=.6cm, level distance = .7cm}]
            	\node {$1$} [grow' = up]
            		child [redEdge] {node {$4$}
            			child [blueEdge] {node {$3$}
            				child [blueEdge] {node {$2$}}
            			}
            		}
            	;
            \end{tikzpicture}	
    	};
    	\node (j) at (4.2,4) {
            \begin{tikzpicture}[level/.style={sibling distance=.6cm, level distance = .7cm}]
            	\node {$2$} [grow' = up]
            		child [blueEdge] {node {$1$}}
            		child [redEdge] {node {$4$}
            			child [blueEdge] {node {$3$}}
            		}
            	;
            \end{tikzpicture}	
    	};
    	\node (k) at (4.2,5.5) {
            \begin{tikzpicture}[level/.style={sibling distance=.6cm, level distance = .7cm}]
            	\node {$1$} [grow' = up]
            		child [redEdge] {node {$3$}
            			child [blueEdge] {node {$2$}}
            		}
            		child [redEdge] {node {$4$}}
            	;
            \end{tikzpicture}	
    	};
    	\node (l) at (5.4,4) {
            \begin{tikzpicture}[level/.style={sibling distance=.6cm, level distance = .7cm}]
            	\node {$3$} [grow' = up]
            		child [blueEdge] {node {$1$}
            			child [redEdge] {node {$2$}}
            		}
            		child [redEdge] {node {$4$}}
            	;
            \end{tikzpicture}	
    	};
    	\node (m) at (6.6,3) {
            \begin{tikzpicture}[level/.style={sibling distance=.6cm, level distance = .7cm}]
            	\node {$3$} [grow' = up]
            		child [blueEdge] {node {$2$}
            			child [blueEdge] {node {$1$}}
            		}
            		child [redEdge] {node {$4$}}
            	;
            \end{tikzpicture}	
    	};
    	\node (n) at (6,6) {
            \begin{tikzpicture}[level/.style={sibling distance=.6cm, level distance = .7cm}]
            	\node {$2$} [grow' = up]
            		child [blueEdge] {node {$1$}}
            		child [redEdge] {node {$3$}}
            		child [redEdge] {node {$4$}}
            	;
            \end{tikzpicture}	
    	};
    	\draw[-, ultra thick] (a) -- (b);
    	\draw[-, ultra thick] (a) -- (h);
    	\draw[-, ultra thick] (a) -- (m);
    	\draw[-, ultra thick] (b) -- (c);
    	\draw[-, ultra thick] (b) -- (l);
    	\draw[-, ultra thick] (c) -- (d);
    	\draw[-, ultra thick] (c) -- (i);
    	\draw[-, ultra thick] (d) -- (e);
    	\draw[-, ultra thick] (e) -- (f);
    	\draw[-, ultra thick] (f) -- (g);
    	\draw[-, ultra thick] (h) -- (d);
    	\draw[-, ultra thick] (h) -- (j);
    	\draw[-, ultra thick] (i) -- (e);
    	\draw[-, ultra thick] (i) -- (k);
    	\draw[-, ultra thick] (j) -- (f);
    	\draw[-, ultra thick] (j) -- (n);
    	\draw[-, ultra thick] (k) -- (g);
    	\draw[-, ultra thick] (l) -- (k);
    	\draw[-, ultra thick] (m) -- (l);
    	\draw[-, ultra thick] (m) -- (n);
    	\draw[-, ultra thick] (n) -- (g);
    \end{tikzpicture}
    }
	\caption{The Tamari lattice on willow posets, defined in \cref{exm:willowPosets}. Increasing edges are colored red, while decreasing ones are colored blue, and there are no two blue edges in parallel nor two red edges in series. Compare to \cref{fig:TamariLatticeBinaryTrees} for an illustration of the bijection between binary trees and willow posets.}
	\label{fig:TamariLatticeWillowPosets}
\end{figure}
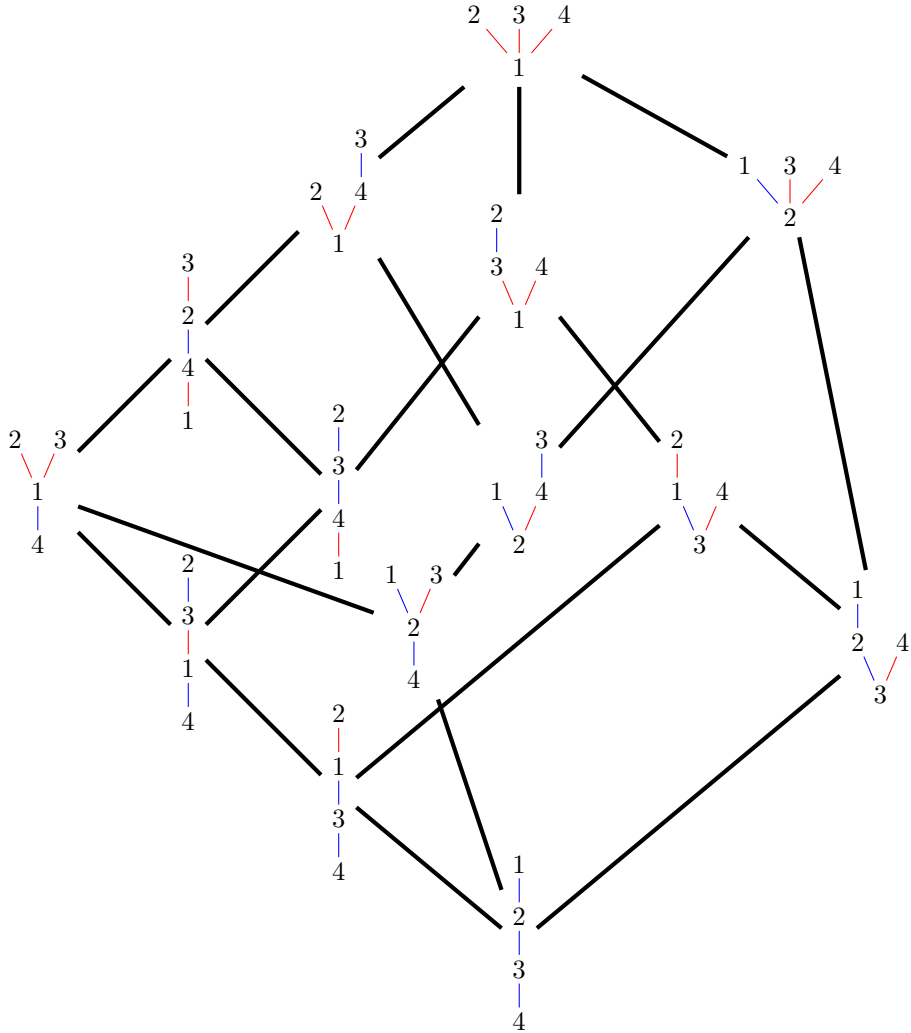
\end{example}

There is an obvious bijection sending a binary tree~$t$ to a willow poset~$w(t)$.
Given a binary tree~$t$, we label its nodes using the infix labeling where all labels in the left (resp.~right) subtree of the $i$-th node are smaller (resp.~larger) than~$i$.
Given a node~$x$ in~$t$, we define:
\begin{itemize}
\item the \defn{right origin} of~$x$: a node and its right child have the same right origin, and a node which is not a right child is its own right origin.
\item the \defn{right outgrowth} of~$x$: all nodes with right origin~$x$, except~$x$ itself.
\end{itemize}
We define the tree~$w(t)$ where the children of a node are its left child in~$t$ (if any) together with its right outgrowth in~$t$.
We claim that $w(t)$ is a willow poset:
\begin{itemize}
\item for any node, its only smaller child in~$w(t)$ is its left child in~$t$, and
\item if a node is larger than its parent, then its parent is its right origin, so that this node has no right outgrowth.
\end{itemize}
The reader can compare \cref{fig:TamariLatticeBinaryTrees,fig:TamariLatticeWillowPosets} for an illustration of this bijection.

In contrast to the $k = 1$ case discussed in \cref{exm:willowPosets}, there seems to be no simple characterization of the series normal $k$-posets for~$k > 1$.
There are still some necessary conditions on the increasing and decreasing edges, but they are more difficult to express.

There is however a simple algorithm to decide whether a given fully $k$-rooted cuttable $k$-poset~$\le_P$ is series normal.
Namely, compute its lexicographically minimal linear extension~$\sigma$, then the normal form~$\tree$ of the tidy series $k$-citelangis rewriting system such that~$\sigma = \tidyEvalPermSeries(\tree)$ by \cref{rem:algoTidyEvalSeriesNormalForms}, and check whether~${\le_P} = \EvalPosetSeries(\tree)$.

Finally, one can also consider the list of all series poset evaluations of the $3^k-1$ quadratic series $k$-signaletic combs distinct from that of \cref{eq:relationPosetsSeries}.
A fully $k$-rooted cuttable $k$-poset is series normal if it does not contain any of these forbidden $k$-posets as patterns.
The meaning of patterns in $k$-rooted $k$-posets should be clear with the poset series compositions defined in \cref{def:seriesRootedPosetOperad}.
Here, we prefer to illustrate intuitively the notion on an example:

\centerline{$
	\EvalPosetSeries \left(
		\begin{tikzpicture}[baseline=-.5cm, level 1/.style={sibling distance = .8cm, level distance = .6cm}, level 2/.style={sibling distance = .8cm, level distance = .6cm}]
			\node [rectangle, draw, fill=red!50] {$\op{r,l}$}
				child {node {\blue $11$}}
				child {node [rectangle, draw, fill=red!50] {$\op{r,l}$}
					child {node {\cyan $22$}}
					child {node {\green $33$}}
				}
			;
		\end{tikzpicture}
	\right)
	=
	\begin{tikzpicture}[baseline={([yshift=-2ex]current bounding box.center)}, level/.style={sibling distance=.7cm, level distance = .7cm}]
		\node {\green $3$} [grow' = up]
			child [edge from parent/.style={green, draw}] {node {\blue $1$}
				child [blueNode] {node {$1$}}
				child {node {\cyan $2$}
					child [cyanNode] {node {$2$}}
					child [greenNode] {node {$3$}}
				}
			}
		;
	\end{tikzpicture}	
	\text{ is a pattern in }
	\EvalPosetSeries \left(
		\begin{tikzpicture}[baseline={([yshift=-.8ex]current bounding box.center)}, level/.style={sibling distance=1cm, level distance = .6cm}, level 3/.style={sibling distance=1.5cm, level distance = .6cm}]
			\node [rectangle, draw] {$\op{r,r}$}
				child {node {}}
				child {node [rectangle, draw] {$\op{r,l}$}
					child {node [rectangle, draw, fill=red!50] {$\op{r,l}$}
						child {node [rectangle, draw, fill=blue!50] {$\op{l,r}$}
							child {node {}}
							child {node {}}
						}
						child {node [rectangle, draw, fill=red!50] {$\op{r,l}$}
							child {node [rectangle, draw, fill=cyan!50] {$\op{r,l}$}
								child {node {}}
								child {node {}}
							}
							child {node [rectangle, draw, fill=green!50] {$\op{r,r}$}
								child {node {}}
								child {node {}}
							}
						}
					}
					child {node {}}
				}
			;
		\end{tikzpicture}
	\right)
	=
	\begin{tikzpicture}[baseline={([yshift=-2ex]current bounding box.center)}, level/.style={sibling distance=1.5cm, level distance = .7cm}]
		\node {$8$} [grow' = up]
			child {node {\green $7$}
				child [edge from parent/.style={green, draw}] {node {$1$}
					child [edge from parent/.style={black, draw}] {node {$1$}}
					child {node {\blue $2$}
						child [blueNode, sibling distance=2cm] {node {$3$}
							child [sibling distance=.8cm] {node {$2$}}
							child [sibling distance=.8cm] {node {$3$}}
						}
						child [sibling distance=2cm] {node {\cyan $5$}
							child [cyanNode, sibling distance=1cm] {node {$4$}
								child [sibling distance=.8cm] {node {$4$}}
								child [sibling distance=.8cm] {node {$5$}}
							}
							child [greenNode, sibling distance=1cm] {node {$7$}
								child {node {$6$}
									child {node {$6$}}
								}
							}
						}
					}
				}
				child {node {$8$}}
			}
		;
	\end{tikzpicture}	
$}

Note that for~$k = 1$, the forbidden posets are
\[
	\EvalPosetSeries \left( \compoR{l}{l} \right)
	=
	\begin{tikzpicture}[baseline=.6cm, level/.style={sibling distance=.5cm, level distance = .7cm}]
		\node {$1$} [grow' = up]
			child [redEdge] {node {$2$}
				child {node {$3$}
				}
			}
		;
	\end{tikzpicture}	
	\qqandqq
	\EvalPosetSeries \left( \compoR{r}{r} \right)
	=
	\begin{tikzpicture}[baseline=.3cm, level/.style={sibling distance=.7cm, level distance = .7cm}]
		\node {$3$} [grow' = up]
			child [blueEdge] {node {$1$}}
			child [blueEdge] {node {$2$}}
		;
	\end{tikzpicture}	
\]
which leads to the description of \cref{exm:willowPosets}.
For~$k = 2$, the forbidden $2$-posets are

\medskip
\centerline{
\begin{tabular}{cccccccc}
	\begin{tikzpicture}[level/.style={sibling distance=.5cm, level distance = .7cm}]
		\node {$1$} [grow' = up]
			child {node {$1$}
				child {node {$2$}
					child {node {$2$}
						child {node {$3$}
							child {node {$3$}
							}
						}
					}
				}
			}
		;
	\end{tikzpicture}	
	&
	\begin{tikzpicture}[level/.style={sibling distance=.5cm, level distance = .7cm}]
		\node {$1$} [grow' = up]
			child {node {$2$}
				child {node {$1$}}
				child {node {$2$}
					child {node {$3$}
						child {node {$3$}
						}
					}
				}
			}
		;
	\end{tikzpicture}	
	&
	\begin{tikzpicture}[level/.style={sibling distance=.5cm, level distance = .7cm}]
		\node {$1$} [grow' = up]
			child {node {$3$}
				child {node {$1$}}
				child {node {$2$}
					child {node {$2$}}
					child {node {$3$}}
				}
			}
		;
	\end{tikzpicture}	
	&
	\begin{tikzpicture}[level/.style={sibling distance=.5cm, level distance = .7cm}]
		\node {$2$} [grow' = up]
			child {node {$1$}
				child {node {$1$}}
				child {node {$2$}
					child {node {$3$}
						child {node {$3$}
						}
					}
				}
			}
		;
	\end{tikzpicture}	
	&
	\begin{tikzpicture}[level/.style={sibling distance=.5cm, level distance = .7cm}]
		\node {$3$} [grow' = up]
			child {node {$1$}
				child {node {$1$}}
				child {node {$2$}
					child {node {$2$}}
					child {node {$3$}}
				}
			}
		;
	\end{tikzpicture}	
	&
	\begin{tikzpicture}[level/.style={sibling distance=.5cm, level distance = .7cm}]
		\node {$2$} [grow' = up]
			child {node {$3$}
				child {node {$1$}
					child {node {$1$}
					}
				}
				child {node {$2$}}
				child {node {$3$}}
			}
		;
	\end{tikzpicture}	
	&
	\begin{tikzpicture}[level/.style={sibling distance=.5cm, level distance = .7cm}]
		\node {$3$} [grow' = up]
			child {node {$2$}
				child {node {$1$}
					child {node {$1$}
					}
				}
				child {node {$2$}}
				child {node {$3$}}
			}
		;
	\end{tikzpicture}	
	&
	\begin{tikzpicture}[level/.style={sibling distance=.5cm, level distance = .7cm}]
		\node {$3$} [grow' = up]
			child {node {$3$}
				child {node {$1$}
					child {node {$1$}
					}
				}
				child {node {$2$}
					child {node {$2$}}
				}
			}
		;
	\end{tikzpicture}
	\\[.3cm]
    \compoR{l,l}{l,l}
	&
    \compoR{l,r}{l,l}
	&
    \compoR{l,r}{r,l}
	&
    \compoR{r,l}{l,l}
	&
    \compoR{r,l}{r,l}
	&
    \compoR{r,r}{l,r}
	&
    \compoR{r,r}{r,l}
	&
    \compoR{r,r}{r,r}
\end{tabular}
}

\medskip
\noindent
were we represented below each $2$-poset the series $k$-signaletic comb from which it was evaluated.


\subsubsection{Series poset operad}
\label{subsubsec:seriesPosetOperad}

As observed above, the operations~$\Operations_k$ on $k$-rooted $k$-posets defined in \cref{def:messyCitelangisSeriesActionPosets} do not satisfy the messy series $k$-citelangis relations.
However, they satisfy a subset of these relations which in turn defines an operad on posets.
This operad can be seen as a suboperad of an operad on all $k$-rooted $k$-posets defined by the following composition rules, illustrated in \cref{fig:seriesPosetOperad}.

\begin{definition}
\label{def:seriesRootedPosetOperad}
Let~$\le_M$ and~$\le_N$ be two $k$-rooted $k$-posets of degrees~$m$ and~$n$ respectively, and let~$i \in [m]$.
Let~$i_1, i_2, \dots, i_k$ denote the~$k$ copies of~$i$ in~$M$ such that~$i_1 <_M i_2 <_M \dots <_M i_k$.
We define the \defn{series composition}~${\le_P} \eqdef {\le_M} \posetSeriesCirc{i} {\le_N}$ by:
\begin{itemize}
\item $P \eqdef 1^{\{k\}} \dots (m+n-1)^{\{k\}}$,
\item $\le_P$ is obtained by inserting~$\le_N$ in~$\le_M$ by placing~$\min_j(\le_N)$ at~$i_j$ for all~$j \in [k]$, and performing the appropriate shift. More precisely, using the notations of \cref{def:shiftMultiset,def:shiftMultisetPlace}, the comparison $x \le_P y$ holds for~$x,y \in P$ if and only if one of the following statements holds:
	\begin{itemize}
	\item $x \in M[i,n]$, $y \in M[i,n]$ and~$\overline{x} \le_M \overline{y}$, or
	\item $x \in N[i-1]$, $y \in N[i-1]$ and~$\overline{x} \le_N \overline{y}$, or
	\item $x \in M[i,n]$, $y \in N[i-1]$, $\overline{x} \le_M i_j$ and~$\min_j(\le_N) \le_N y$ for some~$j \in [k]$ or
	\item $x \in N[i-1]$, $y \in M[i,n]$, $\overline{x} = \min_j(\le_N)$ and~$i_j \le_M \overline{y}$ for some~$j \in [k]$.
	\end{itemize}
\end{itemize}
\end{definition}

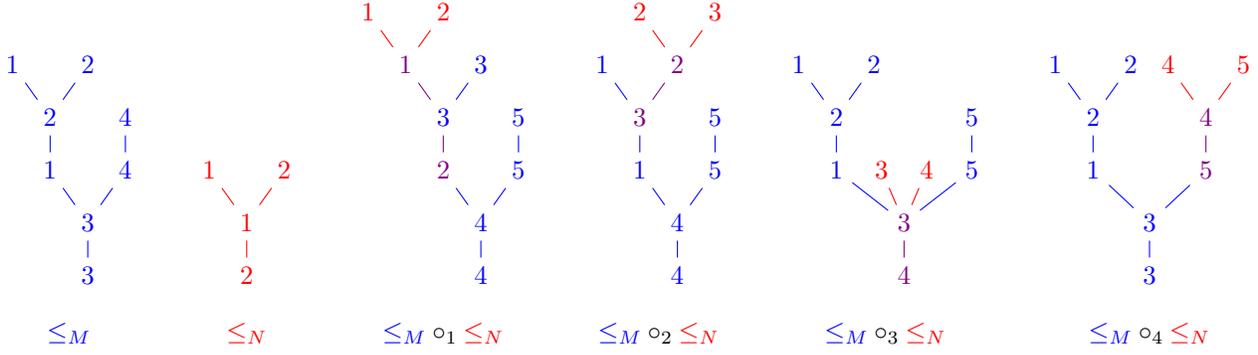
\begin{figure}
	\centerline{
	\begin{tabular}{c@{\qquad}c@{\qquad}c@{\qquad}c@{\qquad}c@{\qquad}c}
	\begin{tikzpicture}[level/.style={sibling distance=1cm, level distance = .7cm}, blue]
		\node {$3$} [grow' = up]
			child {node {$3$}
				child {node {$1$}
					child {node {$2$}
						child {node {$1$}}
						child {node {$2$}}
					}
				}
				child {node {$4$}
					child {node {$4$}}
				}
			}
		;
	\end{tikzpicture}
	&
	\begin{tikzpicture}[level/.style={sibling distance=1cm, level distance = .7cm}, red]
		\node {$2$} [grow' = up]
			child {node {$1$}
				child {node {$1$}}
				child {node {$2$}}
			}
		;
	\end{tikzpicture}
	&
	\begin{tikzpicture}[level/.style={sibling distance=1cm, level distance = .7cm}, blue]
		\node {$4$} [grow' = up]
			child [blueEdge] {node {$4$}
				child {node {\violet $2$}
					child [violetEdge] {node {$3$}
						child [violetEdge] {node {\violet $1$}
							child [redEdge] {node {\red $1$}}
							child [redEdge] {node {\red $2$}}						
						}
						child [blueEdge] {node {$3$}}
					}
				}
				child {node {$5$}
					child {node {$5$}}
				}
			}
		;
	\end{tikzpicture}
	&
	\begin{tikzpicture}[level/.style={sibling distance=1cm, level distance = .7cm}, blue]
		\node {$4$} [grow' = up]
			child {node {$4$}
				child {node {$1$}
					child {node {\violet $3$}
						child {node {$1$}}
						child [violetEdge] {node {\violet $2$}
							child [redEdge] {node {\red $2$}}
							child [redEdge] {node {\red $3$}}						
						}
					}
				}
				child {node {$5$}
					child {node {$5$}}
				}
			}
		;
	\end{tikzpicture}
	&
	\begin{tikzpicture}[level/.style={level distance = .7cm}, level 2/.style={sibling distance=.6cm}, level 4/.style={sibling distance=1cm}, blue]
		\node {\violet $4$} [grow' = up]
			child [violetEdge] {node {\violet $3$}
				child [blueEdge] {node {$1$}
					child {node {$2$}
						child {node {$1$}}
						child {node {$2$}}
					}
				}
				child [redEdge] {node {\red $3$}}
				child [redEdge] {node {\red $4$}}
				child [blueEdge] {node {$5$}
					child {node {$5$}}
				}
			}
		;
	\end{tikzpicture}
	&
	\begin{tikzpicture}[level/.style={level distance = .7cm}, level 2/.style={sibling distance=1.5cm}, level 4/.style={sibling distance=1cm}, blue]
		\node {$3$} [grow' = up]
			child {node {$3$}
				child {node {$1$}
					child {node {$2$}
						child {node {$1$}}
						child {node {$2$}}
					}
				}
				child {node {\violet $5$}
					child [violetEdge] {node {\violet $4$}
						child [redEdge] {node {\red $4$}}
						child [redEdge] {node {\red $5$}}					
					}
				}
			}
		;
	\end{tikzpicture}
	\\[.2cm]
	$\blue \le_M$ & $\red \le_N$ & ${\blue \le_M} \posetSeriesCirc{1} {\red \le_N}$ & ${\blue \le_M} \posetSeriesCirc{2} {\red \le_N}$ & ${\blue \le_M} \posetSeriesCirc{3} {\red \le_N}$ & ${\blue \le_M} \posetSeriesCirc{4} {\red \le_N}$
	\end{tabular}
	}
	\caption{Examples of compositions in the series $2$-poset operad. See \cref{def:seriesRootedPosetOperad}.}
	\label{fig:seriesPosetOperad}
\end{figure}

The following statement is left to the reader.

\begin{lemma}
For any two $k$-rooted $k$-posets~$\le_M$ and~$\le_N$ of degrees~$m$ and~$n$ respectively, and any~$i \in [m]$, the series composition ${\le_M} \posetSeriesCirc{i} {\le_N}$ is a $k$-rooted $k$-poset of degree~$m+n-1$.
\end{lemma}

\begin{proposition}
The series compositions~$\posetSeriesCirc{i}$ of \cref{def:seriesRootedPosetOperad} define an operad structure~$\rootedPosSeries$ on $k$-rooted $k$-posets.
\end{proposition}

\begin{proof}
Rather than a formal proof, we illustrate the argument on pictures.
For the the series axioms (see \cref{subsec:operads}), we have:

\bigskip
\centerline{$
	\left(
	\begin{tikzpicture}[baseline={([yshift=-.8ex]current bounding box.center)}, scale=.4]
		\draw[fill=green!30] (3.9,0) -- (3.9,2) -- (1.9,6) -- (6.1,6) -- (4.1,2) -- (4.1,0) -- (3.9,0);
		\node at (4.7,1) {\blue $i_1$};
		\node at (5.7,4) {\blue $i_k$};
		\node [circle, fill=blue, inner sep=1.5pt] at (4,1) {};
		\node [circle, fill=blue, inner sep=1.5pt] at (4.5,3) {};
		\node [circle, fill=blue, inner sep=1.5pt] at (5,4) {};
		\draw [blue] (4,1) -- (4,2) -- (5,4);
	\end{tikzpicture}
	\posetSeriesCirc{\blue i}
	\begin{tikzpicture}[baseline={([yshift=-.8ex]current bounding box.center)}, scale=.4]
		\draw[fill=blue!30] (3.9,0) -- (3.9,2) -- (2.4,5) -- (5.6,5) -- (4.1,2) -- (4.1,0) -- (3.9,0);
		\node at (3.5,1) {\red $j_1$};
		\node at (2.5,4) {\red $j_k$};
		\node [circle, fill=blue, inner sep=1.5pt] at (4,0) {};
		\node [circle, fill=blue, inner sep=1.5pt] at (4,1) {};
		\node [circle, fill=blue, inner sep=1.5pt] at (4,2) {};
		\draw [blue] (4,0) -- (4,2);
		\node [circle, fill=red, inner sep=1pt] at (4,1) {};
		\node [circle, fill=red, inner sep=1pt] at (3.5,3) {};
		\node [circle, fill=red, inner sep=1pt] at (3,4) {};
		\draw [red] (4,1) -- (4,2) -- (3,4);
	\end{tikzpicture}
	\right)
	\posetSeriesCirc{{\blue i}+{\red j}-1}
	\begin{tikzpicture}[baseline={([yshift=-.8ex]current bounding box.center)}, scale=.4]
		\draw[fill=red!30] (3.9,0) -- (3.9,2) -- (2.9,4) -- (5.1,4) -- (4.1,2) -- (4.1,0) -- (3.9,0);
		\node [circle, fill=red, inner sep=1.5pt] at (4,0) {};
		\node [circle, fill=red, inner sep=1.5pt] at (4,1) {};
		\node [circle, fill=red, inner sep=1.5pt] at (4,2) {};
		\draw [red] (4,0) -- (4,2);
	\end{tikzpicture}
	=
	\begin{tikzpicture}[baseline={([yshift=-.8ex]current bounding box.center)}, scale=.4]
		\draw[fill=green!30] (3.9,0) -- (3.9,2) -- (1.9,6) -- (6.1,6) -- (4.1,2) -- (4.1,0) -- (3.9,0);
		\draw[fill=blue!30] (3.9,1) -- (3.9,2) -- (5,4.2) -- (8.2,7.3) -- (9.7,4) -- (5.1,4) -- (4.1,2) -- (4.1,1) -- (3.9,1);
		\draw[fill=red!30] (4.4,3) -- (5,4.2) -- (6.8,6) -- (5.8,8) -- (8.1,8) -- (7.1,6) -- (5.1,4) -- (4.6,3) -- (4.4,3);
		\node [circle, fill=blue, inner sep=1.5pt] at (4,1) {};
		\node [circle, fill=blue, inner sep=1.5pt] at (4.5,3) {};
		\node [circle, fill=blue, inner sep=1.5pt] at (5,4) {};
		\draw [blue] (4,1) -- (4,2) -- (5,4);
		\node [circle, fill=red, inner sep=1pt] at (4.5,3) {};
		\node [circle, fill=red, inner sep=1pt] at (5.97,5) {};
		\node [circle, fill=red, inner sep=1pt] at (6.95,6) {};
		\draw [red] (4.5,3) -- (5.05,4.1) -- (6.95,6);
	\end{tikzpicture}
	=	
	\begin{tikzpicture}[baseline={([yshift=-.8ex]current bounding box.center)}, scale=.4]
		\draw[fill=green!30] (3.9,0) -- (3.9,2) -- (1.9,6) -- (6.1,6) -- (4.1,2) -- (4.1,0) -- (3.9,0);
		\node at (4.7,1) {\blue $i_1$};
		\node at (5.7,4) {\blue $i_k$};
		\node [circle, fill=blue, inner sep=1.5pt] at (4,1) {};
		\node [circle, fill=blue, inner sep=1.5pt] at (4.5,3) {};
		\node [circle, fill=blue, inner sep=1.5pt] at (5,4) {};
		\draw [blue] (4,1) -- (4,2) -- (5,4);
	\end{tikzpicture}
	\posetSeriesCirc{\blue i}
	\left(
	\begin{tikzpicture}[baseline={([yshift=-.8ex]current bounding box.center)}, scale=.4]
		\draw[fill=blue!30] (3.9,0) -- (3.9,2) -- (2.4,5) -- (5.6,5) -- (4.1,2) -- (4.1,0) -- (3.9,0);
		\node at (3.5,1) {\red $j_1$};
		\node at (2.5,4) {\red $j_k$};
		\node [circle, fill=blue, inner sep=1.5pt] at (4,0) {};
		\node [circle, fill=blue, inner sep=1.5pt] at (4,1) {};
		\node [circle, fill=blue, inner sep=1.5pt] at (4,2) {};
		\draw [blue] (4,0) -- (4,2);
		\node [circle, fill=red, inner sep=1pt] at (4,1) {};
		\node [circle, fill=red, inner sep=1pt] at (3.5,3) {};
		\node [circle, fill=red, inner sep=1pt] at (3,4) {};
		\draw [red] (4,1) -- (4,2) -- (3,4);
	\end{tikzpicture}
	\posetSeriesCirc{\red j}
	\begin{tikzpicture}[baseline={([yshift=-.8ex]current bounding box.center)}, scale=.4]
		\draw[fill=red!30] (3.9,0) -- (3.9,2) -- (2.9,4) -- (5.1,4) -- (4.1,2) -- (4.1,0) -- (3.9,0);
		\node [circle, fill=red, inner sep=1.5pt] at (4,0) {};
		\node [circle, fill=red, inner sep=1.5pt] at (4,1) {};
		\node [circle, fill=red, inner sep=1.5pt] at (4,2) {};
		\draw [red] (4,0) -- (4,2);
	\end{tikzpicture}
	\right)
$}

\bigskip
\noindent
For the the parallel axiom (see \cref{subsec:operads}), we have:

\bigskip
\centerline{$
	\left(
	\begin{tikzpicture}[baseline={([yshift=-.8ex]current bounding box.center)}, scale=.4]
		\draw[fill=green!30] (3.9,0) -- (3.9,2) -- (1.9,6) -- (6.1,6) -- (4.1,2) -- (4.1,0) -- (3.9,0);
		\node at (3.5,1) {\blue $i_1$};
		\node at (2.5,4) {\blue $i_k$};
		\node at (5,2.5) {\red $j_1$};
		\node at (5.5,4) {\red $j_k$};
		\node [circle, fill=blue, inner sep=1.5pt] at (4,1) {};
		\node [circle, fill=blue, inner sep=1.5pt] at (3.5,3) {};
		\node [circle, fill=blue, inner sep=1.5pt] at (3,4) {};
		\draw [blue] (4,1) -- (4,2) -- (3,4);
		\node [circle, fill=red, inner sep=1.5pt] at (4.25,2.5) {};
		\node [circle, fill=red, inner sep=1.5pt] at (4.75,3.5) {};
		\node [circle, fill=red, inner sep=1.5pt] at (5,4) {};
		\draw [red] (4.25,2.5) -- (5,4);
	\end{tikzpicture}
	\posetSeriesCirc{\blue i}
	\begin{tikzpicture}[baseline={([yshift=-.8ex]current bounding box.center)}, scale=.4]
		\draw[fill=blue!30] (3.9,0) -- (3.9,2) -- (2.9,4) -- (5.1,4) -- (4.1,2) -- (4.1,0) -- (3.9,0);
		\node [circle, fill=blue, inner sep=1.5pt] at (4,0) {};
		\node [circle, fill=blue, inner sep=1.5pt] at (4,1) {};
		\node [circle, fill=blue, inner sep=1.5pt] at (4,2) {};
		\draw [blue] (4,0) -- (4,2);
	\end{tikzpicture}
	\right)
	\posetSeriesCirc{{\red j}+{\blue q}-1}
	\begin{tikzpicture}[baseline={([yshift=-.8ex]current bounding box.center)}, scale=.4]
		\draw[fill=red!30] (3.9,0) -- (3.9,2) -- (2.9,4) -- (5.1,4) -- (4.1,2) -- (4.1,0) -- (3.9,0);
		\node [circle, fill=red, inner sep=1.5pt] at (4,0) {};
		\node [circle, fill=red, inner sep=1.5pt] at (4,1) {};
		\node [circle, fill=red, inner sep=1.5pt] at (4,2) {};
		\draw [red] (4,0) -- (4,2);
	\end{tikzpicture}
	=
	\begin{tikzpicture}[baseline={([yshift=-.8ex]current bounding box.center)}, scale=.4]
		\draw[fill=green!30] (3.9,0) -- (3.9,2) -- (1.9,6) -- (6.1,6) -- (4.1,2) -- (4.1,0) -- (3.9,0);
		\draw[fill=blue!30] (3.9,1) -- (3.9,2) -- (2.9,4) -- (0,4) -- (1,6) -- (3,4.2) -- (4.1,2) -- (4.1,1) -- (3.9,1);
		\draw[fill=red!30] (4.15,2.5) -- (5,4.2) -- (7,6) -- (8,4) -- (5.1,4) -- (4.3,2.4) -- (4.15,2.5);
		\node [circle, fill=blue, inner sep=1.5pt] at (4,1) {};
		\node [circle, fill=blue, inner sep=1.5pt] at (3.5,3) {};
		\node [circle, fill=blue, inner sep=1.5pt] at (3,4) {};
		\draw [blue] (4,1) -- (4,2) -- (3,4);
		\node [circle, fill=red, inner sep=1.5pt] at (4.25,2.5) {};
		\node [circle, fill=red, inner sep=1.5pt] at (4.75,3.5) {};
		\node [circle, fill=red, inner sep=1.5pt] at (5,4) {};
		\draw [red] (4.25,2.5) -- (5,4);
	\end{tikzpicture}
	=	
	\left(
	\begin{tikzpicture}[baseline={([yshift=-.8ex]current bounding box.center)}, scale=.4]
		\draw[fill=green!30] (3.9,0) -- (3.9,2) -- (1.9,6) -- (6.1,6) -- (4.1,2) -- (4.1,0) -- (3.9,0);
		\node at (3.5,1) {\blue $i_1$};
		\node at (2.5,4) {\blue $i_k$};
		\node at (5,2.5) {\red $j_1$};
		\node at (5.5,4) {\red $j_k$};
		\node [circle, fill=blue, inner sep=1.5pt] at (4,1) {};
		\node [circle, fill=blue, inner sep=1.5pt] at (3.5,3) {};
		\node [circle, fill=blue, inner sep=1.5pt] at (3,4) {};
		\draw [blue] (4,1) -- (4,2) -- (3,4);
		\node [circle, fill=red, inner sep=1.5pt] at (4.25,2.5) {};
		\node [circle, fill=red, inner sep=1.5pt] at (4.75,3.5) {};
		\node [circle, fill=red, inner sep=1.5pt] at (5,4) {};
		\draw [red] (4.25,2.5) -- (5,4);
	\end{tikzpicture}
	\posetSeriesCirc{\red j}
	\begin{tikzpicture}[baseline={([yshift=-.8ex]current bounding box.center)}, scale=.4]
		\draw[fill=red!30] (3.9,0) -- (3.9,2) -- (2.9,4) -- (5.1,4) -- (4.1,2) -- (4.1,0) -- (3.9,0);
		\node [circle, fill=red, inner sep=1.5pt] at (4,0) {};
		\node [circle, fill=red, inner sep=1.5pt] at (4,1) {};
		\node [circle, fill=red, inner sep=1.5pt] at (4,2) {};
		\draw [red] (4,0) -- (4,2);
	\end{tikzpicture}
	\right)
	\posetSeriesCirc{\blue i}
	\begin{tikzpicture}[baseline={([yshift=-.8ex]current bounding box.center)}, scale=.4]
		\draw[fill=blue!30] (3.9,0) -- (3.9,2) -- (2.9,4) -- (5.1,4) -- (4.1,2) -- (4.1,0) -- (3.9,0);
		\node [circle, fill=blue, inner sep=1.5pt] at (4,0) {};
		\node [circle, fill=blue, inner sep=1.5pt] at (4,1) {};
		\node [circle, fill=blue, inner sep=1.5pt] at (4,2) {};
		\draw [blue] (4,0) -- (4,2);
	\end{tikzpicture}
$}
\end{proof}

The following statements connect this operad~$\rootedPosSeries$ with the series poset evaluation of \cref{def:posetSeries}.

\begin{lemma}
\label{lem:seriesPosetOperad}
For any operation~${\operation \in \Operations_k}$ and any $k$-rooted $k$-posets~$\le_M$ and~$\le_N$, we have
\[
{\le_M} \; \operation \; {\le_N} = \EvalPosetSeries(\operation) \, \posetSeriesCirc{} \, ({\le_M}, {\le_N}),
\]
where the posets~$\EvalPosetSeries(\operation) \eqdef {\le_{I_k}} \; \operation \; {\le_{I_k}}$ are illustrated in \cref{fig:exmMessyCitelangisSeriesActionPosets} for~$k = 2$.
\end{lemma}

\begin{proof}
By \cref{def:seriesRootedPosetOperad}, the series composition~$\EvalPosetSeries(\operation) \, \posetSeriesCirc{} \, ({\le_M}, {\le_N})$ places $\min_j(\le_M)$ at~$1_j$ and ${\min_j(\le_N)}$ at~$2_j$ in the $k$-poset~$\EvalPosetSeries(\operation)$ for all~$j \in [k]$.
This coincides with the result of the operation~${\le_M} \; \operation \; {\le_N}$ described in \cref{def:messyCitelangisSeriesActionPosets,rem:sumsPosetsSeries}.
\end{proof}

\cref{lem:seriesPosetOperad} enables us to interpret the operations on $k$-rooted $k$-posets of the previous section as a suboperad of~$\rootedPosSeries$.
We say that two syntax trees~$\tree, \tree'$ on~$\Operations_k$ with the same arity are \defn{series poset equivalent} and we write~$\tree \simeq^{\series} \tree'$ if they have the same series poset evaluation.

\begin{theorem}
\label{thm:seriesPosetOperad}
The series poset equivalence is compatible with the grafting of syntax trees: for any syntax trees~$\tree \simeq^\series \tree'$ of arity~$p$ and~$\tree[s] \simeq^\series \tree[s]'$ of arity~$q$ and~$i \in [p]$, we have~$\tree \circ_i \tree[s] \simeq^\series \tree' \circ_i \tree[s]'$.
Therefore, there exists a \defn{series $k$-poset operad}
\[
\posSeries \eqdef \EvalPosetSeries \big( \Syntax[\Operations_k] \big).
\]
\end{theorem}

\begin{proof}
By induction on \cref{lem:seriesPosetOperad}, the map~$\EvalPosetSeries$ coincides with the unique operad morphism from the free operad~$\Free[\Operations_k]$ to~$\rootedPosSeries$ that sends a generator~$\operation \in \Operations_k$ to~$\EvalPosetSeries(\operation)$.
Therefore, $\posSeries$ is the suboperad of~$\rootedPosSeries$ generated by~$\bigset{\EvalPosetSeries(\operation)}{\operation \in \Operations_k}$.
\end{proof}

By construction, the series $k$-poset operad~$\posSeries$ satisfies the relation of \cref{eq:relationPosetsSeries}.
\cref{prop:uniqueSeriesPosetEvaluation} implies that it is the only relation in~$\posSeries$.
This generalizes the $L$-algebras of P.~Leroux~\cite{Leroux-Lalgebras}.

\begin{theorem}
\label{thm:presentationSeriesPosetOperad}
The series $k$-poset operad~$\posSeries$ is generated by~$\Operations_k$ with the unique relation
\[
\begin{tikzpicture}[baseline=-.5cm, level 1/.style={sibling distance = 1cm, level distance = .7cm}, level 2/.style={sibling distance = .8cm, level distance = .6cm}]
	\node [rectangle, draw] {$\op{l}\!\!^k$}
		child {node [rectangle, draw] {$\op{r}\!\!^k$}
			child {node {}}
		child {node {}}
		}
		child {node {}}
	;
\end{tikzpicture}
=
\begin{tikzpicture}[baseline=-.5cm, level 1/.style={sibling distance = 1cm, level distance = .7cm}, level 2/.style={sibling distance = .8cm, level distance = .6cm}]
	\node [rectangle, draw] {$\op{r}\!\!^k$}
		child {node {}}
		child {node [rectangle, draw] {$\op{l}\!\!^k$}
			child {node {}}
		child {node {}}
		}
	;
\end{tikzpicture}
.
\]
\end{theorem}

\begin{corollary}\label{corol:dual-poset}
For any $k$, the $\K$-linearlized operad~$\posSeries$ is Koszul. Its Koszul dual~$(\posSeries)^\koszul$ is a finite dimensional operad generated by~$\Operations_k$ with a unique arity~$3$ element
\[
\begin{tikzpicture}[baseline=-.5cm, level 1/.style={sibling distance = 1cm, level distance = .7cm}, level 2/.style={sibling distance = .8cm, level distance = .6cm}]
	\node [rectangle, draw] {$\op{l}\!\!^k$}
		child {node [rectangle, draw] {$\op{r}\!\!^k$}
			child {node {}}
		child {node {}}
		}
		child {node {}}
	;
\end{tikzpicture}
=
\begin{tikzpicture}[baseline=-.5cm, level 1/.style={sibling distance = 1cm, level distance = .7cm}, level 2/.style={sibling distance = .8cm, level distance = .6cm}]
	\node [rectangle, draw] {$\op{r}\!\!^k$}
		child {node {}}
		child {node [rectangle, draw] {$\op{l}\!\!^k$}
			child {node {}}
		child {node {}}
		}
	;
\end{tikzpicture}\,,
\]
all the other arity~$3$ compositions being zero. In particular its Hilbert series is
\[
\Hilbert_{(\posSeries)^\koszul}(t) = t + 2^k t^2 + t^3.
\]
\end{corollary}

\begin{proof}
The only non trivial fact is the Koszulity of~$\posSeries$. We prove it on the dual $(\posSeries)^\koszul$ where it is plain thanks to the Poincar\'e\,--\,Birkhoff\,--\,Witt basis rewriting criterion. Indeed, all trees of arity $4$ rewrites ultimately to $0$.
\end{proof}

\begin{remark}
Combinatorially, the dual series $k$-poset operad~$(\posSeries)^\koszul$ can be realized as a quotient of the messy series $k$-signaletic operad~$\messySignaleticSeries$ spanned by the destination vectors of the form~$\destVect{1^k}{1}$, $\destVect{w}{2}$ with $w \in \{1,2\}^k$, and $\destVect{2^k}{3}$, which we call \defn{$L$-destination vectors}.
The surjective morphism $\messySignaleticSeries \longrightarrow (\posSeries)^\koszul$ sends a destination vector to itself if it is an $L$-destination vector and to $0$ otherwise.
\end{remark}

Thanks to the preceding corollary, we can now identify~$\posSeries$ using Manin powers (see \cref{subsec:ManinProducts}).

\begin{proposition}
The operad~$\posSeries$ is isomorphic to the $k$-th black or white power of the operad~$\posSeries[1]$ of $L$-algebras:
\[
\posSeries \simeq {\posSeries[1]}^{\whiteManin k} \simeq {\posSeries[1]}^{\blackManin k}.
\]
\end{proposition}

\begin{proof}
We consider the map~$\Free[\Operations_k] \to \Free[\Operations]^{\otimes k}$ which sends~$\operation$ to~$\operation_1\otimes\dots\otimes \operation_k$. This map is compatible with the relations of~$\posSeries$ (\cref{thm:presentationSeriesPosetOperad}) as well as the relations of~$(\posSeries)^\koszul$ (\cref{corol:dual-poset}) giving two maps
\[
\posSeries \to {\posSeries[1]}^{\otimes k}
\qqandqq
(\posSeries)^\koszul \to {(\posSeries[1])^\koszul}^{\otimes k}.
\]
Thanks to \cref{prop:whiteManin}, the statement amounts to the fact that these two maps are injective. This is clear from the relations.
\end{proof}

\begin{remark}
The map $\posSeries \to {\posSeries[1]}^{\otimes k}$ of the preceeding proof can be seen combinatorially on $k$-posets. Indeed, it sends a $k$-poset $\le_P$ of degree $n$ to the tensor product~$\le_P^{(1)} \otimes \dots \otimes \le_P^{(k)}$ where~$\le_P^{(i)}$ is the restriction of~$\le_P$ to the set~$\{1_i,\dots, n_i\}$ of the $i$-th copy of the integers of~$[n]$. Thanks to the relation this map is injective. This is not obvious on posets.
\end{remark}


\subsubsection{Hilbert series and numerology in the series $k$-poset operads}
\label{subsubsec:HilbertSeriesPosetOperad}

We now explore some numerological facts about the series $k$-poset operads~$\posSeries$ defined in \cref{subsubsec:seriesPosetOperad}. Thanks to \cref{corol:dual-poset}, one can compute its Hilbert series by Koszul duality.

\begin{corollary}
\label{coro:HilbertSeriesPosetOperad}
The Hilbert series $\Hilbert_{{\posSeries}}(t)$ it the unique solution of the form $H=t + O(t^2)$ of the algebraic equation $H^3 - 2^k H^2 + H - t = 0$.
\end{corollary}

It is well known that any algebraic univariate power series is holonomic (some authors call them $D$-finite), which means that it satisfies a linear differential equation whose coefficients are in the field~$\C(t)$ of rational functions over~$\C$. Moreover, the differential equation can be automatically computed~\cite[Remark~B.12]{FlajoletSedgewick}. In the situation of~\cref{coro:HilbertSeriesPosetOperad}, this can actually be done in a more general setting by replacing $2^k$ by any number $K \ge 2$. Thus we set $H \eqdef t + \sum_{n>1} u_n t^n$ to be the solution of the equation $H^3 - KH^2 + H - t = 0$. Then $H$ is the solution of the differential equation
\[
(27 t^2 + 2C t - D) \frac{\partial^2H}{\partial t^2} + (27 t + C) \frac{\partial H}{\partial t} - 3H + K = 0,
\]
where $C = K(2K^2-9)$ and $D = K^2-4$
An extraction of coefficients shows that $u_n$ satisfies the recurrence 
\[
D (n + 1) (n + 2) u_{n + 2} = C (2n + 1)(n + 1) u_{n + 1}  + 3 (3n - 1) (3n + 1) u_{n},
\]
for~$n > 1$, together with the initial conditons $u_1 = 1$, $u_2 = K$.
\cref{table:dposetk} gives the first values of $\dim \big( \posSeries(n) \big)$.
Note that in the special case when $K = 2$ (that is $k = 1$), we have $C = -2$ and $D = 0$, so that this recurrence degenerates into
\[
2 (2n + 1)(n + 1) u_{n + 1} = 3 (3n - 1) (3n + 1) u_{n}.
\]
which yields the closed form
\[
u_n = \frac{1}{n}\binom{3n-2}{n-1}.
\]

\begin{table}[t]
	\[
    \begin{array}{l|rrrrrrrr}
      \raisebox{-.1cm}{$k$} \backslash \, \raisebox{.1cm}{$n$}
      & 1 &  2 &    3 &      4 &       5 &         6 &          7  &             8 \\[.1cm]
      \hline
      1 & 1 & 2 & 7 & 30 & 143 & 728 & 3876 & 21318 \\
      2 & 1 & 4 & 31 & 300 & 3251 & 37744 & 459060 & 5773548 \\
      3 & 1 & 8 & 127 & 2520 & 56003 & 1333472 & 33262836 & 858011352 \\
      4 & 1 & 16 & 511 & 20400 & 912131 & 43696576 & 2193011700 & 113813345712 \\
      5 & 1 & 32 & 2047 & 163680 & 14658563 & 1406534528 & 141388074996 & 14697175640928
    \end{array}
	\]
	\caption{The values of~$\dim \big( \posSeries(n) \big)$ for~$k \in [5]$ and~$n \in [8]$.}
	\label{table:dposetk}
\end{table}

\begin{remark}
This raises the open question whether there are interesting combinatorial operads whose Hilbert series corresponds to arbitrary integer~$K \ge 2$.
\end{remark}

\begin{remark}
The numerology presented here is also valid for the parallel $2$-poset operad~$\posParallel[2]$ defined in \cref{subsubsec:parallelPosetOperad} since the operads~$\posParallel[2]$ and~$\posSeries[2]$ are isomorphic by \cref{thm:presentationParallelPosetOperad,thm:presentationSeriesPosetOperad}.
\end{remark}


\pagebreak
\subsubsection{Series \texorpdfstring{$k$}{k}-Zinbiel operads}
\label{subsubsec:seriesZinbiel}

As in \cref{subsubsec:parallelZinbiel}, we look for an operad $\messyZinbielSeries$ on $k$-permutations which contains the operad $\messyCitelangisSeries$ as a suboperad generated in degree $2$ and which moreover closes the following diagram of operad morphisms. We will also find a tidy version $\tidyZinbielSeries$, closing the right square where the dashed arrows are not operad morphisms, but bijections of normal forms.

\begin{center}
\begin{tikzcd}[column sep=2cm, row sep=.3cm]
	\phantom{|}\rootedPosSeries\phantom{|} \arrow[r, "\LinExt",  two heads]
	& \phantom{|}\messyZinbielSeries\phantom{|} \arrow[r, dashrightarrow, "\LexMin", hook, two heads]
	& \phantom{|}\tidyZinbielSeries\phantom{|} \\
	\phantom{|}\posSeries\phantom{|} \arrow[r, "\LinExt", two heads] \arrow[u, hook]
	& \phantom{|}\messyCitelangisSeries\phantom{|} \arrow[r, dashrightarrow, "\LexMin", hook, two heads] \arrow[u, hook]
	& \phantom{|}\tidyCitelangisSeries\phantom{|} \arrow[u, hook]
\end{tikzcd}
\end{center}

\paraul{Messy series $k$-Zinbiel operad}
We start with the messy version.

\begin{definition}
\label{def:compositionMessySeriesZinbiel}
Let~$\sigma \in \Perm_k(m)$ and~$\tau \in \Perm_k(n)$ be two $k$-permutations and~$i \in [m]$.
Write~$\sigma = \omega_0 \, i \, \omega_1 \, i \, \omega_2 \cdots \omega_{k-1} \, i \, \omega_k$ and~$\tau = \tau_1 \, \cdots \, \tau_k \, \theta$, where~$\tau_1, \dots, \tau_k$ are the first $k$ letters of~$\tau$.
Then the \defn{messy series $i$-th composition} of $\sigma$ and $\mu$ is defined by
\[
\sigma \messySeriesCirc{i} \tau = \omega_0[i,n] \, \tau_1[i-1] \, \omega_1[i,n] \, \tau_2[i-1] \, \omega_2[i,n] \cdots \omega_{k-1}[i,n] \, \tau_k[i-1] \, (\omega_k[i,n] \shuffle \theta[i-1]).
\]
We extends this definition by linearity.
\end{definition}

Here are some examples of messy series compositions on $k$-permutations:
\begin{align*}
  {\blue3}1{\blue232}1{\blue44} \messySeriesCirc{1} 31{\red3122} = \; &
  {\blue5}3{\blue454}1{\red3122}{\blue66}  +
  {\blue5}3{\blue454}1{\red312}{\blue6}{\red2}{\blue6}  +
  {\blue5}3{\blue454}1{\red312}{\blue66}{\red2} + \dots +
  {\blue5}3{\blue454}1{\red3}{\blue6}{\red12}{\blue6}{\red2}
  \\ & + \cdots (\text{$15$~terms}) \cdots +
  {\blue5}3{\blue454}1{\blue6}{\red3}{\blue6}{\red122}  +
  {\blue5}3{\blue454}1{\blue66}{\red3122}, \\
  {\blue31}2{\blue3}2{\blue144} \messySeriesCirc{2} 31{\red3122} = \; &
  {\blue51}4{\blue5}2{\red4233}{\blue1}{\blue66} +
  {\blue51}4{\blue5}2{\red423}{\blue1}{\red3}{\blue66} +
  {\blue51}4{\blue5}2{\red423}{\blue16}{\red3}{\blue6} + \dots +
  {\blue51}4{\blue5}2{\red4}{\blue1}{\red23}{\blue6}{\red3}{\blue6}
  \\ & + \cdots (\text{$35$~terms}) \cdots +
  {\blue5}3{\blue454}1{\blue6}{\red3}{\blue6}{\red122}  +
  {\blue5}3{\blue454}1{\blue66}{\red3122}, \\
  3{\blue12}3{\blue2144} \messySeriesCirc{3} 31{\red3122} = \; &
  5{\blue12}3{\red5344}{\blue2166} +
  5{\blue12}3{\red534}{\blue2}{\red4}{\blue166} +
  5{\blue12}3{\red534}{\blue21}{\red4}{\blue66} + \dots +
  5{\blue12}3{\red5}{\blue2}{\red34}{\blue16}{\red4}{\blue6} +
  \\ & + \cdots (\text{$70$~terms}) \cdots +
  5{\blue12}3{\blue216}{\red5}{\blue6}{\red344} +
  5{\blue12}3{\blue2166}{\red5344}, \\
  {\blue312321}44 \messySeriesCirc{4} 31{\red3122} = \; &
  {\blue312321}64{\red6455}.
\end{align*}

Let us stress out that, even when~$k = 2$, the messy parallel composition~$\sigma \messyParallelCirc{i} \tau$ of \cref{def:compositionMessyParallelZinbiel} and messy series composition~$\sigma \messySeriesCirc{i} \tau$ of \cref{def:compositionMessySeriesZinbiel} differ.

The following two lemmas relate the messy series composition~$\messySeriesCirc{i}$ of $k$-permutations of \cref{def:compositionMessySeriesZinbiel} and the series composition~$\posetSeriesCirc{i}$ of $k$-rooted $k$-posets of \cref{def:seriesRootedPosetOperad}.
They play the same role for the composition as \cref{lem:LinExtEvalPosetSeries} played for the operations.

\begin{lemma}
\label{lem:morphismSeriesPermPos}
For any two $k$-permutations~$\sigma \in \Perm_k(m)$ and~$\tau \in \Perm_k(n)$ and~$i \in [m]$,
\[
\sigma \messySeriesCirc{i} \tau = \LinExt \big( {\le_\sigma} \posetSeriesCirc{i} {\le_\tau} \big).
\]
\end{lemma}

\begin{proof}
As in \cref{def:compositionMessySeriesZinbiel}, write~$\sigma = \omega_0 \, i \, \omega_1 \, i \, \omega_2 \cdots \omega_{k-1} \, i \, \omega_k$ and~$\tau = \tau_1 \, \cdots \, \tau_k \, \theta$.
From the description of \cref{def:seriesRootedPosetOperad}, the poset~${\le_\sigma} \posetSeriesCirc{i} {\le_\tau}$ decomposes into the ordered sum
\[
{\le_{\omega_0[i,n]}} + \big\{ \tau_1[i-1] \big\} + {\le_{\omega_1[i,n]}} + \cdots + {\le_{\omega_{k-1}[i,n]}} + \big\{ \tau_k[i-1] \big\} + \big( {\le_{\omega_k[i,n]}} \, \sqcup \, {\le_{\theta[i-1]}} \big)
\]
Hence, the statement follows from \cref{lem:disjointUnionOrderedSumLinearExtensions}.
\end{proof}

\begin{lemma}
\label{lem:morphismSeriesPermPos2}
For any two $k$-rooted $k$-posets~${\le_M} \in \rootedPosSeries(m)$ and~${\le_N} \in \rootedPosSeries[k](n)$ and~$i \in [m]$,
\[
\LinExt(\le_M) \messySeriesCirc{i} \LinExt(\le_N) = \LinExt({\le_M} \posetSeriesCirc{i} {\le_N}),
\]
\end{lemma}

\begin{proof}
Fix three integers $m$, $n$, and $i\in[m]$.
Remark that if $\nu$ is a permutation appearing in the messy series composition $\sigma \messySeriesCirc{i} \tau$ for some $\sigma \in \Perm_k(m)$ and~$\tau \in \Perm_k(n)$, one can recover uniquely these two permutations from $\nu$.
As a consequence, $\LinExt(\le_M) \messySeriesCirc{i} \LinExt(\le_N)$ cannot have multiplicities, so that we can argue by double inclusion.

Observe from \cref{def:seriesRootedPosetOperad} that the inclusion of relations~${\le_M} \subseteq {\le'_M}$ and~${\le_N} \subseteq {\le'_N}$ imply the inclusion~${{\le_M} \posetSeriesCirc{i} {\le_N}} \subseteq {{\le'_M} \posetSeriesCirc{i} {\le'_N}}$.
For~$\sigma \in \LinExt(\le_M)$ and~$\tau \in \LinExt(\le_N)$, we have~${\le_M} \subseteq {\le_\sigma}$ and~${\le_N} \subseteq {\le_\tau}$, hence~${{\le_M} \posetSeriesCirc{i} {\le_N}} \subseteq {{\le_\sigma} \posetSeriesCirc{i} {\le_\sigma}}$.
We thus get~$\LinExt({\le_\sigma} \posetSeriesCirc{i} {\le_\tau}) \subseteq \LinExt({\le_M} \posetSeriesCirc{i} {\le_N})$.
As~$\LinExt({\le_\sigma} \posetSeriesCirc{i} {\le_\tau}) = \sigma \messySeriesCirc{i} \tau$ by \cref{lem:morphismSeriesPermPos}, we obtain that
\[
\LinExt(\le_M) \messySeriesCirc{i} \LinExt(\le_N)  \; = \bigcup_{\substack{\sigma \in \LinExt(\le_M) \\ \tau \in \LinExt(\le_N)}} \sigma \messySeriesCirc{i} \tau \; \subseteq \; \LinExt({\le_M} \posetSeriesCirc{i} {\le_N}).
\]

Conversely, starting from any linear extension~$\rho$ of~${\le_M} \posetSeriesCirc{i} {\le_N}$, restricting~$\rho$  to~$M[i,n]$ and~$N[i-1]$ and renumbering respectively yields linear extensions~$\sigma$ of~$\le_M$ and~$\tau$ of~$\le_N$.
A straightforward case analysis of on \cref{def:seriesRootedPosetOperad} shows that~$\rho$ appears in~$\sigma \messySeriesCirc{i} \tau$.
\end{proof}

As a consequence, \cref{def:compositionMessySeriesZinbiel} actually defines an operad composition.

\begin{proposition}
\label{prop:morphismSeriesLinExt}
The family $\big( \K\Perm_k(n) \big)_{n>0}$ endowed with the messy series composition rules~$\messySeriesCirc{i}$ of \cref{def:compositionMessySeriesZinbiel} defines an operad $\messyZinbielSeries[k]$, called the \defn{messy series $k$-Zinbiel operad}. Moreover, $\LinExt$ is a surjective operad morphism from $\rootedPosSeries[k]$ to $\messyZinbielSeries[k]$.
\end{proposition}

\begin{proof}
Thanks to \cref{lem:morphismSeriesPermPos,lem:morphismSeriesPermPos2}, we have
\[
\sigma \messySeriesCirc{i} (\tau \messySeriesCirc{j} \mu) = \LinExt \big( {\le_\sigma} \posetSeriesCirc{i} ({\le_\tau} \posetSeriesCirc{j} {\le_\mu}) \big).
\]
Using similar equalities for the other compositions, we prove the operad axioms.
The morphism property is just \cref{lem:morphismSeriesPermPos}.
The surjectivity follows from $\LinExt({\le_\sigma}) = \sigma$.
\end{proof}

\begin{proposition}
\label{prop:fullyRootedCuttableBasisMessy}
The operad $\messyCitelangisSeries[k]$ is the suboperad of $\messyZinbielSeries[k]$ generated by the four elements $1122$, $1212+1221$, $2112+2121$ and $2211$.
\end{proposition}

\begin{proof}
The generators are given by the linear extensions of the $k$-rooted $k$-posets of \cref{fig:exmMessyCitelangisSeriesActionPosets}:
\[
\begin{array}{l@{\;=\;}l@{\quad}c@{\quad}l@{\;=\;}l}
\LinExt({\le_I} \op{l,l} {\le_I}) & 1122,
&& \LinExt({\le_I} \op{l,r} {\le_I}) & 1212+1221, \\
\LinExt({\le_I} \op{r,l} {\le_I}) & 2112+2121
& \text{and}
& \LinExt({\le_I} \op{r,r} {\le_I}) & 2211.
\end{array}
\qedhere
\]
\end{proof}

\paraul{Tidy series $k$-Zinbiel operad}
The goal of this section is to generalize to any $k$-permutations the composition formula of~\cref{prop:combinatorialModelCompositionsTidySeriesCitelangis} for fully $k$-rooted cuttable $k$-permutations, as suggested by \cref{rem:tidySeriesZinbiel}.

\begin{definition}
\label{def:compositionTidySeriesZinbiel}
Let~$\sigma \in \Perm_k(m)$ and~$\tau \in \Perm_k(n)$ be two $k$-permutations and~$i \in [m]$.
Write~$\sigma = \omega_0 \, i \, \omega_1 \, i \, \omega_2 \cdots \omega_{k-1} \, i \, \mu \, \nu$ where $\mu$ is the maximal factor such that~$\mu_p < i$ for all~$p \in [|\mu|]$.
Therefore either $\nu$ is empty or $i < \nu_1$.
Write moreover~$\tau = \tau_1 \, \cdots \, \tau_k \, \theta$.
Then the \defn{tidy series $i$-th composition} of $\sigma$ and $\mu$ is defined by
\[
\begin{array}{c@{\;}c@{\;}l@{\,}l@{\,}l@{\,}l@{\,}l@{\,}l@{\,}l@{\,}l@{\,}l}
\sigma \tidySeriesCirc{i} \tau
& \eqdef & \omega_0[i,n] & \tau_1[i-1] & \omega_1[i,n] & \tau_2[i-1] & \omega_2[i,n] \cdots \omega_{k-1}[i,n] & \tau_k[i-1] & \mu[i,n] & \theta[i-1] & \nu[i,n] \\
& =      & \omega_0[i,n] & \tau_1[i-1] & \omega_1[i,n] & \tau_2[i-1] & \omega_2[i,n] \cdots \omega_{k-1}[i,n] & \tau_k[i-1] & \mu      & \theta[i-1] & \nu[n-1].
\end{array}
\]
\end{definition}

Note that this is the same composition formula as in~\cref{prop:combinatorialModelCompositionsTidySeriesCitelangis}.
Here are some examples of tidy series compositions on $k$-permutations:
\begin{gather*}
  {\blue3}1{\blue232}1|{\blue44} \tidySeriesCirc{1} 31{\red3122} = {\blue5}3{\blue454}1{\red3122}{\blue66}, \\
  {\blue31}2{\blue3}2{\blue1}|{\blue44} \tidySeriesCirc{2} 31{\red3122} = {\blue51}4{\blue5}2{\blue1}{\red4233}{\blue66}, \\
  3{\blue12}3{\blue21}|{\blue44} \tidySeriesCirc{3} 31{\red3122} = 5{\blue12}3{\blue21}{\red5344}{\blue66}, \\
  {\blue312321}44| \tidySeriesCirc{4} 31{\red3122} = {\blue312321}64{\red6455},
\end{gather*}
where we have marked by a vertical bar the separation $\mu|\nu$.

Let us stress out again that, even when~$k = 2$, the tidy parallel composition~$\sigma \tidyParallelCirc{i} \tau$ of \cref{def:compositionTidyParallelZinbiel} and tidy series composition~$\sigma \tidySeriesCirc{i} \tau$ of \cref{def:compositionTidySeriesZinbiel} differ.

Similarly to \cref{lem:LexMinSeries}, there is a relation between the tidy series composition~$\tidySeriesCirc{i}$ of \cref{def:compositionTidySeriesZinbiel} and the messy series composition~$\messySeriesCirc{i}$ of \cref{def:compositionMessySeriesZinbiel}.

\begin{lemma}
\label{lem:morphismSeriesLexMin}
For any two $k$-permutations~$\sigma \in \Perm_k(m)$ and~$\tau \in \Perm_k(n)$ and~$i \in [m]$,
\[
\sigma \tidySeriesCirc{i} \tau = \LexMin(\sigma \messySeriesCirc{i} \tau).
\]
\end{lemma}

As a consequence, the tidy series composition rules of \cref{def:compositionTidySeriesZinbiel} define an operad.

\begin{proposition}
\label{prop:fullyRootedCuttableBasisTidy}
The family $\big( \Perm_k(n) \big)_{n>0}$ endowed with the tidy series composition rules~$\tidySeriesCirc{i}$ of~\cref{def:compositionTidySeriesZinbiel} defines a non-symmetric set operad~$\tidyZinbielSeries[k]$, called the \defn{tidy series $k$-Zinbiel operad}. Moreover~$\tidyCitelangisSeries[k]$ is the suboperad of~$\tidyZinbielSeries[k]$ given by fully $k$-rooted cuttable $k$-permutations.
\end{proposition}


\newpage
\section{Series versus parallel}
\label{sec:seriesParallel}

In this section, we show that the series and parallel circulation rules strongly disagree but are reconcilable.


\subsection{Series and parallel are not isomorphic}
\label{subsec:seriesParallelNotIsomorphic}

In this section, we study witnesses for the non-isomorphism of the four signaletic operads~$\messySignaleticParallel$, $\tidySignaleticParallel$, $\messySignaleticSeries$ and~$\tidySignaleticSeries$ for~$k \ge 2$.
These witnesses are isomorphism invariants that turn out to differ in these four operads.
We have seen in \cref{coro:HilbertSignaletic} that these four operads have the same Hilbert series, so we need finer invariants.
For this, we study two families of binary operations in these operads:
\begin{itemize}
\item \defn{associative} operations such that $\operation[a] \circ_1 \operation[a] = \operation[a] \circ_2 \operation[a]$,
\item \defn{left-bipotent} (resp.~\defn{right-bipotent}) operations such that $\operation[a] \circ_1 \operation[a] = \operation[0]_3$ (resp.~$\operation[a] \circ_2 \operation[a] = \operation[0]_3$), where~$\operation[0]_n$ is the zero operation of arity~$n$.
\end{itemize}


\subsubsection{The $k = 1$ case}

As a warm-up, we start with the $k = 1$ case.
In this case, there is no distinction between the parallel and series situations, and we therefore abbreviate~$\messySignaletic = \messySignaleticParallel[1] = \messySignaleticSeries[1] = \Dias$ and~$\tidySignaletic = \tidySignaleticParallel[c] = \tidySignaleticSeries[c] = \Dupdual_c$ for~$c \in \{\op{l}, \odot \op{r}\}$.
We first express explicitly the two self-compositions of a generic binary operation in these four operads.

\begin{lemma}
\label{lem:selfCompositionSignaletic1}
The self-compositions of a generic operation~$\operation[a] \eqdef \lambda \, {\op{l}} + \mu \, {\op{r}}$ (with~$\lambda, \mu \in \K$) in the $1$-signaletic operads~$\messySignaletic$, $\tidySignaletic[\op{l}]$, $\tidySignaletic[\odot]$ and~$\tidySignaletic[\op{r}]$ extends in terms of destination vectors~as
\begin{center}
\renewcommand{\arraystretch}{1.3}
\begin{tabular}{c|c|c}
& $\operation[a] \circ_1 \operation[a]$ & $\operation[a] \circ_2 \operation[a]$ \\
\hline
in $\messySignaletic$ & $\lambda^2 \destVect{1}{3} + \lambda\mu \destVect{2}{3} + (\lambda\mu + \mu^2) \destVect{3}{3}$ & $(\lambda^2 + \lambda\mu) \destVect{1}{3} + \lambda\mu \destVect{2}{3} + \mu^2 \destVect{3}{3}$ \\
in $\tidySignaletic[\op{l}]$ & $\lambda^2 \destVect{1}{3} + \lambda\mu \destVect{2}{3} + \lambda\mu \destVect{3}{3}$ & $\lambda^2 \destVect{1}{3} + \lambda\mu \destVect{2}{3} + \mu^2 \destVect{3}{3}$ \\
in $\tidySignaletic[\odot]$ & $\lambda^2 \destVect{1}{3} + \lambda\mu \destVect{2}{3} + \mu^2 \destVect{3}{3}$ & $\lambda^2 \destVect{1}{3} + \lambda\mu \destVect{2}{3} + \mu^2 \destVect{3}{3}$ \\
in $\tidySignaletic[\op{r}]$ & $\lambda^2 \destVect{1}{3} + \lambda\mu \destVect{2}{3} + \mu^2 \destVect{3}{3}$ & $\lambda\mu \destVect{1}{3} + \lambda\mu \destVect{2}{3} + \mu^2 \destVect{3}{3}$
\end{tabular}
\end{center}
\end{lemma}

\begin{proof}
We just need to expand the composition and compute the corresponding destination vectors.
For instance, in~$\messySignaletic$ we have
\[
\operation[a] \circ_1 \operation[a] = \lambda^2 \!\! \compoL{l}{l} \!\! + \lambda\mu \!\! \compoL{r}{l} \!\! +\lambda\mu \!\! \compoL{l}{r} \!\! + \mu^2 \!\! \compoL{r}{r} \!\! = \lambda^2 \destVect{1}{3} + \lambda\mu \destVect{2}{3} + (\lambda\mu + \mu^2) \destVect{3}{3}.
\]
The proof is identical for the seven other compositions.
\end{proof}

By coefficients comparisons, we immediately derive from \cref{lem:selfCompositionSignaletic1} the following statements.

\begin{proposition}
\label{prop:associativeSignaletic1}
Up to scalar multiples, the only associative operations are precisely
\begin{itemize}
\item the operations~$\op{l}$ and~$\op{r}$ in the messy $1$-signaletic operad~$\messySignaletic$,
\item the operations~$\op{l}$ and~${\op{l}} + {\op{r}}$ in the $\op{l}$-tidy signaletic operad~$\tidySignaletic[\op{l}]$,
\item any operation in the $\odot$-tidy signaletic operad~$\tidySignaletic[\odot]$ (\ie dual duplicial operad~$\Dupdual$),
\item the operations~$\op{r}$ and~${\op{l}} + {\op{r}}$ in the $\op{r}$-tidy signaletic operad~$\tidySignaletic[\op{r}]$.
\end{itemize}
\end{proposition}

\begin{proposition}
\label{prop:bipotentSignaletic1}
Up to scalar multiples, the only non-trivial left-bipotent or right-bipotent operations in the $1$-signaletic operads~$\messySignaletic$, $\tidySignaletic[\op{l}]$, $\tidySignaletic[\odot]$ and~$\tidySignaletic[\op{r}]$ are the following:
\begin{itemize}
\item the operation~$\op{r}$ is left-bipotent in the $\op{l}$-tidy signaletic operad~$\tidySignaletic[\op{l}]$,
\item the operation~$\op{l}$ is right-bipotent in the $\op{r}$-tidy signaletic operad~$\tidySignaletic[\op{r}]$.
\end{itemize}
\end{proposition}

As the dimensions of the spaces of associative, left-bipotent and right-bipotent operations are operad isomorphism invariants, \cref{prop:associativeSignaletic1,prop:bipotentSignaletic1} enable to distinguish between the messy and tidy $1$-signaletic operads.

\begin{corollary}
\label{coro:non-isomorphic1}
The messy and tidy $1$-signaletic operads~$\messySignaletic$, $\tidySignaletic[\op{l}]$, $\tidySignaletic[\odot]$ and~$\tidySignaletic[\op{r}]$ are not isomorphic.
\end{corollary}

\begin{remark}
By Koszul duality, \cref{coro:non-isomorphic1} implies a similar statement on the $1$-citelangis operads.
In contrast to our analysis on $1$-signaletic operads, we note that the spaces of associative, left-bipotent and right-bipotent would not suffice to distinguish the $1$-citelangis operads.
\end{remark}


\subsubsection{The $k = 2$ case}

To prepare the general case, we now consider the $k = 2$ situation.
Note that since the Manin product is commutative up to operad isomorphism, the tidy parallel $2$-signaletic operads~$\tidySignaleticParallel[cd] = \Dupdual_{c} \whiteManin \Dupdual_{d}$ and~$\tidySignaleticParallel[dc] = \Dupdual_{d} \whiteManin \Dupdual_{c}$ are isomorphic for any~${c, d \in \{\op{l}, \odot, \op{r}\}}$.
We show below that these are the only isomorphisms between the $2$-signaletic operads.
Our argument is again based on the associative, left-bipotent and right-bipotent binary operations in these operads.

\begin{proposition}
\label{prop:associativeSignaletic2}
All associative operations in the $2$-signaletic operads are given in \cref{table:associativeBipotent}.
\end{proposition}

\begin{proof}
Consider a generic associative operation~$\operation[a] \eqdef \lambda \, {\op{l,l}} + \mu \, {\op{l,r}} + \nu \, {\op{r,l}} + \omega \, {\op{r,r}}$ with~$\lambda, \mu, \nu, \omega \in \K$ in the messy parallel $2$-signaletic operad~$\messySignaleticParallel[2]$.
Comparing the coefficients of the destination vectors that appear in the expansions of~$\operation[a] \circ_1 \operation[a]$ and~$\operation[a] \circ_2 \operation[a]$, we obtain the following $9$ equations:
\[
\begin{array}{r@{\;=\;}l@{\qquad}r@{\;=\;}l@{\qquad}r@{\;=\;}l}
\lambda(\lambda+\mu+\nu+\omega) & \lambda\lambda, &
\mu(\lambda+\nu) & \mu\lambda, &
\mu(\mu+\omega) & (\lambda+\mu)\mu, \\
\nu(\lambda+\mu) & \nu\lambda, &
\omega\lambda & \omega\lambda, &
\omega\mu & (\nu+\omega)\mu, \\
\nu(\nu+\omega) & (\lambda+\nu)\nu, &
\omega\nu & (\mu+\omega)\nu, &
\omega\omega & (\lambda+\mu+\nu+\omega)\omega.
\end{array}
\]
Up to scaling, this leads to $6$ solutions for~$(\lambda, \mu, \nu, \omega)$:
\[
(1,0,0,0), \quad
(0,1,0,0), \quad
(0,0,1,0), \quad
(0,0,0,1), \quad
(1,-1,0,1), \qandq
(1,0,-1,1).
\]
The proof is identical for the other cases.
\end{proof}

\begin{proposition}
\label{prop:bipotentSignaletic2}
All left-bipotent and right-bipotent operations in the $2$-signaletic operads are given in \cref{table:associativeBipotent}.
\end{proposition}

\begin{proof}
They are again obtained by comparisons of the coefficients of the destination vectors in the expansions of~$\operation[a] \circ_1 \operation[a]$ and~$\operation[a] \circ_2 \operation[a]$.
\end{proof}

As the dimensions of the spaces of associative, left-bipotent, and right-bipotent operations are operad isomorphism invariants, \cref{prop:associativeSignaletic2,prop:bipotentSignaletic2} enable to distinguish between the parallel and series and the messy and tidy $2$-signaletic operads.

\begin{corollary}
\label{coro:non-isomorphic2}
Among all $2$-signaletic operads of \cref{table:associativeBipotent}, the only isomorphisms are
\[
\tidySignaleticParallel[\op{l,o}] \simeq \tidySignaleticParallel[\op{o,l}],
\qquad
\tidySignaleticParallel[\op{l,r}] \simeq \tidySignaleticParallel[\op{r,l}]
\qqandqq
\tidySignaleticParallel[\op{o,r}] \simeq \tidySignaleticParallel[\op{r,o}].
\]
\end{corollary}

\hvFloat[floatPos=p, capWidth=h, capPos=right, capAngle=90, objectAngle=90, capVPos=c, objectPos=c]{figure}
{
\renewcommand{\arraystretch}{1.2}
\begin{tabular}{c|c|c|c}
& associative operations & left-bipotent operations & right-bipotent operations \\
\hline
\multirow{2}{*}{in $\messySignaleticParallel[2]$} & $\K \, {\op{l,l}}$, \quad $\K \, {\op{l,r}}$, \quad $\K \, {\op{r,l}}$, \quad $\K \, {\op{r,r}}$, & \multirow{2}{*}{$\{\operation[0]_2\}$} & \multirow{2}{*}{$\{\operation[0]_2\}$} \\
& $\K \, ({\op{l,l}} - {\op{l,r}} + {\op{r,r}})$ \quad and \quad $\K \, ({\op{l,l}} - {\op{r,l}} + {\op{r,r}})$ & & \\
\hline
\multirow{2}{*}{in $\tidySignaleticParallel[\op{l,l}]$} & $\K \, {\op{l,l}}$, \quad $\K \, ({\op{l,l}} + {\op{l,r}})$, \quad $\K \, ({\op{l,l}} + {\op{r,l}})$, \quad $\K \, ({\op{l,l}} + {\op{r,r}})$, & \multirow{2}{*}{$\K \, {\op{l,r}} + \K \, {\op{r,r}}$ and $\K \, {\op{r,l}} + \K \, {\op{r,r}}$} & \multirow{2}{*}{$\{\operation[0]_2\}$} \\
& and \quad $\K \, ({\op{l,l}} + {\op{l,r}} + {\op{r,l}} + {\op{r,r}})$ & & \\
\hline
\multirow{2}{*}{in $\tidySignaleticParallel[\op{l,o}]$} & $\K \, {\op{l,l}} + \K \, {\op{l,r}}$, \quad $\K \, {\op{l,l}} + \K \, ({\op{l,r}} + {\op{r,r}})$, & \multirow{2}{*}{$\K \, {\op{r,l}} + \K \, {\op{r,r}}$} & \multirow{2}{*}{$\{\operation[0]_2\}$} \\
& and \quad $\K \, ({\op{l,l}} + {\op{r,l}}) + \K \, ({\op{l,r}} + {\op{r,r}})$ & & \\
\hline
\multirow{2}{*}{in $\tidySignaleticParallel[\op{l,r}]$} & $\K \, {\op{l,r}}$, \quad $\K \, ({\op{l,r}} + {\op{r,r}})$, \quad $\K \, ({\op{l,l}} + {\op{l,r}})$, & \multirow{2}{*}{$\K \, {\op{r,l}} + \K \, {\op{r,r}}$} & \multirow{2}{*}{$\K \, {\op{l,l}} + \K \, {\op{r,l}}$} \\
& and \quad $\K \, {\op{r,l}} + \K \, ({\op{l,l}} + {\op{l,r}} + {\op{r,r}})$ & & \\
\hline
\multirow{2}{*}{in $\tidySignaleticParallel[\op{o,l}]$} & $\K \, {\op{l,l}} + \K \, {\op{r,l}}$, \quad $\K \, {\op{l,l}} + \K \, ({\op{r,l}} + {\op{r,r}})$, & \multirow{2}{*}{$\K \, {\op{l,r}} + \K \, {\op{r,r}}$} & \multirow{2}{*}{$\{\operation[0]_2\}$} \\
& and \quad $\K \, ({\op{l,l}} + {\op{l,r}}) + \K \, ({\op{r,l}} + {\op{r,r}})$ & & \\
\hline
in $\tidySignaleticParallel[\op{o,o}]$ & $\K \, {\op{l,l}} + \K \, {\op{l,r}} + \K \, {\op{r,l}} + \K \, {\op{r,r}}$ & $\{\operation[0]_2\}$ & $\{\operation[0]_2\}$ \\
\hline
\multirow{2}{*}{in $\tidySignaleticParallel[\op{o,r}]$} & $\K \, {\op{l,r}} + \K \, {\op{r,r}}$, \quad $\K \, ({\op{l,l}} + {\op{l,r}}) + \K \, {\op{r,r}}$, & \multirow{2}{*}{$\{\operation[0]_2\}$} & \multirow{2}{*}{$\K \, {\op{l,l}} + \K \, {\op{r,l}}$} \\
& and \quad $\K \, ({\op{l,l}} + {\op{l,r}}) + \K \, ({\op{r,l}} + {\op{r,r}})$ & & \\
\hline
\multirow{2}{*}{in $\tidySignaleticParallel[\op{r,l}]$} & $\K \, {\op{r,l}}$, \quad $\K \, ({\op{l,l}} + {\op{r,l}})$, \quad $\K \, ({\op{r,l}} + {\op{r,r}})$, & \multirow{2}{*}{$\K \, {\op{l,r}} + \K \, {\op{r,r}}$} & \multirow{2}{*}{$\K \, {\op{l,l}} + \K \, {\op{l,r}}$} \\
& and \quad $\K \, ({\op{l,l}} + {\op{r,l}} + {\op{r,r}}) + \K \, {\op{l,r}}$ & & \\
\hline
\multirow{2}{*}{in $\tidySignaleticParallel[\op{r,o}]$} & $\K \, {\op{r,l}} + \K \, {\op{r,r}}$, \quad $\K \, ({\op{l,l}} + {\op{r,l}}) + \K \, {\op{r,r}}$, & \multirow{2}{*}{$\{\operation[0]_2\}$} & \multirow{2}{*}{$\K \, {\op{l,l}} + \K \, {\op{l,r}}$} \\
& and \quad $\K \, ({\op{l,l}} + {\op{r,l}}) + \K \, ({\op{l,r}} + {\op{r,r}})$ & & \\
\hline
\multirow{2}{*}{in $\tidySignaleticParallel[\op{r,r}]$} & $\K \, {\op{r,r}}$, \quad $\K \, ({\op{l,l}} + {\op{r,r}})$, \quad $\K \, ({\op{l,r}} + {\op{r,r}})$, \quad $\K \, ({\op{r,l}} + {\op{r,r}})$ & \multirow{2}{*}{$\{\operation[0]_2\}$} & \multirow{2}{*}{$\K \, {\op{l,l}} + \K \, {\op{l,r}}$ and $\K \, {\op{l,l}} + \K \, {\op{r,l}}$} \\
& and \quad $\K \, ({\op{l,l}} + {\op{l,r}} + {\op{r,l}} + {\op{r,r}})$ & & \\
\hline
in $\messySignaleticSeries[2]$ & $\K \, {\op{l,l}}$ \quad and \quad $\K \, {\op{r,r}}$ & $\K \, ({\op{r,l}} - {\op{r,r}})$ & $\K \, ({\op{l,l}} - {\op{l,r}})$ \\
\hline
\multirow{2}{*}{in $\tidySignaleticSeries[2]$} & $\K \, {\op{l,l}}$, \quad $\K \, ({\op{l,l}} + {\op{r,r}})$, & \multirow{2}{*}{$\K \, {\op{l,r}} + \K \, {\op{r,r}}$} & \multirow{2}{*}{$\K \, {\op{l,r}}$} \\
& and \quad $\K \, ({\op{l,l}}  + {\op{r,l}} + {\op{r,r}}) + \K \, {\op{l,r}}$ & &
\end{tabular}
}
{The spaces of associative, left-bipotent and right-bipotent operations in all $2$-signaletic operads.}
{table:associativeBipotent}

\begin{remark}
By Koszul duality, \cref{coro:non-isomorphic2} implies a similar statement on the $2$-citelangis operads.
In contrast to our analysis on $2$-signaletic operads, we note that the spaces of associative, left-bipotent and right-bipotent would not suffice to distinguish the $2$-citelangis operads.
\end{remark}


\subsubsection{The $k \ge 3$ case}

As $k$ grows, the description of the associative, left-bipotent and right-bipotent operations becomes more and more complicated.
We now collect some relevant observations about these operations.

\paraul{Associative operations in parallel}
For the parallel situation, we can make use of the insertion operations~$\Ins^{\operation[a]}_i$ of \cref{rem:extensionMessySignaleticParallel,rem:extensionTidySignaleticParallel} to create associative and bipotent operations in arity~$k$ from those in arity~$\ell$ for~$\ell < k$.
The following two statements are immediate from the fact that these insertion operations are operad morphisms.

\begin{proposition}
\label{prop:associativeMessySignaleticParallel}
For any~$k \ge 1$ and~$i \in [k+1]$, if~$\operation[a]$ is associative in~$\messySignaleticParallel$, then~$\Ins^{\op{l}}_i(\operation[a])$ and~$\Ins^{\op{r}}_i(\operation[a])$ are associative in~$\messySignaleticParallel[k+1]$.
Therefore, any operation of~$\{\op{l}, \op{r}\}^k$ is associative in~$\messySignaleticParallel$.
\end{proposition}

\begin{proposition}
\label{prop:associativeTidySignaleticParallel}
For any~$\constraint \in \{\op{l}, \odot, \op{r}\}^k, \constraint[d] \in \{\op{l}, \odot, \op{r}\}^\ell$ and~$\operation[a] \in \{\op{l}, \op{r}\}^k, \operation[b] \in \{\op{l}, \op{r}\}^\ell$, if the operation~$\operation[a]\operation[b]$ is associative in~$\tidySignaleticParallel[{\constraint[c] \constraint[d]}]$, then~$\operation[a] {\op{l}} \operation[b]$ and~$\operation[a] {\op{l}} \operation[b] + \operation[a] {\op{r}} \operation[b]$ (resp.~$\operation[a] {\op{l}} \operation[b]$ and~$\operation[a] {\op{r}} \operation[b]$, resp.~$\operation[a] {\op{r}} \operation[b]$ and~$\operation[a] {\op{l}} \operation[b] + \operation[a] {\op{r}} \operation[b]$) are associative in~$\tidySignaleticParallel[{\constraint[c] \op{l} \constraint[d]}]$ (resp.~in~$\tidySignaleticParallel[{\constraint[c] \odot \constraint[d]}]$, resp.~in~$\tidySignaleticParallel[{\constraint[c] \op{r} \constraint[d]}]$).
\end{proposition}

It would be interesting to characterize which operations in~$\messySignaleticParallel$ and~$\tidySignaleticParallel[\constraint]$ are non-trivially associative, in the sense that they do not arise as consequences of \cref{prop:associativeMessySignaleticParallel,prop:associativeTidySignaleticParallel}.
For instance, when~$k = 2$, the only non-trivially associative operations in parallel $2$-signaletic operads are up to symmetries:
\begin{itemize}
\item the operations~${\op{l,l}} - {\op{l,r}} + {\op{r,r}}$ and~${\op{l,l}} - {\op{r,l}} + {\op{r,r}}$ in~$\messySignaleticParallel[2]$,
\item the operation~${\op{l,l}} + {\op{r,r}}$ in~$\tidySignaleticParallel[\op{l,l}]$,
\item all operations of~$\K \, {\op{r,l}} + \K \, ({\op{l,l}} + {\op{l,r}} + {\op{r,r}})$ in~$\tidySignaleticParallel[\op{l,r}]$.
\end{itemize}

\paraul{Associative operations in series}
For the series situation, we have no simple way to construct associative operations.
We can however make the following observations, whose proofs are immediate on destination vectors.

\begin{proposition}
\label{prop:associativeMessySignaleticSeries}
For any~$k \ge 1$, the operations~$\op{l}\!^k$ and~$\op{r}\!^k$ are associative in~$\messySignaleticSeries$.
All other operations of~$\{\op{l}, \op{r}\}^k$ are not associative.
\end{proposition}

\begin{proposition}
\label{prop:associativeTidySignaleticSeries}
For any~$k \ge 1$, the operations~$\op{l}\!^k$ and~${\op{l}\!^k} + {\op{r}\!^k}$ are associative in~$\tidySignaleticSeries$.
\end{proposition}

Note that there are other associative operations in the series $k$-signaletic operads.
For instance:
\begin{itemize}
\item the associative operations in the tidy series $3$-signaletic operad~$\tidySignaleticSeries[3]$ are given by:
\[
\qquad
\K \, {\op{l,l,l}},
\qquad
\K \, {\op{r,r,r}},
\qqandqq
\K \, ({\op{r,l,l}} - {\op{r,l,r}}) + \K \, ({\op{l,r,l}} - {\op{l,r,r}}).
\]
In fact, the compositions of any two operations in~$\K \, ({\op{r,l,l}} - {\op{r,l,r}}) + \K \, ({\op{l,r,l}} - {\op{l,r,r}})$ vanish.
\item the associative operations in the tidy series $3$-signaletic operad~$\tidySignaleticSeries[3]$ are given by:
\begin{gather*}
\qquad
\K \, {\op{l,l,l}},
\qquad
\K \, ({\op{l,l,l}} + {\op{r,r,r}}),
\qquad
\K \, ({\op{l,l,l}} + {\op{r,l,l}} + {\op{r,r,l}} + {\op{r,r,r}}),
\\
\qquad
\K \, {\op{l,l,r}} + \K \, {\op{l,r,r}} + \K \, {\op{r,l,r}},
\qquad
\K \, ({\op{l,l,l}} + {\op{r,r,l}}) + \K \, {\op{l,r,r}} + \K \, {\op{r,l,r}},
\qquad\text{and}
\\
\qquad
\bigset{\lambda \, ({\op{l,l,l}} + {\op{r,l,l}} + {\op{r,r,l}} + {\op{r,r,r}}) + \mu \, ({\op{l,l,r}} + {\op{l,r,r}}) \pm \sqrt{\lambda\mu} \, ({\op{l,r,l}} + {\op{r,l,r}})}{\lambda, \mu \in \K}.
\end{gather*}
Note that the latter set is not a linear subspace of~$\tidySignaleticSeries[3](2)$.
\end{itemize}

\paraul{Messy parallel and messy series are not isomorphic}
We finally prove that the messy parallel and series $k$-signaletic operads are not isomorphic, studying their left-bipotent and right-bipotent operations.
On the one hand, we observe that the messy parallel $k$-signaletic operad have no such operations.

\begin{proposition}
\label{prop:bipotentMessySignaleticParallel}
The messy parallel $k$-signaletic operad~$\messySignaleticParallel$ has no non-trivial left-bipotent or right-bipotent operation.
\end{proposition}

\begin{proof}
The proof works by induction on~$k$ and heavily relies on \cref{rem:restrictionMessySignaleticSeries}.
The result was already stated in \cref{prop:bipotentSignaletic1} when~$k = 1$ and in \cref{prop:bipotentSignaletic2,table:associativeBipotent} when~${k = 2}$.
Assume now that we have proved the statement for the parallel~$(k-1)$-signaletic operad.
Consider a generic left-bipotent binary operation~$\operation[a] \eqdef \sum_{\operation[b] \in \{\op{l}, \op{r}\}^k} \lambda_{\operation[b]} \, \operation[b]$ of the messy parallel $k$-signaletic operad.
We observe that
\begin{itemize}
\item $\lambda_{{\op{l}}^k} = 0$ since the coefficient of~$\destVect{1^{\{k\}}}{3}$ in~${\operation[a] \circ_1 \operation[a]}$ is~$(\lambda_{{\op{l}}^k})^2$.
\item for any~$i \in [k]$, applying the restriction morphism to the positions~$[k] \ssm \{i\}$ as defined in \cref{rem:restrictionMessySignaleticSeries}, we obtain an operation
\[
\operation[a]' \eqdef \Res_{[k] \ssm \{i\}}(\operation[a]) = \sum_{\substack{\operation[c] \in \{\op{l}, \op{r}\}^{i-1} \\ \operation[d] \in \{\op{l}, \op{r}\}^{k-i}}} \big( \lambda_{\operation[c]{\op{l}}\operation[d]} + \lambda_{\operation[c]{\op{r}}\operation[d]} \big) \operation[c]\operation[d]
\]
of the parallel $(k-1)$-signaletic operad satisfying~$\operation[a]' \circ_1 \operation[a]' = \operation[0]_3$. This shows that any two operations of~$\{\op{l}, \op{r}\}^k$ which only differ at position~$i \in [k]$ have the opposite coefficients in~$\operation[a]$.
\end{itemize}
We conclude that all coefficients of~$\operation[a]$ vanish, so that the only left-bipotent operation in~$\messySignaleticParallel$ is the zero operation~$\operation[0]_2$. The proof is symmetric for right-bipotent operations.
\end{proof}

In contrast, we observe that the messy parallel $k$-signaletic operad has (many) non-trivial left-bipotent and right-bipotent operations.
This requires the following observation, whose immediate proof is left to the reader.

\begin{lemma}
\label{lem:bipotentMessySignaleticSeries}
For any two binary operations~$\operation[a]$ and~$\operation[b]$ of the $k$-signaletic operad and any~$\ell,r \ge 0$, 
\begin{itemize}
\item the destination vector of~$\big( {\op{l}\!^\ell} \, \operation[a] \, {\op{r}\!^r} \big) \circ_1 \big( {\op{l}\!^\ell} \, \operation[b] \, {\op{r}\!^r} \big)$ is obtained from the destination vector of~$\operation[a] \circ_1 \operation[b]$ by adding the prefix~$1^{\{\ell\}}$ and the suffix~$3^{\{r\}}$,
\item the destination vector of~$\big( {\op{r}\!^\ell} \, \operation[a] \, {\op{l}\!^r} \big) \circ_2 \big(  {\op{r}\!^\ell} \, \operation[b] \, {\op{l}\!^r} \big)$ is obtained from the destination vector of~$\operation[a] \circ_2 \operation[b]$ by adding the prefix~$3^{\{\ell\}}$ and the suffix~$1^{\{r\}}$.
\end{itemize}
\end{lemma}

\begin{proposition}
\label{prop:bipotentMessySignaleticSeries}
For any~$\ell, r \in \N$, the binary operation~${\op{l}}^\ell{\op{r,l}}{\op{r}}^r - {\op{l}}^\ell{\op{r,r}}{\op{r}}^r$ is left-bipotent and the operation~${\op{r}}^\ell{\op{l,r}}{\op{l}}^r - {\op{r}}^\ell{\op{l,l}}{\op{l}}^r$ is right-bpotent in the messy series $k$-signaletic operad~$\messySignaleticSeries$.
\end{proposition}

\begin{proof}
This was already observed in \cref{prop:bipotentSignaletic2,table:associativeBipotent} when~$p = q = 0$.
The general case follows by \cref{lem:bipotentMessySignaleticSeries}.
\end{proof}

From~\cref{prop:bipotentMessySignaleticParallel,prop:bipotentMessySignaleticSeries}, we managed to distinguish the series and parallel messy signaletic operads.

\begin{corollary}
\label{coro:non-isomorphic}
For any~$k \ge 1$, the messy parallel and series $k$-signaletic operads~$\messySignaleticParallel$ and~$\messySignaleticSeries$ are not isomorphic.
\end{corollary}

Based on computer experiments, we conjecture that there are in fact no non-trivial isomorphisms between all our $k$-signaletic operads.

\begin{conjecture}
\label{conj:non-isomorphic}
For any~$k \ge 1$, the only isomorphisms among all $k$-signaletic operads are of the form
\[
\tidySignaleticParallel[\constraint] \simeq \tidySignaleticParallel[\sigma(\constraint)]
\]
for some~$\constraint \in \{\op{l}, \odot, \op{r}\}^k$ and~$\sigma \in \fS_k$.
\end{conjecture}

\begin{remark}
By Koszul duality, we have shown that the messy parallel and series $k$-citelangis operads~$\messyCitelangisParallel$ and~$\messyCitelangisSeries$ are not isomorphic.
\end{remark}


\subsection{Parallel-series and series-parallel operads}
\label{subsec:seriesParallelOperads}

In this section, we briefly indicate how to reconcile the series and parallel circulation rules.

\subsubsection{Parallel-series operads}

First of all, it is not difficult to define parallel-series operads.
The input is a sequence of integers $\b{k} = (k_1, \dots, k_\ell)$ whose sum is denoted by~$|\b{k}|$.
We then consider syntax trees on $\Operations_{|\b{k}|} \simeq \Operations_{k_1} \otimes \dots \otimes \Operations_{k_\ell}$, where each node is a traffic signal decomposed into $\ell$ blocks of sizes~$k_1, \dots, k_\ell$.
As usual, $|\b{k}|$ cars drive through the syntax tree, starting at its root, and following the traffic signals at each node.
But the cars are colored with~$\ell$ colors: there are $k_i$ cars colored~$i$ for all~$i \in [\ell]$.
Cars of the same color follow the series traffic rule on their block: the $k_i$ cars of color~$i$ start one after the others and follow and erase the first remaining traffic signal of the $i$-th block in each node, never crossing barriers between blocks.
This defines two \defn{$\b{k}$-signaletic parallel-series operads}~$\messySignaleticParallelSeries$ and~$\tidySignaleticParallelSeries$ according to the ordering rules of the traffic signals not traversed by a car.
The complete \cref{subsec:HilbertSeriesKoszulitySignaletic} on Hilbert series and Kozulity applies with the corresponding definition for right $\b{k}$-signaletic combs.

\begin{theorem}
The $\b{k}$-signaletic parallel-series operads~$\messySignaleticParallelSeries$ and~$\tidySignaleticParallelSeries$ are quadratic and Koszul with Hilbert series $\sum_{p \ge 1} p^{|\b{k}|} t^p$.
\end{theorem}

As a consequence, we can consider their dual \defn{$\b{k}$-citelangis parallel-series operads}~$\messyCitelangisParallelSeries$ and~$\tidyCitelangisParallelSeries$ which share the same Hilbert series with $\citelangis[|\b{k}|]$.
The definition clearly translates into the following description of these operads in terms of Manin products.

\begin{proposition}
The $\b{k}$-signaletic and $\b{k}$-citelangis parallel-series operads are Manin products of series signaletic and citelangis operads:
\[
\messySignaleticParallelSeries = \underset{i \in [\ell]}{\scalebox{2}{$\square$}} \messySignaleticSeries[k_i],
\quad
\tidySignaleticParallelSeries = \underset{i \in [\ell]}{\scalebox{2}{$\square$}} \tidySignaleticSeries[k_i],
\quad
\messyCitelangisParallelSeries = \underset{i \in [\ell]}{\scalebox{2}{$\blacksquare$}} \messyCitelangisSeries[k_i]
\qandq
\tidyCitelangisParallelSeries = \underset{i \in [\ell]}{\scalebox{2}{$\blacksquare$}} \tidyCitelangisSeries[k_i].
\]
\end{proposition}

The particular case where $\b{k} = (k,\ell)$ is of length $2$ has interesting combinatorial properties. It acts on $k+\ell$ multipermutations where the first $k$ signals chose the first $k$ letters and the last $\ell$ signals chose the last $\ell$ letters of the shuffles.
Most of all the results of \cref{sec:actions} can be adapted in this setting.

\begin{theorem}
The algebra~$\FQSym_{k+\ell}$ can be endowed with both messy and tidy parallel-series $(k,\ell)$-citelangis algebra structures which are both free.
\end{theorem}

There is also a notion of $(k,\ell)$-rooted bounded posets and the messy and tidy $(k,\ell)$-Zinbiel operads, getting the following diagram:

\medskip
\begin{center}
\begin{tikzcd}[column sep=2cm, row sep=.3cm]
	\phantom{|}\boundedPosParallel[k,\ell]\phantom{|} \arrow[r, "\LinExt",  two heads]
	& \phantom{|}\messyZinbielParallel[k,\ell]\phantom{|} \arrow[r, dashrightarrow, "\LexMin", hook, two heads]
	& \phantom{|}\tidyZinbielParallel[k,\ell]\phantom{|} \\
	\phantom{|}\posParallel[k,\ell]\phantom{|} \arrow[r, "\LinExt", two heads] \arrow[u, hook]
	& \phantom{|}\messyCitelangisParallel[k,\ell]\phantom{|} \arrow[r, dashrightarrow, "\LexMin", hook, two heads] \arrow[u, hook]
	& \phantom{|}\tidyCitelangisParallel[k,\ell]\phantom{|} \arrow[u, hook]
\end{tikzcd}
\end{center}

\medskip
Details are left to the reader.

\subsubsection{Series-parallel operads}

Under appropriate conditions, we can also adapt the series setting to work with blocks of parallel operads.
In a first version, the input is a weakly decreasing sequence of integers $\b{k} = (k_1 \ge \dots \ge k_\ell)$ whose sum is denoted by~$|\b{k}|$.
We consider again syntax trees on $\Operations_{|\b{k}|} \simeq \Operations_{k_1} \otimes \dots \otimes \Operations_{k_\ell}$, where each node is a traffic signal decomposed into $\ell$ blocks of sizes~$k_1, \dots, k_\ell$.
As before, $|\b{k}|$ cars drive through the syntax tree, starting at its root, and following the traffic signals at each node, but the traffic rule is now as follows.
The first $k_1$ cars start in parallel at the root and follow the signals in parallel in the first block.
After this first round, we erase the entire block of~$k_1$ letters at each node that has been traversed by at least one car.
The next $k_2$ cars start again in parallel from the root and follow the signals in parallel in the first remaining block.
Note that in the nodes not traversed in the first round, the cars of the second round are reading traffic signal from the first block.
This causes no problem since the size of the first block is at least~$k_2$.
Again, we erase the entire block (now of size $k_1$ or $k_2$) at each node that has been traversed by at least one car.
We repeat until the $\ell$-th round.
This defines two \defn{$\b{k}$-signaletic series-parallel operads}~$\messySignaleticSeriesParallel$ and~$\tidySignaleticSeriesParallel$.
Again, the complete \cref{subsec:HilbertSeriesKoszulitySignaletic} on Hilbert series and Koszulity applies with the corresponding definition for right $\b{k}$-signaletic combs.

\begin{theorem}
The $\b{k}$-signaletic series-parallel operads~$\messySignaleticSeries[\b{k}]$ and~$\tidySignaleticSeries[\b{k}]$ are quadratic and Koszul with hilbert series $\sum_{p \ge 1} p^{|\b{k}|} t^p$.
\end{theorem}

As a consequence, we can consider their dual \defn{$\b{k}$-citelangis series-parallel operads}~$\messyCitelangisSeriesParallel$ and~$\tidyCitelangisSeriesParallel$ which share the same Hilbert series with $\citelangis[|\b{k}|]$.
However it is not clear how these $\b{k}$-citelangis series-parallel operads can act on multipermutations.
Finally, thanks to~\cref{rem:extensionMessySignaleticParallel}, we can relax the decreasing condition on~$\b{k}$, by adding some~$\op{l}$ or~$\op{r}$ (not changing the operad) to comply with the decreasing assumption.


\clearpage
\section{Open questions}
\label{sec:openQuestions}

This works raises a lot of open questions. We present some of them below.


\subsection{A general notion of series Manin products?}

All the parallel constructions presented here are $k$-th powers of certain operads for Manin products.
The similarities between the parallel and series situations suggest that the series constructions could be $k$-th powers of the same operads for a general twisted Manin product defined for any non-symmetric operads or only in some specific cases.
It would be interesting to have a general construction or at least some more examples with a similar behavior.


\subsection{Symmetric $k$-Zinbiel operads?}

In this paper, we did not provide a thorough analysis of the $k$-Zinbiel operads and in particular the messy ones.
The main reason is that we think that they should be investigated as symmetric operads.
Indeed, in the $k = 1$ case, the non-symmetric suboperad generated in degree~$2$ within the Zinbiel operad is the dendriform operad.
One has to use the symmetric groups action to generate all permutations.
For larger values of $k$, the series $k$-citelangis operad is generated by a strict subspace of degree~$2$.
For example for~$k = 2$, there are four generators but six $2$-permutations.
So it may be interesting to investigate the suboperad generated by the full component of degree~$2$.
However, we think that an in depth study of the series $k$-Zinbiel operad as a symmetric operad might be much more interesting.
Is it Koszul? What is its dual (some kind of $k$-Leibniz operad)?


\subsection{\texorpdfstring{$k$}{k}-tridendriform operads?}

A straightforward extension of the signaletic interpretation of the diassociative operad is obtained by adding a splitting signal~$\bot$.
When a car reads a signal~$\bot$, it splits into two cars, one following the left branch and the other one following the right branch.
In this setting, for a tree of arity~$n$, each destination is replaced by a non-empty subset of~$[n]$ (because at least one car arrives at the leaves of the tree).
This operad was already considered by J.-L.~Loday and M.~Ronco under the name of triassociative operad~\cite{LodayRonco-tridendriform}.
It is quadratic, Koszul and its dual is knows as the tridendriform operad.
Typical tridendriform algebras include quasi-shuffle algebras and in particular the algebra of quasi-symmetric functions.
Its usual Manin square was considered by P.~Leroux under the name of ennea operad~\cite{Leroux-ennea}.

It is very likely that both parallel and series stories generalize in this setting.
Moreover, J.-L.~Loday and M.~Ronco considered the triduplicial operad which after some adaptation should provide the base case for the tidy setting.
With respect to action and combinatorial realization, the analogue of $\FQSym$ and the Zinbiel operad should be provided by a suitable $k$-version of the algebra of word quasi-symmetric function $\WQSym$~\cite{DuchampHivertThibon-NCSF6, MenousNovelliThibon, Mlodecki-these} where permutations are replaced by packed words or equivalently ordered set partitions.

Finally, a similar extension would be to add a stoping signal~$\ominus$.
A car that reads a signal~$\ominus$ stops, so that the destination set of each car can now be empty.
It would be worth to study the parallel and series extensions of this rule as well.


\subsection{\texorpdfstring{$k$}{k}-duplicial operads?}
\label{subsec:dupl}

As discussed in \cref{subsubsec:tidyParallelSignaleticOperads,subsubsec:tidyParallelCitelangisOperads}, there are natural ways to extend any of the three duplicial operads~$\Dup$, $\Dup_{\op{l}}$ and~$\Dup_{\op{r}}$ and their duals in the parallel situation.
In contrast, in \cref{subsubsec:tidySeriesSignaleticOperads,subsubsec:tidySeriesCitelangisOperads}, we could only extend the twisted duplicial operad~$\Dup_{\op{l}}$ and its dual in the series situation.

A natural question is thus to look for a generalization of the classical duplicial operad~$\Dup$ and its dual in series.
The idea would be to force all signals not traversed by the series route to point towards these series route.
The difficulty is that given a signal not traversed by the series routes, there is no natural way to decide towards which series route it should point.
Natural choices include the following:
\begin{itemize}
\item in each node, the $i$-th letter points towards the $i$-th route,
\item in each node~$v$, the $i$-th free letter points towards the $i$-th route not passing through~$v$,
\item in each node, each letter goes up in the tree until it reaches a route, and then points towards this route,
\item all signals that are on the left (resp.~right) of all routes to point to the right (resp.~left).
\end{itemize}
However, none of these rules lead to an operad.
An appropriate rule still remains to be found.


\subsection{Interpolating with \texorpdfstring{$q$}{q}-analogues between messy and tidy?}

Recall that there exists a $q$-deformation of the shuffle product defined inductively by
\[
X \shuffle_q \varepsilon = X,
\qquad
\varepsilon \shuffle_q Y = Y,
\qqandqq
xX \shuffle_q yY \eqdef x(X \shuffle_q yY) + q^{|xX|}\,y(xX \shuffle Y),
\]
for any two words~$X,Y$ and any two letters~$x,y$.
This quantum shuffle, which is the simplest case of M.~Rosso's construction~\cite{Rosso-quantumGroups}, is associative and interpolates between the concatenation when~$q = 0$ and the shuffle when~$q = 1$.
When~$q$ is not a root of the unity, the $q$-shuffle algebra is isomorphic to the concatenation algebra.
One can similarly define a shifted $q$-shuffle of two permutations.
This gives a quantum deformation~$\FQSym_q$ of the algebra of free quasi-symmetric functions~\cite{ThibonUng, HivertNovelliThibon-commutativeCombinatorialHopfAlgebras}.

The tidy rules return~$0$ or a single permutation defined using a concatenation, whereas the messy rules involve a shuffle.
The comparison with the $q$-shuffle suggests that it may be possible to interpolate between the tidy and the messy rules for the three families of operads: signaletic, citelangis and Zinbiel.


\newpage
\bibliographystyle{alpha}
\bibliography{signaleticOperads}
\label{sec:biblio}

\end{document}